\documentclass[
11pt,  reqno]{amsart}
\usepackage[margin=1.2in,marginparwidth=1.5cm, marginparsep=0.5cm]{geometry}
\pdfoutput=1

\usepackage[implicit=true]{hyperref}

\usepackage{todonotes}

\usepackage{booktabs} 
\usepackage{microtype}
\usepackage{amssymb}
\usepackage{mathrsfs}
\usepackage[most]{tcolorbox}
\usepackage{soul}

\usepackage{footmisc}

\usepackage{upgreek}
\usepackage{dsfont}

\usepackage{color}

\usepackage{xsavebox}

\usepackage{cases}

\allowdisplaybreaks[1]

\definecolor{gr}{rgb}   {0.,   0.69,   0.23 }
\definecolor{bl}{rgb}   {0.,   0.5,   1. }
\definecolor{mg}{rgb}   {0.85,  0.,    0.85}
\definecolor{yl}{rgb}   {0.8,  0.7,   0.}
\definecolor{or}{rgb}  {0.7,0.2,0.2}

\newcommand{\Sym}{\textup{\texttt{Sym}}}

\usepackage{tikz}
\usepackage{marginnote}
\usepackage{scalerel} 

\usetikzlibrary{shapes.misc}
\usetikzlibrary{shapes.symbols}
\usetikzlibrary{decorations}
\usetikzlibrary{decorations.markings}

\usepackage{tikz}

\usetikzlibrary{shapes.misc}
\usetikzlibrary{shapes.symbols}
\usetikzlibrary{shapes.geometric}

\tikzset{
	dot/.style={circle,fill=black,draw=black,inner sep=0pt,minimum size=1mm},
	>=stealth,
	}

\tikzset{
	ddot/.style={circle,fill=white,draw=black,inner sep=0pt,minimum size=1mm},
	>=stealth,
	}

\tikzset{decision/.style={ 
        draw,
        diamond,
        aspect=1.5
    }}

\tikzset{decision/.style={ 
        draw,
        square,
        aspect=1.5
    }}

   \tikzset{square/.style
={regular polygon, regular polygon sides=4,  fill=white,draw=black,inner sep=0pt,minimum size=1.5mm},
	>=stealth,
	}

\tikzset{dia2/.style
={diamond,fill=white,draw=black,inner sep=0pt,minimum size=1mm},
	>=stealth,
	}

\tikzset{dia/.style
={star,fill=black,draw=black,inner sep=0pt,minimum size=1.5mm},
	>=stealth,
	}

\makeatletter
\def\DeclareSymbol#1#2#3{\expandafter\gdef\csname MH@symb@#1\endcsname{\tikz[baseline=#2,scale=0.25]{#3}}}
\def\<#1>{\csname MH@symb@#1\endcsname}
\makeatother

\DeclareSymbol{X}{-2.4}{\node[dot] {};}
\DeclareSymbol{1}{0}{\draw[white] (-.4,0) -- (.4,0); \draw (0,0)  -- (0,1.2) node[dot] {};}
\DeclareSymbol{2}{0}{\draw (-0.5,1.2) node[dot] {} -- (0,0) -- (0.5,1.2) node[dot] {};}

\DeclareSymbol{4'}{0}{\draw (0,0) node[dia]{} -- (0,1.2) node[square] {};
  }

\DeclareSymbol{4}{0}{\draw (0,0) node[dot]{} -- (0,1.2) node[square] {};
  }

\DeclareSymbol{1}{-2.7}
 {
  \draw (0,0) node[dot]{};}

\DeclareSymbol{1'}{-2.7}
 {\draw (0,0) node[ddot]{};}

\DeclareSymbol{1''}{-2.7}
 {\draw (0,0) node[dia]{};}

 \DeclareSymbol{4''}{-2.7}
 {\draw (0,0) node[square]{};}

\DeclareSymbol{3}{0}
 {\draw (0,0) node[dot]{} -- (0,1.2) node[ddot] {};
 \draw (-1.2,0) node[dot] {} -- (0,1.2)node[ddot] {} -- (1.2,0) node[dot] {};
 }

\DeclareSymbol{3'}{0}
 {\draw (0,0) node[dia]{} -- (0,1.2) node[ddot] {};
 \draw (-1.2,0) node[dia] {} -- (0,1.2)node[ddot] {} -- (1.2,0) node[dia] {};}

\DeclareSymbol{31}{-3}{
\draw (0,-1)node[dot] {} -- (0,0) node[ddot] {}-- (0.9, -1) node[dot] {};
\draw (0,-1)node[dot] {} -- (0,0) node[ddot] {}-- (-0.9, -1) node[dot] {};
\draw (0,0)node[ddot] {} -- (1.3,1) node[ddot] {}-- (1.3, 0) node[dot] {};
\draw  (1.3,1) node[ddot] {}-- (2.6, 0) node[dot] {};
}

\DeclareSymbol{32}{-3}{
\draw (0,-1)node[dot] {} -- (0,0) node[ddot] {}-- (0.9, -1) node[dot] {};
\draw (0,-1)node[dot] {} -- (0,0) node[ddot] {}-- (-0.9, -1) node[dot] {};
\draw (0,0)node[ddot] {} -- (0,1) node[ddot] {}-- (0.9, 0) node[dot] {};
\draw (0,0)node[ddot] {} -- (0,1) node[ddot] {}-- (-0.9, 0) node[dot] {};
}

\DeclareSymbol{33}{-3}{
\draw (2.6,-1)node[dot] {} -- (2.6,0) node[ddot] {}-- (3.5, -1) node[dot] {};
\draw (2.6,-1)node[dot] {} -- (2.6,0) node[ddot] {}-- (1.7, -1) node[dot] {};
\draw (0,0)node[dot] {} -- (1.3,1) node[ddot] {}-- (1.3, 0) node[dot] {};
\draw  (1.3,1) node[ddot] {}-- (2.6, 0) node[ddot] {};
}

\DeclareSymbol{3''}{0}
 {\draw (0,0) node[dia]{} -- (0,1.2) node[ddot] {};
 \draw (-1.2,0) node[dot] {} -- (0,1.2)node[ddot] {} -- (1.2,0) node[dia] {};}

\DeclareSymbol{3'''}{0}
 {\draw (0,0) node[dot]{} -- (0,1.2) node[ddot] {};
 \draw (-1.2,0) node[dot] {} -- (0,1.2)node[ddot] {} -- (1.2,0) node[dia] {};}

\DeclareSymbol{31'}{-3}{
\draw (0,-1.2)node[dia] {} -- (0,0) node[ddot] {}-- (0.9, -1.2) node[dia] {};
\draw (0,-1.2)node[dia] {} -- (0,0) node[ddot] {}-- (-0.9, -1.2) node[dia] {};
\draw (0,0)node[ddot] {} -- (1.3,1.2) node[ddot] {}-- (1.3, 0) node[dia] {};
\draw  (1.3,1.2) node[ddot] {}-- (2.6, 0) node[dia] {};
}

\DeclareSymbol{32'}{-3}{
\draw (0,-1.2)node[dia] {} -- (0,0) node[ddot] {}-- (0.9, -1.2) node[dia] {};
\draw (0,-1.2)node[dia] {} -- (0,0) node[ddot] {}-- (-0.9, -1.2) node[dia] {};
\draw (0,0)node[ddot] {} -- (0,1.2) node[ddot] {}-- (0.9, 0) node[dia] {};
\draw (0,0)node[ddot] {} -- (0,1.2) node[ddot] {}-- (-0.9, 0) node[dot] {};
}

\DeclareSymbol{33'}{-3}{
\draw (2.6,-1.2)node[dia] {} -- (2.6,0) node[ddot] {}-- (3.5, -1.2) node[dia] {};
\draw (2.6,-1.2)node[dia] {} -- (2.6,0) node[ddot] {}-- (1.7, -1.2) node[dia] {};
\draw (0,0)node[dot] {} -- (1.3,1.2) node[ddot] {}-- (1.3, 0) node[dot] {};
\draw  (1.3,1.2) node[ddot] {}-- (2.6, 0) node[ddot] {};
}

\DeclareSymbol{523}{-3}{
\draw (0,0)node[dot] {} -- (0,1) node[ddot] {};
\draw (0,1)node[ddot] {} -- (1.3,0) node[square] {}-- (1.3, -1) node[dot] {};
\draw (0,1)node[ddot] {} -- (-1.3,0) node[ddot] {}-- (-2.2, -1) node[dot] {};
\draw  (-1.3,0) node[ddot] {}-- (-1.3, -1) node[dot] {};
\draw  (-1.3,0) node[ddot] {}-- (-0.4, -1) node[dot] {};
}

\DeclareSymbol{311'''}{-3}{
\draw (0,0)node[dia] {} -- (0,1) node[ddot] {};
\draw (0,1)node[ddot] {} -- (-1.3,0) node[square] {}-- (-1.3, -1.1) node[dia] {};
\draw (0,1)node[ddot] {} -- (1.3,0) node[dia] {};
}

\DeclareSymbol{311}{-3}{
\draw (0,0)node[dot] {} -- (0,1) node[ddot] {};
\draw (0,1)node[ddot] {} -- (-1.3,0) node[square] {}-- (-1.3, -1.1) node[dot] {};
\draw (0,1)node[ddot] {} -- (1.3,0) node[dot] {};
}

\DeclareSymbol{61}{-3}{
\draw (0,1)node[dot] {} -- (0,0) node[square] {} -- (0,-1) node[dot]{};
\draw (0,1)node[ddot] {} -- (-1.3,0) node[square] {}-- (-1.3, -1) node[dot] {};
\draw (0,1)node[ddot] {} -- (1.3,0) node[dot] {};
}

\DeclareSymbol{62}{-3}{
\draw (0,1)node[dot] {} -- (-1.3,0) node[square] {} -- (-1.3,-1) node[dot]{};
\draw (0,1)node[ddot] {} -- (0, 0) node[dot] {};
\draw (0,1)node[ddot] {} -- (1.3,0) node[square] {} -- (1.3, -1) node[dot]{};
}

\DeclareSymbol{63}{-3}{
\draw (0,1)node[dot] {} -- (0,0) node[square] {} -- (0,-1) node[dot]{};
\draw (0,1)node[ddot] {} -- (-1.3, 0) node[dot] {};
\draw (0,1)node[ddot] {} -- (1.3,0) node[square] {} -- (1.3, -1) node[dot]{};
}

\DeclareSymbol{34}{-3}{
\draw (0,1)node[ddot] {} -- (0,0) node[square] {} -- (0,-1) node[dot]{};
\draw (0,1)node[ddot] {} -- (-1.3, 0) node[square] {} -- (-1.3, -1) node[dot]{};
\draw (0,1)node[ddot] {} -- (1.3,0) node[square] {} -- (1.3, -1) node[dot]{};
}

\DeclareSymbol{441'}{-3}{
\draw (0,1)node[square] {} -- (0,0) node[square] {}-- (0, -1) node[dia] {};
}

\DeclareSymbol{441}{-3}{
\draw (0,1)node[square] {} -- (0,0) node[square] {}-- (0, -1) node[dot] {};
}

\DeclareSymbol{413'''}{-3}{
\draw (0,1)node[square] {} -- (0,0) node[ddot] {}-- (0, -1) node[dia] {};
\draw (0,0)node[ddot] {} -- (1.3,-1) node[dia] {};
\draw (0,0)node[ddot] {} -- (-1.3,-1) node[dia] {};
}

\DeclareSymbol{413}{-3}{
\draw (0,1)node[square] {} -- (0,0) node[ddot] {}-- (0, -1) node[dot] {};
\draw (0,0)node[ddot] {} -- (1.3,-1) node[dot] {};
\draw (0,0)node[ddot] {} -- (-1.3,-1) node[dot] {};
}

 \DeclareSymbol{141''}{-3}{
\draw (0,1)node[ddot] {} -- (0,0) node[square] {}-- (0, -1) node[dia] {};
\draw (0,1)node[ddot] {} -- (-1.3,0) node[dot] {};
\draw (0,1)node[ddot] {} -- (1.3,0) node[dia] {};
}

 \DeclareSymbol{141}{-3}{
\draw (0,1)node[ddot] {} -- (0,0) node[square] {}-- (0, -1) node[dot] {};
\draw (0,1)node[ddot] {} -- (-1.3,0) node[dot] {};
\draw (0,1)node[ddot] {} -- (1.3,0) node[dot] {};
}

 \DeclareSymbol{114'}{-3}{
\draw (0,1)node[ddot] {} -- (0,0) node[dot] {};
\draw (0,1)node[ddot] {} -- (-1.3,0) node[dot] {};
\draw (0,1)node[ddot] {} -- (1.3,0) node[square] {} -- (1.3, -1) node[dia] {};
}

 \DeclareSymbol{114}{-3}{
\draw (0,1)node[ddot] {} -- (0,0) node[dot] {};
\draw (0,1)node[ddot] {} -- (-1.3,0) node[dot] {};
\draw (0,1)node[ddot] {} -- (1.3,0) node[square] {} -- (1.3, -1) node[dot] {};
}

 \DeclareSymbol{81}{-3}{
\draw (0,1)node[ddot] {} -- (0,0) node[square] {} -- (0,-1) node[dot]{};
\draw (0,1)node[ddot] {} -- (-1.3,0) node[ddot] {} -- (-2, -1) node[dot]{};
\draw (0,1)node[ddot] {} -- (-1.3,0) node[ddot] {} -- (-1.3, -1) node[dot]{};
\draw (0,1)node[ddot] {} -- (-1.3,0) node[ddot] {} -- (-0.6, -1) node[dot]{};
\draw (0,1)node[ddot] {} -- (1.3,0) node[square] {} -- (1.3, -1) node[dot] {};
}

 \DeclareSymbol{82}{-3}{
\draw (0,1)node[ddot] {} -- (-1.3,0) node[square] {} -- (-1.3, -1) node[dot]{};
\draw (0,1)node[ddot] {} -- (0,0) node[ddot] {} -- (-0.7, -1) node[dot]{};
\draw (0,1)node[ddot] {} -- (0,0) node[ddot] {} -- (0, -1) node[dot]{};
\draw (0,1)node[ddot] {} -- (0,0) node[ddot] {} -- (0.7, -1) node[dot]{};
\draw (0,1)node[ddot] {} -- (1.3, 0) node[dot] {};
}

 \DeclareSymbol{83}{-3}{
\draw (0,1)node[ddot] {} -- (-1.3,0) node[square] {} -- (-1.3,-1) node[dot]{};
\draw (0,1)node[ddot] {} -- (1.3,0) node[ddot] {} -- (2, -1) node[dot]{};
\draw (0,1)node[ddot] {} -- (1.3,0) node[ddot] {} -- (1.3, -1) node[dot]{};
\draw (0,1)node[ddot] {} -- (1.3,0) node[ddot] {} -- (0.6, -1) node[dot]{};
\draw (0,1)node[ddot] {} -- (0,0) node[dot] {} ;
}

 \DeclareSymbol{84}{-3}{
\draw (0,1)node[ddot] {} -- (-1.3,0) node[ddot] {} -- (-2, -1) node[dot]{};
\draw (0,1)node[ddot] {} -- (-1.3,0) node[ddot] {} -- (-1.3, -1) node[dot]{};
\draw (0,1)node[ddot] {} -- (-1.3,0) node[ddot] {} -- (-0.6, -1) node[dot]{};
\draw (0,1)node[ddot] {} -- (0,0) node[dot] {} ;
\draw (0,1)node[ddot] {} -- (1.3,0) node[square] {} --(1.3,-1) node[dot]{};
}

 \DeclareSymbol{85}{-3}{
\draw (0,1)node[ddot] {} -- (0,0) node[ddot] {} -- (-0.7, -1) node[dot]{};
\draw (0,1)node[ddot] {} -- (0,0) node[ddot] {} -- (0, -1) node[dot]{};
\draw (0,1)node[ddot] {} -- (0,0) node[ddot] {} -- (0.7, -1) node[dot]{};
\draw (0,1)node[ddot] {} -- (-1.3,0) node[dot] {} ;
\draw (0,1)node[ddot] {} -- (1.3,0) node[square] {} --(1.3,-1) node[dot]{};
}

 \DeclareSymbol{86}{-3}{
\draw (0,1)node[ddot] {} -- (1.3,0) node[ddot] {} -- (2, -1) node[dot]{};
\draw (0,1)node[ddot] {} -- (1.3,0) node[ddot] {} -- (1.3, -1) node[dot]{};
\draw (0,1)node[ddot] {} -- (1.3,0) node[ddot] {} -- (0.6, -1) node[dot]{};
\draw (0,1)node[ddot] {} -- (-1.3,0) node[dot] {} ;
\draw (0,1)node[ddot] {} -- (0,0) node[square] {} --(0,-1) node[dot]{};
}

\DeclareSymbol{91}{-3}{
\draw (0,1)node[ddot] {} -- (-1.3,0) node[ddot] {} -- (-2, -1) node[dot]{};
\draw (0,1)node[ddot] {} -- (-1.3,0) node[ddot] {} -- (-1.3, -1) node[dot]{};
\draw (0,1)node[ddot] {} -- (-1.3,0) node[ddot] {} -- (-0.6, -1) node[dot]{};
\draw (0,1)node[ddot] {} -- (0,0) node[square] {} --(0,-1) node[dot]{};
\draw (0,1)node[ddot] {} -- (1.3,0) node[square] {} --(1.3,-1) node[dot]{};
}

\DeclareSymbol{92}{-3}{
\draw (0,1)node[ddot] {} -- (0,0) node[ddot] {} -- (-0.7, -1) node[dot]{};
\draw (0,1)node[ddot] {} -- (0,0) node[ddot] {} -- (0, -1) node[dot]{};
\draw (0,1)node[ddot] {} -- (0,0) node[ddot] {} -- (0.7, -1) node[dot]{};
\draw (0,1)node[ddot] {} -- (-1.3,0) node[square] {} --(-1.3,-1) node[dot]{};
\draw (0,1)node[ddot] {} -- (1.3,0) node[square] {} --(1.3,-1) node[dot]{};
}

\DeclareSymbol{93}{-3}{
\draw (0,1)node[ddot] {} -- (1.3,0) node[ddot] {} -- (2, -1) node[dot]{};
\draw (0,1)node[ddot] {} -- (1.3,0) node[ddot] {} -- (1.3, -1) node[dot]{};
\draw (0,1)node[ddot] {} -- (1.3,0) node[ddot] {} -- (0.6, -1) node[dot]{};
\draw (0,1)node[ddot] {} -- (0,0) node[square] {} --(0,-1) node[dot]{};
\draw (0,1)node[ddot] {} -- (-1.3,0) node[square] {} --(-1.3,-1) node[dot]{};
}

\DeclareSymbol{101}{-3}{
\draw (0,1)node[ddot] {} -- (1.3,0) node[ddot] {} -- (2, -1) node[dot]{};
\draw (0,1)node[ddot] {} -- (1.3,0) node[ddot] {} -- (1.3, -1) node[dot]{};
\draw (0,1)node[ddot] {} -- (1.3,0) node[ddot] {} -- (0.6, -1) node[dot]{};
\draw (0,1)node[ddot] {} -- (-1.3,0) node[ddot] {} -- (-2, -1) node[dot]{};
\draw (0,1)node[ddot] {} -- (-1.3,0) node[ddot] {} -- (-1.3, -1) node[dot]{};
\draw (0,1)node[ddot] {} -- (-1.3,0) node[ddot] {} -- (-0.6, -1) node[dot]{};
\draw (0,1)node[ddot] {} -- (0,0) node[dot] {} ;
}

\DeclareSymbol{102}{-3}{
\draw (0,1)node[ddot] {} -- (1.8,0) node[ddot] {} -- (2.4, -1) node[dot]{};
\draw (0,1)node[ddot] {} -- (1.8,0) node[ddot] {} -- (1.8, -1) node[dot]{};
\draw (0,1)node[ddot] {} -- (1.8,0) node[ddot] {} -- (1.2, -1) node[dot]{};
\draw (0,1)node[ddot] {} -- (0,0) node[ddot] {} -- (-0.6, -1) node[dot]{};
\draw (0,1)node[ddot] {} -- (0,0) node[ddot] {} -- (0, -1) node[dot]{};
\draw (0,1)node[ddot] {} -- (0,0) node[ddot] {} -- (0.6, -1) node[dot]{};
\draw (0,1)node[ddot] {} -- (-1.5,0) node[dot] {} ;
}

\DeclareSymbol{103}{-3}{
\draw (0,1)node[ddot] {} -- (-1.8,0) node[ddot] {} -- (-2.4, -1) node[dot]{};
\draw (0,1)node[ddot] {} -- (-1.8,0) node[ddot] {} -- (-1.8, -1) node[dot]{};
\draw (0,1)node[ddot] {} -- (-1.8,0) node[ddot] {} -- (-1.2, -1) node[dot]{};
\draw (0,1)node[ddot] {} -- (0,0) node[ddot] {} -- (-0.6, -1) node[dot]{};
\draw (0,1)node[ddot] {} -- (0,0) node[ddot] {} -- (0, -1) node[dot]{};
\draw (0,1)node[ddot] {} -- (0,0) node[ddot] {} -- (0.6, -1) node[dot]{};
\draw (0,1)node[ddot] {} -- (1.5,0) node[dot] {} ;
}

\DeclareSymbol{111}{-3}{
\draw (0,1)node[ddot] {} -- (-1.8,0) node[ddot] {} -- (-2.4, -1) node[dot]{};
\draw (0,1)node[ddot] {} -- (-1.8,0) node[ddot] {} -- (-1.8, -1) node[dot]{};
\draw (0,1)node[ddot] {} -- (-1.8,0) node[ddot] {} -- (-1.2, -1) node[dot]{};
\draw (0,1)node[ddot] {} -- (0,0) node[ddot] {} -- (-0.6, -1) node[dot]{};
\draw (0,1)node[ddot] {} -- (0,0) node[ddot] {} -- (0, -1) node[dot]{};
\draw (0,1)node[ddot] {} -- (0,0) node[ddot] {} -- (0.6, -1) node[dot]{};
\draw (0,1)node[ddot] {} -- (1.5,0) node[square] {}--(1.5, -1)  node[dot]{}  ;
}

\DeclareSymbol{112}{-3}{
\draw (0,1)node[ddot] {} -- (1.8,0) node[ddot] {} -- (2.4, -1) node[dot]{};
\draw (0,1)node[ddot] {} -- (1.8,0) node[ddot] {} -- (1.8, -1) node[dot]{};
\draw (0,1)node[ddot] {} -- (1.8,0) node[ddot] {} -- (1.2, -1) node[dot]{};
\draw (0,1)node[ddot] {} -- (0,0) node[ddot] {} -- (-0.6, -1) node[dot]{};
\draw (0,1)node[ddot] {} -- (0,0) node[ddot] {} -- (0, -1) node[dot]{};
\draw (0,1)node[ddot] {} -- (0,0) node[ddot] {} -- (0.6, -1) node[dot]{};
\draw (0,1)node[ddot] {} -- (-1.5,0) node[square] {}--(-1.5, -1)  node[dot]{}  ;
}

\DeclareSymbol{113}{-3}{
\draw (0,1)node[ddot] {} -- (1.3,0) node[ddot] {} -- (2, -1) node[dot]{};
\draw (0,1)node[ddot] {} -- (1.3,0) node[ddot] {} -- (1.3, -1) node[dot]{};
\draw (0,1)node[ddot] {} -- (1.3,0) node[ddot] {} -- (0.6, -1) node[dot]{};
\draw (0,1)node[ddot] {} -- (-1.3,0) node[ddot] {} -- (-2, -1) node[dot]{};
\draw (0,1)node[ddot] {} -- (-1.3,0) node[ddot] {} -- (-1.3, -1) node[dot]{};
\draw (0,1)node[ddot] {} -- (-1.3,0) node[ddot] {} -- (-0.6, -1) node[dot]{};
\draw (0,1)node[ddot] {} -- (0,0) node[square] {}--(0, -1)  node[dot]{}  ;
}

\DeclareSymbol{7}{-3}{
\draw (0,1)node[ddot] {} -- (-1.8,0) node[ddot] {} -- (-2.4, -1) node[dot]{};
\draw (0,1)node[ddot] {} -- (-1.8,0) node[ddot] {} -- (-1.8, -1) node[dot]{};
\draw (0,1)node[ddot] {} -- (-1.8,0) node[ddot] {} -- (-1.2, -1) node[dot]{};
\draw (0,1)node[ddot] {} -- (0,0) node[ddot] {} -- (-0.6, -1) node[dot]{};
\draw (0,1)node[ddot] {} -- (0,0) node[ddot] {} -- (0, -1) node[dot]{};
\draw (0,1)node[ddot] {} -- (0,0) node[ddot] {} -- (0.6, -1) node[dot]{};
\draw (0,1)node[ddot] {} -- (1.8,0) node[ddot] {} -- (2.4, -1) node[dot]{};
\draw (0,1)node[ddot] {} -- (1.8,0) node[ddot] {} -- (1.8, -1) node[dot]{};
\draw (0,1)node[ddot] {} -- (1.8,0) node[ddot] {} -- (1.2, -1) node[dot]{};
}

\newtheorem{theorem}{Theorem} [section]

\newtheorem{lemma}[theorem]{Lemma}
\newtheorem{proposition}[theorem]{Proposition}
\newtheorem{remark}[theorem]{Remark}

\newtheorem{definition}[theorem]{Definition}

\DeclareMathOperator*{\supp}{supp}
\DeclareMathOperator{\med}{med}

\DeclareMathOperator{\Id}{Id}

\DeclareMathOperator{\HS}{HS}

\newcommand{\noi}{\noindent}
\newcommand{\Z}{\mathbb{Z}}
\newcommand{\R}{\mathbb{R}}
\newcommand{\bbC}{\mathbb{C}}
\newcommand{\T}{\mathbb{T}}
\newcommand{\bul}{\bullet}

\let\Re=\undefined\DeclareMathOperator*{\Re}{Re}
\let\Im=\undefined\DeclareMathOperator*{\Im}{Im}

\let\P= \undefined
\newcommand{\P}{\mathbf{P}}

\newcommand{\E}{\mathbb{E}}

\renewcommand{\L}{\mathcal{L}}

\newcommand{\F}{\mathcal{F}}

\newcommand{\tf}{\mathfrak{t}}

\newcommand{\hf}{\mathfrak{h}}

\newcommand{\al}{\alpha}
\newcommand{\be}{\beta}
\newcommand{\dl}{\delta}
\newcommand{\updl}{\updelta}

\newcommand{\nb}{\nabla}

\newcommand{\scrH}{\mathscr{H}}
\newcommand{\Hs}{\mathscr{H}}

\newcommand{\Dl}{\Delta}
\newcommand{\eps}{\varepsilon}
\newcommand{\kk}{\kappa}
\newcommand{\g}{\gamma}
\newcommand{\G}{\Gamma}
\newcommand{\bG}{\mathbf{\Gamma}}
\newcommand{\ld}{\lambda}
\newcommand{\Ld}{\Lambda}
\newcommand{\s}{\sigma}

\newcommand{\zb}{\mathbf{z}}

\newcommand{\ub}{\mathbf{u}}

\newcommand{\LOP}{\mathcal{L}}

\newcommand{\CRP}{\mathcal{D}}

\newcommand{\vb}{\mathbf{v}}
\newcommand{\wb}{\mathbf{w}}

\newcommand{\bbX}{\mathbb{X}}
\newcommand{\vbbX}{\vec{\mathbb{X}}}

\def\doublestroke#1{\pdfliteral{1 Tr .35 w}#1\pdfliteral{0 Tr 0 w}}

\newcommand{\Taa}{\doublestroke{\Theta}}

\newcommand{\Ta}{\Theta}

\newcommand{\Si}{\Sigma}
\newcommand{\ft}{\widehat}

\newcommand{\Ft}{{\mathcal{F}}}
\newcommand{\wt}{\widetilde}
\newcommand{\cj}{\overline}
\newcommand{\dx}{\partial_x}
\newcommand{\dt}{\partial_t}
\newcommand{\dd}{\partial}

\newcommand{\embeds}{\hookrightarrow}

\newcommand{\ta}{\theta}

\renewcommand{\l}{\ell}
\renewcommand{\o}{\omega}

\newcommand{\Om}{\Omega}

\newcommand{\les}{\lesssim}
\newcommand{\ges}{\gtrsim}

\newcommand{\jb}[1]
{\langle #1 \rangle}

\renewcommand{\b}{\beta}
\newcommand{\ind}{\mathbf 1}

\renewcommand{\S}{\mathcal{S}}

\newcommand{\ff}{\mathfrak{g}}

\newcommand{\N}{\mathbb{N}}

\newcommand{\ze}{\zeta}

\renewcommand{\H}{\mathcal{H}}

\newtheorem*{ackno}{Acknowledgements}

\newcommand{\I}{\mathcal{I}}

\newcommand{\If}{\mathfrak{I}}

\newcommand{\RR}{\mathcal{R}}

\newcommand{\C}{\mathcal{C}}
\numberwithin{equation}{section}
\numberwithin{theorem}{section}

\newcommand{\Q}{\mathbb{Q}}
\newcommand{\PP}{\mathbb{P}}

\newcommand{\V}{\mathcal{V}}

\newcommand{\U}{\mathcal{U}}

\newcommand{\NN}{\mathcal{N}}
\newcommand{\D}{\mathcal{D}}

\newcommand{\too}{\longrightarrow}

\newcommand{\Xc}{\mathcal{X}}
\newcommand{\Yc}{\mathcal{Y}}

\newcommand{\Zc}{\mathcal{P}}

\usepackage{comment}

\newcommand{\pPsi}{\mathbf{\Psi}}

\newcommand{\XX}{\mathbf{X}}

\newcommand{\vXX}{\vec{\mathbf{X}}}

\newcommand{\ito}{\textup{Ito}}

\makeatletter
\@namedef{subjclassname@2020}{%
  \textup{2020} Mathematics Subject Classification}
\makeatother

\begin{document}
\baselineskip = 14pt

\title[Fourier restriction norm method adapted to controlled paths]
{Fourier restriction norm method adapted to controlled paths: stochastic wave equations}

\author[A.~Chapouto, J.~Li, and T.~Oh]
{Andreia Chapouto, Jiawei Li, and Tadahiro Oh}

\address{
Andreia Chapouto, School of Mathematics\\
The University of Edinburgh\\
and The Maxwell Institute for the Mathematical Sciences\\
James Clerk Maxwell Building\\
The King's Buildings\\
Peter Guthrie Tait Road\\
Edinburgh\\
EH9 3FD\\
 United Kingdom\\
and
CNRS, Laboratoire de math\'ematiques de Versailles, UVSQ, Universit\'e Paris-Saclay, CNRS, 45 avenue des 
\'Etats-Unis, 78035 Versailles Cedex, France, 
and School of Mathematics, Monash University, VIC 3800, Australia}

\email{andreia.chapouto@monash.edu}

\address{
Jiawei Li, School of Mathematics\\
The University of Edinburgh\\
and The Maxwell Institute for the Mathematical Sciences\\
James Clerk Maxwell Building\\
The King's Buildings\\
Peter Guthrie Tait Road\\
Edinburgh\\
EH9 3FD\\
 United Kingdom}

\email{jiawei.li@ed.ac.uk}

\address{
Tadahiro Oh, 
School of Mathematics\\
The University of Edinburgh\\
and The Maxwell Institute for the Mathematical Sciences\\
James Clerk Maxwell Building\\
The King's Buildings\\
Peter Guthrie Tait Road\\
Edinburgh\\
EH9 3FD\\
 United Kingdom\\
and School of Mathematics and Statistics, Beijing Institute of Technology, Beijing 100081, China}


\email{hiro.oh@ed.ac.uk}



\begin{abstract}

We investigate the pathwise well-posedness issue of the stochastic nonlinear wave equation (SNLW) with a multiplicative noise. While the Ito solution theory (= random field solution theory) was established  in the 
'80s, its pathwise well-posedness has remained  a challenging open problem for over forty years.
By building a unified framework for the Fourier restriction norm method adapted to the $U^p$- and $V^p$-spaces, due to Koch and Tataru (2007), and the Young\,/\,rough integration theory via the sewing lemma and controlled paths due to Gubinelli (2004)
along with the random tensor estimate 
for multiple stochastic integrals with respect to (fractional) Brownian motions,   
we establish pathwise local well-posedness of SNLW in optimal regularity ranges. In particular, in the one-dimensional case with a white-in-time noise, our result covers the case of an almost space-time white noise, which is optimal within the framework of one-parameter rough paths.

\end{abstract}

%
\maketitle
%

 \tableofcontents

\section{Introduction}

\subsection{Stochastic nonlinear wave equation} 
We consider the Cauchy problem for the following
stochastic nonlinear wave equation  (SNLW) with a multiplicative noise on $\T^d = (\R/(2\pi\Z))^d$:

\noi
\begin{equation}\label{SNLW}
\begin{cases}
\dt^2 u +(1-\Dl )u  \pm u^k = u\Phi \zeta
\\
(u, \dt u) \vert_{t=0} = (\phi_0, \phi_1), 
\end{cases}
\end{equation}
where  $k\ge2$ is an integer
 and $\Phi$ is a  Hilbert-Schmidt operator from $L^2(\T^d)$ to $H^\s(\T^d)$ for some $\s\in\R$
 such  that\footnote{More precisely, $\Phi W^\be(t)$ has spatial regularity $\s$, 
 where $W^\be$ is as in \eqref{W0}.} 
 the noise $\Phi\zeta$ has spatial regularity $\s$. 
 Here, 
  $\zeta$ 
  denotes a fractional\,/\,white-in-time and white-in-space noise on $\R_+\times \T^d$.
  Informally, one can think of the fractional-in-time noise $\zeta$ as
\begin{align*}
\text{``} \zeta  = \jb{\dt}^{-\ta} \xi \text{''},
\end{align*}

\noi
for some $0\le\ta<\frac12$, where $\xi$ is a Gaussian space-time white noise on $\R_+ \times \T^d$, with covariance formally given by
\begin{align}
\label{white}
\E\big[ \xi(t_1,x_1) \xi  (t_2,x_2)\big] = \dl(t_1-t_2) \dl(x_1-x_2)
\end{align}

\noi
for $t_1,t_2 \in\R_+$ and $x_1,x_2 \in\T^d$, where $\dl$ denotes the Dirac delta function.
See \eqref{zeta} for a precise meaning of the noise $\ze$.

Our main goal in this paper is to study {\it pathwise} 
well-posedness issues for \eqref{SNLW}.
We say that $u$ is a solution to \eqref{SNLW} if it 
satisfies  the following 
Duhamel formulation (= mild formulation):
\begin{align}
\label{mild0}
u(t) = \dt S(t) \phi_0 + S(t) \phi_1 \mp \int_0^t S(t-t') u^k(t') \, dt' + \Psi(u)(t),
\end{align}
where $S(t)$ denotes the linear wave propagator, given by 
\begin{align}
\label{S0}
S(t)  = \frac{\sin(t\jb{\nabla})}{\jb{\nabla}}.
\end{align}

\noi
Here, 
 $\jb{\nabla} = \sqrt{1-\Dl}$ is  the Fourier multiplier operator with symbol $\jb{n} = (1+|n|^2)^\frac 12$.
The term $\Psi(u) = \Psi^\be(u)$  in~\eqref{mild0} is often called the stochastic convolution, which accounts for the effect of the stochastic forcing in \eqref{SNLW},
and is formally given by 
\begin{align}
\label{psi1}
\Psi(u)(t) 
= \int_0^t S(t-t') \big[ u(t') \Phi d  W^\be (t') \big].
\end{align}
Here,
$W^\be $ denotes the cylindrical  process on $L^2(\T^d)$,  given by\footnote{By convention, we endow
$\T^d$ with the normalized Lebesgue measure $  (2\pi)^{-d}dx$
such that we do not need to carry factors involving $2\pi$.
With a slight abuse of notation, we simply use $dx$ to denote
the normalized Lebesgue measure on $\T^d$ in the following.}
\begin{align}
\label{W0}
W^\be (t) = \sum_{n\in\Z^d} B_n(t) e_n,
\end{align}

\noi
where $e_n(x) = e^{in\cdot x}$  and $\{B_n\}_{n\in\Z^d}$ denotes
 a family of independent complex-valued fractional Brownian motions 
with the Hurst parameter $\frac 12 \le \be < 1$
 on a fixed probability space $(\Om, \PP, \F)$,  conditioned that 
\begin{align}
\label{real0}
B_{-n} = \cj {B_n}, \quad n\in \Z^d.
\end{align}

\noi
See Subsection~\ref{SUBSEC:FBM} for a review on  fractional Brownian motions
and stochastic integrals with respect to the family $\{B_n\}_{n\in\Z^d}$.
With this notation, 
we interpret the noise $\zeta$ in \eqref{SNLW} as the (distributional) temporal  derivative of $W^\be $:
\begin{align}
\label{zeta}
\zeta  = \dt  W^\be . 
\end{align}

When $\be =\frac12$, 
the family $\{B_n\}_{n\in\Z^d}$ corresponds to  a family of independent Brownian motions, satisfying~\eqref{real0}, and 
the noise $\ze$ in \eqref{SNLW} reduces
to a space-time white noise $\xi$.
In this case,  
since the work~\cite{Walsh86}, 
there have been intensive research activities
on SNLW with (more general) multiplicative noises from the viewpoint of
the Ito solution theory and 
the random field solution theory;
see, for example, \cite{CN88, CN88sing, MilletSanz, 
Mueller97, DF98,  MSS99,  Dalang99, PZ00, MM01, BO07, DSS09,  
On1, On2, 
BR, BO13, Mueller2, DNZ, MSS21,  BR22, HOO}.
See also 
 \cite{Dalang09, Sanz1, DZ14, DSS26} for further references therein.
See 
\cite{Ba12, BJQS15, LHW22, CDST25} 
for results in the fractional-in-time case.
Here, the Ito solution theory and 
the random field solution theory
refer to approaches
where 
the stochastic convolution $\Psi(u)$ in~\eqref{psi1} 
(possibly with $u$ replaced by a more general function $F(u)$) is constructed
as a (suitable) limit in $L^2(\Om)$
via stochastic calculus.
We also point out that 
 the study on \eqref{SNLW}
goes further back in the literature
due to its direct connection to 
stochastic integration
with respect to a multi-parameter white noise;
see, for example, 
\cite{Cai1, WZ, Cai2, Cai3, Hajek, FN, Norris}.

Our main interest in this paper
is to study  the question of pathwise well-posedness
of SNLW~\eqref{SNLW}.
Informally, we say that 
SNLW~\eqref{SNLW} is pathwise locally well-posed
(given $(\phi_0, \phi_1)$ and $\Phi$)
if there exists a set $\Si \subset \Om$ with $\PP(\Si) = 1$
such that, for each $\o \in \Si$, 
there exists a solution $u^\o$ to the Duhamel formulation \eqref{mild0}
on a time interval $[0, T]$
for some almost surely positive stopping time $T$, 
where the 
stochastic convolution $\Psi^\o(u^\o)$ in~\eqref{psi1}
is constructed in the pathwise manner (namely, 
using information of $W^{\be, \o}$ in \eqref{W0}
and the unknown $u^\o$ for this particular $\o \in \Si$, 
not as a limit in $L^2(\Om)$).
Here, we intentionally remained vague
on relevant function spaces, etc.;
see Subsection \ref{SUBSEC:1.3}
for a further discussion.
The difficulty of pathwise well-posedness theory
as compared to the Ito solution theory
can be seen even at 
the level of 
ordinary\,/\,stochastic differential equations:
\begin{align}
\dt f = f \dt g
\label{OD1}
\end{align}

\noi
where $f$ is an unknown and $g$ is a given function.
A solution to \eqref{OD1} is formally given 
by 
\begin{align}
f(t) = f(0) + I(f, g)(t) := f(0) + \int_0^t f(t') \dt g(t') dt', 
\label{OD2}
\end{align}

\noi
where the integral $I(f, g)$ is well defined if $g$ (and hence $f$)
is sufficiently regular.
Suppose that  $g$ is a Brownian motion $B$.
The second order term in an iterative scheme for \eqref{OD2}
contains
$I(B, B) = \int_0^t B(t') dB(t')$.
It is, however, known that there exists no separable Banach space $X \subset C([0, 1])$
such that 
(i)~sample paths of a Brownian motion 
belong to $X$, almost surely, and
(ii) the bilinear operator  $I$ in \eqref{OD2}, a priori defined
for smooth functions, extends 
to a continuous map on $X\times X$ into $C([0, 1])$;
see \cite{Lyons91}
and 
\cite[Proposition 1.1]{FH20}.
This shows that when $g = B$, 
\eqref{OD1} is ill-posed in the pathwise sense, 
if we were to take only $f(0)$ and $B$ as given data.
See the next subsection for a further discussion.
We point out that one important consequence
of the  pathwise solution theory
is that it allows for almost sure convergence of 
various approximation schemes
such as  mollification, piecewise linear approximation, and many other
approximations.
See, for example, \cite{GT}, \cite[Sections 10.3, 12.2, 13.3, and 17.7]{FV10}, 
and \cite[Sections 8.7 and 9.2]{FH20}.

Regarding pathwise well-posedness of SNLW
with multiplicative noises, 
there are very few results in the literature (see, for example,  \cite{QST07, CDST23}), 
limited to the fractional-in-time case
with the Hurst parameter $\be > \frac 12$
(namely, the so-called Young case). 
In particular, 
pathwise well-posedness of 
SNLW~\eqref{SNLW} with a multiplicative white-in-time noise
has remained 
a challenging open problem
for the last forty years
since the work \cite{Walsh86}.
In this paper, 
we establish the first pathwise local well-posedness
result for SNLW \eqref{SNLW}
with a multiplicative white-in-time noise
(Theorem~\ref{THM:2}), 
as well as
a new general 
 pathwise local well-posedness
result in the fractional-in-time case
(Theorem \ref{THM:1}), 
thus answering the open question of forty years.
Our approach is based on 
building a unified framework for

\smallskip
\begin{itemize}
\item[(i)]
 the Fourier restriction norm method,
independently introduced by Bourgain \cite{BO93}
and by Klainerman and Machedon \cite{KM}
and further developed by Koch and Tataru 
\cite{KT07}
in the setting of the $V^p$-spaces of functions
of bounded $p$-variation
and 
their  preduals (the so-called $U^p$-spaces), 
and

\smallskip
\item[(ii)] 
the Young\,/\,rough integration theory via the sewing lemma and controlled paths due to Gubinelli \cite{G04}, subsequent to a seminal work \cite{Lyons98}
by Lyons on rough path theory, 
\end{itemize}

\smallskip

\noi
where
 the random tensor estimate
for multiple stochastic integrals with respect to (fractional) Brownian motions  
(Lemma \ref{LEM:RTE})
 plays a crucial role in (ii). 
We note that our pathwise well-posedness results 
are essentially sharp in terms of spatial regularity.
See Remark \ref{REM:ill2}\,(i).
Furthermore, 
 in the one-dimensional case with a white-in-time noise, our result covers the case of an almost space-time white noise, which is optimal within the framework of one-parameter rough paths;
 see Remark \ref{REM:opt}.

\begin{remark}\label{REM:foc} \rm
In the literature, the wave equation often appears with the linear operator $\dt^2 -\Dl$, while the linear operator $\dt^2 +1 - \Dl$ in \eqref{SNLW} corresponds to the Klein-Gordon equation. 
We note that the same pathwise local well-posedness results with inessential modifications hold for 
both equations, 
while the Klein-Gordon version does not require separate treatment 
at the zeroth spatial frequency.
For this reason, 
we choose to work with  the Klein-Gordon version  \eqref{SNLW}
in this paper, and we simply refer to it as 
the stochastic nonlinear wave
equation.

Our main results (Theorems \ref{THM:1} and \ref{THM:2})
are on pathwise {\it local} well-posedness
of SNLW~\eqref{SNLW}.
As such, the defocusing\,/\,focusing nature of the equation
does not play any role in our analysis, 
and thus we simply work with the ``focusing'' case (namely, with $-u^k$
in \eqref{SNLW}) in the following.

\end{remark}

\subsection{Overview} 

The study of {\it random dispersive PDEs}, 
 broadly interpreted
with random initial data and\,/\,or  stochastic forcing, 
was initiated 
in seminal works \cite{BO94, BO96}
by Bourgain 
in  the context of  invariant Gibbs measures
for nonlinear Schr\"odinger equations (NLS).
This study lying at the intersection of dispersive PDEs
and stochastic analysis has attracted much attention and has been studied intensively;
see, for example, 
\cite{BT1,Oh0,  CO, BT3, BOP2, Poc, GKO1,  OTh2, 
OTz, GKOT22,   
Bring0,   OOT1, ORT23, Bring, 
OOT2, 
STzX, OTWZ,  FT, Zine, LLOT, LLO}.
See also surveys \cite{BOP4, Tzv1}.
In particular, over the last five years, 
we have witnessed several breakthrough
results
\cite{GKO2, DNY2, DNY22, 
BDNY24}, 
introducing novel tools and ideas such as paracontrolled
calculus in the dispersive setting, 
random averaging operators, 
and the theory of random tensors.
We point out that 
many works in the literature focused
on probabilistic well-posedness with random initial data and\,/\,or
{\it additive} noises, 
where the difficulty comes from 
{\it  spatial} roughness
 in making sense of products and nonlinearities. 
In contrast, 
the main difficulty 
of the pathwise well-posedness
of SNLW~\eqref{SNLW} with a {\it multiplicative} noise
comes from {\it temporal} roughness
in making sense of the stochastic convolution $\Psi(u)$
in \eqref{psi1} in the pathwise manner.

\medskip

Focusing only on temporal regularity, 
let us revisit the 
 integral
$I= I(f, g)$ in \eqref{OD2}:
\begin{align}
\label{inc1}
I (t) =  \int_0^t f(t') d  g (t')  = \int_0^t f(t') \dt  g (t') dt'.
\end{align}

%
\noi
From the differential calculus viewpoint, 
$I$ is given as the unique solution to 
the ODE: $\dt I = f \dt g$ with $I(0) = 0$.
%
We can also interpret  the integral $I$ from the 
incremental viewpoint (namely, 
moving from derivatives to finite increments).
Namely, we define the integral $I$ in~\eqref{inc1} 
as
the  (unique) function, satisfying 
\begin{align}
I(t) - I(r) = f(r)(g(t) - g(r)) + R_{t, r}
 \quad \text{with \, $I(0) = 0$}
\label{inc2}
\end{align}

\noi
for any $0 \le r \le t \le 1$, where the error term $R$
satisfies 
\begin{align}
R_{t, r} = o(|t- r|)
\label{inc3}
\end{align}

\noi
as $|t-r| \to 0$, 
uniformly
 in $0 \le r \le t \le 1$.
It is clear that this property
is satisfied by the usual integral 
for  $f$ and $g$ with sufficient regularity.
We note that, while the error term $R$ is not unique,  
the relation \eqref{inc2}  with \eqref{inc3} 
uniquely determines $I$, which is 
indeed given as the unique limit of the Riemann-Stieltjes type sums:
\begin{align*}
I(t) = \lim_{|\Pi([0,t])|\to 0} 
\sum_{j=0}^{k-1} f(t_{j+1})\big( g(t_j) - g(t_{j+1})\big).
\end{align*}

\noi
See \eqref{con3}.
Here,  the limit is taken over any partition $\Pi([0,t])$ of $[0,t]$
of the form 
$\Pi([0,t])= \{0 = t_{k} < \dots < t_1 <  t_0 = t\}$
whose mesh size $|\Pi([0,t])|  = \sup_j |t_j - t_{j+1}|$ goes to $0$.
Namely, the integral $I = I(f, g)$
 is the only function whose increments match the 
 ``germ''  $f (r)(g(t) - g(r))$ modulo a negligible
error.

In the setting of H\"older continuous functions, 
the integral $I = I(f, g)$ in \eqref{inc1}
can be understood as 
the so-called \textit{Young integral} \cite{Y36}
if $f \in C^{\al_1}$ and 
 $g \in C^{\al_2}$
 with $\al_1 + \al_2 > 1$.
When $\al_1 = \al_2 = \al$, this condition reduces to $\al > \frac 12$.
In a seminal work 
 \cite{Lyons98}, Lyons introduced the theory of rough paths, 
which allows us to interpret the integral $I = I(f, g)$ in~\eqref{inc1} 
as
 a \textit{rough integral}  when $0 < \al \le \frac 12$.
In a subsequent breakthrough work 
\cite{G04}, 
Gubinelli introduced the notion of controlled paths, 
which provided a convenient framework to make sense of
a rough integral $I = I(f, g)$ via 
 the so-called sewing lemma (Lemma \ref{LEM:sew}).\footnote{Indeed, 
 the relation \eqref{inc2} states that $f$ is controlled by $g$;
 see Definition \ref{DEF:RP}\,(ii).}
 The theory of rough paths has been applied
 widely in the  study of 
 pathwise well-posedness of stochastic differential equations driven by rough noises such as 
 a Brownian motion
 and 
 pathwise well-posedness of  
stochastic parabolic PDEs; see, for example,  \cite{GLT, CF, GT10, CFO, H11, DGT, HW, OW, HB}.
See also  the  monographs  \cite{LQ02, FV10, FH20}
for further discussions on the subject.
We also point out that the theory of rough paths
has led to the development of 
 the theory of regularity structures by Hairer \cite{H14}
and paracontrolled distributions by Gubinelli, Imkeller, and Perkowski \cite{GIP}.

Let us come back to our problem.
One of the main tasks is to provide a pathwise meaning
to the stochastic convolution $\Psi(u)$ in \eqref{psi1}.
For now, we ignore the linear wave propagator $S(t - t')$ in \eqref{psi1}
and consider the integral:
\begin{align}
\label{inc4}
I(u)(t) 
= \int_0^t  u(t') \Phi d  W^\be (t').
\end{align}

\noi
Given $\Phi \in \HS(L^2(\T^d);H^\s(\T^d))$, 
it is easy to see 
from  \eqref{W0}
that 
 $\Phi W^\be  \in C^{\al}(\R_+;H^\s(\T^d))$ 
for any $0 < \al < \be$,  
 almost surely.
By ignoring spatial regularities, 
if $u$ has the same temporal regularity $\al$
as  $\Phi W^\be$, 
then we can make sense of $I(u)$ in \eqref{inc4}
as (i)~a Young integral for  $\frac 12 < \be < 1$, 
and (ii)~a rough integral for $\be = \frac 12$
(by imposing an appropriate controlled structure on $u$).
For this reason, 
we will refer to 

\smallskip
\begin{itemize}
\item the fractional-in-time case with 
the Hurst parameter $\frac 12 < \be < 1$
as the {\it Young case}, and 

\smallskip

\item the white-in-time case with 
the Hurst parameter $\be = \frac 12$
as the {\it rough case}

\end{itemize}

\noi
in the following.

The stochastic convolution $\Psi(u)$ in \eqref{psi1}
comes with the linear wave propagator $S(t - t')$
as compared to \eqref{inc4}.
In the usual study of dispersive PDEs, 
dispersion induces oscillations
providing certain cancellation (say, via integration by parts).
However, the integrator $dW^\be (t')$ in \eqref{psi1} is of negative regularity, 
and as a result, 
the linear wave propagator $S(t - t')$ in~\eqref{psi1} is an enemy in our problem, 
preventing tools such as integration by parts.
We circumvent this issue by working with the so-called
interaction representation (see \eqref{int})
in the vectorial formulation, 
and construct the stochastic convolution 
$\Psi(u)$  (at the level of the interaction representation)
as a Young\,/\,rough integral, depending on the value of the Hurst parameter $\be$.

In order to handle the nonlinear Duhamel integral term 
(= the third term on the right-hand side of \eqref{mild0})
for $ d\ge 2$, 
we make use of dispersion in the form of the Strichartz estimates
(Lemma \ref{LEM:Slin}).
However, the usual Strichartz spaces 
(not capturing temporal regularity)
or the $X^{s, b}$-spaces in the Fourier restriction norm method
are not compatible with the pathwise construction 
of the stochastic convolution $\Psi(u)$
as a Young\,/\,rough integral.
We build a unified framework
for treating both the nonlinear Duhamel integral term
and the stochastic convolution $\Psi(u)$
by working with 
 the refined Fourier restriction norm method
 adapted to the $U^p$- and $V^p$-spaces, 
 introduced by 
 Koch and Tataru 
\cite{KT07}; 
see Subsection \ref{SUBSEC:Up}
for the definitions and basic properties
of these function spaces.
In the rough case ($\be = \frac 12$), 
we reduce~\eqref{mild0}
to system \eqref{bd9a}-\eqref{bd9b}
for three unknowns.
We note that the $U^p$-$V^{p'}$ duality 
allows us to make sense of the integral
$I = I(f, g)$ in \eqref{inc1}
for 
$f \in U^p$ and $g \in V^{p'}$
even at  the endpoint 
$\frac 1p + \frac 1 {p'} =  1$, 
where the Young integration theory (in the $V^p$-setting)
fails, 
and such duality plays an important
role
in studying the nonlinear Duhamel integral term; 
see~\eqref{lin2} in Lemma~\ref{LEM:lin}.

\subsection{Main results}
\label{SUBSEC:1.3}

In this subsection, we state our main results
on pathwise local well-posedness of SNLW \eqref{SNLW}.

Let us first recall the 
 notion of a critical regularity associated with the (deterministic) nonlinear wave equation (NLW):
\begin{equation}\label{NLW}
\dt^2 u  -\Dl u \pm u^k = 0
\end{equation}

\noi
by separately considering the cases  $d \ge 2$ and  $d= 1$.

\smallskip

\noi
$\bul$ {\bf $\pmb{d\ge 2}$.}
Let us first consider the case $d\ge2$. 
When posed on $\R^d$, \eqref{NLW} enjoys the following scaling symmetry; 
  if $u$ is a solution to  \eqref{NLW} with initial data $(\phi_0,\phi_1)$, then 
the scaled function $u^\ld(t,x)= \ld^{- \frac{2}{k-1}}u( \ld^{-1}t, \ld^{-1} x)$ 
is also a solution to 
 \eqref{NLW} with the scaled initial data $(\phi_0^\ld(x), \phi_1^\ld(x))
 = \big(\ld^{-\frac{2}{k-1}}\phi_0(\ld^{-1}x),\ld^{-\frac{2}{k-1}-1} \phi_1(\ld^{-1}x)\big)$. 
 This scaling symmetry induces the following scaling critical Sobolev index: 
 \begin{align}
 s_{\rm scaling} = s_{\rm scaling}(d,k) =  \frac{d}{2} - \frac{2}{k-1}
 \label{scaling} 
 \end{align} 
 
 \noi
 such that the homogeneous Sobolev $\dot H^s(\R^d) \times \dot H^{s-1}(\R^d)$-norm
 remains invariant 
 for the scaled initial data 
  $(\phi_0^\ld, \phi_1^\ld)$.
 It is known that 
 \eqref{NLW} also enjoys the  conformal symmetry (= Lorentzian invariance;  
 see \cite[(1.4)]{CCT}), 
 which induces another critical regularity:
 \begin{align}
 s_{\rm conf} = s_{\rm conf}(d,k) = \frac{d+1}{4} - \frac{1}{k-1}.
 \label{conf} 
 \end{align} 
 
 \noi
 See \cite{LS95, Sogge, CCT} for further discussions. 
We note that the heuristics related to the
scaling and conformal symmetries
also holds on $\T^d$ due to the finite speed of propagation.

This discussion leads us to 
 define the following critical Sobolev index $s_{\rm crit} = s_{\rm crit}(d, k)$:
\begin{equation}\label{crit1}
s_{\rm crit} = \max (s_{\rm scaling}, s_{\rm conf}, 0) 
= \max \bigg( \frac{d}{2} - \frac{2}{k-1}, \frac{d+1}{4} - \frac{1}{k-1}, 0 \bigg).
\end{equation}

\noi
The last restriction of $s_{\rm crit} \ge 0$ comes from 
the need to make sense of the nonlinearity $u^k$ in~\eqref{NLW}.
Note that we have $s_\text{scaling} <  s_\text{conf}$ in the following cases:
\begin{align}
(d, k) \in 
\mathcal{A}_0
: = \big\{
 (2, 2), (2, 3), (2, 4), (3, 2), (4, 2)
\big\}.
\label{crit2x}
\end{align}

\smallskip

\noi
$\bul$ {\bf $\pmb{d = 1}$.}
In this case, the wave equation is not dispersive.
As a consequence, there are no Strichartz estimates available, 
and Sobolev's inequality is  the only available tool. 
This leads to a new critical regularity $s_{\rm sob} = \frac12-\frac1k$, 
where $s_{\rm sob}$ is the lowest regularity such that 
 the embedding $H^{s_{\rm sob}}(\T) \subset L^k(\T)$ holds. Hence, 
 by noting that $s_{\rm sob} > s_{\rm conf} > s_{\rm scaling}$, 
 we define the critical Sobolev index  in one-dimension by setting
\begin{align}
\label{crit3}
s_{\rm crit} = s_{\rm sob} = \frac12 - \frac1k. 
\end{align}

\smallskip

The following  local well-posedness result is known for NLW \eqref{NLW}.

\begin{proposition}\label{PROP:LWP}
Let  $d \ge 1$.  Given  an integer  $k \geq 2$, 
 let $s_\textup{crit}$ be as in
\eqref{crit1} when $d \ge 2$
and~\eqref{crit3} when $d = 1$.
Suppose that one of the following conditions holds{\rm:}
\begin{align}
\begin{split}
 \textup{(i) }& 
s >  s_\textup{crit}, 
\ \text{ when  $(d,k) \in \{(2,2), (2,3), (3,2)\}$, \  or}\\
 \textup{(ii) } & 
 s \ge   s_\textup{crit}, 
\ \text{ otherwise}.
\end{split}
\label{reg1}
\end{align}

\noi
Then, NLW \eqref{NLW}
is  locally well-posed in 
$\H^s(\T^d) := H^s(\T^d) \times H^{s-1}(\T^d)$.

\end{proposition}

Proposition \ref{PROP:LWP} follows 
from a standard contraction argument, 
using 
Sobolev's inequality for $d = 1$
and the Strichartz estimates (Lemma \ref{LEM:Slin})
for $d\ge 2$.
See the proof of~Proposition~\ref{PROP:LWPx} 
on local well-posedness of SNLW~\eqref{SNLW}
in the Ito sense, 
where we establish the nonlinear estimates.
In Section \ref{SEC:nonlin}, 
we revisit
the nonlinear estimates necessary 
for proving 
Proposition~\ref{PROP:LWP}
in the context of the Fourier restriction norm method 
adapted to the $U^p$- and $V^p$-spaces;
see Subsection \ref{SUBSEC:Up}
for the definitions of these spaces.

\begin{remark}\label{REM:ill}\rm

In the focusing case (with $-u^k$ in \eqref{NLW}), 
Proposition \ref{PROP:LWP} 
is essentially sharp;\footnote{Namely, except 
at the critical case 
$s = s_\text{crit} = \max(s_\text{conf}, 0)$
when
$(d,k) \in \{(2,2), (2,3), (3,2)\}$.}
see \cite{LS95, CCT, OOTz, FO}
for the known ill-posedness results.
See also \cite[Exercise 3.64 and 3.67]{TAO}.
In the defocusing case
(with $+u^k$ in \eqref{NLW}), 
the 
question 
of local well-posedness of NLW~\eqref{NLW}
still remains open 
in the range (i)~$\max (s_\text{scaling}, 0) \le  s \le  \max(s_\text{conf}, 0)$
for  $(d,k) \in \{(2,2), (2,3), (3,2)\}$
and (ii)~$s_\text{scaling} \le  s <   s_\text{conf}$
for  $(d,k) \in \{(2,4),  (4,2)\}$.

\end{remark}

We are now ready to state our main results.
In order to simplify the presentation,  we only consider the case when 
$\Phi$ is spatially homogeneous. 
Namely, 
$\Phi$ is a Fourier multiplier operator
given by 
\begin{align}
\Phi (e_n) = \phi_n e_n   \quad \text{with} \ \ \phi_{-n} = \cj{\phi_n}.
\label{phi1}
\end{align}

\noi
Under this assumption, we have 
\begin{align}
\| \Phi \|_{\HS(L^2; H^\s)}= \|\jb{n}^\s\phi_n\|_{\l^2_n}.
\label{phi1a}
\end{align}

\noi
We leave the general case to interested readers.

%
%
%
%
%
%
%
%

\medskip

\noi
\underline{\bf Part 1: Young case.}
Our first main result is pathwise local 
well-posedness of SNLW \eqref{SNLW}
in the Young case
(namely, 
the  fractional-in-time case with the Hurst parameter $\frac 12 < \be < 1$).

\begin{theorem}[Young case]
\label{THM:1}
Let  $d \ge 1$ and $\frac 12 < \be < 1$.
Let $\zeta  = \dt  W^\be$ be as in \eqref{zeta}, 
where $W^\be$ is the cylindrical process defined in \eqref{W0}.
  Given  an integer  $k \geq 2$, 
suppose that one of the conditions 
in \eqref{reg1} holds{\rm :}
\begin{align*}
 \textup{(i) }& 
s >  s_\textup{crit}, 
\ \text{ when  $(d,k) \in \{(2,2), (2,3), (3,2)\}$, \  or}\\
 \textup{(ii) } & 
 s \ge   s_\textup{crit}, 
\ \text{ otherwise}, 
\end{align*}

\noi
where  $s_\textup{crit}$ is as in
\eqref{crit1} when $d \ge 2$
and~\eqref{crit3} when $d = 1$.
Furthermore, suppose that 
$ \s \in \R$ satisfies 
\begin{align}
\label{sigma0}
\s>\max (-s, s-1).
\end{align}

\noi
Then, 
given  $\Phi \in \HS(L^2(\T^d); H^\s(\T^d))$, 
satisfying~\eqref{phi1}, 
 SNLW \eqref{SNLW}
  with a multiplicative fractional-in-time noise
{\rm (}with the Hurst parameter $\frac 12 < \be < 1${\rm)}  
is pathwise locally well-posed in $\H^s(\T^d)$.

\end{theorem}

In the following, we discuss our strategy and 
explain a meaning of pathwise solutions
constructed in Theorem \ref{THM:1}.
A natural attempt 
would be 
to try to construct the stochastic convolution $\Psi(u)$
in \eqref{psi1}
as a Young integral, 
for which 
we need to make use of 
 the temporal regularity of the unknown $u$ in  \eqref{mild0}.
However, 
due to the presence of the linear wave propagator  $S(t)$ 
in~\eqref{mild0}, 
we expect positive temporal regularity 
for $u$ only at the expense of spatial regularity.
To bypass this issue, 
we will instead work with the so-called \textit{interaction representation}, which is morally given by ``$S(t)^{-1} u(t)$"; see \eqref{int} for a precise definition.

In order to introduce this change of 
unknowns, 
 we work with the following vectorial formulation of \eqref{mild0} for 
 a vector $\vec u  = (u,v)$ with $v=\dt u$:
\begin{align}
\label{mild1}
\begin{split}
\vec u(t) = 
\begin{pmatrix}
u(t) \\
v(t)
\end{pmatrix}  
& = \S(t)
\begin{pmatrix}
\phi_0\\
\phi_1
\end{pmatrix}
 + \int_0^{t} \S(t-t')
\begin{pmatrix}
0\\
u^k(t')
\end{pmatrix}
dt'\\
& \quad +
\int_0^t \S(t-t')
\bigg[
\begin{pmatrix}
0 \\
u(t')
\end{pmatrix} \Phi dW^\be (t')\bigg],
\end{split}
\end{align}

\noi
where  $\S(t)$ is the matrix-valued
linear wave propagator, given by 
\begin{align}
\label{Sdef}
\begin{split}
\S(t) 
 =
\begin{pmatrix}
\dt S(t) & S(t)
\\
\dt^2 S(t) & \dt S(t)
\end{pmatrix}
=
\begin{pmatrix}
\cos(t\jb{\nabla}) & \frac{\sin(t\jb{\nabla})}{\jb{\nabla}} \\
- \jb{\nabla} \sin(t\jb{\nabla}) & \cos(t\jb{\nabla})
\end{pmatrix}.
\end{split}
\end{align}

\noi
We recall that $\S(t)$ is unitary on 
 $\H^s(\T^d) = H^s(\T^d) \times H^{s-1}(\T^d)$ (see \eqref{Hs})
with $\S(t)^{-1} = \S(-t)$.
We then define the  
 \textit{interaction representation} $\vec\ub = (\ub,\vb)$ 
 of $\vec u  = (u,v)$ by 
\begin{align}
\label{int}
\vec\ub(t) = \S(t)^{-1} \vec u(t) = \S(-t) \vec u(t).
\end{align}

\noi
The interaction representation \eqref{int} has played a crucial role in the study of dispersive PDEs, 
in particular,    in the context of the Fourier restriction norm method
\cite{BO93, KM, KT07} and the normal form method \cite{BIT, GuoKO, OSTz, OW1}.
For example, (an equivalent version of) the $X^{s, b}$-norm 
(also known as the hyperbolic Sobolev norm)
in the Fourier restriction norm method 
can be expressed as the bi-parameter space-time Sobolev norm
of the interaction representation $\vec \ub = (\ub, \vb)$:
\begin{align}
\|\vec u \|_{X^{s, b}}
= \big\| \jb{\dt}^b (\jb{\nb_x}^s  \ub, \jb{\nb_x}^{s-1}  \vb)  \big\|_{L^2_{t, x}\times L^2_{t, x}}.
\label{Xsb}
\end{align}

In terms of the interaction representation $\vec \ub$, 
\eqref{mild1} is expressed as 
\noi
\begin{align}\label{mild2}
\vec\ub(t)
=
\begin{pmatrix}
\phi_0
\\
\phi_1
\end{pmatrix}
+ \int_0^t \S(-t')
\begin{pmatrix}
0 \\
( \pi_1 [ \S(t') \vec\ub(t') ] )^k
\end{pmatrix}
dt'
+
\vec{\pPsi} (\vec\ub) (t), 
\end{align}
where $\pi_1 \vec\ub$ denotes the projection onto the first component of the vector $\vec\ub$,  
and $\vec{\pPsi} (\vec\ub) = \vec{\pPsi}^\be (\vec\ub) $ 
is formally given by 
\begin{align}
\label{psi2}
\vec{\pPsi} (\vec\ub) (t)  
= \int_0^t
\S(-t')
\bigg[\begin{pmatrix}
0
\\
\pi_1 [ \S(t') \vec\ub(t') ]
\end{pmatrix}
\Phi dW^\be (t')\bigg]
.
\end{align}

\noi
From \eqref{psi1} with \eqref{Sdef}, we
have 
\begin{align}
\begin{split}
\vec{\pPsi} (\vec\ub)(t)
& =
\begin{pmatrix}
\displaystyle \int_0^t \frac{- \sin(t'\jb{\nabla}) }{\jb{\nabla}} 
\Big[\big( \cos(t'\jb{\nabla}) \ub(t')  + \frac{\sin(t'\jb{\nabla})}{\jb{\nabla}} \vb(t') \big)
 \Phi d W^{\be} (t') \Big] \rule[-5mm]{0pt}{0pt}\\
\displaystyle
\int_0^t \cos(t'\jb{\nabla}) 
\Big[  \big( \cos(t'\jb{\nabla}) \ub(t') + \frac{\sin(t'\jb{\nabla})}{\jb{\nabla}} \vb(t') \big)  \Phi d W^\be (t')   \Big]
\end{pmatrix}.
\end{split}
\label{psi3}
\end{align}

\noi
We 
 provide a pathwise meaning to the stochastic term
$\vec{\pPsi} (\vec\ub)$
by interpreting it as a {\it Young integral}.
Here, the main tools are 
 the sewing lemma (Lemma \ref{LEM:sew})
due to Gubinelli \cite{G04}
and the random tensor estimate (Lemma \ref{LEM:RTE})
for (multiple) stochastic integrals with respect to fractional Brownian motions, 
which we develop in Subsections \ref{SUBSEC:FBM}
and \ref{SUBSEC:RTE}.

For this purpose, we introduce an 
operator-valued 
two-parameter process 
 $\vXX$ (called a driver) by setting
\begin{align}
\label{introX}
\begin{split}
\vXX_{t,r} (\vec f) 
& = \int_r^t
\S(-t')
\bigg[\begin{pmatrix}
0
\\
\pi_1 [ \S(t') \vec f ]
\end{pmatrix}
\Phi dW^\be (t')\bigg]\\
& =
\begin{pmatrix}
\displaystyle \int_r^t \frac{- \sin(t'\jb{\nabla}) }{\jb{\nabla}} 
\Big[\big( \cos(t'\jb{\nabla}) f_1  + \frac{\sin(t'\jb{\nabla})}{\jb{\nabla}} f_2  \big)
 \Phi d W^{\be} (t') \Big] \rule[-5mm]{0pt}{0pt}\\
\displaystyle
\int_r^t \cos(t'\jb{\nabla}) 
\Big[  \big( \cos(t'\jb{\nabla}) f_1 + \frac{\sin(t'\jb{\nabla})}{\jb{\nabla}}f_2 \big)  \Phi d W^\be (t')   \Big]
\end{pmatrix}
\end{split}
\end{align}

\noi
for $t \ge r \ge 0$ and   a pair  $\vec f = (f_1,f_2)$ of functions $f_1$ and $f_2$ on $\T^d$, 
where we simply replaced $ \ub (t')$ and $ \vb (t')$ in \eqref{psi3}
by $f_1$ and $f_2$, respectively.
See also \eqref{XX0}, \eqref{XX0a}, and \eqref{XX0b}.
Note that the stochastic term 
$\vec{\pPsi} (\vec\ub)$ in \eqref{psi2}
is then formally given by 
\begin{align}
\vec{\pPsi} (\vec\ub)(t) 
= \vXX_{t, 0} (\vec \ub_\bul), 
\label{bd1x}
\end{align}

\noi
where $\bul$ denotes the variable of integration (in time).
{\it Suppose} that there exists $\al > \frac 12$ such that 
\begin{align}
\| \vXX_{t,r}\|_{\H^s \to \H^s} \le C_\o |t-r|^\al
\label{bd1}
\end{align}

\noi
for any $0 \le r \le t \le T$ (with some $T > 0$), 
where $C_\o$ is an almost surely finite random constant.
Then, 
a standard Young integration theory for H\"older continuous functions, 
based on the sewing lemma (see \cite{G04, GT10, FH20}), 
allows us to construct 
$\vec{\pPsi} (\vec\ub)$
for $\vec\ub \in C^\al([0, T]; \H^s(\T^d))$, in the pathwise manner, 
as the Young integral $\I^{\vXX}(\vec \ub)$
of  $\vec\ub$
with the Young driver $\vXX$, 
given as the unique limit of 
Riemann-Stieltjes type sums; see \eqref{intY1f}.

This approach, however, leaves an issue
in studying the nonlinear problem \eqref{mild2}, 
where we need to make sense of {\it both} 
the nonlinear  term in \eqref{mild2}:
\begin{align}
\vec{\NN}({\vec{\ub}})(t) 
=  \int_0^t \S(-t')
\begin{pmatrix}
0 \\
( \pi_1 [ \S(t') \vec\ub(t') ] )^k
\end{pmatrix}
dt'
\label{bd2}
\end{align}

\noi
and the stochastic term $\vec{\pPsi} (\vec\ub)$
in the pathwise manner.
The Strichartz spaces (see  \eqref{it1}, 
\eqref{it3}, and \eqref{it4} in 
Subsection~\ref{SUBSEC:ito1})
for proving Proposition \ref{PROP:LWP}
are not suitable for this purpose since they do not provide
any positive temporal regularity, 
causing an issue in making sense of the stochastic term 
$\vec{\pPsi} (\vec\ub)$ as a Young integral via the sewing lemma.
As for the $X^{s, b}$-norm in~\eqref{Xsb}, 
while we may employ the sewing lemma
in the Sobolev (and Besov) scale 
in \cite{FS}, 
the $X^{s, b}$-space is not suitable
to treat the scaling critical problem (namely, $s = s_\text{scaling}$), 
which would require $b = \frac12$, but we have $X^{s, \frac 12}\not \subset C(\R; \H^s(\T^d))$
and the transference principle fails for $b \le \frac 12$.
Moreover, in the rough case
(namely, the white-in-time case with the Hurst parameter $\be = \frac 12$), 
the Sobolev  scale
sewing lemma   \cite{FS} requires $q > 2$
for the $L^q_t$-integrability, 
and thus the $X^{s, b}$-norm is not suitable even in this sense.

Next, 
consider  the following variant of the $X^{s, b}$-norm:
\begin{align}
\|\vec u \|_{\C^\al_{\Box, T} \H^s_x}
= \|  \vec \ub \|_{\C^\al_T \H^s_x}, 
\notag 
\end{align}

\noi
used by Gubinelli in \cite{G12}, 
where $\C^\al([0, T]) = C^\al([0, T]) \cap L^\infty([0, T])$
denotes the Lipschitz space (see \eqref{Lip1}).
On the one hand, 
the $\C^\al_T \H^s_x$-norm allows us to construct the stochastic 
term $\vec{\pPsi} (\vec\ub)$ in \eqref{psi2}
as a Young integral via the sewing lemma.
On the other hand, 
the Lipschitz space\footnote{More precisely, the H\"older semi-norm.}
scales much worse than $L^\infty$
and, as a result, it is not possible to treat the nonlinear 
term
$\vec{\NN}({\vec{\ub}})$
in \eqref{bd2} 
at a low regularity (in particular, $s$ close to the scaling critical regularity 
$s_\text{scaling}$).

We overcome this issue by 
employing
the refined Fourier restriction norm method
adapted to the $V^p$-spaces of functions
of bounded $p$-variation
and 
their  preduals (the so-called $U^p$-spaces), 
introduced by Koch and Tataru
with their collaborators \cite{KT05, KT07, HHK09, HTT11}.
See Subsection~\ref{SUBSEC:Up}
for the definitions and basic properties
of these spaces.
Moreover,  these spaces enjoy the following properties:

\smallskip

\begin{itemize}
\item[(i)]
The $U^p$- and $V^p$-spaces
scale like the $L^\infty$-space, 
and thus the 
$U^p_{\Box} (\R; \H^{s}(\T^d))$-space
 defined in \eqref{Dl1}
is scaling-invariant\footnote{Strictly speaking, 
we would need to replace the nonhomogeneous Sobolev space 
$\H^{s}(\T^d)$ by its homogeneous counterpart.}
for $s = s_\text{scaling}$, allowing us to treat 
the nonlinear  term 
$\vec{\NN}({\vec{\ub}})$ in \eqref{bd2}
even at 
the scaling critical 
case (for $s$ satisfying~\eqref{reg1}) for $d \ge 2$ via the Strichartz estimates and the transference principle.
See Section \ref{SEC:nonlin}
for deterministic analysis on 
the nonlinear  term 
$\vec{\NN}({\vec{\ub}})$ in~\eqref{bd2}.
We note that the $U^p$-$V^{p'}$ duality 
plays a crucial role  in (the proof of) the linear estimate \eqref{lin2}
of Lemma \ref{LEM:lin}.

\smallskip

\item[(ii)] 
The $V^p$-spaces
are also adapted to the construction of 
Young (and rough) integrals; see \cite{FV10, DGHT19}.
In Section \ref{SEC:int}, 
we present a review on 
the  Young\,/\,rough integration theory for functions of bounded $p$-variation, 
 based on the sewing lemma 
(Lemma \ref{LEM:sew}), 
in particular for  readers in dispersive PDEs
who may not be familiar with the subject.
We pay particular attention to 
make notations and spaces
introduced in Subsection~\ref{SUBSEC:Up}
(for PDE analysis) 
and Section \ref{SEC:int} 
(for rough analysis) compatible.

\end{itemize}

\smallskip

\noi
Furthermore, 

\begin{itemize}
\item[(iii)]
In Subsections \ref{SUBSEC:FBM}
and \ref{SUBSEC:RTE}, 
we set up notations and establish 
 the random tensor estimate (Lemma \ref{LEM:RTE})
for (multiple) stochastic integrals with respect to fractional Brownian motions, 
which we used in Section \ref{SEC:ops}
to show that the bound \eqref{bd1} on 
the random driver $\vXX$ indeed holds
almost surely, provided that \eqref{sigma0} holds.

\end{itemize}

\smallskip

\noi
The points (i), (ii), and (iii) 
allow us to interpret 
the integral equation \eqref{mild2}
for the interaction representation $\vec \ub = \S(-t) \vec u(t)$
as
the following {\it Young differential equation} (YDE):
\begin{align}
\vec{\ub}(t)= (\phi_0,\phi_1) + 
\vec{\NN}({\vec{\ub}})(t) 
+  
\I^{\vXX^\o}(\vec\ub)(t)
\label{bd3}
\end{align}

\noi
for almost every $\o \in \Om$, 
where 
$\vec{\NN}({\vec{\ub}})$ is as in \eqref{bd2}
and 
$\I^{\vXX^\o}(\vec\ub)$
denotes the Young integral of 
$\vec\ub$ with the Young driver $\vXX^\o$.
In Section \ref{SEC:young}, 
we prove pathwise local well-posedness
of the YDE \eqref{bd3}
and present a proof of Theorem \ref{THM:1}.

\medskip


\noi
\underline{\bf Part 2: Rough case.}
Next, we state our  main result on pathwise local 
well-posedness of SNLW \eqref{SNLW}
in the rough case
(namely, the white-in-time case with the Hurst parameter $\be = \frac 12$).

\begin{theorem}[rough case]
\label{THM:2}

Let  $d \ge 1$ and $\be = \frac 12$. 
Let $\zeta  = \xi = \dt  W^\frac 12 $ be as in \eqref{zeta}, 
where 
$\xi$ is the space-time white noise, satisfying~\eqref{white}, 
and 
$W^\frac 12 $ is the $L^2$-cylindrical Wiener process  in \eqref{W0}.
  Given  an integer  $k \geq 2$, 
suppose that one of the conditions 
in \eqref{reg1} holds{\rm :}
\begin{align*}
 \textup{(i) }& 
s >  s_\textup{crit}, 
\ \text{ when  $(d,k) \in \{(2,2), (2,3), (3,2)\}$, \  or}\\
 \textup{(ii) } & 
 s \ge   s_\textup{crit}, 
\ \text{ otherwise}, 
\end{align*}

\noi
where  $s_\textup{crit}$ is as in
\eqref{crit1} when $d \ge 2$
and~\eqref{crit3} when $d = 1$.
Furthermore, suppose that 
$ \s \in \R$ satisfies \eqref{sigma0}{\rm :}
\begin{align*}
\s>\max (-s, s-1).
\end{align*}

\noi
Then, 
given  $\Phi \in \HS(L^2(\T^d); H^\s(\T^d))$, 
satisfying~\eqref{phi1}, 
 SNLW \eqref{SNLW}
 with a multiplicative white-in-time noise
is pathwise locally well-posed in $\H^s(\T^d)$.

\end{theorem}

In the following, we discuss  a meaning of pathwise solutions
constructed in Theorem \ref{THM:2}.
In the rough case (= the white-in-time case
with the Hurst parameter $\be = \frac 12$), 
the bound \eqref{bd1} on the random driver $\vXX$
holds
only for $\al < \frac 12$.
By noting that 
 $\Phi W^\frac 12  $ is $\al$-H\"older continuous in time
 for any $\al < \frac 12$, 
we expect from \eqref{mild2}
that   
 $\vec\ub$  also has temporal regularity\footnote{In this informal discussion, 
``regularity'' refers to the H\"older regularity (in time).  For our actual application, 
it means the $V^\frac 1\al$-regularity.  See \eqref{reg9}.
} $\al < \frac 12$.
Hence, due to the deficiency of (temporal) regularities
(namely,  $\al + \al <  1$), 
the Young integral $\I^{\vXX}(\vec \ub)$, 
formally given in \eqref{bd1x}, 
does not make sense in this case, 
and thus we need to resort to rough path theory, 
introduced in a seminal work by Lyons
\cite{Lyons98}, 
and controlled paths, introduced by  Gubinelli \cite{G04};
 see also 
\cite{LCL07, FV10, FH20}
for textbook accounts on rough path theory.

Define the second order 
operator-valued 
two-parameter process 
$\vbbX$ by setting
\begin{align}
\label{bX0}
\vbbX_{t, r} = 
\vXX_{t,r}\circ \vXX_{\bullet, r}, 
\end{align}

\noi
where 
 $\bul$ denoting the variable of integration
 and $\circ$ denotes the composition of
 matrix-valued operators.
See \eqref{bbX0} and \eqref{bbX0a}.

Fix $\frac 13 < \al < \frac 12$.
Suppose that we have 
\begin{align}
\| \vbbX_{t,r}\|_{\H^s \to \H^s} \le C_\o |t-r|^{2\al}
\label{bd4}
\end{align}

\noi
for any $0 \le r \le t \le T$, 
where $C_\o$ is an almost surely finite random constant.
Then, together with the bound \eqref{bd1}, \eqref{introX}, 
and \eqref{bX0}, 
one can check  that the pair $(\vXX, \vbbX)$ is an 
operator-valued $\al$-H\"older rough path (see \cite[Definition 2.1]{FH20})
and thus is an operator-valued $\frac 1\al$-variation rough path
(see Definition \ref{DEF:RP}\,(i)).

In view of \eqref{mild2}, we now impose
the following controlled structure on $\vec \ub$
(namely, $\vec \ub$ is controlled by $\vXX$ in the sense of 
Definition \ref{DEF:RP}\,(ii)):
\begin{align}
\label{bd5}
\vec\ub(t) - \vec\ub(r)  = \vXX_{t,r} (\vec\ub'(r)) + \vec R^{\vXX, \vec\ub}_{t,r},
\end{align}

\noi
for $t\ge r\ge 0$, 
where $\vec\ub'$ is  the so-called Gubinelli derivative of $\vec\ub$ and $\vec R^{\vXX,  \vec\ub}$ is a smoother (in time) two-parameter process with temporal regularity $\g > \frac 12 \, (> \al)$;
see \eqref{Ho2} in the H\"older regularity context
and \eqref{V2def} with \eqref{reg9} in the $p$-variation context.
Then, from~\eqref{bd5} with \eqref{bd1}, 
we expect that 
$\vec \ub$ has regularity $\al$.
Hence, 
if $0 < \al < \frac 12 < \g < 1$ satisfies
\begin{align}
\min ( \al + \g, \, 3\al )>1, 
\notag 
\end{align}

\noi
a standard rough integration theory for H\"older continuous functions, 
based on the sewing lemma (see \cite{G04, GT10, FH20}), 
allows us to construct 
the stochastic term $\vec{\pPsi} (\vec\ub)$ in \eqref{psi2},  in the pathwise manner, 
as the {\it rough integral} $\I^{\vXX, \vbbX}(\vec \ub)$
of  $\vec\ub$
with the rough driver $(\vXX, \vbbX)$, 
given as the unique limit of 
Riemann-Stieltjes type sums; see \eqref{int2d}.
See Subsection \ref{SUBSEC:sewR}
for a review on the construction of rough integrals.

This discussion leads
us to 
 interpret 
the integral equation \eqref{mild2}
for the interaction representation $\vec \ub = \S(-t) \vec u(t)$
as
the following {\it rough differential equation} (RDE):\footnote{In this pathwise discussion, 
we work with
a fixed realization 
 $(\vXX^\o, \vbbX^\o)$ (for almost every $\o \in \Om$).
 For simplicity of notation, however, we drop the superscript $\o$ in the following.}

\begin{align}
\vec{\ub}(t) &= (\phi_0,\phi_1) + 
\vec{\NN}({\vec{\ub}})(t) 
+ \I^{\vXX, \vbbX}(\vec\ub)(t), 
\label{bd7}
\end{align}

\noi
where $\vec{\NN}({\vec{\ub}}) $
is the nonlinear  term in \eqref{bd2}.
In most cases, 
we can in fact study the system \eqref{bd7} and \eqref{bd5}
for two unknowns $\vec \ub$  and $\vec \ub'$, 
using the unified framework
(namely, using the $U^p$- and $V^p$-spaces)
developed in the Young case described above.\footnote{We 
note that we need
to take both $\al$ and $\g$ to be sufficiently close to $\frac 12$
to prove local well-posedness of the 
system \eqref{bd7} and \eqref{bd5} near or at the critical regularity $s_\text{crit}$.}

Unfortunately, 
this approach fails in the 
 following case:
\begin{align}
\label{bd8}
d \ge 6, \quad
k=2, \quad
\text{and}
\quad
s=s_{\text{crit}} \ge 1 ,
\end{align}
where $s_{\text{crit}}$ is as in \eqref{crit1}.
In view of the 
linear estimate \eqref{lin3}
of Lemma~\ref{LEM:lin}, 
we need to estimate $u^2$ in $L^1_t$.
Namely, the only useful Strichartz
admissible pair $(q, r)$
in this case 
must come with $q = 2$.
In view of the  transference principle (Lemma \ref{LEM:STR}), 
we need to impose
$\vec\ub \in U^{2}_T \H^s(\T^d)$.
On the other hand, 
assuming that 
the rough integral $\I^{\vXX, \vbbX}(\vec\ub)$ 
described above 
makes sense, 
we expect 
$\vec\ub \in U^{q }_T \H^s(\T^d)\setminus U^2_T \H^s(\T^d)$
for $q= \frac 1\al >2$, 
which implies that 
$\vec\ub$ does {\it not} have sufficient temporal regularity 
to treat the nonlinear  term 
$\vec{\NN}({\vec{\ub}})$
in \eqref{bd2}
via the Strichartz estimates.

We overcome  this issue by imposing a further structure on the unknown
$\vec \ub$. 
More precisely, we write $\vec \ub$ as  
\begin{align}
\vec\ub = \vec\vb + \vec\wb, 
\label{exp0}
\end{align} 

\noi
where $\vec\vb$,  $\vec\wb$, 
and 
the Gubinelli derivative $\vec\ub'$ (of $\vec \ub  =\vec\vb + \vec\wb$)
satisfy the following RDE-PDE system:
\begin{align}
\begin{split}
\vec\vb & = \I^{\vXX, \vbbX}(\vec\vb + \vec\wb),  \\
\vec\wb(t)& 
 = (\phi_0, \phi_1) + \vec\NN(\vec\vb + \vec\wb)(t),
\end{split}
\label{bd9a}
\end{align}

\noi
and
\begin{align}
(\vec \vb(t) + \vec \wb(t))
& - (\vec \vb(r) + \vec \wb(r))
  = \vXX_{t,r} (\vec\ub' (r))   + \vec{R}^{\vXX,  \vec\vb+ \vec \wb}_{t,r}
\label{bd9b}
\end{align}

\noi
\noi
for $0 \le r\le t \le T$.
Here, 
$\vec{R}^{\vXX, \vec\vb+ \vec \wb}$
denotes a smoother (in time)
 two-parameter process
 with temporal regularity $\g > \frac 12 \, (> \al)$
 and 
$\I^{\vXX, \vbbX}(\vec\vb + \vec\wb)$
denotes the rough integral of $\vec\vb + \vec\wb$
with the rough  driver $(\vXX, \vbbX)$.
Furthermore, 
we impose the following regularity conditions on $\vec \vb$ 
and $\vec \wb$: 
\smallskip
\begin{itemize}
\item $\vec \vb$ is rougher in time  but smoother in space, 
namely,   $\vec \vb \in U^\frac 1\al \H^{s+\eps}(\T^d)$
for some small $\eps > 0$, 
\smallskip
\item $\vec \wb$ is smoother in time but rougher in space, 
namely,   $\vec \wb \in U^2 \H^{s}(\T^d)$.

\end{itemize}

\noi
We verify these regularity conditions 
by 
establishing that  the rough  driver $(\vXX, \vbbX)$
exhibits spatial smoothing
(see \eqref{AS2})
and 
performing a contraction argument
to the RDE-PDE system \eqref{bd9a}-\eqref{bd9b}.
A key ingredient here is the improved
nonlinear estimates; see \eqref{nonl000} in Proposition \ref{PROP:nonlin3}.
See also Remark \ref{REM:nonlin5}.
In Section \ref{SEC:rough}, 
we prove pathwise local well-posedness
of the RDE-PDE system \eqref{bd9a}-\eqref{bd9b}; 
see Proposition \ref{PROP:LWP2}.
We note that 
it is crucial to take both $\al < \frac 12$ and 
$\g > \frac12$
to be  sufficiently close to $\frac 12$.

As a consequence of a contraction argument, 
a solution $(\vec \vb, \vec \wb, \vec \ub')$
satisfies
the compatibility condition 
$ \vec \ub' = \vec \vb+ \vec \wb = \vec \ub$, where the second equality follows from \eqref{exp0}.
Moreover, from~\eqref{exp0} and \eqref{bd9a}, 
we see that 
our pathwise solution 
$\vec\ub = \vec\vb + \vec\wb$ satisfies the RDE \eqref{bd7}, 
where we interpret the stochastic term $\vec{\pPsi} (\vec\ub)$
in \eqref{psi2}
as the rough integral 
$\I^{\vXX, \vbbX}(\vec\ub)$. 
We note that uniqueness
of the solution 
$\vec\ub$
follows from the uniqueness of the solution 
$(\vec \vb, \vec \wb, \vec \ub')$
to the RDE-PDE system 
\eqref{bd9a}-\eqref{bd9b}
via the identifications $\vec\ub = \vec\vb + \vec\wb$ 
and $\vec\ub = \vec\ub'$.
See 
Section \ref{SEC:rough}
for a proof of Theorem \ref{THM:2}.

 \medskip

\subsection{Remarks \& comments}

In this subsection, we state remarks and comments.

\begin{remark} \rm
(i)
The framework  introduced in this paper is general and can be
adapted to study pathwise well-posedness of other stochastic dispersive PDEs
with multiplicative noises.
See \cite{CGLLO2,  CLOO}.
In the case of SNLW \eqref{SNLW}, 
the smoothing coming from the 
linear wave propagator $S(t)$ in \eqref{S0}
in fact plays a crucial role
in establishing \eqref{bd1} and \eqref{bd4}
on the random drivers $\vXX$ and $\vbbX$;
see Propositions \ref{PROP:driver1} and \ref{PROP:driver2}.
For  the  stochastic NLS with a multiplicative white-in-time noise, 
the bounds
\eqref{bd1} and \eqref{bd4} do not hold
for the associated first order and second order drivers
$\XX$ and $\bbX$.
Such bounds hold only with a slight loss in spatial regularity.
Namely, temporal roughness
induces   a loss in spatial regularity, causing
a substantial difficulty as compared to SNLW
studied in this paper.
We will address this issue in forthcoming works
\cite{CLO2, CLOZ}.

\smallskip

\noi
(ii) In \cite{GT10}, 
Gubinelli and Tindel developed
the convolution Young\,/\,rough integration theory
and proved pathwise local well-posedness
of the stochastic heat equation (SHE)  with a multiplicative noise, 
where the operator norms  of the random drivers
were estimated by studying their Hilbert-Schmidt norms.
In \cite{SLO}, 
the second and third authors with Y.\,Shao 
applied the random tensor estimate
(Lemma \ref{LEM:RTE})
for multiple stochastic integrals with respect to fractional Brownian motions
to study pathwise local well-posedness of 
SHE and improved the results in \cite{GT10}
in both the Young and rough cases.
See \cite{SLO} for a further discussion.

\smallskip

\noi
(iii) In the study  of dispersive PDEs, 
the sewing lemma was used in several contexts.
In \cite{G12}, Gubinelli developed
nonlinear rough paths (see also \cite{GT10, DGT, NX}) and used the sewing lemma, 
as a replacement of the Fourier restriction norm method, 
to study local well-posedness of  the deterministic KdV equation.
See also \cite{OShao}.
In   \cite{CG1, CGLLO1, GGLO}, 
the authors studied 
well-posedness issues of modulated dispersive PDEs
by constructing the nonlinear Duhamel integral terms
as a nonlinear Young integral via the sewing lemma.

\end{remark}

\begin{remark}\label{REM:ill2}\rm
(i) 
In view of Remark \ref{REM:ill}
on the deterministic NLW \eqref{NLW}, 
the spatial regularity condition 
\eqref{reg1}
imposed in Theorems \ref{THM:1} and \ref{THM:2}
is essentially sharp.
See also Remark \ref{REM:opt}.

\smallskip

\noi
(ii) 
In this paper, we considered SNLW \eqref{SNLW}
with an integer $k \ge 2$ for simplicity of presentation.
It is possible to adapt our analysis to the nonlinearity 
of the form 
$|u|^{k-1} u$ for a non-integer value of $k > 1$.
In doing so, one needs to 
impose the compatibility condition $k > \lfloor s\rfloor + 1$ when $k$ is not an odd integer
to ensure smoothness of the nonlinearity.
See~\cite{CCT} for a further discussion.

\smallskip

\noi
(iii) 
While  we 
chose to work with 
a noise of the form 
$u \Phi \ze$
in this paper,
our approach also extends to treat a more general noise
of the form $F(u) \Phi \ze$
for a bounded  continuous function~$F$
with a suitable regularity.
A required modification is straightforward, 
and thus we omit details.

It is of interest to study pathwise local well-posedness
of SNLW \eqref{SNLW}
with 
a noise of the form $u^\l\Phi \ze $ for an integer $\l \ge 2$.
In the white-in-time case, 
one needs to employ 
 nonlinear rough path theory developed in 
\cite{G12, GT10}; see also \cite{OShao}.
We mention  a recent work \cite{CLOO}, 
where we
studied pathwise local well-posedness
of the stochastic KdV equation with 
a bilinear noise 
$u^2 \Phi \ze $ by making use of multilinear dispersion on the noise term.

\end{remark}

\begin{remark}\rm 
(i) In the rough case (= white-in-time case with $\be = \frac 12$), we introduced 
 the second order driver $\vbbX$
  in \eqref{bX0}.
In view of \eqref{bbX0a} with \eqref{XX0a}, 
each component of the matrix-valued operator
$\vbbX$ is given by 
an iterated stochastic integral.
In this paper, we interpreted it
as an iterated Wiener-Ito integral.
In Appendix \ref{SEC:ito}, 
we show that our pathwise solutions constructed 
in Theorem \ref{THM:2} 
agree
with Ito solutions to SNLW \eqref{SNLW}, 
where we interpret
the stochastic convolution $\Psi(u)$ in \eqref{psi1}
as an Ito integral; see Proposition \ref{PROP:ito}.
Theorem \ref{THM:2} also holds
if we interpret
the pair $(\vXX, \vbbX)$ as a geometric rough path;
see \cite[Section 2.2]{FH20}.
In this case, pathwise solutions agree
with Stratonovich solutions to SNLW \eqref{SNLW}, 
where we interpret
the equation \eqref{SNLW}
in the  Stratonovich sense.
See \cite[Chapter 5]{FH20} for a further discussion.

\medskip

\noi
(ii)
Let  $\be = \frac 12$.
In the defocusing case (with $+ u^k$, $k \in 2\N+1$, in \eqref{SNLW}), 
in view of the agreement with Ito solutions, 
pathwise local-in-time solutions to SNLW \eqref{SNLW}
constructed in Theorem \ref{THM:2} with $s = 1$
can be extended globally in time
in the energy (sub)critical regime (namely, when $s_\text{scaling} \le1$, 
where $s_\text{scaling}$ is as in \eqref{scaling}).
While it is not explicitly written in the literature, 
one can adapt the argument by 
de Bouard and Debussche \cite{DD1, DD2}, 
developed in the context of stochastic NLS, 
and establish an a priori $\H^1$-bound
on Ito solutions to SNLW~\eqref{SNLW}
via the energy conservation for the deterministic NLW \eqref{NLW}, 
Ito's lemma, 
and the Burkholder-Davis-Gundy inequality.
In the energy-subcritical case ($s_\text{scaling} <1$), 
this allows us to 
 conclude global well-posedness
in $\H^1(\T^d)$.
One may also adapt the argument in 
\cite{CLOx} via the $I$-method in the stochastic setting
to extend
global well-posedness
to some  $s < 1$.
See 
\cite{ALPZ}
for  the energy-critical case ($s_\text{scaling} =1$).
In the energy-supercritical case 
($s_\text{scaling} > 1$), 
we do not expect global well-posedness in general
in view of a recent work 
\cite{SWZ}
on the energy-supercritical defocusing NLW.

It is of interest to investigate 
the issue of global well-posedness
of SNLW \eqref{SNLW}
(with the defocusing nonlinearity $+u^k$, $k \in 2\N+1$)
in the Young case.

\end{remark}

\begin{remark}\rm
(i) In the rough case, 
the decomposition \eqref{exp0}
is needed to treat the case~\eqref{bd8}.
In other situations, 
Theorem \ref{THM:2}
follows from  directly working on $\vec \ub$ (and 
its Gubinelli derivative $\vec \ub'$)
and studying the RDE \eqref{bd7} coupled with the controlled structure \eqref{bd5}.
For simplicity of presentation, however, 
we present a unified approach and work on the RDE-PDE system
\eqref{bd9a}-\eqref{bd9b}.

\smallskip

\noi
(ii)
Let us mention
and compare several works, where 
a decomposition of an unknown into two unknowns
as in \eqref{exp0} plays a crucial role.
In \cite{GKO2}, paracontrolled calculus, originally introduced
for stochastic parabolic PDEs \cite{GIP, CC, MW1}, 
was adapted to study SNLW with an additive noise.
After the second order expansion via the explicit stochastic terms (see \cite[(1.13)]{GKO2}), 
the remainder part $v$ was decomposed into $v = X+ Y$, 
where $X$ is rougher-in-space but structured, 
and $Y$ is smoother-in-space;
see \cite[(1.25) and  (1.26)]{GKO2}.
In this work, the fact that the $X$-equation
has a structure played an important role
in deriving the final system 
\cite[(1.36)]{GKO2}.
See \cite{OOT1, Bring, OOT2, BDNY24}
for  subsequent works in this direction.
In \cite{ORW21}, 
the authors studied the hyperbolic Liouville equation
(SNLW with an additive space-time white noise
and an exponential nonlinearity), 
and the remainder part $v$ (after the first order expansion
as in \cite{McK, BO96, DD})
was decomposed into $v = X+ Y$, 
where, once again,  $X$ is rougher-in-space but structured, 
and $Y$ is smoother-in-space;
see \cite[(1.54) and (1.56)]{ORW21}.
In this work, the sign-definite structure of the rougher part
$X$ played a crucial role.
We also mention the work~\cite{DNY1}, 
where the (deterministic) derivative NLS was studied
(with deterministic initial data).
In this work, the equation for the (gauged) unknown $v$
was decomposed into a system for $v$
and the additional part $w$
where~$w$ is smoother in time; 
see \cite[(4.16) and (4.17)]{DNY1}.
See also \cite{Cha}.

In the current paper, the reduction of \eqref{mild2}
to the RDE-PDE system \eqref{bd9a}-\eqref{bd9b}
was done in two steps.
In the first step, we imposed
 the controlled structure \eqref{bd5}
 and reduced~\eqref{mild2}
 into a system \eqref{bd7} and \eqref{bd5}
 for two unknowns $\vec \ub$ and $\vec \ub'$
 (or we can view this as a system for two 
 unknowns 
 $\vec \ub'$ and $\vec R^{\vXX, \vec\ub}$, 
 where the latter is smoother in time).
 We note that this first step can be replaced
 by a paracontrolled ansatz (in time).
 In the second step, 
 we applied 
the decomposition \eqref{exp0}
and arrived at 
 the RDE-PDE system \eqref{bd9a}-\eqref{bd9b}
for three unknowns
$\vec \vb$, 
$\vec \wb$,  and $\vec \ub'$.
One notable difference
with the works \cite{GKO2, ORW21, DNY1}
is that 
 $\vec \vb$ is rougher in time  but smoother in space, 
while  $\vec \wb$ is smoother in time but rougher in space.
Namely, each part has a (small) gain of regularity in one of the temporal
or spatial direction.

\end{remark}

\begin{remark}\label{REM:opt}\rm

Note that the condition \eqref{sigma0}  implies $\s > - \frac 12$.
In particular, when $d = 1$, 
Theorem \ref{THM:2} in the rough case (with $\be = \frac 12$) 
covers  the case of an almost space-time white noise
(for example, when $\Phi = \jb{\dx}^{-\eps}$ for any $\eps > 0$).
We note that this is optimal 
 within the framework of one-parameter rough paths.

Let us briefly describe 
the main difficulty of SNLW with a multiplicative space-time white noise, 
posed on the one-dimensional torus $\T$:
\begin{align}
\dt^2 u +(1-\dx^2 )u  \pm u^k = u \xi.
\label{SNLW2}
\end{align}

\noi
See \cite{Walsh86, Oh} for the Ito solution theory.
See also Remark \ref{REM:white1}.
The space-time white noise $\xi$
has regularity $- \frac 12 - \eps$
in both the temporal and spatial directions, 
from which we expect that 
the unknown 
has regularity $\frac 12 - \eps$
in both the temporal and spatial directions
(which is precisely the regularity 
of the stochastic convolution $\Psi(u)$ 
in~\eqref{psi1} with $u \equiv 1$).
Namely,
we have 
 the deficiency of regularities
in {\it both} the temporal and spatial directions, 
and thus the usual one-parameter rough path analysis is not sufficient.
See~\cite{CG}, where a partial progress was made in the bi-parameter setting.
We will address this issue in a forthcoming work.
We mention 
a recent  work 
\cite{BLS24}, 
where such bi-parameter analysis was developed in
the study of the one-dimensional wave maps equation.
See also 
 \cite{HOO}
 for a further discussion on \eqref{SNLW2}.

\end{remark}

\begin{remark}\rm

In this paper,  we only consider $\be = \frac 12$
as  the rough case.
As long as we can extend the bounds
\eqref{bd1} and \eqref{bd4}
to lower values of $\al$ (with $\al = \be - \eps$ for any small $\eps > 0$), 
the rough integration theory based on the second order rough path\footnote{We regard the Young driver $\vXX$ 
as the first order rough path.}
($\vXX, \vbbX)$ in Subsection~\ref{SUBSEC:sewR}
allows us to construct the rough integral $\I^{\vXX, \vbbX}(\vec \ub)$ 
in the pathwise manner, provided that  $\be > \frac 13$.
We note that as $\be$ decreases, 
higher spatial regularity is needed for  pathwise local well-posedness
of SNLW \eqref{SNLW}
(in order to handle 
the nonlinear  term in~\eqref{bd2}).
We will address this issue in a future  work.
Lastly, we mention 
the works \cite{Deya, Deya2, Deya3} 
on SNLW, forced by an additive fractional-in-time noise
with the Hurst parameter $\be < \frac 12$.

\end{remark}

\subsection{Organization}

In Section \ref{SEC:2}, 
we introduce basic notations and  function spaces, 
including spaces for multi-parameter functions (Subsection \ref{SUBSEC:2.2})
and the $U^p$- and $V^p$-spaces (Subsection~\ref{SUBSEC:Up}).
We then provide a brief review on 
multiple stochastic integrals with respect to fractional Brownian motions
(Subsection \ref{SUBSEC:FBM})
and establish the random tensor estimate (Lemma \ref{LEM:RTE})
in this setting.
We also recall Kolmogorov's continuity criterion 
for rough paths (Lemma \ref{LEM:kolm}).
In Section \ref{SEC:int}, 
we go over
 Young\,/\,rough integration theory for functions of bounded $p$-variation, 
 based on the sewing lemma 
(Lemma \ref{LEM:sew}).
The materials presented in Section \ref{SEC:int} may be standard by
now, but we decided to include them in an accessible manner for readers in dispersive PDEs
who may not be familiar with the subject.
In Section \ref{SEC:ops}, 
we study regularity properties
of  the random operators 
$\vXX$ and $\vbbX$, 
 defined in \eqref{introX} and~\eqref{bX0}, respectively, 
 using the random tensor estimate.
In Section \ref{SEC:nonlin}, 
we establish (deterministic) nonlinear estimates
in the $U^p$-setting, in particular
using the Strichartz estimates for $d \ge 2$.
In Section~\ref{SEC:young}, 
by reducing~\eqref{mild2}
to the YDE~\eqref{duhamelY}, 
we prove pathwise local well-posedness
of SNLW \eqref{SNLW} in the Young case
(Theorem~\ref{THM:1}).
In Section \ref{SEC:rough}, 
by reducing \eqref{mild2}
to the RDE-PDE system  \eqref{bd9a}-\eqref{bd9b}, 
we prove pathwise local well-posedness
of SNLW \eqref{SNLW} in the rough case
(Theorem~\ref{THM:2}). 
 In 
Appendix~\ref{SEC:ito}, 
after going over the Ito solution theory (see Proposition~\ref{PROP:LWPx} and 
Remark~\ref{REM:white1}), 
we show that, in the white-in-time case (= rough case),  our pathwise solutions agree
with Ito solutions.
In Appendix
\ref{SEC:B}, 
we provide a brief discussion 
on an optimizer for the local-in-time $V^p$-norm
defined in \eqref{local}.
In Appendix \ref{SEC:sew}, we present a
proof of the sewing lemma (Lemma \ref{LEM:sew})
adapted to our setting.
In Appendix \ref{SEC:D}, 
we go over the Burkholder-Davis-Gundy
inequality on Banach spaces, which we used in 
Lemma \ref{LEM:stoconv1}
on the construction of the stochastic convolution 
as an Ito integral.


\section{Notations and preliminaries}
\label{SEC:2}

\subsection{Basic notations}

We use $A \les B$ to denote an estimate of the form $A \le CB$ for some constant $C>0$. We write $A\sim B$ if $A \les B$ and $B \les A$, and $A \ll B$ if $A\le c B$ for some small constant $c> 0$.
We may write  $\les_{\al}$ 
to emphasize the dependence on an external parameter $\al$.
We often use $C>0$ to denote various constants, which may vary line by line, 
and use $C_\al$
to denote the  dependence on an external parameter $\al$.
We use  $\al+$ and $\al-$ 
to denote $\al+\eps$ and $\al-\eps$ for arbitrarily small $\eps>0$, respectively.

Given $a, b \in \R$, we set 
\begin{align}
a \land b = \min(a,b)
\qquad \text{and} \qquad a \vee b = \max(a,b).
\label{ord2}
\end{align}

Given a time dependent function $u$, we often use the short-hand notation $u_t = u(t)$, which is common practice in stochastic analysis. Similarly, for a function parametrized by $n$-variables 
$(t_1,\ldots, t_n)$, we write $u_{t_1, \ldots, t_n} = u(t_1, \ldots, t_n)$.

We use $\vec f$ to denote a vector $\vec f = (f_1, f_2)$.
Given a $C^1$-function $f$, 
we set $\vec f = (f, \dt f)$.
We may use 
column vectors and row vectors interchangeably 
with the understanding that 
they all represent column vectors.
Given  a vector $\vec{f} = (f_1,f_2)$, we use $\pi_1$ and $\pi_2$ to denote the projection onto 
its first and second  components, namely
\begin{align}
\pi_1 \vec f = f_1 \qquad \text{and} \qquad \pi_2 \vec f = f_2.
\label{vec1}
\end{align}

Given $N\in \N$, we define $\P_N$ as the Dirichlet projector onto frequencies $\{n\in\Z^d: \ |n| \le N\}$:
\begin{align*}
\P_N f (x) = \sum_{\substack{n\in\Z^d \\ |n| \le N} } \ft{f}(n) e^{i n\cdot x},
\end{align*}
and we set $\P_N^\perp = \Id - \P_N$.
With a slight abuse of notation, 
we also use  $\P_N$ and $\P_N^\perp$ to
denote frequency projectors  acting on vectors 
$\vec f = (f_1, f_2)$ as follows:
\begin{align}
\begin{split}
\P_N \vec{f}  = (\P_N f_1, \P_N  f_2)
\qquad \text{and}\qquad
\P_N^\perp \vec{f}  = (\P^\perp_N f_1, \P^\perp_N f_2).
\end{split}
\notag 
\end{align}

Given $n_{j_1}, \dots, n_{j_k} \in \Z$, 
we set 
\begin{align}
n_{j_1 \cdots j_k} = 
n_{j_1}+ \cdots+ n_{j_k}.
\label{ord3}
\end{align}

\noi
For example, 
$n_{123} = n_1 + n_2 + n_3$.

Given a dyadic number $N \ge 1$, 
we use the notation $|n| \sim N$
to denote a dyadic spatial frequency
restriction via the Littlewood-Paley projector
(which we omit for notational convenience)
with the understanding that $|n| \sim N$ implies 
$|n| \les N$ when $N = 1$.
Given dyadic numbers $N, N_1, N_2\ge 1$, 
we use $N_{\max}$, $N_{\med}$, and $N_{\min}$
to denote their decreasing rearrangement:
\begin{align}
N_{\max} \ge N_{\med} \ge N_{\min}.
\label{ord1}
\end{align}

When writing the norm of a space-time function, we usually use the short-hand notation
 such as $L^q_I L^r_x = L^q(I; L^r(\T^d))$ for a time interval $I\subset\R$ and $L^q_T L^r_x = L^q([0,T]; L^r(\T^d))$.

Let $X, Y, Z$ be Banach spaces. We denote by $\LOP(X;Y)$ the space of bounded linear operators from $X$ to $Y$, equipped with the operator norm:
\begin{align*}
\| T \|_{\LOP(X;Y)} = \sup_{\|f\|_{X} = 1} \| T f \|_{Y}.
\end{align*}

\noi
When $X=Y$, we simply set  $\LOP(X) = \LOP(X;X)$.
Given $T_1 \in \L(X; Y)$
and $T_2 \in \L(Y; Z)$, 
we use $T_2 T_1$ to denote the composition
$T_2 \circ T_1$ of the operators $T_1 $ and $T_2$.

For $s\in\R$, $1\le p \le \infty$,  we define the Sobolev spaces $H^s(\T^d)$, 
$W^{s,p}(\T^d)$, and $\H^s(\T^d) = H^s(\T^d) \times H^{s-1}(\T^d)$ via the norms:
\begin{align}
\| f \|_{H^s(\T^d) } & = 
\| \jb{\nb}^s f \|_{L^2(\T^d)}
= \| \jb{n}^s \ft{f}(n) \|_{\l^2_n(\Z^d)} , 
\notag \\
\| f\|_{W^{s,p}(\T^d)} &  = 
\| \jb{\nb}^s f \|_{L^p(\T^d)}
= 
\| \Ft^{-1}(\jb{n}^s \ft{f}(n)) \|_{L^p(\T^d)}, 
\notag \\
\| \vec{f} \|_{\H^s(\T^d)} 
& = 
\Big(\| \pi_1 \vec{f} \|_{H^s(\T^d)}^2 + \| \pi_2 \vec{f} \|_{H^{s-1}(\T^d)}^2\Big)^\frac 12 ,
\label{Hs}
\end{align}
where $\jb{n} = (1+|n|^2)^{\frac12}$ and $\Ft^{-1}$ denotes the inverse Fourier transform. 
We note that $W^{s,2}(\T^d) = H^s(\T^d)$. 
We use 
\[(\H^s(\T^d))^* = H^{-s}(\T^d) \times H^{-s+1}(\T^d)\]
to denote the dual space of $\H^s(\T^d)$.

\medskip

We conclude this subsection by recalling
some basic deterministic estimates.
We first recall the  fractional Leibniz rule; see \cite[Proposition~2]{BOZ}.
See also \cite[Lemma 3.4]{GKO1}.

\begin{lemma}[fractional Leibniz rule]
\label{LEM:leibniz}

Let $d\ge1$, $s \ge 0$, 
$1<p_j, q_j \le \infty$,  
and $1\le r \le \infty$ with  $\frac{1}{p_j} + \frac{1}{q_j} = \frac{1}{r}$,  $j=1,2$. Then, we have
\begin{align*}
\| \jb{\nabla}^s (fg) \|_{L^r(\T^d)} & \les \| f \|_{L^{p_1}(\T^d)} \| \jb{\nabla}^s g \|_{L^{q_1}(\T^d) } +  \|\jb{\nabla}^s  f \|_{L^{p_2}(\T^d)} \| g \|_{L^{q_2}(\T^d) } .
\end{align*}

\end{lemma}

\smallskip

We also recall an elementary calculus result. 
See, for example,  \cite[Lemma 4.2]{GTV97} for a proof.

\begin{lemma}\label{LEM:conv}

Let $\al\ge\be\ge0$ and $\al+\be>1$. Then,
\begin{align*}
\int_\R \frac{1}{\jb{\xi_1}^\al \jb{\xi - \xi_1}^\be} d\xi_1 \les \frac{1}{\jb{\xi}^{\g}},
\end{align*}
where $\g \in\R$ is given by
\begin{align*}
\g =
\begin{cases}
\be, &\text{if } \al > 1, \\
\be-, & \text{if } \al = 1, \\
\al+ \be -1 , &\text{if } \al <1.
\end{cases}
\end{align*}
\end{lemma}

\subsection{Increments}
\label{SUBSEC:2.2}

 Let $X$ be a Banach space and $T>0$.
 For 
 $n\in\N$, we define the $n$-simplex: 
\begin{align*}
\Dl_{n, T} &:= \big\{ (t_1, \ldots, t_n) \in [0,T]^n: \ t_i \ge  t_j \text{ for } i < j\big\}.
\end{align*}

\noi We denote by $C_{n,T} X$ the space of continuous functions 
$f$
from $\Dl_{n,T}$ to $X$ such that 
\begin{align}
f_{t_1, \dots, t_n} = 0
\label{vani1}
\end{align}

\noi
whenever $t_j = t_{j+1}$ for some $j = 1, \dots, n-1$.
When $n=1$, we may write $ C_T X$ and equip  the space  with the norm:
\begin{align*}
\| f \|_{C_T X} := \| f\|_{L^\infty_T X } = \sup_{0\le t < T } \| f(t)\|_{X}.
\end{align*}

\noi
Given $f\in C_{n,T}X$ and $(t_1, \ldots, t_n)$, we often use the short-hand notation
$f_{t_1, \ldots, t_n} = f(t_1, \ldots, t_n)$.
\noi
We also introduce the coboundary operator 
$\updl: C_{n,T} X \to C_{n+1,T} X$
by setting \begin{align*}
(\updl f)_{t_1, \ldots, t_{n+1}} 
= \sum_{k=1}^{n+1} (-1)^{k} f_{t_1, \ldots,  t_{k-1}, t_{k+1}, \ldots, t_{n+1}}
\end{align*}

\noi
for $f\in  C_{n,T} X$ and $(t_1,\ldots, t_{n+1}) \in \Dl_{n+1,T}$.
For example, given  $f\in C_{1,T} X$ and $g\in C_{2,T} X$, we have
\begin{align}
\label{updl23}
(\updl f)_{t,r} & = f_t - f_r \qquad \text{and} \qquad
(\updl g)_{t_1,t_2,t_3}  = g_{t_1,t_3} - g_{t_1,t_2} - g_{t_2,t_3}
\end{align}

\noi
for  $(t,r) \in \Dl_{2,T}$ and $(t_1,t_2,t_3) \in \Dl_{3,T}$.
Recall from   \cite[Lemma 2.1]{GT10}  that
 the following cochain complex is exact:
\begin{align}
0 \too X \to C_{1,T} X \overset{\updl}{\too} C_{2,T} X \overset{\updl}{\too} 
C_{3,T} X \overset{\updl}{\too} \cdots.
\label{chain1}
\end{align}
In particular, $\updl\circ \updl =0$ and if $f \in C_{n,T} X$ with $\updl f =0$, then there exists (non-unique) $g\in C_{n-1, T}X$ such that $f= \updl g$.

\smallskip

Given $0<  \al \le 1$, we denote by $C^\al_{T} X = C^{\al}([0,T];X)$ the space of $\al$-H\"older continuous functions with values in $X$, endowed with the semi-norm:
\begin{align*}
\| f\|_{C^\al_{T} X} = \sup_{(t,r)\in\Dl_{2,T}} \frac{ \| (\updl f)_{t,r} \|_{X}}{|t-r|^\al}.
\end{align*}
We also  define the Lipschitz space\footnote{Also known 
as the H\"older-Zygmund space or the H\"older-Besov space.}  $\C_T^\al X = \C^\al([0,T]; X)$ via the norm:
\begin{align}
\| f \|_{\C^\al_{T} X} = \| f\|_{L^\infty_T X} + \| f \|_{C^\al_{T} X}.
\label{Lip1}
\end{align}

\noi
For $n = 2, 3$, we  define the spaces $C^\al_{n,T} X\subset C_{n,T} X$ equipped with the following H\"older-type norms:
\begin{align}
\begin{split}
\| g\|_{C^\al_{2,T} X} & = \sup_{(t,r) \in \Dl_{2,T}} \frac{\|g_{t,r} \|_{X} }{|t-r|^{\al}}, \\
\| h\|_{C^\al_{3,T}X } & = \inf_{0<\mu<\al} \sup_{(t_1,t_2,t_3) \in \Dl_{3,T}} \frac{\| h_{t_1,t_2,t_3}\|_{X}}{|t_1-t_2|^{\mu} |t_2-t_3|^{\al-\mu}}.
\end{split}
\label{Ho2}
\end{align}

\subsection{$U^p$- and $V^p$-spaces and related estimates}
\label{SUBSEC:Up}

In this subsection, we review the definitions and basic properties of the $U^p$- and $V^p$-spaces, developed by Tataru, Koch, and their collaborators \cite{KT05, KT07, HHK09, HTT11} in the context of dispersive PDEs.
The $V^p$-space denotes the space of functions of bounded $p$-variation, whose introduction dates back to the classical work of Wiener \cite{Wiener1924}, and whose restriction to continuous functions is often denoted by $C^{p-\textup{var}}$ in the rough path literature; see, for example,  \cite{FV10}.
The $U^p$-space, introduced in \cite{KT05, KT07}, is the predual of $V^{p'}$, where $p'$
denotes the H\"older conjugate of $p$.
We refer readers to \cite{HHK09, HTT11, Koch14, KT18} for detailed proofs of the relevant properties of these function spaces.

In the following, let $X$ be a Banach space.
Let $\Zc$ denote the collection of finite partitions $\{t_k\}_{k=0}^K$ of $\R$
of the form:
\begin{align}
 -\infty<t_K < t_{K-1} < \cdots< t_0 \le \infty.
 \label{U1}
\end{align}

\noi
 If $t_0 = \infty$, we use the convention that $u(t_0) = u(\infty) :=0$ for any function $u :\R \to X$. Note that the value of the limit $\lim_{t\to \infty} u(t)$ need not be 0.
Lastly, we use $\ind_I:\R\to\R$ to denote the sharp characteristic function on a set $I\subset \R$.

\begin{definition}
\rm
\label{DEF:Up}

Let $1 \le p <\infty$.

\smallskip
\noi{\rm(i)} 
We define a $U^p$-atom to be a step  function $a: \R\to X$ of the form
\begin{align*}
a = \sum_{k=0}^{K-1} \phi_{k} \ind_{[t_{k+1}, t_{k})}
\end{align*}

\noi
where $\{t_k\}_{k=0}^K \subset \Zc$ and $\{\phi_k\}_{k=0}^{K-1}\subset X$ 
satisfy $\sum_{k=0}^{K-1} \|\phi_k\|^p_X =1 $.

\medskip
\noi{\rm(ii)} We define the $p$-atomic space $U^p(\R; X)$ 
by 
\begin{align}
\begin{split}
U^p(\R; X) = \bigg\{u: \R \to X: \, 
& u = \sum_{j=1}^\infty \ld_j a_j\\
& \text{for $U^p$-atoms $\{a_j\}_{j=1}^\infty$, 
$\{\ld_j\}_{j=1}^\infty \in  \l^1(\N)$}\bigg\} 
\end{split}
\notag 
\end{align}

\noi
endowed with the norm:
\begin{align}
\begin{split}
\| u \|_{{U}^p(\R; X)}  =  
\inf \bigg\{ \|\ld\|_{\l^1}: \, 
& u = \sum_{j=1}^\infty \ld_j a_j\\
&  \text{for $U^p$-atoms $\{a_j\}_{j=1}^\infty$, 
$\{\ld_j\}_{j=1}^\infty \in  \l^1(\N)$}\bigg\},  
\end{split}
\label{U1a}
\end{align}

\noi
where the infimum is taken over all possible representations for $u$.

\medskip
\noi{\rm(iii)} We define ${V}^p(\R;X)$ as the space of functions of finite $p$-variation $u:\R \to X$ for which $\lim_{t\to\pm\infty}u(t)$ exist and  $u(\infty):=0$, 
endowed with the norm:
\begin{align}
\| u\|_{{V}^p(\R;X)} = \sup_{\{t_k\}_{k=0}^K \in \Zc} 
\bigg( \sum_{k=0}^{K-1} \| u_{t_{k}} - u_{t_{k+1}} \|^p_X \bigg)^\frac1p.
\label{U2}
\end{align}

\noi
We remark
that the functions in $V^p(\R; X)$ have left- and right-limits everywhere, and at most countably many
discontinuities.

We also define $V^p_{\rm rc}(\R;X)$ as the closed subspace of 
 right-continuous functions in $V^p(\R; X)$ such that $\lim_{t\to-\infty} u(t) =0$.

\medskip

\noi(iv) Let $s_1,s_2\in\R$. We define $U^p_{\Box}\big(\R; H^{s_1}(\T^d)\times H^{s_2}(\T^d)\big)$ and $V^p_{\Box}\big(\R; H^{s_1}(\T^d)\times H^{s_2}(\T^d)\big)$ to be the spaces of functions $\vec{u}, \vec{v}: \R \to H^{s_1}(\T^d)\times H^{s_2}(\T^d)$ such that the following norms are finite:
\begin{align}
\begin{split}
\| \vec{u} \|_{{U}^p_{\Box} (\R; H^{s_1}_x\times H^{s_2}_x)} & = \| \S(-t) \vec{u}\, 
\|_{{U}^p(\R; H^{s_1}_x\times H^{s_2}_x)}, \\
 \| \vec{v} \|_{{V}^p_{\Box} (\R; H^{s_1}_x\times H^{s_2}_x)}  &= \| \S(t)^\tf\,  \vec{v} \,
 \|_{{V}^p(\R; H^{s_1}_x\times H^{s_2}_x)},
\end{split}
\label{Dl1}
\end{align}

\noi
where $\S(t)$ is the matrix-valued linear wave propagator
 in \eqref{Sdef} and $\S(t)^\tf$ denotes its matrix transpose.
We emphasize that the latter space appears as a ``dual'' space;
see Lemma \ref{LEM:lin} and Subsection \ref{SUBSEC:nonl2}.

\end{definition}

Given an interval  $I\subset \R$, 
we define the time restriction norms by setting
\begin{align}
\begin{split}
\| u \|_{U^p_I X} 
& = 
\| u \|_{U^p(I; X)}
= \inf \big\{ \| \wt{u} \|_{U^p(\R;X)} : \ \wt{u}\vert_{I} = u \big\}, \\
\| u \|_{V^p_I X } &=
\| u \|_{V^p(I; X) } 
= 
 \inf \big\{ \| \wt{u} \|_{V^p(\R;X)} : \ \wt{u}\vert_{I} = u \big\}.
\end{split}
\label{local}
\end{align}

\noi
We also define 
the time restriction spaces
$U^p_{\Box, I}  H^{s_1}_x\times H^{s_2}_x$
and $V^p_{\Box, I}  H^{s_1}_x\times H^{s_2}_x$
in an analogous manner.
When $I=[0,T]$ for some $T>0$, we use the short-hand notation 
$U^p_T X$, $V^p_T X$, 
$U^p_{\Box, T} H^{s_1}_x\times H^{s_2}_x$, 
and $V^p_{\Box, T}  H^{s_1}_x\times H^{s_2}_x$.
Recall from \cite[Lemma A.1]{BOP2} that the infimum in the ${U}^p_T X$-norm is attained at 
$\ind_{[0,T]} \cdot u$. 
See Appendix \ref{SEC:B}
on an optimizer for the 
 infimum in the $V^p_T X$-norm.

In the following, we often  work with the closed subspaces 
$ \U^p_T X =\U^p([0,T];X) $ and 
$ \V^p_T X = \V^p([0,T]; X)$ 
of continuous
functions:
\begin{align}
\begin{split}
\U^p_{T} X  & : = U^p_T X \cap C_{T}X,\\
 \V^p_{T} X &: = V^p_T X \cap  C_{T} X.
\end{split}
\label{local2}
\end{align}

\begin{remark}\rm

\noi(i)
We note that the spaces $U^p(\R; X)$, $V^p(\R;X)$, and $V^p_{\rm rc}(\R;X)$ are Banach spaces. 
See, for example,  \cite[Propositions 2.2 and 2.4]{HHK09}. Moreover, the following embeddings hold for $1 \le p <q < \infty$:
\begin{align}
U^p(\R; X) \embeds V^p_{\rm rc}(\R;X) \embeds U^q(\R; X) \embeds L^\infty (\R;X).
\label{Linfty}
\end{align}

\noi
Analogous embeddings hold for the adapted spaces $U^p_\Box(\R; \H^s(\T^d))$ 
and $V^p_\Box( \R;\H^s(\T^d))$.

\smallskip
\noi(ii) Let $0<\al <  1$. 
If $u_{t_0}=0$ for some $t_0 \in [0, T]$, 
then we have 
\begin{align}
\| u \|_{V^{\frac1\al }_{T} X} \les T^\al \| u \|_{C^{\al}_{T} X} .
\notag 
\end{align}

\noi
Note that we need the condition 
$u_{t_0}=0$ due to the boundary condition $u(\infty) = 0$ imposed
by convention (for the $U^p$-$V^{p'}$ duality; see \cite[Theorem 2.8]{HHK09}).

\end{remark}

We now state a  transference result which 
allows us to transfer 
 estimates on linear solutions
to those for 
space-time functions
with $U^q$-regularity in time;
see, for example,  Lemma~\ref{LEM:STR}.
Such transference results
play a fundamental role 
 in applications of the Fourier restriction norm method
in studying dispersive PDEs;  
see, for example, 
\cite[Proposition 3.7]{KS}
and
\cite[Lemma 2.9]{TAO}.
The following lemma  follows from a simple modification of
the multilinear result in \cite[Proposition 2.19]{HHK09}, 
and thus we omit details.

\begin{lemma}[transference]\label{LEM:trf}
Let $s_1, s_2\in \R$ and $T_1,T_2\in \LOP( H^{s_1}(\T^d)\times H^{s_2}(\T^d); 
L^1_\textup{loc}(\T^d)\times L^1_\textup{loc}(\T^d))$ such that
\begin{align*}
\| T_1(\S(t)\vec{\phi}) \|_{L^q_t(\R; L^r_x \times L^r_x)} & \les \| \vec{\phi} \|_{
H^{s_1}\times H^{s_2}},  \\
\| T_2(\S(-t) ^\tf \, \vec{\psi}) \|_{L^q_t(\R; L^r_x \times L^r_x) } & \les \| \vec{\psi} \|_{H^{s_1}\times H^{s_2}},
\end{align*}

\noi
hold for some $1\le q < \infty$ and $1 \le  r\le \infty$, 
where $\S(t)$ is the matrix-valued linear wave propagator
 in \eqref{Sdef} and $\S(t)^\tf$ denotes its matrix transpose.
 Then, we have
\begin{align*}
\| T_1 (\vec{u}) \|_{L^q_t(\R; L^r_x \times L^r_x)} & \les
 \| \vec{u} \|_{U^q_\Box(\R;H^{s_1}_x\times H^{s_2}_x)}
=  \| \S(-t) \vec{u} \|_{U^q(\R;H^{s_1}_x\times H^{s_2}_x)}, \\
\| T_2(\vec{v}) \|_{L^q_t(\R; L^r_x \times L^r_x)} & \les
 \| \S(t)^\tf\, \vec{v} \|_{U^q(\R; H^{s_1}_x\times H^{s_2}_x)}.
\end{align*}

\end{lemma}

We conclude this subsection 
by presenting basic linear estimates.

\begin{lemma}[linear estimates]
\label{LEM:lin}

Let $s\in\R$, $1 < p <\infty$, and $0<T\le \infty$. Then, we have
\begin{align}
\|  \vec{\phi} \|_{U^p_T \H^s_x} & \le \|\vec{\phi} \|_{\H^s},
\label{lin1}
\\
\bigg\| \int_0^t \S(-t') \vec{F}(t') dt' \bigg\|_{U^p_T \H^s_x}
& \le \sup_{ \|  \vec{v}\|_{V^{p'}_{\Box, T} (\H^s_x)^*} =1  }
\bigg| \int_0^T \int_{\T^d} \vec{F}(t)\cdot  \vec{v}(t)\,  dx dt \bigg| ,
\label{lin2}
\\
\bigg\| \int_0^t \S(-t') \vec{F}(t') dt' \bigg\|_{{U}^p_T \H^s_x} & \les  \| \vec{F} \|_{L^1_T \H^{s}_x},
\label{lin3}
\end{align}

\noi
where $\S(t)$ is the matrix-valued linear wave propagator
 in \eqref{Sdef},  $\S(t)^\tf$ denotes its matrix transpose, 
and 
$V^{p'}_{\Box, T} (\H^s_x)^*$ is as in \eqref{Dl1} and \eqref{local}
with 
$(\H^{s}(\T^d))^* = H^{-s}(\T^d) \times H^{1-s}(\T^d)$.

\end{lemma}

The first estimate  \eqref{lin1} follows from 
 \eqref{local}, the fact that the infimum 
for the time restriction norm
(see  \eqref{local})
on the left-hand side of \eqref{lin1}
 is attained at $\ind_{[0,T]} \cdot \vec{\phi}$ (see \cite[Lemma A.1]{BOP2}), and 
the 
atomic structure of the $U^p$-space in 
Definition~\ref{DEF:Up}.
The second estimate~\eqref{lin2} follows
from 
 the duality relation of $U^p(\R; \H^s(\T^d))$ and $V^{p'}(\R; (\H^s(\T^d))^*)$ 
 (see 
 \cite[Theorem 2.8]{HHK09}
and 
\cite[Section 4.2]{Koch14})
 and the fact that $\S(-t)^\tf$ is the adjoint of $\S(-t)$
 along with the definition 
 \eqref{Dl1}
 of the space $V^{p'}_{\Box} (\R; (\H^s_x)^*)$;
 see also Remark \ref{REM:local}.
See, for example, the proof of \cite[Proposition 2.11]{HTT11}.
The third estimate \eqref{lin3} also follows
from 
 the duality relation of $U^p(\R; \H^s(\T^d))$ and $V^{p'}(\R; (\H^s(\T^d))^*)$. 
A more direct way to see \eqref{lin1} and~\eqref{lin3}
may be to use the embedding \eqref{Linfty}
and estimate the $V^1$-norm 
with the boundary condition $u (\infty) = 0$
by using 
 \eqref{ctrl0a}.

\subsection{Fractional Brownian motion and multiple stochastic integrals}
\label{SUBSEC:FBM}

In this subsection, we recall the definition and basic properties of fractional Brownian motions and of Wiener integrals  with respect to fractional Brownian motions. See \cite[Chapter 5]{Nualart06} 
for 
further details.  See also \cite{BHOZ, Nualart13}.  

We first recall the definition of a fractional Brownian motion, 
introduced in 
 \cite{Kolm} and further developed in  \cite{MVN68}.

\begin{definition}\rm
\label{DEF:fBM}
Let $0 < \be < 1$. 
A (real-valued) fractional Brownian motion $\{B(t)\}_{t \in \R_+}$ with 
the Hurst parameter $\be$ is a centered Gaussian process with covariance given~by
\begin{align*}
 \E[B(t_1) B(t_2)]
 =  \frac12 \big( t_1^{2\be} + t_2^{2\be} - |t_1-t_2|^{2\be}\big) .
\end{align*}

\noi
When $\be=\frac12$, this process corresponds to a standard Brownian motion. 
A complex-valued fractional Brownian motion $\{ B(t) \}_{t \in \R_+}$ with 
the Hurst parameter $\be$ is a complex-valued centered Gaussian process such that 
 $\{ \sqrt 2\Re B(t)\}_{t \in \R_+} $ and 
$\{ \sqrt 2 \Im B(t)\}_{t \in \R_+}$ are independent real-valued fractional Brownian motions with the Hurst parameter $\be$.
Thus, we have 
\begin{align*}
\E\big[|B(t_1)- B(t_2)|^2\big]
= |t_1-t_2|^{2\be}.
\end{align*}

\end{definition}

In the following, we set-up a real-valued isonormal framework
but the discussion below can be easily adapted to the complex-valued setting
(by dropping the conditions \eqref{fBM1} and~\eqref{fBM2}).
See also \cite{OTh1}
for a related discussion in the complex-valued setting.

We first introduce the following partition of $\Z^d$ (see \cite{OP1}):
\begin{align*}
\Z^d = (\Z^d)_+  \cup (\Z^d)_- \cup \{0\}^d ,
\end{align*}
where
\begin{align*}
(\Z^d)_+ &
= \bigcup_{k=0}^{d-1} \Z^k \times \Z_+ \times \{0\}^{d-k-1}
\quad \text{and}
\quad
(\Z^d)_-   = - (\Z^d)_+,
\end{align*}
and $\{0\}^d$ denotes the origin in $\Z^d$. We also set $(\Z^d)_{+,0} = (\Z^d)_{+} \cup \{0\}^d$.

Given $0 < \be < 1$, 
let $\{B_n\}_{n\in \Z^d}$ be a family of 
mutually independent  complex-valued fractional Brownian motions with 
the Hurst parameter $\be$ 
conditioned that\footnote{By convention, we take  $B_0$ to be  a real-valued fractional Brownian motion.} 
\begin{align}
B_{-n} = \cj{B_n},\quad n \in \Z^d.
\label{fBM1}
\end{align}

\noi
Then,  $\{B_0\}\cup \{\sqrt 2\Re B_n, \sqrt 2 \Im B_n \}_{n\in (\Z^d)_{+}}$ 
forms a family of  mutually independent real-valued fractional Brownian motions
with the same Hurst parameter $\be$. 
In the following, we 
define stochastic integrals with respect to the family $\{B_n\}_{n\in\Z^d}$ of 
complex-valued fractional Brownian motions
by decomposing them into  stochastic integrals with respect to the family 
$\{\Re B_n, \Im B_n\}_{n\in (\Z^d)_{+,0}} $ of real-valued (scaled) fractional Brownian motions.

In defining stochastic integrals with respect to the family 
$\{B_n\}_{n\in\Z^d}$,  we need the following {\it real} Hilbert space.
For $\frac 12 < \be < 1$, 
let $\scrH^\be(\R_+)$ be the completion of 
linear combinations of 
(real-valued) step functions on $\R_+$ under the following norm:\footnote{By restricting our attention to functions $f$ on $\R$ with $\supp(f) \subset \R_+$, \eqref{BM0} defines a norm, not a semi-norm.}
\begin{align}
\label{BM0}
\begin{split}
\|f\|_{ \scrH^\be (\R_+)}^2 
& = \be(2\be-1)
\int_0 ^\infty \int_0^\infty f(t) f(t') |t-t'|^{2\be -2 } dtdt' \\
& = \be (2\be-1) \int_0^\infty f(t) \, \mathfrak{I}_{2\be-1}(f)(t) dt \\
& = C_\be  \int_0^\infty  |\tau|^{1-2\be} |\ft{f}(\tau)|^2 d\tau \\
& = C_\be\|f\|_{\dot H^{\frac 12 - \be}(\R_+)}^2
\end{split}
\end{align}

\noi
for a function $f$ supported on $\R_+$.
Here, $\mathfrak{I}_{2\be-1}$ denotes the Riesz potential of order $2\be-1 > 0$.
Given $k\in\N$, we also define $\scrH^\be(\R^k_+)$ as the completion of 
linear combinations of products of step functions in $t_j \in \R_+$, $j=1, \ldots, k$,  under the norm:
\begin{align}
\begin{split}
\|f\|_{\scrH^\be(\R_+^k)}^2
&   = \be^k(2\be-1)^k
\int_{\R_+^{2k}}
f(\vec t)
f(\vec t')
\prod_{j = 1}^k |t_j-t_j'|^{2\be-2} dt_jdt_j' \\
&
= C_\be^k
\bigg\| \prod_{j = 1}^k |\dd_{t_j}|^{\frac 12 - \be}f\bigg\|_{L^2(\R^k_+)}^2
\end{split}
\label{BM0a}
\end{align}

\noi
for a function $f$ supported on $\R_+^k$, 
where
$\vec t = (t_1, \dots, t_k)$
and $\vec t' = (t_1', \dots, t_k')$.
When $\be=\frac12$, we set 
\begin{align}
\scrH^\frac12 (\R^k_+) = L^2(\R^k_+).
\label{BM0x}
\end{align}

%

We say that a sequence $f = \{f_n\}_{n \in \Z^d}$ of complex-valued functions
$f_n$ on $\R_+$ belongs to 
$ \l^2(\Z^d;\scrH^\be(\R_+))$, 
if we have 
\begin{align}
f_{-n} = \cj{f_n}, \quad  n \in \Z^d,
\label{fBM2} 
\end{align}

\noi
and 
$\Re f_n, \Im f_n \in \scrH^\be(\R_+)$
(note from \eqref{fBM2}  that $\Im f_0 = 0$)
such that 
\begin{align*}
\|f\|_{\l^2(\Z^d;\Hs^\be(\R_+))}
:\! & = \bigg(\sum_{n \in \Z^d} \| f_n\|_{\Hs^\be(\R_+)}^2\bigg)^\frac 12\\
& = \bigg(\sum_{n \in \Z^d} \|\Re  f_n\|_{\Hs^\be(\R_+)}^2
+ \sum_{n \in \Z^d} \|\Im  f_n\|_{\Hs^\be(\R_+)}^2\bigg)^\frac 12 < \infty.
\end{align*}

\noi
Then, we define the Wiener integral  of 
 $f = \{f_n\}_{n \in \Z^d} \in \l^2(\Z^d;\scrH^\be(\R_+))$
 with respect to $\{B_n\}_{n\in\Z^d}$
 with the Hurst parameter $\frac 12 \le \be < 1$
 by setting
\begin{align}
\begin{split}
I_1[f]  
& =  \sum_{n\in \Z^d}J_n^{(r)}(f_n)
+ \sum_{n\in \Z^d}J_n^{(i)}(f_n)\\
& = \sum_{n\in \Z^d} \int _0^\infty f_n(t) d \Re B_n(t)
+ i 
 \sum_{n\in \Z^d}
\int _0^\infty f_n(t) d \Im B_n(t).
\end{split}
\label{BM1}
\end{align}

\noi
Note that $I_1[f]$ is real-valued
in view of the conditions \eqref{fBM1} and \eqref{fBM2}.
In \eqref{BM1},  each summand $J_n^{(r)}(f_n)$ 
or $J_n^{(i)}(f_n)$ is  understood as a Wiener integral;
namely, $\{J_0^{(r)}(f_0)\}\cup \{\sqrt 2J_n^{(r)}(f_n), \sqrt 2 J_n^{(i)}(f_n)\}_{n \in (\Z^d)_{+}}$ is a family of independent mean-zero Gaussian random variables
with variance $ \|f_n\|_{\Hs^\be}^2$.
In particular, the map $I_1$ is an isometry from $\l^2(\Z^d; \scrH^\be(\R_+))$ 
into $L^2(\Om, \s(I_1), \PP)$.
Here, 
 $\s(I_1)$ denotes the $\s$-algebra generated by the process $I_1 = \{ I_1[f]: \, f \in \l^2(\Z^d; \scrH^\be(\R_+)) \}$, 
where  we view $I_1$ as a process
indexed by  $f \in \l^2(\Z^d;\Hs^\be(\R_+))$.
 
  The process $I_1$ is known as an isonormal Gaussian process associated with the Hilbert space $\l^2(\Z^d;\scrH^\be(\R_+))$, satisfying
\begin{align}
\notag
\E\big[I_1[f] {I_1[g]} \big] = \jb{f,g}_{\l^2_n \scrH^\be_t}
\end{align}

\noi
for  $f,g\in \l^2(\Z^d;\scrH^\be(\R_+))$; see \cite[Definition 1.1.1]{Nualart06}.
See \cite[Chapter 5]{Nualart06} and \cite[Section 2.1]{BHOZ} for further details on the construction of the Wiener integral $I_1$.

Given $k\in\N\cup\{0\}$, we define the $k$th homogeneous Wiener chaos $\H_k$ as the closed linear subspace of $L^2(\Omega)$ generated by
\begin{align*}
\big\{ H_k(I_1[f]) : \, f\in \l^2(\Z^d; \scrH^\be(\R_+)), \, \|f\|_{\l^2_n\scrH^\be_t} = 1   \big\},
\end{align*}
where $H_k$ denotes the $k$th Hermite polynomial defined via the generating function:
\begin{align*}
e^{tx - \frac12 t^2} = \sum_{k=0}^\infty \frac{t^k}{k!} H_k(x).
\end{align*}

\noi
For readers' convenience, we write the first four Hermite polynomials:
\begin{align*}
H_0(x) = 1,
\quad H_1(x) = x,
\quad H_2(x) = x^2-1 ,
\quad  \text{and}
\quad H_3(x) =  x^3 - 3x.
\end{align*}

\noi
We also set 
\begin{align}
\H_{\le k} = \bigoplus_{j = 0}^k \H_j.
\label{chaos1}
\end{align}

\noi
Note that for any $f\in \l^2 (\Z^d; \scrH^\be (\R_+))$, the stochastic integral $I_1[f]$ is an element of the first homogeneous Wiener chaos~$\H_1$.
When $k\neq j$, 
the spaces $\H_k$ and $\H_j$
are orthogonal in $L^2(\Om)$.
We recall the   Wiener-Ito decomposition
for 
the real-valued Hilbert space $L^2(\Omega, \s(I_1), \PP)$:
\begin{align*}
L^2(\Omega, \s(I_1), \PP) = \bigoplus_{k=0}^\infty \H_k.
\end{align*}

\noi
where the right-hand side is  an orthogonal sum of the subspaces $\H_k$.

We now state the Wiener chaos estimate
for elements in the homogeneous Wiener chaos $ \H_k$, 
 which follows from the hypercontractivity of the Ornstein-Uhlenbeck semigroup due to Nelson \cite{Nelson2}. 
For its  proof, see, for example,  \cite[Theorem I.22]{Simon} and \cite[Corollary 2.8.14]{NP12}.

\begin{lemma}[Wiener chaos estimate]
\label{LEM:WCE}
Let $k\in\N$. Given any  $1 \le p <\infty$ and $F \in \H_{\le k}$, we have
\begin{align*}
\| F \|_{L^p(\Omega)} \le p^{\frac k 2 } \| F \|_{L^2(\Omega)}.
\end{align*}

\end{lemma}

Next,  we introduce multiple stochastic integrals with respect to
 the family 
  $\{B_n\}_{n\in\Z^d}$ of 
 fractional Brownian motions  with the Hurst parameter $\frac 12 \le \be < 1$, 
 satisfying \eqref{fBM1}. 
Given  $f\in \l^2( (\Z^d)^{k} ; \Hs^\be(\R_+^{k})  )$, where 
$\Hs^\be(\R_+^{k})$ 
is as in  \eqref{BM0a}
and \eqref{BM0x},  
 we define its symmetrization by
\begin{align}
\label{sym}
\Sym(f) (z_1, \ldots, z_k) & = \frac{1}{k!} \sum_{\s\in \S_k} f(z_{\s(1)},  \ldots, z_{\s(k)}),
\end{align}
where $z_j = (t_j, n_j)$, $j=1, \ldots,k$,  and the sum is taken over all permutations $\s$ in the symmetric group $\S_k$ on $\{1, \ldots, k\}$.
We denote by  
$\big(\l^2( (\Z^d)^{k} ; \Hs^\be(\R_+^{k})  )\big)^{ \Sym}$  the subspace of symmetric functions in 
$\l^2( (\Z^d)^{k} ; \Hs^\be(\R_+^{k})  )$.

We now introduce the notion of a multiple Wiener integral $I_k$.

\begin{definition}\rm
\label{DEF:MSI}
Let $k\in\N$. The $k$th multiple Wiener integral $I_k$ is an
isometry (up to a constant factor; see \eqref{BM2b})
from 
$\big(\l^2( (\Z^d)^{k} ; \Hs^\be(\R_+^{k})  )\big)^\Sym$
 into the homogeneous Wiener chaos $\H_k$, uniquely determined by
\begin{align*}
I_k\big[ \Sym(h_1^{\otimes k_1} \otimes \cdots \otimes h_m^{\otimes k_m})  \big] = \prod_{j=1}^m H_{k_j} (I_1[h_j]),
\end{align*}

\noi
for any orthogonal elements $h_1, \ldots, h_m \in \l^2(\Z^d; \scrH^\be(\R_+))$ and $k_1, \ldots,k_m \in \N\cup\{0\}$ such that $k_1+\cdots + k_m = k$,
where $H_{k}$ denotes the $k$th Hermite polynomial and $I_1$ is as in \eqref{BM1}.

\end{definition}

The multiple stochastic integral $I_k$ satisfies the following properties:
\begin{itemize}
\item The stochastic integral of (possibly non-symmetric) $f\in 
\l^2( (\Z^d)^{k} ; \Hs^\be(\R_+^{k})  )$
 is defined as 
 \begin{align}
 I_k[f] = I_k [ \Sym(f)].
 \label{sym1}
 \end{align}

\item Given  $f\in \l^2( (\Z^d)^{k} ; \Hs^\be(\R_+^{k})  )$, 
 and $g\in \l^2( (\Z^d)^{\l} ; \Hs^\be(\R_+^{\l})  )$
 for some $k, \l \in \N$,  
 we have
\begin{align}
\label{BM2b}
\E \big[ I_k[f] I_\l[g] \big] = k!\cdot  \ind_{k=\l}
\cdot  \jb{\Sym(f), \Sym(g)}_{\l^2_{n_1, \dots, n_k}  \scrH^\be_{t_1, \dots, t_k}}, 
\end{align}

\noi
where
$\l^2_{n_1, \dots, n_k} \Hs^\be_{t_1, \dots, t_k}$
is a short-hand notation for 
$\l^2( (\Z^d)^{k} ; \Hs^\be(\R_+^{k})  )$.

\item When $\be=\frac{1}{2}$, the multiple Wiener integral agrees with the iterated Wiener-Ito integral with respect to the family $\{B_n\}_{n\in\Z^d}$ of 
mutually independent complex-valued  Brownian motions. 
Moreover, if $f$ is symmetric, then we have
\begin{align*}
I_k[f]
&
=
k! \sum_{n_1, \ldots, n_k \in \Z^d} 
\int_{\Dl_k} f(z_1, \dots, z_k) dB_{n_k}(t_k ) \cdots dB_{n_1}(t_1),
\end{align*}
\noi
where 
$\Dl_k
= \big\{ (t_1, \ldots, t_k) \in \R_+^k: \ t_i > t_j 
\text{ for } i < j\big\}$
 and $z_j = (t_j, n_j)$.
Here,  the iterated integral on the right-hand side is understood as an iterated 
Wiener-Ito integral; see \cite[p.\,23]{Nualart06}.

\item From Jensen's inequality with \eqref{sym}, we have the following estimate:
\begin{align}
\label{Jensen}
\| \Sym(f) \|_{\l^2_{n_1, \dots, n_k}  \scrH^\be_{t_1, \dots, t_k}}
\le 
 \| f \|_{\l^2_{n_1, \dots, n_k}  \scrH^\be_{t_1, \dots, t_k}}
\end{align}

\noi
for $f\in \l^2( (\Z^d)^{k} ; \Hs^\be(\R_+^{k})  )$.
See \cite[(B.2)]{OWZ}.

\end{itemize}

We conclude this subsection by stating a product formula for iterated stochastic integrals.
See \cite[Proposition 1.1.3]{Nualart06}.

\begin{lemma}
\label{LEM:prod}
Given  $  k,\l \in \N\cup\{0\}$,  let $f \in \big(\l^2( (\Z^d)^{k} ; \Hs^\be(\R_+^{k})  )\big)^{\Sym}$
and $g\in 
\big(\l^2( (\Z^d)^{\l} ; \Hs^\be(\R_+^{\l})  )\big)^{\Sym}$.
Then, we have
\begin{align}
\notag
I_k[f ] I_\l [g] = \sum_{r=0}^{\min(k, \l)} \binom{k}{r} \binom{\l}{r} r!
\cdot I_{ k + \l - 2 r} ( f \otimes_r g),
\end{align}

\noi
where $\otimes_r$ denotes the $r$-contraction of $f$ and $g$, defined by
\begin{align}
\label{contr}
\begin{split}
& (f \otimes_r g ) (z_1, \dots, z_{k+\l-2r})\\
& \quad =
c_\be^r 
 \sum_{m_1, \ldots, m_r \in \Z^d} \int_{\R_+^{2r}} 
f(z_1, \ldots, z_{k-r}, s_1, m_1, \ldots, s_r, m_r) \\
& \hphantom{XXXXXXXX}
  \times g(z_{k-r+1}, \ldots, z_{k+\l-2r}, 
s_1', - m_1, \ldots, s_r', - m_r) \\
& \hphantom{XXXXXXXX}
\times \prod_{j=1}^r |s_j - s_j' |^{2\be-2 } ds_j ds_j'
\end{split}
\end{align}

\noi
for $\frac 12 <   \be < 1$
with an obvious modification when $\be = \frac 12$, 
where 
 $z_j = (t_j,n_j)$ for $j=1,\ldots, k+\l -2r$
and  $c_\be = \be(2\be-1)$.

\end{lemma}

We note that the $r$-contraction of $f$ and $g$ is not necessarily symmetric.

\subsection{Random tensor estimate}
\label{SUBSEC:RTE}

In this subsection, we provide a basic definition of tensors and 
present a random tensor estimate 
adapted to multiple stochastic integrals
with respect to fractional Brownian motions
with the Hurst parameters $\frac 12 \le \be < 1$
(Lemma~\ref{LEM:RTE}).
 See \cite[Sections 2 and 4]{DNY22}, \cite[Section 4]{Bring},  \cite[Appendix C]{OWZ}, and \cite[Section 2]{OW25} for further discussions.

 Let $A$ denote a finite index set and $n_A$ 
 and $t_A$ denote the tuples $\{n_j\}_{ j\in A}$
and   $\{t_j\}_{ j\in A}$.
  We  use the short-hand notation:
\begin{align}
z_A = (t_A, n_A) \in (\R \times \Z^d)^{A}
\qquad \text{and} \qquad f_{z_A} = f(z_A) = f(t_A,n_A).
\label{short1}
\end{align}

\begin{definition}\rm
\label{DEF:tensor}

Let $A$ be a finite index set.

\smallskip

\noi{\rm(i)} A tensor $h = h_{n_A}$ is a function from $(\Z^d)^A$ to $\bbC$ with input variables $n_A$. 

\smallskip

\noi{\rm(ii)} Let $(B,C)$ be a partition of $A$.
We define the norms $\|\cdot\|_{n_A}$ and $\|\cdot\|_{n_B \to n_C}$ by
\begin{align*}
\| h \|_{n_A} & = \| h\|_{\l^2_{n_A}} = \bigg( \sum_{n_A} |h_{n_A}|^2 \bigg)^\frac12, \\
\| h\|_{n_B \to n_C} & =  \bigg( \sup_{\|f\|_{n_B} = 1} \sum_{n_C} \Big| \sum_{n_B} h_{n_A} f_{n_B} \Big|^2 \bigg)^\frac12,
\end{align*}
where we used the short-hand notation $\sum_{n_Z} = \sum_{n_Z \in (\Z^d)^Z}$ for a finite index set $Z$.
\end{definition}

Note that, by duality, we have
$\|h\|_{n_B \to n_C} = \| h\|_{n_C \to n_B}$ for any tensor $h=h_{n_A}$. 
If $B=\varnothing$ or $C=\varnothing$, then we have $\|h\|_{n_B \to n_C} = \|h\|_{n_A}$. 
Moreover, by duality, we can rewrite the norm $\| \cdot \|_{n_B \to n_C}$ as 
\begin{align}\label{op-norm-dual}
\| h \|_{n_B \to n_C} = \sup_{\|f\|_{n_B} = \|g\|_{n_C} = 1 }
\bigg| \sum_{n_B,n_C} h_{n_A} f_{n_B} g_{n_C} \bigg|.
\end{align}

\noi
In particular, from \eqref{op-norm-dual} and Cauchy-Schwarz's inequality, we have
\begin{align}
\| h \|_{n_B \to n_C} 
\le \| h \|_{n_B, n_C} .
\label{HS2x}
\end{align}

Before proceeding further, 
let us  rewrite the contraction in \eqref{contr}, 
using the tensor notation.
Given  finite index sets, 
 $A$ and $D$, 
 let  $f\in (\l^2_{n_A} \scrH^\be_{t_A})^{\Sym}$ and 
 $g\in( \l^2_{n_D} \scrH^\be_{t_D})^{\Sym}$.
Given  $0 \le r \le \min(|A|, |D|)$, 
let $(A_1,A_2)$ and  $(D_1,D_2)$ be partitions of $A$ and $D$, respectively, 
such that $|A_2| = |D_2| = r$.
Then, define 
 an index set $B$ with $|B| = |A| + |D| -2r$
 by setting $B = (B_1, B_2) := (A_1, D_1)$, 
 where $(B_1, B_2)$ forms a partition of $B$.
Then, with $z_A = (z_{A_1}, z_{A_2}) 
= (z_{B_1}, z_{A_2}) $
and
 $z_D = (z_{D_1}, z_{D_2}) 
= (z_{B_2}, z_{D_2}) $, 
we can write the $r$-contraction of $f$ and $g$
in~\eqref{contr}  as 
\begin{align}
\label{contr2}
\begin{split}
(f \otimes_r g) (z_B)
&
= (f_{z_A} \otimes_r g_{z_D}) (z_B )
\\
 & = \jb{  f \vert_{z_{A_1} = z_{B_1}, z_{A_2} = z_C}   , g \vert_{z_{D_1}  = z_{B_2} , z_{D_2} = z_C}    }_{\l^2_{n_C} \scrH^\be_{t_C}} \\
 &= \jb{  f_{z_{B_1} z_{C}} , g_{z_{B_2} \wt z_C} }_{\l^2_{n_C} \scrH^\be_{t_C}}
,
\end{split}
\end{align}

\noi
for an index set $C$ with $|C|= r$, 
where $\wt z_C = (t_C, - n_C)$.
Here, 
 the inner product on $\l^2_{n_C} \scrH^\be_{t_C}$
is given by 
\begin{align*}
& \jb{  f_{ z_{C}} , g_{ z_C} }_{ \l^2_{n_C}\scrH^\be_{t_C}}
\\
&\quad 
=
c_\be^{|C|} 
\sum_{n_C}
\int_{\R_+^{2|C|}} f( t_C, n_C) g( t'_C, n_C) \prod_{j \in C} |t_j-t_j'|^{2\be-2} dt_C 
dt_C' 
\end{align*}

\noi
for $\frac 12 < \be < 1$
with an obvious modification when $\be = \frac 12$, 
where  
 $c_\be = \be(2\be-1)$
 is as in \eqref{contr}.

\medskip

We  now state the random tensor estimate
for multiple stochastic integrals, 
which plays a crucial role in 
studying 
 the mapping properties of the random  operators $\vXX$ and $\vbbX$ in~\eqref{introX} and \eqref{bX0}, respectively, 
 in Section \ref{SEC:ops}.
The random tensor estimate  
  was first introduced in a seminal work \cite{DNY22}  by  Deng, Nahmod, and Yue
   in the context of  random variables; see also \cite{Bring, OW25}. 
 See   
 \cite{BO96} for a precursor of the random tensor estimate.
In \cite{OWZ}, it was extended to the case of  multiple Wiener-Ito integrals
(with the Hurst parameter $\be = \frac 12$).
Here, we further extend  this result to treat multiple stochastic integrals with respect to fractional Brownian motions with Hurst parameter $\frac 12 < \be < 1$, 
by closely following the argument in \cite{DNY22, OWZ}.

\begin{lemma}
\label{LEM:RTE}

Fix $\frac 12 \le \be < 1$ and $A$ a finite index set with $|A|=k \ge 1 $.
Given a  tensor $\hf=\hf_{bc n_A}$ with 
 $(b,c) \in (\Z^d)^{m}$ for some integer $m\ge2$
and $n_A \in (\Z^d)^{A}$ 
satisfying
\begin{align*}
\hf_{-b, -c, -n_A} = \cj{\hf_{bc n_A}}, 
\quad 
(b,c) \in (\Z^d)^{m}, \ 
n_A \in (\Z^d)^{A}, 
\end{align*}

\noi
 define the random tensor $H = H_{bc}$ given by
\begin{align}
H_{bc} =   I_k \big[ \hf_{bcn_A} \ff_{bc} (t_A, n_A) \big]
\label{Z1-new}
\end{align}

\noi
for $\ff\in \l^\infty_{bc n_A}( (\Z^d)^{k+m} ; \Hs^\be(\R_+^{k})  )$, satisfying
\begin{align*}
\ff_{-b, -c, -n_A} = \cj{\ff_{bc n_A}}, 
\quad 
(b,c) \in (\Z^d)^{m}, \ 
n_A \in (\Z^d)^{A}, 
\end{align*}

\noi
where 
$\Hs^\be(\R_+^{k})$ 
is as in  \eqref{BM0a} and \eqref{BM0x}
and  $I_k$ denotes the $k$th multiple stochastic integral associated with an isonormal process over $\l^2(\Z^d; \scrH^\be(\R_+))$ defined  in Subsection~\ref{SUBSEC:FBM}.
Then, given any $\theta>0$ and $1 \le p <\infty$, we have
\begin{align}\label{Z1a-new}
\begin{aligned}
\big\| \| H_{bc} \|_{b \to c} \big\|_{L^p(\Om)}
&
\les p^\frac k2
\big\| \hf_{bcn_A} \times \|\ff_{bc}(t_A, n_A)\|_{\scrH^\be_{t_A}} \big\|_{\l^2_{bcn_A}}^{\theta}\\
 &\quad \times
 \Big( \max_{(B,C)} \big\|    \hf_{bcn_A} 
 \times \|\ff_{bc}(t_A, n_A)\|_{\scrH^\be_{t_A}} \big\|_{b n_B \to c n_C} \Big)^{1-\theta},
\end{aligned}
\end{align}

\noi
where the maximum is taken over  all partitions $(B,C)$ of $A$.

\end{lemma}

Before proceeding further, let us remark that, as pointed out in  
 \cite[Remark 6.6]{Bring2}, 
the random tensor estimate 
is closely related to operator bounds for structured random matrices
and  can be proven, using 
the 
non-commutative 
Khintchine  inequality 
\cite[Theorem~3.2]{vH}.
See \cite{Bring2, Kaneshiro}
for a further discussion.

\begin{proof}[Proof of Lemma \ref{LEM:RTE}]

The case $\be=\frac12$ was treated  in \cite[Lemma C.3]{OWZ}.
The general case 
 $\frac 12 < \be < 1$ follows
 from  a straightforward modification of 
 the proof of \cite[Lemma~C.3]{OWZ}, 
based on  a higher order version of Bourgain's $TT^*$-argument \cite{BO96}
via an induction argument  originally introduced in~\cite{DNY22}, 
where the main modification comes 
from the change in the underlying Hilbert space.
In the fractional-in-time case ($\frac 12 < \be < 1$), 
we need to work with 
 $\l^2((\Z^d)^k;\scrH^\be(\R_+^k))$ 
 associated with multiple stochastic integrals  with respect to fractional Brownian motions of 
 the Hurst parameter $\be$.
For the sake of completeness and for readers' convenience, 
we include a proof.

Fix $\frac 12 < \be < 1$.
Define  the linear operator $T$ with kernel $H_{bc}$:
\begin{align}
(Tg)_b  = \sum_{c} H_{bc} g_c, \quad g \in \l^2_c.
\label{Z2}
\end{align}

\noi
For  $j \in \N$,  we define the operator $T_j$ by
\begin{align*}
T_j =
\begin{cases}
(TT^*)^{ \frac j  2 }, & \text{if $j$ is even}, \\
(TT^*)^{\frac{j-1}{2} } T , & \text{if $j$ is odd},
\end{cases}
\end{align*}

\noi
where $T^*$ denotes the adjoint of $T$
with kernel $H^*_{bc} = H_{cb}$.

We start by showing that the kernel of $T_j$ is given by a linear combination of terms 
$H^{(j)}$
of the~form:
\begin{align}
H^{(j)}_{bc} = I_\l \big[y_{bc}(z_D)\big]
\label{Z1a}
\end{align}
for some finite index set $D$ with $|D| = \l \le kj$, 
where 
 $y_{bc}(z_D)$ satisfies the following bound:
\begin{align}
\begin{split}
 \| y_{bc}(n_D, t_D) \|_{\l^2_{bcn_D} \scrH^\be_{t_D}}
&
\les 
 \big\|  \hf_{bcn_A}  \times
 \| \ff_{bc}(z_A)\|_{\scrH^\be_{t_A}} \big\|_{\l^2_{bcn_A}}\\
&\quad \times 
\Big(  \max_{(B,C)} \big\| \hf_{bcn_A} \times \|\ff_{bc}(z_A)\|_{\scrH^\be_{t_A}}
      \big\|_{b n_B \to c n_C} \Big)^{j-1}, 
 \end{split}
\label{Z1b}
\end{align}

\noi
where $z_D = (t_D, n_D)$ as in \eqref{short1} 
and the maximum is taken over  all partitions $(B,C)$ of $A$.
Here, the implicit constant depends on $k$, $j$, and $\l$
and grows in $j$ (and $\l$).
This, however, does not cause an issue since, 
given small $\ta > 0$ in \eqref{Z1a-new}, 
we fix $j = j(\ta) \gg1$.

Let $j=1$.
From \eqref{Z1-new}
with \eqref{sym1}, 
 we have only one term of the form \eqref{Z1a} with $D=A$, $\l=k$, and 
 $y_{bc}(z_D) = \Sym(\hf_{bc n_A} \ff_{bc}(z_A))$, 
 where $\Sym$ is as in \eqref{sym}.
In this case,   the estimate~\eqref{Z1b}  follows from~\eqref{Jensen}.
Note from \eqref{BM2b} with \eqref{Z1-new}
and \eqref{Z2}
that 
\begin{align*}
\| \Sym(\hf_{bc n_A} \ff_{bc}(z_A)) \|_{\l^2_{bcn_A} \scrH^\be_{t_A}}
\sim_k \big\|\|H_{bc} \|_{\l^2_{bc}}\big\|_{L^2(\Om)}
=  \big\|\|T \|_{\HS}\big\|_{L^2(\Om)}.
\end{align*}

\noi
We also note that the first factor on the right-hand side
of \eqref{Z1a-new} (and \eqref{Z1b}) essentially represents
the Hilbert-Schmidt norm of $T$ (modulo symmetrization).

Given an integer $j \ge 2$, 
assume that \eqref{Z1a} and \eqref{Z1b}
hold for $T_{j - 1}$.
We only consider the case when $j$ is odd.
When $j$ is even, the claim follows from a similar consideration.
Noting that  $T_{j} = T_{j-1}T$, 
it follows from \eqref{Z1a}, \eqref{Z1-new}, 
\eqref{sym1}, and 
Lemma \ref{LEM:prod}
that the kernel of $T_j$ is given by
a linear combination of terms $H^{(j)}$ of the form:
\begin{align}\label{pf-DNY-aux}
\begin{aligned}
H^{(j)}_{bc} & = \sum_{b'} H^{(j-1)}_{bb'} H_{b'c}   \\
&= \sum_{b'} I_{\ell} \big[ y_{bb'}(z_D) \big] \cdot 
 I_k \big[ \hf_{b'c n_A} \ff_{b'c }(z_A) \big]\\
& =  \sum_{r=0}^{\min(k,\ell)} \binom{k}{r}  \binom{\ell}{r}  
 r!\cdot 
I_{k + \ell - 2r}
 \bigg[ \sum_{b'} \big( \Sym(y_{bb'}) \otimes_r  \Sym(\hf_{b'c} \ff_{b'c}) \big) \bigg] 
 \end{aligned}
\end{align}

\noi 
for a finite index set $D$ with $|D| = \l \le k(j-1)$. 
Hence, it remains to show that the integrand 
of the stochastic integral $I_{k + \ell - 2r}$ in \eqref{pf-DNY-aux}
satisfies \eqref{Z1b}
for each $r = 0, 1, \dots, \min(k,\l)$.

In view of \eqref{Jensen}, 
 we drop $\Sym$ in the following.
Fix $0\le r \le \min(k, \l)$. 
From \eqref{contr2} with~\eqref{short1}, 
we have
\begin{align}
\label{Z1cc}
\big( y_{bb'}  \otimes_r \hf_{b'c} \ff_{b'c}  \big)(z_B) 
= \big\langle y_{bb'} (z_{B_1}, z_C) ,
\, \hf_{b'c n_{B_2} \wt n_C } \ff_{b'c} (z_{B_2} , \wt z_C) 
\big\rangle_{\l^2_{n_C} \scrH^\be_{t_C}},
\end{align}
where 
$\wt n_C = -n_C$, 
$\wt z_C = (t_C, - n_C)$, 
and 
$B_1$, $B_2$, and $C$ are pairwise disjoint sets, with $|B_1| = \l-r$, $|B_2| = k-r$, and $|C| = r$, $B = B_1 \sqcup B_2$ such that  (after a suitable relabelling), we have
\begin{align}
\label{Z1c}
D  = B_1 \sqcup C  \quad \text{and} \quad  A = B_2 \sqcup C.
\end{align}

\noi
Here, $\sqcup$ denotes a disjoint union.
Then, from
\eqref{Z1cc}, 
Cauchy-Schwarz's inequality (in $\scrH^\be_{t_C}$), 
Minkowski's inequality, 
and the identification in~\eqref{Z1c},
we have
\begin{align}
& \Big\|   \sum_{b'} \big( y_{bb'} \otimes_r \hf_{b'c} \ff_{b'c} \big)(t_B, n_B) \Big\|_{\ell^2_{bc n_B} \scrH^\be_{t_B}}
\notag \\
& \quad 
= \Big\|   \sum_{b'}  \big\langle y_{bb'}(z_{B_1}, z_{C}), 
\, \hf_{b'c  n_{B_2} \wt  n_C } \ff_{b'c}(z_{B_2}, \wt z_{C}) \big\rangle_{\l^2_{n_C} \scrH^\be_{t_C} } \Big\|_{\ell^2_{bc n_B} \scrH^\be_{t_B}} 
\notag\\
& \quad 
  \les
 \Big\|  \sum_{b', n_C}
\| y_{bb'}(z_{B_1}, z_{C})\|_{\scrH^\be_{t_{B_1}t_{C}}} 
\notag\\
& \hphantom{LXXXX}  \times
\|\hf_{b'c n_{B_2}\wt n_C} \ff_{b'c}(z_{B_2},\wt z_C  )\|_{\scrH^\be_{t_{B_2}t_C}}
   \Big\|_{\ell^2_{bc n_{B}} } 
\label{pf_DNY4-new}
   \\
& \quad 
 \le
  \|y_{bb'}(z_{B_1}, z_C)
 \|_{\ell^2_{bb'n_{B_1} n_C}\scrH^\be_{t_{B_1}t_C} } 
 \notag\\
& \quad \quad  \times \Big\| \| \hf_{b'c n_{B_2} \wt n_C} \ff_{b'c}(z_{B_2}, \wt z_C)\|_{\scrH^\be_{t_{B_2}t_C}} \Big\|_{b' n_C \to cn_{B_2}}
\notag \\
& \quad 
\le
    \|y_{bb'}(z_D ) \|_{\ell^2_{bb'n_D} \scrH^\be_{t_D}}
\Big\| \hf_{b'cn_A} \times  \|\ff_{b'c}(z_A )\|_{\scrH^\be_{t_A}} \Big\|_{b' n_C \to cn_{A\setminus C}} 
\notag\\
& \quad 
 \leq \|y_{bb'}(z_D ) \|_{\ell^2_{bb'n_D} \scrH^\be_{t_D}}
\times  \max_{(B,C)}\Big\|  \hf_{b'cn_A}  
\times   \|\ff_{b'c}(z_A )\|_{\scrH^\be_{t_A}} \Big\|_{b' n_B \to cn_{ C}}, 
\notag
\end{align}

\noi
where, in the second step, 
we also used the following fact (which immediately follows from~\eqref{BM0a}):
\[ \| f_1(t_{B_1}) f_2(t_{B_2})\|_{\scrH^\be_{t_B}}
= \prod_{j = 1}^2 \| f_j(t_{B_j}) \|_{\scrH^\be_{t_{B_j}}}.\]

\noi
Hence, 
 \eqref{Z1b} for $T_j$ 
 follows from \eqref{pf_DNY4-new} 
 and the inductive hypothesis
 on $T_{j - 1}$.

We now prove \eqref{Z1a-new}. 
Given 
 $m\in\N$, consider $T_{2m} = (TT^*)^m$ with kernel $\RR_{2m}$
 of the form: 
\begin{align}
\label{Z2a}
(\RR_{2m})_{bb'} = \sum_{j=1}^{J} I_{ 2  \l_j} \big[ y_{bb'}^{(j)} ( z_{D^{(j)}} ) \big],
\end{align}
for some $J\ge 1$, $0 \le \l_j \le m k$, 
$|D^{(j)} | = 2\l_j$, and 
$y_{bb'}^{(j)}$ satisfying \eqref{Z1b}.
Note that $\RR_{2m} \in \H_{\le 2mk}$, 
where $\H_{\le 2mk}$ is as in \eqref{chaos1}.
Then, given any finite  $p \ge 4m$, 
it follows from 
 Minkowski's integral inequality and the Wiener chaos estimate (Lemma~\ref{LEM:WCE})
 that 
\begin{align}
\label{Z2aab}
\begin{split}
 \big\| \| H_{bc} \|_{b \to c} \big\|_{L^p(\Om)}
& = \big\| \| H_{bc} \|_{b \to c}^{2m} \big\|_{L^{\frac{p}{2m}}(\Om)}^{\frac{1}{2m}}
 = \big\| \| T \|_{\ell^2_b \to \ell^2_c }^{2m} \big\|_{L^{\frac{p}{2m}}(\Om)} ^{\frac{1}{2m}} \\
&  =  \big\| \| (T T^*)^m\|_{\ell^2_{b'} \to \ell^2_{b} }\big\|_{L^{\frac{p}{2m}}(\Om)} ^{\frac{1}{2m}}
 \le \big\| \| (\mathcal{R}_{2m})_{bb'} \|_{\l^2_{bb'}} \big\|_{L^{\frac{p}{2m}}(\Om)} ^{\frac{1}{2m}} \\
& \le  p^{\frac k 2} \big\| \| (\mathcal{R}_{2m})_{bb'} \|_{L^{2}(\Om)} \big\|_{\l^2_{bb'}} ^{\frac{1}{2m}}.
\end{split}
\end{align}

\noi
Hence, 
from 
\eqref{Z2aab}, \eqref{Z2a}, 
\eqref{BM2b}, and \eqref{Z1b}, we obtain
\begin{align}
\begin{split}
&  \big\| \| H_{bc} \|_{b \to c} \big\|_{L^p(\Om)} 
  \les p^{\frac k 2}
\bigg( \sum_{j = 1}^J
 \Big\| \big\| I_{ 2 \l_j} \big[y_{bb'}^{(j)}(z_{D^{(j)}})\big] \big\|_{L^{2}(\Om)} \Big\|_{\l^2_{bb'}}
 \bigg)^{\frac{1}{2m}}  \\
&  \quad 
\sim p^{\frac k 2}
\bigg( \sum_{j = 1}^J
 \| y_{bb'}^{(j)}(z_{D^{(j)}}) \|_{\l^2_{bb' n_{D^{(j)}}} \scrH^\be _{t_{D^{(j)}}}} \bigg)^{\frac{1}{2m}} \\
&  \quad 
 \les p^{\frac k 2} 
\big\|  \hf_{bcn_A}  \times \| \ff_{bc}(t_A, n_A)\|_{\scrH^\be_{t_A}} \big\|_{\l^2_{bcn_A}}^{\frac{1}{2m}}\\
&  \quad \quad 
 \times 
\Big(  \max_{(B,C)} \big\| \hf_{bcn_A} \times \|\ff_{bc}(t_A, n_A)\|_{\scrH^\be_{t_A}}
      \big\|_{b n_B \to c n_C} \Big)^{1-\frac{1}{2m}} 
\end{split}
\label{Z2ax}
\end{align}

\noi
where the maximum is over all partitions $(B,C)$ of $A$.
Therefore, given $\ta > 0$,  the bound~\eqref{Z1a-new}
follows from 
choosing $m \gg 1$ such that $\frac 1{2m} \le \ta$, 
\eqref{Z2ax}, and 
\eqref{HS2x}.
\end{proof}

\subsection{Kolmogorov's continuity criterion}

We conclude the section by recalling a version of Kolmogorov's continuity criterion
for a pair of two-parameter 
operator-valued stochastic processes $(\XX, \bbX)$ satisfying 
Chen's relation:
\begin{align}
\label{KCC0}
\begin{split}
(\updl \XX)_{t_1,t_2,t_3}& =0, \\
(\updl \bbX)_{t_1,t_2,t_3} &  = \XX_{t_1,t_2} \circ \XX_{t_2,t_3}
\end{split}
\end{align}

\noi
for $t_1\ge t_2\ge t_3\ge 0$.
See Definition \ref{DEF:RP}\,(i).
Since 
a straightforward modification of the proof of 
\cite[Theorem~3.1]{FH20} yields the following lemma, 
we omit its proof.
Furthermore, 
as pointed out on \cite[p.\,41]{FH20}, 
the usual Kolmogorov continuity criterion 
for $\XX$ is contained in the proof of 
\cite[Theorem~3.1]{FH20}
by ignoring all considerations related to the second-order
process $\bbX$.
Namely, if the first conditions in \eqref{KCC0}
and \eqref{KCC1} (and in 
\eqref{KCC0}
and \eqref{KCC3})
hold, 
then the first bound in \eqref{KCC2}
(and in \eqref{KCC4}, respectively)
holds.

\begin{lemma}\label{LEM:kolm}

Let $X, Y$ be  Banach spaces with $Y\subset X$ and $T>0$.
Let $(\XX, \bbX)
\in  C_{2,T}\L(X; Y)\times C_{2,T}\L(X; Y)$ be a pair of two-parameter stochastic processes, 
satisfying
\eqref{KCC0}
 for  $(t_1,t_2,t_3)\in\Dl_{3,T}$.
Suppose that there exist $M_1, M_2>0$, $p \ge 1$, and $\frac 1p < \g \le 1$
such that
\begin{align}
\begin{split}
\big\| \|  \XX_{t,r} \|_{\L(X; Y)} \big\|_{L^p(\Om)} & \le M_1  |t-r|^\g, \\ 
\big\| \|  \bbX_{t,r} \|_{\L(X; Y)} \big\|_{L^\frac{p}{2}(\Om)} 
&  \le M_2  |t-r|^{2\g }
\end{split}
\label{KCC1}
\end{align}

\noi
for any $(t, r) \in \Dl_{2, T}$.
Then, for any $0< \al<\g -\frac{1}{p}$, there exists
a constant $C_{\g ,\al, T} >0$
 such that
\begin{align}
\begin{split}
\big\| \|\XX\|_{C^\al_{2,T} \L(X; Y)} \big\|_{L^p(\Om)} 
&  \le C_{\g ,\al, T} M_1, \\
\big\| \|\bbX\|_{C^{2\al}_{2,T} \L(X; Y)} \big\|_{L^\frac p2(\Om)}
&   \le C_{\g ,\al, T}
(M_1^2 +M_2).
\end{split}
\label{KCC2}
\end{align}

\noi
In particular,
suppose that  given $0 < \g \le 1$,  there exist $M, \ta_1, \ta_2 > 0$ such that
\begin{align}
\begin{split}
\big\| \|  \XX_{t,r} \|_{\L(X; Y)} \big\|_{L^p(\Om)} 
& \le M p^{\ta_1} |t-r|^\g\\
\big\| \|  \bbX_{t,r} \|_{\L(X; Y)} \big\|_{L^\frac p2(\Om)} 
& \le M p^{\ta_2} |t-r|^{2\g }
\end{split}
\label{KCC3}
\end{align}

\noi
for any  finite $p > \frac1\g$
and $(t, r) \in \Dl_{2, T}$.
Then, 
 for any $0<\al<\g -\frac{1}{p}$, there exists
a constant $C_{\g ,\al, T} >0$ such that
\begin{align}
\begin{split}
\big\| \|\XX\|_{C^\al_{2,T} \L(X; Y)} \big\|_{L^p(\Om)} 
& \le C_{\g ,\al, T} Mp^{\ta_1}\\
\big\| \|\bbX\|_{C^{2\al}_{2,T} \L(X; Y)} \big\|_{L^\frac p2(\Om)} 
& \le C_{\g ,\al, T} M
(p^{2\ta_1} + p^{\ta_2}).
\end{split}
\label{KCC4}
\end{align}

\noi
Consequently, there exists a version of $\XX$ and $\bbX$ such that 
$(\XX, \bbX) \in C^\al_{2,T}\L(X; Y)\times C^{2\al}_{2,T} \L(X; Y)$,  almost surely.
\end{lemma}

\section{Young and rough integration}
\label{SEC:int}

In this section,
we provide a review 
on  the construction of (linear) Young and rough integrals, based on the sewing lemma 
(Lemma \ref{LEM:sew})  and the theory of controlled  paths in~\cite{G04, GT10};
see also  \cite{FLP06, FH20}.
While the Young and rough integration theory developed in  \cite{G04, GT10} was based on 
 H\"older continuous functions, 
in view of our application to dispersive PDEs, 
we will instead 
go over 
the
 Young\,/\,rough integration theory for functions of bounded $p$-variation, 
 which is
 compatible with the Fourier restriction norm method
 adapted to the $U^p$- and $V^p$-spaces (see Subsection \ref{SUBSEC:Up}).
The materials presented in this section may be standard by
now, but we decided to include them in an accessible manner for readers in dispersive PDEs
who may not be familiar with the Young\,/\,rough integration theory.

In Subsection~\ref{SUBSEC:sew}, we recall the notion of controls and 
state a version of the sewing lemma adapted to the $V^p$-spaces. 
In Subsection~\ref{SUBSEC:sewY}, we go over 
the construction of Young integrals 
via the sewing lemma.
In Subsection \ref{SUBSEC:sewR}, 
by introducing the notion of controlled paths 
in the $p$-variation setting (Definition \ref{DEF:RP}), 
we go over the construction of rough integrals
for controlled paths.

\subsection{Sewing lemma}
\label{SUBSEC:sew}

We first recall the notion of a regular control
function and a suitable generalization of bounded $p$-variation functions in 
the $n$-parameter setting, $n=2,3$, adapted to our setting.
See, for example,  \cite[Subsection 2.1]{DGHT19}
for a presentation close to our setting.
See also
\cite{Lyons98, LQ02, LCL07, Koch14}
 and \cite[Subsection 1.2.1 and 5.1.1]{FV10} for a further 
discussion.

\begin{definition}\rm
\label{DEF:V23}
Let  $X$ be a Banach space and fix $T> 0$.

\smallskip

\noi{\rm(i)}
We call a continuous function $\o:\Dl_{2,T}\to \R_+$ a regular control function on $[0,T]$,  if it is super-additive, namely, we have
\begin{equation}
\o(t_1,t_2) + \o(t_2,t_3) \le \o(t_1,t_3)
\label{con1}
\end{equation}

\noi
 for any $(t_1,t_2,t_3) \in \Dl_{3,T}$.
 For simplicity, we will refer to regular control functions
 as  {\it controls} in the following.
Note from \eqref{con1} and the non-negativity of $\o$ that $\o(t, t) = 0$ for any $t \in [0, T]$
and 
\begin{equation}
\max\big(\o(t_1,t_2) ,   \o(t_2,t_3)\big) \le \o(t_1,t_3)
\label{con1b}
\end{equation}

\noi
 for any $(t_1,t_2,t_3) \in \Dl_{3,T}$.
We also recall the following fact from 
\cite[Exercise~1.8 and 1.9]{FV10}:
\begin{align}
\begin{split}
& \text{given two controls
 $\o_1$ and $\o_2$,
 the function $\o_1^a \o_2^b$}\\ 
& \hphantom{XXXX}\text{
is also a control
for any 
$a, b \ge 0$ with $a + b \ge1$.
}
\end{split}
\label{con1a}
\end{align}

\smallskip
\noi{\rm(ii)}
Let $ 1\le p < \infty$.
Given  $u \in \V^{p}_T X $ (see \eqref{local2}), 
we define 
$\o^{(1)}_{X, p}(u):\Dl_{2,T} \to \R_+$ 
by setting
\begin{align}
\label{ctrl0}
\o_{X, p}^{(1)}(u;t, r) =  
 \sup_{\Pi([r,t])}
 \sum_{j=0}^{k-1} \|u_{t_{j}} - u_{t_{j+1}} \|_X^p
\end{align}

\noi
for any $(t,r) \in \Dl_{2,T}$, where 
 the supremum is taken over  partitions
 $\Pi([r,t])$  
 of  $[r,t]$\textup{:} 
\begin{align}
\Pi ([r,t]) = \{r = t_k < \dots < t_1 <  t_0 = t\}.
\label{part1}
\end{align}

\noi
Then,  $\o^{(1)}_{X, p}(u)$ is a control.
See 
\cite[Propositions 1.12 and 5.8]{FV10}.

\smallskip
\noi{\rm(iii)} 
Let $0 <  p<\infty$.
Given an interval $I = [r, t]\subset \R_+$, we denote by $ \V^{p}_2(I ; X)$ the set of 
two-parameter continuous   functions $u$ from $\Dl_{2,I} = \{ (t_1,t_2)\in I^2: \, t_1\ge  t_2 \}$ to $X$, 
vanishing on the diagonal: 
 $u_{t, t} = 0$ for any $t \in I$, 
such that 
 the following (quasi-)norm is finite:
\begin{align}
\label{V2def}
\| u \|_{\V_2^{p }(I; X) }
 = \sup_{\Pi(I)}
\bigg(\sum_{j=0}^{k-1} \| u_{t_{j}, t_{j+1}} \|^{p}_{X}\bigg)^\frac 1p , 
\end{align}

\noi
 where 
 the supremum is taken over  partitions
 $\Pi(I)$  of the interval $I = [r, t]$ of the form  \eqref{part1}. 
When $I = [0,T]$,  we use the 
following short-hand notation: 
\[\V^p_{2,T} X = \V^p_2([0,T]; X).\]

Let $1 \le p < \infty$.
Given  $u \in \V^{p}_{2, T} X $, 
we define 
$  \o^{(2)}_{X, p}(u):\Dl_{2,T} \to \R_+$ 
by 
\begin{align}
\label{ctrl1}
\o_{X, p}^{(2)}(u;t, r) 
: = \| u\|^p_{\V^p_2([r,t];X)}.
\end{align}

\noi
%
%
By adapting the proof of 
\cite[Proposition  5.8]{FV10}
to functions in $ \V^{p}_{2, T}X$, 
we see that 
$\o^{(2)}_{X, p}(u)$ is a control.
See also
\cite[Lemma 3.3.1] {LQ02}
under the assumption of Chen's relation
(but without assuming the vanishing on the diagonal: $u_{t, t}=0$).

\smallskip

Given $0 < p < \infty$,  we also define the space $\V_{3,T}^{p}X = \V^p_{3} ([0,T]; X)
\subset C_{3, T}X$ 
to be the class of  three-parameter continuous functions $u:\Dl_{3,T}\to X$, 
satisfying \eqref{vani1},  
such that 
there exists a  control  $\o$  in the sense of Part~(i),  satisfying 
\begin{align}
\|u_{t_1,t_2,t_3}\|_X^p \le \o(t_1,t_3)
\label{con2}
\end{align}

\noi
for any $(t_1,t_2,t_3) \in \Dl_{3,T}$.

%
%
%
%
%
%

Lastly, for $n = 2, 3$, 
we set 
\begin{align}
\V^{1-}_{n,T} X = \V^{1-}_{n}([0,T]; X)= \bigcup_{ 0<p<1} \V^{p}_{n,T} X.
\notag 
\end{align}

\end{definition}

\begin{remark}\label{REM:sew2}
\rm

Recall  the boundary condition $u(\infty) = 0$ 
imposed 
for the $V^p$-norm 
in Definition~\ref{DEF:Up}\,(iii).
Then, together with \eqref{local}, we have  
\begin{align}
\|u\|_{L^\infty_I X} \le \| u\|_{V^p_IX}
\label{ctrl0b}
\end{align}

\noi
for any interval $I \subset \R$.
Then, 
it follows from \eqref{U2},  \eqref{ctrl0}, and 
\eqref{ctrl0b} with \eqref{U1} and~\eqref{part1}
(note that the latter partition include the endpoints
$r$ and $t$ of the interval $[r, t]$)
that 
\begin{align}
\| u \|^p_{V^p([r,t]; X )}\sim  \o_{X, p}^{(1)}(u;t, r) + \| u \|^p_{L^\infty([r, t]; X)}
\label{ctrl0a}
\end{align}

\noi
for any  $(t,r) \in \Dl_{2,T}$ and $u \in V^p_T X$,
where the time restriction norm on the left-hand side is as in 
\eqref{local}.
Here, the inequality $\les$ in \eqref{ctrl0a}
follows from 
consider the following extension of~$u$ onto $\R$:
\begin{align}
\label{ext1}
\wt u = \ind_{(-\infty, r)} \cdot u_r + \ind_{[r,t)} \cdot u + \ind_{[t, \infty)} \cdot u_t.
\end{align}

\noi
In Appendix \ref{SEC:B}, we
show that the extension $\wt u$ in \eqref{ext1}
indeed attains the infimum in \eqref{local}.



\smallskip

Let  $u \in \V^p_T X$.
Then, from \eqref{ctrl0} and \eqref{V2def} that 
\begin{align}
\o_{X, p}^{(1)}(u;T, 0)
&=
\| \updl u \|_{\V^p_{2,T} X}^p,
\label{ctrl000}
\end{align}

\noi
where 
 $\updl$ is as in  \eqref{updl23}.
Moreover, if $u_{t_0}=0$ for some $t_0 \in [0, T]$, then 
it follows  from~\eqref{ctrl0a} and \eqref{ctrl000} that 
\begin{align}
\| u\|_{V^p_T X}^p \les \| \updl u \|_{\V^p_{2,T} X}^p = \o_{X, p}^{(1)}(u;T, 0).
\label{ctrl000a}
\end{align}

Lastly, for $0 < \al < 1$, we note from
\eqref{V2def} and H\"older's inequality with  \eqref{Ho2} that 
\begin{align}
\|  u \|_{\V^\frac 1\al _{2,T} X} \le  T^\al \|u\|_{C^\al_{2,T} X}.
\label{con1x}
\end{align}

\end{remark}

\begin{remark}\label{REM:sew3}
\rm 
 Let $0 < p < 1$ and $I = [r, t]\subset \R_+$ be an interval.
Recall from \cite[Proposition~5.2]{FV10}
that  any function $u \in \V_T^pX$ is a constant function.
Here, the continuity assumption is important.
Similarly,  it follows from 
a slight modification of the proof of \cite[Proposition~5.2]{FV10} for functions in $\V_T^pX$
 that if $R \in \V^p_{2, I}X$, 
then we have 
\begin{align}
\lim_{|\Pi([r,t])| \to 0} \sum_{j=0}^{k-1} R_{t_j, t_{j+1}} = 0, 
\label{con3}
\end{align}

\noi
where the limit is taken over any partition $\Pi([r,t])$ of $[r,t]$
of the form \eqref{part1}
whose mesh size $|\Pi([r,t])|  = \sup_j |t_j - t_{j+1}|$ goes to $0$.

\end{remark}

Before stating the sewing lemma, 
we recall that the $p$-variation norm 
is closely related to the $p^{-1}$-H\"older norm.
From H\"older's inequality (see \eqref{con1x}), we have
\begin{align*}
\o_{X, p}^{(1)}(u;T, 0)^\frac 1p \le T^\frac 1p \|u \|_{C^{p^{-1}}_T X}.
\end{align*}

\noi
Furthermore, 
 $u$ is a function of bounded $p$-variation 
if and only if it is a reparametrization of a $p^{-1}$-H\"older continuous function;
see \cite[Propositions 1.21 and 5.14]{FV10}.
In Subsections~\ref{SUBSEC:sewY}
and~\ref{SUBSEC:sewR}, 
we briefly go over the construction of Young\,/\,rough integrals, 
using the sewing lemma (Lemma~\ref{LEM:sew}).
For this purpose, it is often more convenient
to state results
in terms of the reciprocal of $p$:
\begin{align}
\al = \frac 1p
\label{reg9}
\end{align}

\noi
rather than the summability index $p$ itself.
We also note that, in 
Section \ref{SEC:ops}, we study almost sure 
mapping  properties
of the random drivers
$\vXX$ and $\vbbX$, defined in \eqref{introX} and~\eqref{bX0}, respectively, 
in terms of H\"older regularity $\al$.

\medskip

We now present a version of the sewing lemma  adapted to the
current $p$-variation setting.
See, for example, 
\cite[Lemma 2.2]{DGHT19} 
for an analogous statement whose
proof is based on an 
adaptation of the proof of 
\cite[Lemma 4.2]{FH20}
in the H\"older regularity setting.
Our statement corresponds to 
\cite[Proposition~2.3 and  Corollaries 2.4 and 2.5]{GT10}
in the H\"older regularity setting; 
see also \cite{CGLLO1, CGLLO2}.


\begin{lemma}[sewing lemma]
\label{LEM:sew}
Let $X$ be a Banach space and fix $T>0$.
 Then, there exists a unique linear map {\rm (}called the \emph{sewing map}{\rm ):}
\begin{align*}
\Ld: \V^{1-}_{3,T} X \cap \Im \updl |_{C_{2,T} X} \too \V_{2, T}^{1-}X
\end{align*}
such that

\smallskip
\begin{enumerate}
\item[\textup{(i)}]
We have $\updl \Ld h  = h $ for each $h \in  \V^{1-}_{3,T} X \cap \Im \updl |_{C_{2,T} X}$.

\medskip
\item[\textup{(ii)}]
Given $\al>1$, let $h \in \V^{\frac 1\al}_{3,T} X
\cap \Im \updl |_{C_{2,T} X} $ with a control $\o$ such that
\begin{align*}
\| h_{t_1,t_2,t_3} \|_{X} \le \o(t_1,t_3)^\al
\end{align*}
for any $(t_1,t_2,t_3) \in \Dl_{3,T}$ as in \eqref{con2}. 
Then, there exists $C_\al>0$, depending only on $\al$, such that the following estimate holds{\rm :}
\begin{align}
\label{sew1}
\| (\Ld h)_{t,r} \|_{X} \le C_\al \o(t,r)^\al
\end{align}

\noi
for any $(t, r) \in \Dl_{2,T}$.

\medskip
\item[\textup{(iii)}] Given any $g \in C_{2,T}X$ with $\updl g\in \V^{\frac1\al}_{3,T} X$ for some $\al>1$, there exists a unique $f\in C([0, T];  X) $ {\rm (}up to an additive constant{\rm )} such that
\begin{align*}
\updl f = (\Id - \Ld \updl) g.
\end{align*}

\noi
In addition, we have
\begin{align}
\label{sew2}
(\updl f)_{t,r} = \lim_{|\Pi([r,t]) | \to 0} \sum_{j=0}^{k-1} g_{t_{j}, t_{j+1}}
\end{align}

\noi
for any $(t,r) \in \Dl_{2,T}$, 
where the limit is taken over any partition $\Pi([r,t])$ of $[r,t]$
of the form \eqref{part1}
whose mesh size $|\Pi([r,t])|  = \sup_j |t_j - t_{j+1}|$ goes to $0$.

\end{enumerate}

\end{lemma}

For the sake of completeness, we sketch a proof 
of Lemma \ref{LEM:sew}
in Appendix \ref{SEC:sew}, 
with a particular attention 
in showing 
continuity and vanishing on the diagonal
(as required in Definition \ref{DEF:V23}\,(ii)).
See also Remark~\ref{REM:GT}.

\subsection{Young integral}
\label{SUBSEC:sewY}

In this subsection, we briefly go over the construction of Young integrals 
in the current $p$-variation setting,
which follows from a 
slight modification of the presentations  in \cite[Section 3]{G04} and \cite[Subsection 2.3.1]{GT10}
for the H\"older regularity setting.

Let $X$ be a Banach space and fix $T > 0$.
Given  $\frac 12 <\al\le 1$, let  $\XX\in \V^{\frac1\al}_{2,T} \LOP(X)$ be a  driver, satisfying
\begin{align}
\label{Chen}
(\updl \XX)_{t_1,t_2,t_3} =0 
\end{align}

\noi
for any $(t_1,t_2,t_3) \in \Dl_{3,T}$.
In the following, we assume that the driver $\XX_{t,r}$ is given by an
integral operator over the time interval $[r,t]$;
 see, for example, \eqref{XX0a}.
Given $u \in \V^{\frac1{\al_0}}_T X$ for some $0 < \al_0 \le 1$, 
our goal is to construct 
 the Young integral $\I^{\XX}(u)$ of $u$ with  the driver $\XX$, 
 which is 
a (unique) function, 
formally given by 
\begin{align}
\label{y1aa}
\I^{\XX}(u )(t)  = \XX_{t,0} (u_{\bullet}).
\end{align}

\noi
Here,  $\bul$ denotes  the variable of integration.
Note that the expression on the right-hand side of~\eqref{y1aa}
is merely formal.

The increment of the Young integral $\I^{\XX}(u)$ 
in \eqref{y1aa}
is given by
\begin{align}
\label{y1a}
\big(\updl  \I^{\XX} ( u ) \big)_{t,r}  = \XX_{t,r} (u_{\bullet})
\end{align}

\noi
for $(t,r) \in\Dl_{2,T}$, where $\updl$ is as in \eqref{updl23}.
As mentioned above,  the right-hand side of \eqref{y1a} is merely formal.
By replacing $\bul$ with the left endpoint $r$, we have 
\begin{align}
\label{y1b}
\XX_{t,r}(u_\bul) = \XX_{t,r}(u_r) + R_{t,r} =: \Theta_{t,r} + R_{t,r}  
\end{align}

\noi
for some two-parameter function $R = R^{\XX, u}$.
Our goal is to find {\it one} such error term $R$ with sufficient regularity 
which will allow us to define the Young integral $\I^{\XX}(u)$
as the unique limit of Riemann-Stieltjes type sums;
see Remark \ref{REM:sew3}.

By applying the coboundary operator $\updl$  to 
both sides of \eqref{y1b} with \eqref{y1a} and the fact that $\updl^2= 0$, 
we have 
\begin{align}
(\updl R)_{t_1,t_2,t_3}
= -  (\updl \Theta)_{t_1,t_2,t_3}
 = \XX_{t_1,t_2} (\updl u)_{t_2,t_3}
\notag 
\end{align}

\noi
for  $(t_1,t_2,t_3) \in \Dl_{3,T}$, 
where we used 
 \eqref{Chen} at the second equality.
It follows from the regularity assumptions on $\XX$ and $u$
that 
$\updl R \in \V_{3,T}^{\frac{1}{\al+\al_0}} X$;
see the proof of Lemma \ref{LEM:intY}\,(i) below.
Hence, if $\al+\al_0 >1$, the sewing lemma (Lemma~\ref{LEM:sew}) 
allows us to define an error term~$R$ by the relation:
\begin{align}
\label{y1c}
R  = - \Ld \updl \Theta \in \V_{2,T}^{\frac{1}{\al+\al_0}} X,
\end{align}

\noi
where $\Ld$ denotes the sewing map.
Consequently,
by combining \eqref{y1aa}, \eqref{y1b}, and \eqref{y1c}, 
we define the Young integral $\I^{\XX}(u)$ 
by 
\begin{align}
 \I^{\XX} (u)(t) = \big[(\Id - \Ld\updl) \Theta\big]_{t, 0}.
\notag 
\end{align}

\noi
Note from \eqref{y1a}, \eqref{y1b}, and \eqref{y1c}
that the increment of $\I^\XX(u)$ is given by 
\begin{align}
\label{intY1aa}
\updl  \I^{\XX} (u) = (\Id - \Ld\updl) \Theta.
\end{align}

\noi
Therefore, from \eqref{intY1aa}
and 
Remark \ref{REM:sew3} with $\al + \al_0 > 1$
(which shows that the contribution from the error term $R=- \Ld \updl \Ta$ in \eqref{y1c}
vanishes in the limit), 
we see that 
the Young integral $\I^\XX(u)$
is 
indeed given as the unique limit of the Riemann-Stieltjes type sums:
\begin{align}
\label{intY1f}
\I^{\XX}(u)(t) = \lim_{|\Pi([0,t])| \to 0} \sum_{j=0}^{k-1} \Theta_{t_j, t_{j+1}}
=
\lim_{|\Pi([0,t])| \to 0} \sum_{j=0}^{k-1} \XX_{t_j, t_{j+1}} (u_{t_{j+1}} ), 
\end{align}

\noi
where the limit is in the sense of   Lemma \ref{LEM:sew}\,(iii).
We refer to 
a driver $\XX\in \V^{\frac1\al}_{2,T} \LOP(X)$, satisfying~\eqref{Chen}, 
as a {\it Young driver}.

\medskip

In the following lemma, 
we  summarize basic  properties of the Young integral $\I^{\XX}( u )$ obtained via the construction above.

\begin{lemma}[Young integral]
\label{LEM:intY}
Let $X$ be a Banach space.
Given
 $0<\al, \al_0\le1$ and  $T>0$,
let $\XX \in \V_{2,T}^{\frac1\al} \LOP(X)$
be a  driver, satisfying \eqref{Chen}, and $u \in \V^{\frac1{\al_0}}_T X$
{\rm (}see \eqref{local2}{\rm )}.
Suppose that 
 $\al + \al_0 >1$. 
 Then, the following statements hold{\rm :}

\smallskip
\noi\textup{(i)}
Let $\Theta$ be as in \eqref{y1b}{\rm :}
\begin{align}
\label{Ta}
\Theta_{t,r} = \XX_{t,r} (u_r)
\end{align}
for $(t,r) \in \Dl_{2,T}$.
Then, we have
\begin{align}
\Theta \in \V^{\frac1\al}_{2,T} X\qquad\text{and}\qquad  \updl\Theta \in \V^{\frac1{\al+\al_0}}_{3,T} X.
\label{Ta2}
\end{align}

\smallskip
\noi\textup{(ii)}
There exists a unique function $\I^\XX (u) \in \V_{T}^{\frac1\al} X$ 
with $\I^\XX(u)(0)=0$, satisfying~\eqref{intY1aa}
and~\eqref{intY1f},  such that
\begin{align}
\label{intY1a}
\| \updl \I^\XX(u) - \Theta \|_{\V^{\frac1{\al+\al_0}}_{2,T} X } & \les
\|\XX\|_{\V^{\frac1\al}_{2,T} \LOP(X)} \| u \|_{V^{\frac1{\al_0}}_{T} X},
 \\
\label{intY1b}
\| \I^\XX(u) \|_{V^{\frac1\al}_{T} X} 
& \les
 \| \XX\|_{\V^{\frac1\al}_{2,T} \LOP(X)}
 \|u \|_{V^{\frac1{\al_0}}_{T} X }.
\end{align}

\noi
We refer to 
 $\I^{\XX}(u)$ as the Young integral of $u$ with the Young driver $\XX$.

\smallskip
\noi\textup{(iii)} 
Given two drivers $\XX^{(j)} \in \V^{\frac1\al}_{2,T} \L(X)$ satisfying $\updl\XX^{(j)} =0 $, $j=1,2$, we have
\begin{align}
\| \I^{\XX^{(1)}}(u) - \I^{\XX^{(2)}}(u) \|_{V^{\frac1\al}_{T} X} &
\les
 \| \XX^{(1)} - \XX^{(2)}\|_{\V^{\frac1\al}_{2,T} \LOP(X)} \|u\|_{V^{\frac1{\al_0}}_{T} X}
\label{intY1c}
\end{align}

\noi
for any $u \in \V^{\frac1{\al_0}}_T X$,
where $\I^{\XX^{(j)}}(u)$ denotes the Young integral of $u$ with the Young driver $\XX^{(j)}$ constructed in Part \textup{(ii)}.

\end{lemma}

\begin{proof}

The proof essentially follows from the strategy described in \cite[Section~3]{G04} and \cite[Subsection~2.3.1]{GT10} in the H\"older regularity setting.
We present some details for readers' convenience.

\smallskip

\noi
(i) From  \eqref{Ta} with \eqref{ctrl0b},  we have
\begin{align}
\label{intY1cc}
\begin{split}
\| \Theta_{t,r}\|_{ X} 
& = \|\XX_{t,r} (u_r) \|_{X} 
\le \| \XX_{t,r} \|_{\LOP(X)} \|u\|_{L^\infty_T X}\\
& \les  \| \XX_{t,r} \|_{\LOP(X)}  \|u \|_{V^{\frac1{\al_0}}_{T} X }
\end{split}
\end{align}

\noi
for $(t,r)\in\Dl_{2,T}$.
Hence, we conclude from \eqref{intY1cc}, 
$\XX \in \V^{\frac1\al}_{2,T} \LOP(X)$, 
and the continuity of $u$
that $\Theta \in \V^{\frac1\al}_{2,T} X$.

Given   $\XX \in \V^{\frac1\al}_{2, T}\LOP(X)$
and 
$u \in \V^{\frac1{\al_0}}_T X$, 
let 
$\o_\frac 1\al^{(2)}(\XX)
= \o_{\L(X), \frac 1\al}^{(2)}(\XX)$ and 
$\o_\frac 1{\al_0}^{(1)}(u) = \o_{X, \frac 1{\al_0}}^{(1)}(u) $
be as in 
\eqref{ctrl1} and~\eqref{ctrl0}, respectively.
Recall from Definition \ref{DEF:V23}\,(ii) and (iii) that 
$\o_\frac 1\al^{(2)}(\XX)$ and 
$\o_\frac 1{\al_0}^{(1)}(u) $ are controls.
Then,
it follows from 
\eqref{con1a}
that 
\begin{align}
\o : = \big(\o_\frac 1\al^{(2)}(\XX)\big)^\frac{\al}{\al + \al_0}
\big(\o_\frac 1{\al_0}^{(1)}(u)\big)^\frac{\al_0}{\al + \al_0}  
\label{con4}
\end{align}

\noi
is also a control.

A direct computation with \eqref{Ta} and \eqref{Chen} yields
\begin{align}
 (\updl \Theta)_{t_1,t_2,t_3}
 =-  \XX_{t_1,t_2} \big((\updl u)_{t_2,t_3}\big)
\label{con5}
\end{align}

\noi
for  $(t_1,t_2,t_3) \in\Dl_{3,T}$
and thus  we have 
\begin{align}
\label{intY2ab}
\begin{split}
\| (\updl \Theta)_{t_1,t_2,t_3} \|_{X}^\frac 1{\al + \al_0}
 & \le \| \XX_{t_1,t_2} \|_{\LOP(X)}^\frac 1{\al + \al_0} \| (\updl u )_{t_2,t_3}\|_{X}^\frac 1{\al + \al_0} \\
& \le \big( \o_\frac 1\al^{(2)}(\XX; t_1,t_2)\big)^\frac \al{\al + \al_0} 
\big(\o_\frac 1 {\al_0} ^{(1)}(u;t_2,t_3)\big)^\frac {\al_0}{\al + \al_0} \\
& \le 
\o(t_1,t_3)
\end{split}
\end{align}

\noi
for  $(t_1,t_2,t_3) \in \Dl_{3,T}$, 
where
we used  \eqref{con1b} and \eqref{con4}
in the last step.
Hence, from \eqref{con2}, \eqref{intY2ab}, 
and the continuity of $\XX$ and $u$, 
we  conclude that $\updl \Theta \in \V^{\frac{1}{\al+\al_0}}_{3,T} X$.

\medskip

\noi(ii) 
In view of  \eqref{Ta2} with $\al + \al_0 > 1$, 
we can apply the sewing lemma (Lemma~\ref{LEM:sew}) 
to define the error term $R = - \Ld \updl \Ta
\in \V_{2,T}^{\frac{1}{\al+\al_0}} X$ as in \eqref{y1c}.
Then,  together with \eqref{y1b}, 
we see that there exists a unique function $\I^\XX (u) \in  \V^{\frac1\al}_{T} X$
with $\I^\XX(u)(0) =0$, 
satisfying \eqref{intY1aa} and  \eqref{intY1f}.
The first bound 
\eqref{intY1a} follows from 
\eqref{V2def}, 
\eqref{intY1aa},  and
\eqref{sew1} in the sewing lemma (Lemma~\ref{LEM:sew}) 
with \eqref{intY2ab}, \eqref{con4}, and \eqref{ctrl0a}, 
while the second bound
\eqref{intY1b} follows from 
\eqref{ctrl000a}, 
\eqref{intY1a}, 
and~\eqref{intY1cc}.


\medskip

\noi(iii) 
The bound \eqref{intY1c} follows from a straightforward modification 
of the proof of \eqref{intY1b}
by replacing $\Ta$ in \eqref{Ta}
with
\begin{align*}
\wt \Ta_{t,r} = (\XX^{(1)} - \XX^{(2)})_{t,r} (u_r)
\end{align*}
for $(t,r) \in \Dl_{2,T}$.
We omit details.
\end{proof}

\subsection{Rough integral}
\label{SUBSEC:sewR}

In this subsection, 
we go over the construction of rough integrals in the current $p$-variation setting.
Our goal is to construct 
 an
integral $\I(u)$ of $u$,  whose increment on $[r,t]$ is given by 
$\XX_{t,r} (u_{\bullet})$ (see \eqref{y1a}), under the weaker assumption
 $\al + \al_0 \le 1$ (in the setting of Lemma \ref{LEM:intY}).
In this case, $\updl \Ta$ in \eqref{con5} does not have sufficient regularity
to apply the sewing lemma (Lemma~\ref{LEM:sew}).
In order to overcome this issue, 
we resort to  rough path theory, 
introduced in a seminal work by Lyons
\cite{Lyons98}, 
and controlled paths, introduced by Gubinelli \cite{G04}.
Our presentation closely follows
 \cite[Proposition~1]{G04} and  \cite[Subsection~4.3]{CGLLO2}
 in the H\"older regularity setting.

We first recall
the notion of 
rough paths and controlled  paths.

\begin{definition}
\label{DEF:RP}
\rm

Let $X, Y$ be  Banach spaces with $Y\subset X$ and $T>0$.

\smallskip

\noi\textup{(i)} Given $\frac13 < \al  \le \frac12 $,  we say that 
$(\XX, \bbX) \in \V^{\frac1\al }_{2,T} \LOP(X;Y) \times \V^{\frac{1}{2\al}}_{2,T} \LOP(X;Y)$ 
is a $\frac1\al$-variational rough path  with values in $\LOP(X;Y)$ if 
it satisfies  Chen's relation{\rm :}
\begin{align}
\begin{split}
(\updl \XX)_{t_1,t_2,t_3} &= 0, \\
(\updl \bbX)_{t_1,t_2,t_3} & = \XX_{t_1,t_2}  \XX_{t_2,t_3}
\end{split}
\label{Chen2}
\end{align}

\noi
for  $(t_1,t_2,t_3) \in\Dl_{3,T}$, 
where 
$ \XX_{t_1,t_2}  \XX_{t_2,t_3} = \XX_{t_1,t_2} \circ \XX_{t_2,t_3}$
is understood as  the usual composition of operators.

\smallskip
\noi\textup{(ii)}
Let $0< \al < \g \le 1$ and $\XX \in \V^{\frac1\al}_{2,T} \LOP(X)$ with $\updl \XX =0$.
We say that a function $u \in \V^{\frac1\al}_T X$
 is  controlled by $\XX$, with a Gubinelli derivative\footnote{A Gubinelli derivative is not uniquely determined
 in general. See \cite[Remark 4.7]{FH20}.} $u' \in \V^{\frac1\al}_T X$, if there exists a remainder
 term\footnote{In the literature (see, for example, \cite[Theorem 4.10]{FH20}),
the time regularity of the remainder $R^{\XX, u}$ is often taken as $\g=2\al$, which is too restrictive in our case. See the original paper \cite{G04}, 
where the weaker condition $\g> \al$ is also used.
}
$R^{\XX, u} \in \V^{\frac1\g}_{2,T}X$ such that
\begin{align}
\label{CRP}
(\updl u)_{t,r} = \XX_{t,r} (u'_r) + R^{\XX, u}_{t,r}
\end{align}
for $(t,r)\in\Dl_{2,T}$.
We denote  the space of controlled paths $(u, u')$ by 
$\D^{\al, \g}_{\XX,T}(X) = \D^{\al, \g}_{\XX}([0, T]; X),$ equipped with the norm:
\begin{align}
\label{CRPnorm}
\| (u,u') \|_{\D^{\al, \g}_{\XX, T}(X) }:= \| u(0)\|_{X} + \| u'(0) \|_{X} 
+ \frac14 \| u'  \|_{V^{\frac1\al }_{T} X} + \| R^{\XX, u} \|_{\V^{\frac1\g}_{2,T} X},
\end{align}

\noi
where the $V^{p}_T$- and $\V^{p}_{2,T}$-norms are as in 
\eqref{local} and \eqref{V2def}, respectively.

\end{definition}

\begin{remark} \label{REM:CRP2}\rm

(i)
From \eqref{CRP} and 
$\XX \in \V^{\frac1\al}_{2,T} \LOP(X)$
with \eqref{ctrl0a} and \eqref{ctrl000}, 
we have $u \in \V^{\frac1\al}_T X$.

\noi
(ii)
The factor $\frac14$ in front of the norm of $u'$ in \eqref{CRPnorm} 
is  needed  to close a contraction argument
in the rough case presented in Section \ref{SEC:rough}.
 See Remark~\ref{REM:CRPnorm} for further details.


\end{remark}

Let $X, Y$ be Banach spaces with $Y\subset X$, $\frac 13  <  \al \le \frac12$, and
$\XX
\in \V^{\frac1\al }_{2,T} \LOP(X;Y)$.
Define $\bbX$ by setting
\begin{align}
\bbX_{t,r} = 
\XX_{t,r} \circ \XX_{\bullet, r} 
\label{con6x}
\end{align}

\noi
for $(t,r) \in \Dl_{2,T}$.
In the following, we suppose that 
$(\XX,\bbX)\in \V^{\frac1\al}_{2,T} \LOP(X;Y) \times \V^{\frac{1}{2\al}}_{2,T} \LOP(X;Y)$
is a rough path as in Definition~\ref{DEF:RP}~(i). 
Given 
a controlled rough path $(u,u') \in \CRP^{\al,\g}_{\XX, T}(X)$
(for some $\g \in (\al, 1]$), satisfying \eqref{CRP}
for some remainder term $R^{\XX, u} \in \V^{\frac1\g}_{2,T}X$, 
our goal is to construct the rough integral
$\I^{\XX, \bbX}(u)$ of $u$ with the driver $(\XX, \bbX)$, 
satisfying 
$\I^{\XX, \bbX}(u)(0) = 0$, whose 
 increment is formally given by
\begin{align}
\big(\updl \I^{\XX, \bbX}(u)\big)_{t,r} = \XX_{t,r}(u_\bul)
\label{int2x}
\end{align}

\noi
for $(t,r) \in \Dl_{2,T}$.

By replacing $\bul$ with the left endpoint $r$
and substituting the controlled structure \eqref{CRP} with \eqref{con6x},  
we have
\begin{align}
\label{int2b}
\begin{split}
\XX_{t,r} (u_\bullet)
&
= \XX_{t,r} (u_r) + \XX_{t,r}\big( (\updl u)_{\bullet, r} \big)
\\
& = \XX_{t,r} (u_r) + \bbX_{t,r}(u'_r ) + {Q}_{t,r} \\
& =: \Taa_{t,r} + {Q}_{t,r},
\end{split}
\end{align}

\noi
for 
some two-parameter function $Q=Q^{\XX, u}$,
formally given by $Q _{t, r} = \XX_{t, r}(R^{\XX, u}_{\bul, r})$.

By applying the coboundary operator $\updl$  to 
both sides of \eqref{int2b} with \eqref{int2x} and the fact that $\updl^2= 0$, 
a direct computation with 
\eqref{Chen2} and \eqref{CRP} yields
\begin{align}
\label{int2bb}
\begin{split}
(\updl Q)_{t_1,t_2,t_3}
&
= - (\updl \Taa)_{t_1,t_2,t_3}
\\
&
= \XX_{t_1,t_2} (R^{\XX, u}_{t_2,t_3}) + \bbX_{t_1,t_2} \big( (\updl u')_{t_2,t_3} \big)
\end{split}
\end{align}

\noi
for $(t_1,t_2,t_3) \in\Dl_{3,T}$.
From the regularity assumptions on  $\XX$, $\bbX$, $u'$, and $R^{\XX, u}$, 
we see that  $\updl Q \in \V^{\frac1\kk}_{3,T} Y $, 
where $\kk = (\al+\g) \land 3\al$;
see the proof of Lemma \ref{LEM:intR}\,(i) below.
Hence, if 
$\kk  > 1$,
the sewing lemma (Lemma~\ref{LEM:sew}) 
allows us to define an error term~$Q$ by the relation:
\begin{align}
Q = - \Ld \updl \Taa \in \V^{\frac1\kk}_{2,T} Y,
\label{int2y}
\end{align}
where $\Ld$ denotes the sewing map.
Consequently,
by combining \eqref{int2x}, \eqref{int2b}, and \eqref{int2y}, 
we define the rough integral $\I^{\XX, \bbX}(u)$ of $u$ with the driver $(\XX, \bbX)$ 
by 
\begin{align}
 \I^{\XX, \bbX} (u)(t) = \big[(\Id - \Ld\updl) \Taa\big]_{t, 0}
\label{int2z}
\end{align}

\noi
whose increment is given by 
\begin{align}
\label{int2c}
\updl\I^{(\XX,\bbX)}(u) = (\Id - \Ld\updl) \Taa.
\end{align}

\noi
Therefore, 
 from \eqref{int2c}
and 
Remark \ref{REM:sew3} with $\kk > 1$
(which shows that the contribution from the error term $Q= - \Ld \updl \Taa$ in \eqref{int2y}
vanishes in the limit), 
we see that 
the rough integral $\I^{\XX, \bbX}(u)$
is 
indeed given as the unique limit of the Riemann-Stieltjes type sums:
\begin{align}
\label{int2d}
\begin{split}
\I^{\XX,\bbX} (u) (t)
& = \lim_{|\Pi([0,t])| \to 0} \sum_{j=0}^{k-1} \Taa_{t_{j}, t_{j+1}}\\
& =
\lim_{|\Pi([0,t])| \to 0} \sum_{j=0}^{k-1}
\Big( \XX_{t_{j}, t_{j+1}} (u_{t_{j+1}} )
+  \bbX_{t_{j}, t_{j+1}} (u_{t_{j+1}}' )
\Big), 
\end{split}
\end{align}

\noi
where the limit is in the sense of   Lemma \ref{LEM:sew}\,(iii).
We refer to 
a $\frac1\al$-variational rough path
 $(\XX,\bbX)  \in \V^{\frac1\al}_{2,T} \LOP(X;Y) \times \V^{\frac{1}{2\al}}_{2,T} \LOP(X;Y)$,  satisfying \eqref{Chen2}, 
 as a {\it rough driver}.

\medskip

In the following lemma, 
we  summarize basic  properties of the rough integral $\I^{\XX, \bbX}( u )$ obtained via the construction above.

\begin{lemma}
\label{LEM:intR}

Let $X, Y$ be Banach spaces with $Y \subset X$.
Given $\frac 13 <\al\le \frac 12 $, 
 $\al < \g \le 1$,  and $T>0$, 
let $(\XX,\bbX)  \in \V^{\frac1\al}_{2,T} \LOP(X;Y) \times \V^{\frac{1}{2\al}}_{2,T} \LOP(X;Y)$
be a $\frac1\al$-variational rough path, satisfying \eqref{Chen2}, 
and  $(u,u') \in \CRP^{\al, \g}_{\XX, T} (X)$ be a controlled  path,
satisfying \eqref{CRP} for some remainder term $R^{\XX, u} \in \V^{\frac1\g}_{2,T} X$,
as in  Definition~\ref{DEF:RP}.
 Suppose that
\begin{align}
\label{int2cond}
\kk : = ( \al + \g ) \land(3\al )>1.
\end{align}
Then, the following statements  hold{\rm :}

\smallskip

\noi\textup{(i)} Let $\Taa$ be as in \eqref{int2b}{\rm :}
\begin{align}
\label{intRY0}
\Taa_{t,r} = \XX_{t,r} (u_r )  + \bbX_{t,r} (u_r')
\end{align}
for $(t,r) \in \Dl_{2,T}$. 
Then, we have 
\begin{align}
\Taa \in \V^{\frac1\al}_{2,T} Y
\qquad \text{and}\qquad  \updl\Taa \in \V^{\frac1\kk}_{3,T} Y.
\label{int2e}
\end{align}

\smallskip

\noi\textup{(ii)}
There exists a unique function $\I^{\XX,\bbX}(u) \in \V^{\frac1\al}_{T} Y$ 
{\rm (}see \eqref{local2}{\rm )}
with $\I^{\XX,\bbX}(u)(0) =0$, 
satisfying~\eqref{int2c} and \eqref{int2d}, such that 
\begin{align}
\label{intR0aa}
\begin{split}
\| \updl \I^{\XX,\bbX}(u) - \Taa \|_{ \V^{\frac1\kk}_{2,T} Y} 
& \les \|\XX\|_{ \V^{\frac1\al}_{2,T} \LOP(X;Y)} \| R^{\XX, u} \|_{ \V^{\frac1\g}_{2,T} X}\\
& \quad + \| \bbX \|_{ \V^{\frac{1}{2\al}}_{2,T} \LOP(X;Y)} \|u '\|_{V^{\frac1\al}_{T} X}  ,
\end{split}
\\
\label{intR0ab}
\begin{split}
\| \I^{\XX,\bbX}(u) \|_{V^{\frac1\al}_{T} Y } 
& \les \|\XX\|_{ \V^{\frac1\al}_{2,T} \LOP(X;Y)} \Big(\|u\|_{L^\infty_T X} 
 +  \| {R}^{\XX, u} \|_{ \V^{\frac1\g}_{2,T} X} \Big)\\
& \quad+ \|\bbX\|_{ \V^{\frac{1}{2\al}}_{2,T} \LOP(X;Y)} \| u' \|_{V^\frac1\al_T X}.
\end{split}
\end{align}

\noi
We refer to 
 $\I^{\XX, \bbX}(u)$ as the rough integral of $u$ with the driver $(\XX, \bbX)$.

\smallskip

\noi\textup{(iii)} 
For  $j = 1, 2$, let 
 $(\XX^{(j)}, \bbX^{(j)}) \in \V^{\frac1\al}_{2,T} \LOP(X;Y) \times \V^{\frac{1}{2\al}}_{2,T} \LOP(X;Y)$
 be $\frac 1\al$-variational rough paths, 
 satisfying \eqref{Chen2},
 and 
 $(u^{(j)}, (u^{(j)})') \in \CRP^{\al, \g}_{\XX^{(j)}, T}(X)$ be 
 controlled paths, satisfying 
\begin{align}
\notag
(\updl u^{(j)})_{t,r} = \XX^{(j)}_{t,r} \big( ( u^{(j)})'_r \big) + R^{(j)}_{t,r}
\end{align}
for some remainder term $R^{(j)} = R^{\XX^{(j)}, u^{(j)}} \in \V^{\frac1\g}_{2,T} X$.
Then, we have 
\begin{align}
\label{intR0ac}
\begin{split}
\| & \I^{\XX^{(1)}, \bbX^{(1)}} (u^{(1)}) - \I^{\XX^{(2)}, \bbX^{(2)}} (u^{(2)}) \|_{V^{\frac1\al}_T Y}
\\
& \les \| \XX^{(1)} - \XX^{(2)} \|_{\V^{\frac1\al}_{2,T} \LOP(X;Y)}
\Big(
\| u^{(1)}  \|_{L^\infty_T X}
+\|R^{(1)} \|_{\V^{\frac1\g}_{2,T} X}\Big)
\\
&
\quad
+
\| \XX^{(2)} \|_{\V^{\frac1\al}_{2,T} \LOP(X;Y)}
\Big(\| u^{(1)} - u^{(2)} \|_{L^\infty_T X}
+\| R^{(1)} - R^{(2)} \|_{\V^{\frac1\g}_{2,T} X}\Big)
\\
&
\quad
+ \| \bbX^{(1)} - \bbX^{(2)} \|_{\V^{\frac{1}{2\al}}_{2,T} \LOP(X;Y)}
\| (u^{(1)} )' \|_{V^{\frac1\al}_{T} X}\\
& \quad
+
\|  \bbX^{(2)} \|_{\V^{\frac{1}{2\al}}_{2,T} \LOP(X;Y)}
\|
(u^{(1)})' - (u^{(2)})' \|_{V^{\frac1\al}_T X}, 
\end{split}
\end{align}

\noi
where $\I^{\XX^{(j)}, \bbX^{(j)}}(u^{(j)})$ denotes the rough integral of $u^{(j)}$ with the rough driver 
$(\XX^{(j)}, \bbX^{(j)})$ constructed in Part \textup{(ii)}.

\end{lemma}

\begin{proof}

\noi(i)
From \eqref{intRY0} with \eqref{ctrl0b},  we have
\begin{align}
\label{intR1aa}
\begin{split}
\| \Taa_{t,r} \|_{Y}& \le \| \XX_{t,r} (u_r) \|_{Y} + \| \bbX_{t,r}(u'_r)\|_{Y} \\
&  \le \|\XX_{t,r}\|_{\LOP(X;Y)} \|u\|_{L^\infty_T X} + \| \bbX_{t,r} \|_{\LOP(X;Y)} \| u' \|_{L^\infty_T X}\\
&  \les \|\XX_{t,r}\|_{\LOP(X;Y)} 
\|u \|_{V^{\frac1\al}_{T} X }
+ \| \bbX_{t,r} \|_{\LOP(X;Y)} 
\|u' \|_{V^{\frac1\al}_{T} X }
\end{split}
\end{align}

\noi
for $(t,r) \in \Dl_{2,T}$.
Hence,  from \eqref{intR1aa}
with the regularities and continuity of 
$\XX$, $\bbX$, $u$, and $u'$, 
we conclude
that $\Taa \in \V^{\frac1\al}_{2,T} X$.

Let $\o^{(1)}_{X, p}$ and $\o^{(2)}_{X, p}$ be as in 
\eqref{ctrl0} and \eqref{ctrl1}, respectively.
Then,
it follows from 
\eqref{con1a} with \eqref{int2cond}
and the fact that a sum of controls is a control
that 
\begin{align}
\begin{split}
\o: \!
& = \big(\o^{(2)}_{\L(X;Y), \frac1\al} (\XX)\big)^\frac \al \kk
\big(\o^{(2)}_{X, \frac1\g} (R^{\XX, u})\big)^\frac \g \kk\\
& \quad  + \big(\o^{(2)}_{\L(X;Y), \frac1{2\al}} (\bbX)\big)^\frac {2\al} \kk
\big(\o^{(1)}_{X, \frac1\al} (u')\big)^\frac \al\kk
\end{split}
\label{con6}
\end{align}

\noi
is also a control.
From 
 \eqref{int2bb}, 
\eqref{con1b}, and \eqref{con6}, we have 
\begin{align}
\label{intR1ab}
\begin{split}
\| (\updl \Taa)_{t_1,t_2,t_3} \|_{Y}
&
\le \| \XX_{t_1,t_2} \|_{\LOP(X;Y)} \|R^{\XX, u}_{t_2,t_3} \|_{X} + \| \bbX_{t_1,t_2}\|_{\LOP(X;Y)} \| (\updl u' )_{t_2,t_3} \|_{X}
\\
&
\le \big(\o^{(2)}_{\L(X;Y),  \frac1\al} (\XX; t_1,t_2)\big)^\al 
\big(\o^{(2)}_{X, \frac1\g} (R^{\XX, u}; t_2,t_3)\big)^\g\\
& \quad + \big(\o^{(2)}_{\L(X;Y), \frac1{2\al}} (\bbX; t_1,t_2)\big)^{2\al }
\big(\o^{(1)}_{X, \frac1\al} (u'; t_2,t_3)\big)^\al\\
& \le\big(\o(t_1,t_3)\big)^{\kk}
\end{split}
\end{align}

\noi
for  $(t_1,t_2,t_3) \in \Dl_{3,T}$.
Hence, from \eqref{con2}, \eqref{intR1ab}, and 
the continuity of $\XX$, $\bbX$, $R^{\XX, u}$, and $u'$, we conclude that 
 $\updl\Taa \in \V^{\frac1\kk}_{3,T} Y$.

\medskip
\noi
(ii) 
In view of \eqref{int2e} with \eqref{int2cond}, 
we can apply the sewing lemma (Lemma \ref{LEM:sew})
to define the error term $Q = - \Ld \updl \Taa \in \V^{\frac1\kk}_{2,T} Y$
as in \eqref{int2y}.
Then,  together with \eqref{int2x} and \eqref{int2b}, 
we see that there exists a unique function $\I^{\XX, \bbX} (u) \in  \V^{\frac1\al}_{T} X$
with $\I^{\XX, \bbX}(u)(0) =0$, 
satisfying \eqref{int2c} and  \eqref{int2d}.
The first bound 
\eqref{intR0aa} follows from 
\eqref{V2def}, 
\eqref{int2c},  and
\eqref{sew1} in the sewing lemma (Lemma~\ref{LEM:sew}) 
with \eqref{intR1ab} and \eqref{con6}, 
while the second bound
\eqref{intR0ab} follows from 
\eqref{ctrl000a}, \eqref{int2c}, 
\eqref{intR0aa}
and \eqref{intR1aa}.

\medskip
\noi
(iii)
As in \eqref{intRY0}, set
 \[\Taa^{(j)}_{t,r} = \XX^{(j)}_{t,r} (u^{(j)}_r) +  \bbX^{(j)}_{t,r} \big( (u^{(j)}_r)' \big) \]
 
 \noi
 for  $(t,r) \in \Dl_{2,T}$, $j = 1, 2$.
 Then, 
  from \eqref{int2c}, we have
\begin{align}
\begin{split}
\big( \updl     &  \I^{(\XX^{(1)}, \bbX^{(1)})} (u^{(1)})\big)_{t,r}
-
\big(\updl \I^{(\XX^{(2)}, \bbX^{(2)})} (u^{(2)}) \big)_{t,r}
\\
&
=
(\XX^{(1)}_{t,r} - \XX^{(2)}_{t,r} ) (u_r^{(1)})
+ \XX^{(2)}_{t,r}  (u_r^{(1)} - u_r^{(2)})\\
&
\quad
+
(\bbX^{(1)}_{t,r} - \bbX^{(2)}_{t,r} ) \big( (u_r^{(1)})' \big)
+
\bbX^{(2)}_{t,r} \big( (u_r^{(1)} )' - (u_r^{(2)})' \big)\\
& \quad 
- \big[ \Ld (\updl \Taa^{(1)} -\updl  \Taa^{(2)}) \big]_{t,r}
\end{split}
\label{con7}
\end{align}

\noi
\noi
for  $(t,r) \in \Dl_{2,T}$, 
whereas it follows from \eqref{int2bb} that 
\begin{align}
\begin{split}
&
(\updl \Taa^{(1)})_{t_1,t_2,t_3} - (\updl \Taa^{(2)})_{t_1,t_2,t_3}
\\
&
=
- (\XX^{(1)}_{t_1,t_2} - \XX^{(2)}_{t_1,t_2} ) (R^{(1)}_{t_2,t_3})
-
\XX^{(2)}_{t_1,t_2} (R^{(1)}_{t_2,t_3} - R^{(2)}_{t_2,t_3})
\\
&
\quad
- (\bbX^{(1)}_{t_1,t_2} - \bbX^{(2)}_{t_1,t_2} )
\big( (\updl (u^{(1)})')_{t_2,t_3}\big)\\
& \quad 
-
\bbX^{(2)}_{t_1,t_2} \big( (\updl (u^{(1)})')_{t_2,t_3} - (\updl (u^{(2)})')_{t_2,t_3}\big)
\end{split}
\label{con8}
\end{align}

\noi
\noi
for  $(t_1,t_2,t_3) \in \Dl_{3,T}$.
The bound \eqref{intR0ac} follows from a straightforward
modification of \eqref{intR0aa} and \eqref{intR0ab} with 
\eqref{con7} and \eqref{con8}.
We omit details.
 \end{proof}

\section{Regularities of Young and rough drivers}
\label{SEC:ops}

One of the key steps in proving pathwise local well-posedness
of SNLW \eqref{SNLW}
(at the level of the vectorial integral formulation \eqref{mild2}
for the interaction representation $\vec \ub$ defined in~\eqref{int})
is to make sense of
 the stochastic term
$\vec{\pPsi} (\vec\ub) = \vec{\pPsi}^\be  (\vec\ub)$, formally defined 
in \eqref{psi2}, 
in the pathwise manner.
For this purpose, 
our main goal in this section is to  study almost sure 
mapping properties of  
 the random operators 
$\vXX$ and $\vbbX$, 
 defined in \eqref{introX} and~\eqref{bX0}, respectively.
See Propositions \ref{PROP:driver1} and \ref{PROP:driver2}, 
where the  random tensor estimate (Lemma \ref{LEM:RTE})
plays a crucial role in estimating the operator norms
of these random operators.

In the fractional-in-time case with the Hurst parameter $\frac 12 < \be < 1$, 
our goal is to prove that, given any $\frac 12 < \al < \be$,  
we have 
\begin{align}
\vXX \in C^{\al}_{2,T} \LOP(\H^s(\T^d))
\subset \V^{\frac 1\al}_{2,T} \LOP(\H^s(\T^d))
\label{AS1}
\end{align}

\noi
and $\updl \vXX = 0$, 
almost surely, under an appropriate assumption
on the relevant parameters.
See \eqref{con1x} for the latter inclusion
in \eqref{AS1}.
Then, 
we can apply
Lemma \ref{LEM:intY}
and 
 construct 
 the stochastic term
$\vec{\pPsi} (\vec\ub)$
in \eqref{psi2}
as the Young integral 
$\I^{\vXX}(\vec \ub) \in \V^{\frac 1\al}_T\H^{s}(\T^d)
\subset C([0, T]; \H^{s}(\T^d))$
with the Young driver $\vXX$
(for any given  $\vec \ub \in \V^{\frac 1\al}_T\H^{s}(\T^d)$)
in the pathwise manner.

In the white-in-time case with the Hurst parameter $\be = \frac 12$, 
our goal is to prove that, given any $  0 < \al < \be = \frac 12$,  
we have 
\begin{align}
\begin{split}
(\vXX, \vbbX )  
& \in C^{\al}_{2,T} \LOP(\H^s(\T^d);\H^{s_0}(\T^d))
\times 
C^{2\al}_{2,T} \LOP(\H^s(\T^d);\H^{s_0}(\T^d))\\
& \subset \V^{\frac 1\al}_{2,T} \LOP(\H^s(\T^d);\H^{s_0}(\T^d))
\times \V^{\frac 1{2\al}}_{2,T} \LOP(\H^s(\T^d);\H^{s_0}(\T^d))
\end{split}
\label{AS2}
\end{align}

\noi
for some $s_0 > s$
and $(\vXX, \vbbX )$  satisfies Chen's relation \eqref{Chen2}, 
almost surely, under an appropriate assumption
on the relevant parameters.
Given 
 $\frac 13< \al <  \frac 12 <  \g < 1$ with $\al + \g > 1$, 
 let 
  $(\vec \ub,\vec \ub') \in \CRP^{\al, \g}_{\vec \XX, T} (\H^s(\T^d))$ be a controlled  path,\footnote{As a consequence of a contraction argument, we will show that $\vec \ub' = \vec \ub$ in our application.
See Section~\ref{SEC:rough}.}
satisfying~\eqref{CRP}:
\begin{align*}
(\updl \vec \ub)_{t,r} = \vXX_{t,r} (\vec \ub '_r) + R^{\vXX, \vec \ub}_{t,r}
\end{align*}

\noi
 for some remainder term $R^{\vXX, \vec \ub} \in \V^{\frac1\g}_{2,T} \H^s(\T^d)$.
Then, 
once we prove \eqref{AS2}, 
we can apply
Lemma \ref{LEM:intR}
and 
 construct 
 the stochastic term
$\vec{\pPsi} (\vec\ub)$ in \eqref{psi2}
as the rough integral 
$\I^{\vXX, \vbbX}(\vec \ub) \in \V^{\frac 1\al}_T\H^{s_0}(\T^d)
\subset C([0, T]; \H^{s_0}(\T^d))$
with the rough driver $(\vXX, \vbbX)$
in the pathwise manner.

\medskip

Before proceeding further, we set some notations.
Given a pair $\vec f = (f_1,f_2)$ of functions on~$\T^d$, 
it follows from 
 \eqref{introX} that 
\begin{align}
\label{XX0}
\vXX_{t,r} (\vec f) & =
\begin{pmatrix}
\XX_{t,r}^{(1,1)}  & \XX_{t,r}^{(1,2)}  \\
\XX_{t,r}^{(2,1)} &  \XX_{t,r}^{(2,2)}
\end{pmatrix}
\vec f
=
\begin{pmatrix}
\XX_{t,r}^{(1,1)} (f_1) +  \XX_{t,r}^{(1,2)} (f_2)  \\
\XX_{t,r}^{(2,1)}(f_1)  +   \XX_{t,r}^{(2,2)} (f_2)
\end{pmatrix}
\end{align}

\noi
for $(t,r) \in \Dl_{2,T}$, 
where $\XX^{(i,j)}_{t,r}$, $i,j=1,2$, are given by 
\begin{align}
\label{XX0a}
\begin{split}
\XX^{(1,1)}_{t,r} (f)& = - \int_r^t \frac{\sin(t'\jb{\nabla})}{\jb{\nabla}} 
\big[\cos(t'\jb{\nabla}) f \cdot \Phi dW^\be (t') \big], \\
\XX^{(1,2)}_{t,r} (f)& = - \int_r^t \frac{\sin(t'\jb{\nabla})}{\jb{\nabla}} 
\bigg[ \frac{\sin(t'\jb{\nabla})}{\jb{\nabla}} f \cdot \Phi d W^\be (t') \bigg], \\
\XX^{(2,1)}_{t,r} (f)& =  \int_r^t \cos(t'\jb{\nabla}) 
\big[\cos(t'\jb{\nabla})
f \cdot \Phi d W^\be (t') \big], \\
\XX^{(2,2)}_{t,r}(f) &  =  \int_r^t \cos(t'\jb{\nabla}) 
\bigg[ \frac{\sin(t'\jb{\nabla})}{\jb{\nabla}} f \cdot \Phi d W^\be (t') \bigg]
\end{split}
\end{align}

\noi
for a function $f$ on $\T^d$, 
where $W^\be $ is the cylindrical process in \eqref{W0}.
For example, from~\eqref{XX0a} with \eqref{phi1}, 
we have 
\begin{align}
\XX^{(1,1)}_{t,r} (f)
= - \sum_{n \in \Z^d} e_n
\int_r^t  \sum_{n_1, n_2 \in \Z^d}
\ind_{n = n_{12}}
\frac{\sin(t'\jb{n})}{\jb{n}} 
\cos(t'\jb{n_2})\ft f(n_2) \phi_{n_1} dB_{n_1}(t'), 
\label{XX0b}
\end{align}

\noi
where $n_{12}  = n_1 + n_2$ is as in \eqref{ord3}.

From 
\eqref{bX0}
with \eqref{XX0}  we have 
\begin{align}
\label{bbX0}
\vbbX_{t,r} (\vec f)
=
\begin{pmatrix}
\bbX^{(1,1)}_{t, r} &  \bbX^{(1,2)}_{t, r} \\
\bbX^{(2,1)}_{t, r} & \bbX^{(2,2)}_{t,  r}
\end{pmatrix}
\vec f
=
\begin{pmatrix}
\bbX^{(1,1)}_{t, r} (f_1 ) +  \bbX^{(1,2)}_{t, r} (f_2 )\\
\bbX^{(2,1)}_{t, r} (f_1) +  \bbX^{(2,2)}_{t,  r} (f_2)
\end{pmatrix}
\end{align}

\noi
for $(t,r) \in \Dl_{2,T}$
and  a pair $\vec f = (f_1,f_2)$ of functions on~$\T^d$, 
where  $\bbX^{(i,j)}_{t,r}$, $i,j=1,2$, are given by
\begin{align}
\label{bbX0a}
\bbX^{(i,j)}_{t,r}
& = \XX^{(i,1)}_{t,r}\XX^{(1,j)}_{\bullet, r}  + \XX^{(i,2)}_{t,r}  \XX^{(2,j)}_{\bullet, r} , 
\end{align}

\noi
where $\XX^{(i,j)}$ is as in \eqref{XX0a}
 and $\bul$ denoting the variable of integration.

With this notation, we
view
 the Young driver $\vXX$ in~\eqref{XX0}
(with $\frac 12 < \be < 1$)
and 
the rough driver $(\vXX, \vbbX)$ 
(with $\be = \frac 12$)
as random operators given by 
(multiple) stochastic integrals with 
respect to (fractional) Brownian motions, 
as introduced in Subsection \ref{SUBSEC:FBM}.
We study their almost
sure mapping properties
by applying 
the  random tensor estimate (Lemma \ref{LEM:RTE}).
In Subsection~\ref{SUBSEC:driver1}, 
we consider the Young case, 
while we treat the rough case in 
Subsection~\ref{SUBSEC:driver2}.

\subsection{Young driver} 
\label{SUBSEC:driver1}

In this subsection, we study regularity properties of the Young driver 
 $\vXX$  in \eqref{XX0}
 (with the Hurst parameter   $\frac 12 < \be< 1$).

We recall the notation \eqref{ord2}:
$a \land b = \min(a,b)$
and  $a \vee b = \max(a,b)$, 
 which is frequently used in 
the remaining part of this section.

\begin{proposition}
\label{PROP:driver1}

Let $\frac 12 < \be< 1$
and 
$ \s \in \R$.
Given  $\Phi \in \HS(L^2(\T^d); H^\s(\T^d))$, 
satisfying~\eqref{phi1}, 
let  $\vXX$ be the first order driver in~\eqref{introX} 
with the Hurst parameter $\be$.
Suppose that $s, s_0, \s \in \R$ satisfy 
\begin{align}
\begin{split}
s + \s  > 0
\quad \text{and}\quad 
s_0  < \min \Big((s\wedge 0) + \s + 1, 
s+ (\s\wedge 0) + 1\Big).
\end{split}
\label{YG0}
\end{align}

\noi
 Then, given $\frac 12 < \al < \be$ and small $\eps > 0$ such that $\be + \eps < 1$, we have
\begin{align}
\big\| \|  \vXX_{t,r}  \|_{\LOP(\H^s; \H^{s_0})} \big\|_{L^p(\Om)} 
 & \les   p^\frac 12  (1\vee T)^{\be- \al + \eps}
\| \Phi\|_{\HS(L^2;H^\s)} 
  |t-r|^\al
\label{YG1}
\end{align}

\noi
for any  
 finite $p \ge 1$,  $T > 0$, and 
 $(t, r) \in \Dl_{2, T}$.
Consequently, 
given any $0 < \al < \be$, 
we have 
\begin{align}
\big\| \| \vXX\|_{C^{\al}_{2,T} \LOP(\H^s;\H^{s_0})} \big\|_{L^p(\Om)} 
\les_{\al, T} p^\frac 12 
 \| \Phi\|_{\HS(L^2;H^\s)}
\label{YG2}
\end{align}

\noi
for any finite $p \ge1$ and $T > 0$.
In particular, 
there exists a version of~$\vXX$
such that 
\[\vXX \in C^{\al}_{2,T} \LOP(\H^s(\T^d);\H^{s_0}(\T^d)), \]

\noi
almost surely.
Moreover,  $\vXX$ satisfies the following Chen's relation{\rm:}
\begin{align}
(\updl \vXX)_{t_1,t_2,t_3} =0
\label{YG3}
\end{align}

\noi
for any $(t_1,t_2,t_3) \in \Dl_{3,T}$, almost surely.

\end{proposition}

\begin{proof}

The identity \eqref{YG3} follows from a direct computation, 
using \eqref{XX0} and \eqref{XX0a}, and thus we omit details.
Furthermore, we note that the bound \eqref{YG2} follows from \eqref{YG1}
and Kolmogorov's continuity criterion (Lemma \ref{LEM:kolm}).
Hence, we focus on proving \eqref{YG1} in the following.

In view of \eqref{XX0}, 
the bound \eqref{YG1} follows once 
we 
 establish the following bound  for each component $\XX^{(i,j)}_{t,r}$:
\begin{align}
\label{XX1}
\begin{split}
\sum_{i, j = 1}^2
\big\| \|  \XX^{(i,j)}_{t,r}  \|_{\LOP(H^{s+1-j}; H^{s_0 +1-i })} \big\|_{L^p(\Om)} 
 \les
 p^\frac 12  (1\vee T)^{\be- \al + \eps}
\| \Phi\|_{\HS(L^2;H^\s)} 
  |t-r|^\al
\end{split}
\end{align}

\noi
for any 
 $(t, r) \in \Dl_{2, T}$.
From 
 \eqref{XX0a} (see also \eqref{XX0b}), we have 
\begin{align}
\Ft_x\big(\jb{\nabla}^{s_0+1-i}\XX^{(i,j)}_{t,r} \jb{\nabla}^{-s-1+j}f \big) (n)
& = \sum_{n_2 \in \Z^d} \ft f (n_2) I_1[ \hf_{nn_2}(n_1) \ff^{(i,j)}_{nn_2}(t') ] 
\label{XX1a}
\end{align}

\noi
for 
 $(t, r) \in \Dl_{2, T}$, 
where $I_1$ denotes the Wiener integral with respect to a family of fractional Brownian motions with 
the Hurst parameter $\be$ defined in \eqref{BM1}, and  $\hf_{nn_2}$ and $\ff^{(i,j)}_{nn_2}$, 
$i, j = 1, 2$,  
are given by
\begin{align}
\label{TensorR1}
\begin{split}
\hf_{nn_2}(n_1) & =  \ind_{n=n_{12}}\cdot \frac{\jb{n}^{s_0-1}}{\jb{n_2}^{s}} \phi_{n_1}, \\
\ff^{(i,j)}_{nn_2} (t')  &= (-1)^i \ind_{[r,t]}(t') \cdot
 F^{(i)}(t'\jb{n})  F^{(j+1)}(t'\jb{n_2})  ,
\end{split}
\end{align}

\noi
where $n_{12} = n_1 + n_2$ is as in \eqref{ord3}
and $F^{(k)} $, $k \in \Z$, is given by 
\begin{align}
 F^{(k)}(x) = 
\begin{cases}
\sin x, & \text{if  $k$ is odd},\\
\cos x, & \text{if  $k$ is even}.\\
\end{cases}
\label{ord4}
\end{align}

Hence, 
in view of \eqref{XX1a}, 
we see that 
the bound \eqref{XX1} follows once we prove
\begin{align}
\label{XX2a}
\begin{split}
\ \big\| \| I_1[ \hf_{nn_2}(n_1) \ff_{nn_2}^{(i,j)} (t')  ] \|_{n_2 \to n} \big\|_{L^p(\Om)}
 \les
 p^\frac 12  (1\vee T)^{\be- \al + \eps}
\| \Phi\|_{\HS(L^2;H^\s)} 
  |t-r|^\al
\end{split}
\end{align}
for $i,j=1,2$, where the norm $\|\cdot\|_{n_2 \to n}
= \|\cdot\|_{\l^2_{n_2} \to \l^2_n}$ is as in Definition~\ref{DEF:tensor}.

Given dyadic $N, N_1, N_2 \ge1$, we set 
\begin{align}
\hf^{\bf N}_{n n_2}(n_1) & = 
\hf^{N, N_1, N_2}_{n n_2}(n_1)
= \ind_{E_{\bf N}}
\cdot \hf_{n n_2}(n_1), 
\label{YG3a}
\end{align}

\noi
where $E_{\bf N}$ is defined by 
\begin{align}
\begin{split}
E_{\bf N} = \big\{(n, n_1, n_2) \in (\Z^d)^3:
\ &  n = n_1 + n_2, \\
& 
 |n|\sim N, \, |n_j|\sim N_j, \, j = 1, 2\big\}.
\end{split}
\label{EN1}
\end{align}

\noi
Then, 
given any $1\le p < \infty$ and $\ta > 0$, 
by applying  the random tensor estimate (Lemma~\ref{LEM:RTE}), we have
\begin{align}
\label{XX2b}
\begin{split}
& \big\| \| I_1[ \hf_{nn_2}(n_1) \ff_{nn_2}^{(i,j)} (t')  ] \|_{n_2 \to n} \big\|_{L^p(\Om)} \\
& \quad 
\le \sum_{\substack{N, N_1, N_2\ge 1\\ \text{ dyadic}\\N_{\max}\sim N_{\med}}} 
 \big\| \|I_1[ \hf^{\bf N}_{nn_2}(n_1) \ff^{(i,j)}_{nn_2}(t')] \|_{n_2 \to n} \big\|_{L^p(\Om)}\\
& \quad \les
p^\frac12
\sum_{\substack{N, N_1, N_2\ge 1\\ \text{ dyadic}\\N_{\max}\sim N_{\med}}} 
\| \hf^{\bf N}_{nn_2}(n_1)   \|_{\l^2_{nn_1n_2} }^{\ta}
 \|\ff_{nn_2}^{(i,j)} (t') \|_{\l^\infty_{nn_2} \scrH^\be_{t'}}\\
& \hphantom{XXXXXXXXX}
\times  
 \max\Big( \|  \hf^{\bf N}_{nn_2}(n_1)\|_{n_1n_2 \to n} , 
 \|  \hf^{\bf N}_{nn_2}(n_1) \|_{n_2 \to n n_1}  \Big)^{1- \ta}, 
\end{split}
\end{align}

\noi
where 
$N_{\max}$ and  $N_{\med}$ are as in \eqref{ord1} and 
the $\scrH^\be$-norm is as in \eqref{BM0}.

We first estimate the $\l^\infty_{nn_2} \scrH^\be_{t'}$-norm of  $\ff^{(i,j)}$  in \eqref{XX2b}. 
From \eqref{BM0} and \eqref{TensorR1}, we have
\begin{align}
\label{XX3}
\begin{split}
\| \ff_{nn_2}^{(i,j)}(t')\|_{\scrH^\be_{t'}}^2
& \sim  \int_\R
    	\bigg| \int_\R e^{-i\tau t'}  \ff^{(i,j)}_{nn_2}(t') dt' \bigg|^2
    	|\tau|^{1-2\be} d\tau\\
 & \les \int_\R \bigg\{ \bigg| \int_r^t  e^{- i t' (\tau + \theta_1)} dt' \bigg|^2 + \bigg| \int_r^t  e^{- i t' (\tau - \theta_1)} dt' \bigg|^2  \\
& \quad + \bigg| \int_r^t  e^{- i t' (\tau + \theta_2)} dt' \bigg|^2  + \bigg| \int_r^t  e^{- i t' (\tau -\theta_2)} dt' \bigg|^2  \bigg\} |\tau|^{1-2\be} d\tau, 
\end{split}
\end{align}

\noi
uniformly in $n, n_2 \in \Z^d$, 
where $\theta_1 = \jb{n}+\jb{n_2}$ and $\theta_2 = \jb{n}-\jb{n_2}$. 

For $a \in \R \setminus\{0\}$, we have
\begin{align}
\label{XX3aa}
\bigg| \int_r^t e^{-it'a} dt' \bigg| &= 2\frac{|\sin (\frac12(t-r)a)|}{|a|}
\les_\g |t-r|^{\g } |a|^{-(1-\g)}
\end{align}

\noi
for any $0 \le \g \le 1$.
Noting that $\ta_1 \ge 1$, 
it follows from 
 \eqref{XX3aa} with $\g = \al$ and  Lemma~\ref{LEM:conv} that 
\begin{align}
\label{XX3a}
\begin{split}
\int_\R  \bigg| \int_r^t  e^{- i t' (\tau \pm \theta_1)} dt' \bigg|^2 |\tau|^{1-2\be} d\tau 
& \les |t-r|^{2\al}  \int_\R  |\tau\pm\theta_1|^{-2(1-\al)}  |\tau|^{1-2\be} d\tau \\
& \les |t-r|^{2\al}
\end{split}
\end{align}

\noi
for any $t \ge r \ge 0$, 
uniformly in $n, n_2 \in \Z^d$, 
provided that 
 $\frac12 < \al  < \be < 1$, which in particular guarantees local integrability
 of each factor (at $\tau = 0, \mp \ta_1$).

In the following, we only study the contribution from the third term 
on the right-hand side of \eqref{XX3} since the fourth term can be treated
in an analogous manner.
Since $\ta_2=  \jb{n}-\jb{n_2}$ can be $0$ (and close to $0$), 
we need to proceed with care.
When $|\ta_2|\ge 1$, \eqref{XX3a} holds with $\ta_1$ replaced by $\ta_2$.
It remains to consider the case $|\ta_2 |< 1$.

Without loss of generality, assume that $0 \le \ta_2 < 1$.
Then, 
proceeding as in \eqref{XX3a}, we have 
\begin{align}
\begin{split}
\int_{\R\setminus [-2, 1]}  \bigg| \int_r^t  e^{- i t' (\tau + \theta_2)} dt' \bigg|^2 |\tau|^{1-2\be} d\tau 
& \les |t-r|^{2\al}
\end{split}
\label{XX3ax}
\end{align}

\noi
for any $t \ge r \ge 0$, 
uniformly in $n, n_2 \in \Z^d$ such that $0 \le \ta_2 < 1$, 
provided that 
 $\frac12 < \al  < \be < 1$.
By applying~\eqref{XX3aa} with $\g = \be+\eps \le1 $
for small $\eps > 0$, we have 
\begin{align}
\begin{split}
& \int_{[-2, 1]}  \bigg| \int_r^t  e^{- i t' (\tau + \theta_2)} dt' \bigg|^2 |\tau|^{1-2\be} d\tau \\
& \quad \les |t-r|^{2(\be + \eps)}  \int_{[-2, 1]}  |\tau+\theta_2|^{-2(1-\be-\eps)}  |\tau|^{1-2\be} d\tau \\
& \quad \les |t-r|^{2(\be + \eps)}  \int_{[-2, 1]}  
|\tau+\theta_2|^{-1+2\eps}+   |\tau|^{-1+2\eps} d\tau \\
& \quad \les (1\vee T)^{2(\be -\al + \eps )} |t-r|^{2\al}
\end{split}
\label{XX3ay}
\end{align}

\noi
for any $0 \le r \le t \le T$,
uniformly in $n, n_2 \in \Z^d$ such that $0 \le \ta_2 < 1$, 
 provided that $0 < \al \le \be < 1$.

Hence, from \eqref{XX3}, \eqref{XX3a}, \eqref{XX3ax}, and \eqref{XX3ay}, 
we obtain
\begin{align}
\| \ff_{nn_2}^{(i,j)}(t')\|_{\l^\infty_{n n_2}\scrH^\be_{t'}}
\les (1\vee T)^{\be -\al + \eps } |t-r|^{\al}
\label{XX3az}
\end{align}

\noi
for any $0 \le r \le t \le T$,
provided that 
 $\frac12 < \al  < \be < 1$.

\medskip

Next, we estimate the contributions 
from the factors  in \eqref{XX2b}, 
involving  $\hf^{\bf N}_{nn_2}(n_1)$. 
From~\eqref{YG3a} with \eqref{TensorR1}
and \eqref{phi1a},
 we have 
\begin{align}
\label{XX4a}
\begin{split}
\|  \hf^{\bf N}_{nn_2}(n_1)\|_{\l_{nn_1n_2}^2}
& \lesssim 
\frac{N^{s_0-1+ \frac d2}}{N_1^\s N_2^s} \|\Phi\|_{\HS(L^2;H^\s)}.
\end{split}
\end{align}

\noi
From \eqref{op-norm-dual} and Cauchy-Schwarz's inequality with \eqref{EN1}
and \eqref{phi1a}, we have
\begin{align}
\label{XX4b}
\begin{split}
\|  \hf^{\bf N}_{nn_2}(n_1) \|_{n_1n_2 \to n} 
& \sim 
\frac{N^{s_0-1}}{N_1^\s N_2^s}
 \sup_{\|f\|_{n_1n_2} = \|g\|_{n} =1 } 
\bigg| \sum_{n,n_1,n_2 \in \Z^d}
\ind_{E_{\bf N}}\cdot  \jb{n_1}^\s \phi_{n_1}f_{n_1n_2} g_{n} \bigg| \\
& \les \frac{N^{s_0-1}}{N_1^\s N_2^s} \| \Phi\|_{\HS(L^2;H^\s)}.
\end{split}
\end{align}

\noi
A similar computation yields
\begin{align}
\label{XX4c}
\| \hf^{\bf N}_{nn_2}(n_1) \|_{n_2 \to nn_1} 
& \les \frac{N^{s_0-1}}{N_1^\s N_2^s} \| \Phi\|_{\HS(L^2;H^\s)}.
\end{align}

Therefore, putting \eqref{XX2b}, \eqref{XX3az}, 
\eqref{XX4a}, \eqref{XX4b}, and \eqref{XX4c} together, 
we obtain
\begin{align}
\label{XX6a}
\begin{split}
&  \big\| \| I_1[ \hf_{nn_2}(n_1) \ff_{nn_2}^{(i,j)} (t')  ] \|_{n_2 \to n} \big\|_{L^p(\Om)}\\
& \quad \les
p^\frac12 
(1\vee T)^{\be -\al + \eps }
 \| \Phi\|_{\HS(L^2;H^\s)} 
  |t-r|^{\al}
 \sum_{\substack{N, N_1, N_2 \ge1\\\text{dyadic}\\N_{\max}\sim N_{\med}}} 
 \frac{N^{s_0-1+ \frac d2\ta}}{N_1^\s N_2^s} \\
& \quad \les
p^\frac12 
(1\vee T)^{\be -\al + \eps }
 \| \Phi\|_{\HS(L^2;H^\s)} 
  |t-r|^{\al}, 
\end{split}
\end{align}

\noi
provided that \eqref{YG0} holds
and $\ta > 0$ is sufficiently small.
Here, the second inequality in \eqref{XX6a} follows
from separately considering the cases:
 (a)~$N \sim N_1 \ges N_2$, 
(b)~$N \sim N_2 \ges N_1$, 
and 
(c)~$N_1 \sim N_2 \ges N$
(also depending on the sign of the exponent of $N_{\min}$).
This proves~\eqref{XX2a}
(and hence \eqref{XX1} and \eqref{YG1}).
\end{proof}

\subsection{Rough driver}  
\label{SUBSEC:driver2}

In this subsection, we study regularity properties of the rough driver 
 $(\vXX, \vbbX)$ 
 with the Hurst parameter   $\be= \frac 12$.
In this case, 
the $\scrH^\be_{t_A}$-norm, appearing in 
the random tensor estimate (Lemma \ref{LEM:RTE}), 
reduces to the $L^2_{t_A}$-norm, providing  a certain simplification.

\begin{proposition}
\label{PROP:driver2}

Let $ \s \in \R$.
Given $\Phi \in \HS(L^2(\T^d); H^\s(\T^d))$, 
satisfying \eqref{phi1}, 
let  $\vXX$ and $\vbbX$
be  the first and second order operators  
with the Hurst parameter $\be = \frac 12$, 
defined in~\eqref{XX0} and \eqref{bbX0}, respectively.

\smallskip

\noi
{\rm (i)}
Suppose that $s, s_0 \in \R$ satisfy 
\eqref{YG0}.
 Then, we have 
\begin{align}
\big\| \|  \vXX_{t,r}  \|_{\LOP(\H^s; \H^{s_0})} \big\|_{L^p(\Om)} 
 & \les   p^\frac 12  
\| \Phi\|_{\HS(L^2;H^\s)} 
  |t-r|^\frac 12 
\label{ZG1}
\end{align}

\noi
for any  
 finite $p \ge 1$,  $T > 0$, and 
 $(t, r) \in \Dl_{2, T}$.
Consequently, 
given any $0 < \al < \frac 12 $, 
we have 
\begin{align}
\big\| \| \vXX\|_{C^{\al}_{2,T} \LOP(\H^s;\H^{s_0})} \big\|_{L^p(\Om)} 
& \les_{\al, T} p^\frac 12 
 \| \Phi\|_{\HS(L^2;H^\s)}
 \label{ZG2}
\end{align}

\noi
for any finite $p \ge1$ and $T > 0$.

\smallskip

\noi
{\rm (ii)}
Suppose that $s, s_0 \in \R$ satisfy 
\begin{align}
s + \s  > 0,  
\qquad 
\s > - \frac 12,  
\label{RG0}
\end{align}

\noi
and 

\begin{align}
\begin{split}
s_0 
& <  \min\Big(\
  \s + 1, \, 
(s\wedge \s) 
+  (\s\wedge 0) + 2\Big).
\end{split}
\label{RG1}
\end{align}

\noi
 Then, we have 
\begin{align}
\big\| \|  \vbbX_{t,r}  \|_{\LOP(\H^s; \H^{s_0})} \big\|_{L^p(\Om)} 
 & \les   p  
\| \Phi\|_{\HS(L^2;H^\s)}^2 
  |t-r|
\label{ZG1a}
\end{align}

\noi
for any  
 finite $p \ge 1$,  $T > 0$, and 
 $(t, r) \in \Dl_{2, T}$.
Consequently, 
given any $0 < \al < \frac 12 $, 
we have 
\begin{align}
\big\| \| \vbbX\|_{C^{2\al}_{2,T} \LOP(\H^s;\H^{s_0})} \big\|_{L^p(\Om)} 
& \les_{\al, T} p
 \| \Phi\|_{\HS(L^2;H^\s)}^2
\label{ZG2a}
\end{align}

\noi
for any finite $p \ge1$ and $T > 0$.
In particular, under \eqref{YG0}, \eqref{RG0}, and \eqref{RG1},
there exist versions of~$\vXX$ and $\vbbX$
such that 
\begin{align}
(\vXX, \vbbX )  \in C^{\al}_{2,T} \LOP(\H^s(\T^d);\H^{s_0}(\T^d))
\times 
C^{2\al}_{2,T} \LOP(\H^s(\T^d);\H^{s_0}(\T^d)),
\label{ZG3}
\end{align}

\noi
almost surely.

\smallskip

\noi\textup{(iii)} 
In addition to 
\eqref{YG0}, \eqref{RG0}, and \eqref{RG1},
suppose that  $s_0 \ge s$
and $\frac 13  < \al < \frac12$.
Then,  the pair $(\vXX, \vbbX)$ is 
almost surely 
a $\frac1\al$-variational rough path with values in $\LOP(\H^{s}(\T^d); \H^{s_0}(\T^d))$ 
in the sense of Definition~\ref{DEF:RP}\,(i),
satisfying  Chen's relation  \eqref{Chen2}.

\end{proposition}

\begin{remark}\label{REM:reg1}\rm 
Let $s \ge 0$.
Then, under \eqref{RG0}, 
the condition \eqref{RG1}
reduces to 
\begin{align}
s_0 < \begin{cases}
\s + 1, & \text{if } s\ge \s, \\
\min(s + 2, \s + 1), & \text{if } \s > s \ge 0.
\end{cases}
\label{RG1x}
\end{align}

\end{remark}

\begin{proof}[Proof of Proposition \ref{PROP:driver2}]

We first note that the bounds  \eqref{ZG2}
and \eqref{ZG2a}
 follow from 
  Kolmogorov's continuity criterion (Lemma \ref{LEM:kolm})
  with 
 \eqref{ZG1} and 
\eqref{ZG1a}, respectively.
A direct computation with 
\eqref{bbX0a}
shows that 
$(\vXX, \vbbX)$ satisfies Chen's relation \eqref{Chen2}.
(Here, we need the condition $s_0 \ge s$
such that the operator compositions in~\eqref{bbX0a}
are well defined.)
Moreover, from \eqref{ZG3} and \eqref{con1x}, we have 
\begin{align*}
(\vXX, \vbbX )  \in \V^{\frac 1\al}_{2,T} \LOP(\H^s(\T^d);\H^{s_0}(\T^d))
\times 
\V^{\frac 1 {2\al}}_{2,T} \LOP(\H^s(\T^d);\H^{s_0}(\T^d)),
\end{align*}

\noi
almost surely.  
Hence, we focus on proving \eqref{ZG1} and \eqref{ZG1a} in the following.

The first bound  \eqref{ZG1}
follows from proceeding as in the proof of Proposition \ref{PROP:driver1}.
From~\eqref{TensorR1}, we have 
\begin{align}
\|\ff_{nn_2}^{(i,j)} (t') \|_{\l^\infty_{nn_2} L^2_{t'}}
 \le |t-r|^\frac 12 
\label{ZG4}
\end{align}

\noi
for $i,j=1,2$.
Then, \eqref{ZG1} follows from 
\eqref{XX2b} (with $\be = \frac 12$), \eqref{XX4a}, 
\eqref{XX4b}, \eqref{XX4c}, 
and \eqref{ZG4}, provided that \eqref{YG0} holds.

In view of \eqref{bbX0}, 
the second bound \eqref{ZG1a} follows once 
we 
 establish the following bound  for 
 each component $\bbX^{(i,j)}_{t,r}$:
\begin{align}
\label{bbX0aa}
\begin{split}
\sum_{i, j= 1}^2
\big\| \|  \bbX^{(i,j)}_{t,r}  \|_{\LOP(H^{s+1-j}; H^{s_0+1-i})} \big\|_{L^p(\Om)} 
 & \les
p 
\| \Phi\|_{\HS(L^2;H^\s)}^2  |t-r| 
\end{split}
\end{align}

\noi
for any 
 $(t, r) \in \Dl_{2, T}$.
From  \eqref{bbX0a}
 with  \eqref{XX0a}, we have 
\begin{align}
\label{bbX0b}
\begin{split}
& \Ft_x( \jb{\nabla}^{s_0+1-i } \bbX^{(i,j)}_{t,r} \jb{\nabla}^{-s-1+j} f) (n)\\
& \quad  = \sum_{n_3} \ft{f}(n_3) I_2 [ \wt{\hf}_{nn_3}(n_A) \wt{\ff}^{(i,j)}_{nn_3}(t_A,n_A)  ], 
\end{split}
\end{align}

\noi
where $I_2$ denotes the second order multiple stochastic integral defined in Definition~\ref{DEF:MSI}, $A = \{1,2\}$, and
$\wt{\hf}_{n n_3}$ and  $\wt{\ff}_{n n_3}^{(i, j)}$, $i, j = 1, 2$,  are given by
\begin{align}
\label{tensor2b}
\begin{split}
\wt{\hf}_{nn_3} (n_A) &  =   \ind_{ n=n_{123} } \frac{  \jb{n}^{s_0-1}}{ \jb{n_{23}} \jb{n_3}^{s} }
 \phi_{n_1} \phi_{n_2} , \\
\wt{\ff}^{(i,j)}_{nn_3}(n_A,t_A) 
& =  (-1)^{i} \ind_{\Dl_{2, [r, t]}}(t_1, t_2)\\
& \quad \times   F^{(i)}(t_1\jb{n}) F^{(j+1)}(t_2\jb{n_{3}}) \sin((t_1-t_2)\jb{n_{23}}) , \\
\end{split}
\end{align}

\noi
where
$n_{123} = n_1 + n_2+n_3$ as in \eqref{ord3}, 
$\Dl_{2, [r, t]}$ is as in 
Definition \ref{DEF:V23}\,(iii), 
and $F^{(k)} $ is as in~\eqref{ord4}.
Hence, 
in view of~\eqref{bbX0b}, 
we see that 
the  bound \eqref{bbX0aa} follows once we prove
\begin{align}
\label{bbX1}
\big\| \|
I_2[ \wt{\hf}_{nn_3}(n_A) \wt{\ff}^{(i,j)}_{nn_3}(n_A,t_A)  ]
 \|_{n_3 \to n } \big\|_{L^p(\Om)}
\les
p
 \|\Phi\|_{\HS(L^2;H^\s)}^2|t-r|
\end{align}

\noi
for  $i,j=1,2$.

Given dyadic $N, N_1, N_2, N_3, N_{23} \ge1$, we set 
\begin{align}
\begin{split}
\wt \hf^{\bf N}_{n n_3}(n_A) & = 
\wt\hf^{N, N_1, N_2, N_3, N_{23}}_{n n_3}(n_A)
= \ind_{\wt E_{\bf N}}
\cdot 
\wt \hf_{n n_3}(n_A), 
\end{split}
\label{bbX1x}
\end{align}

\noi
where $\wt E_{\bf N}$ is defined by 
\begin{align*}
\wt E_{\bf N} = \big\{(n, n_1, n_2, n_3) \in \Z^4:
\ & n = n_{123},  \, |n|\sim N, \\
&  |n_j|\sim N_j, \, j = 1, 2, 3, \, |n_{23}|\sim N_{23}\big\}.
\end{align*}

\noi
Under $n = n_1 + n_{23}$ 
and $n_{23} = n_2 + n_3$, we have 
\begin{align}
\begin{split}
\max(N, N_1, N_{23})
& \sim \med(N, N_1, N_{23}),\\
\max(N_{23}, N_2, N_{3})
& \sim \med(N_{23}, N_2, N_{3}), 
\end{split}
\label{EN3}
\end{align}

\noi
where 
$\max(A, B, C)$ (and 
$\med(A, B, C)$) denotes
the largest element (and the second largest element, respectively)  in $A$, $B$, and $C$.
Then, 
given any $1\le p < \infty$ and $\ta > 0$, 
by applying  the random tensor estimate (Lemma~\ref{LEM:RTE}), we have
\begin{align}
\begin{split}
&
\big\| \| I_2[ \wt{\hf}_{nn_3}(n_A) \wt{\ff}^{(i,j)}_{nn_3}(n_A,t_A)  ] \|_{n_3 \to n} \big\|_{L^p(\Om)} \\
&\quad 
\les p  \sum_{\substack{N, N_1, N_2, N_3, N_{23} \ge1 \\\text{dyadic}\\\eqref{EN3}}}  
\| \wt{\hf}_{nn_3}^{\bf N}(n_A) \|_{\l^2_{nn_3 n_A}}^{\ta} \| \wt{\ff}_{nn_3}^{(i,j)} (n_A,t_A) \|_{\l^\infty_{nn_3 n_A} L^2_{t_A}}\\
& \hphantom{XXXXXXXXXX} \times 
 \Big( \max_{(B,C)} \| \wt{\hf}_{nn_3}^{\bf N}(n_A) \|_{n_3 n_B \to n n_C}   \Big)^{1-\ta} \\
\end{split}
\label{bbX1a}
\end{align}

\noi
where the maximum is taken over all partitions $(B,C)$ of $A = \{1, 2\}$.

From  \eqref{tensor2b}, we have 
\begin{align}
\| \wt{\ff}_{nn_3}^{(i,j)} (n_A,t_A) \|_{ \l^\infty_{nn_3 n_A }L^2_{t_A}}^2
& \le \int_r^t \int_r^{t_1} 1 \, dt_2 d t_1
 \sim |t-r|^2.
 \notag 
\end{align}

\noi
From \eqref{bbX1x} with  \eqref{tensor2b}
and \eqref{phi1a},
 we have 
\begin{align}
& \|  \wt{\hf}_{nn_3}^{\bf N}(n_A) \|_{\l^2_{nn_3 n_A}}
\sim  \frac{  N^{s_0-1+\frac d2}}{ N_{23}N_1^\s N_2^\s N_3^{s} }
\|\Phi \|_{\HS(L^2;H^\s)}^2.
\label{bbX1y}
\end{align}

\noi
From \eqref{op-norm-dual} and Cauchy-Schwarz's inequality with \eqref{EN3}
and \eqref{phi1a}, we have
\begin{align}
\label{bbX1c}
\begin{split}
& \| \wt{\hf}^{\bf N}_{nn_3}(n_A) \|_{n_3 n_B \to nn_C} \\
& \quad = 
\sup_{\|f\|_{\l^2_{n_3 n_B}} = \|g\|_{\l^2_{nn_C}} = 1} \Big| 
\sum_{n,n_3, n_B, n_C} \ind_{\wt E_{\bf N}}\cdot \wt\hf_{nn_3}(n_A) f_{n_3 n_B} g_{n n_C} \Big| \\
& \quad  \les 
\frac{  N^{s_0-1}}{ N_{23}N_1^\s N_2^\s N_3^{s} }\\
& \quad  \quad \times  \sup_{\|f\|_{\l^2_{n_3 n_B}} = \|g\|_{\l^2_{nn_C}} = 1} \sum_{n,n_3, n_A}
 \ind_{\wt E_{\bf N}}\cdot
\prod_{j= 1}^2 \jb{n_j}^\s |\phi_{n_j}|\cdot   |f_{n_3 n_B} g_{nn_C} | \\
& \quad  \le 
\frac{  N^{s_0-1}}{ N_{23}N_1^\s N_2^\s N_3^{s} }
\|\Phi \|_{\HS(L^2;H^\s)}^2.
\end{split}
\end{align}

Therefore, putting \eqref{bbX1a}, \eqref{bbX1a}
\eqref{bbX1y}, and \eqref{bbX1c} together, 
we obtain
\begin{align}
\begin{split}
&
\big\| \| I_2[ \wt{\hf}_{nn_3}(n_A) \wt{\ff}^{(i,j)}_{nn_3}(n_A,t_A)  ] \|_{n_3 \to n} \big\|_{L^p(\Om)} \\
&\quad 
\les p  
  \|\Phi\|_{\HS(L^2;H^\s)}^2 |t-r|
\sum_{\substack{N, N_1, N_2, N_3, N_{23} \ge1 \\\text{dyadic}\\\eqref{EN3}}}  
\frac{  N^{s_0-1+ \frac d2 \ta}}{ N_{23}N_1^\s N_2^\s N_3^{s} }\\
& \quad \les
p
 \|\Phi \|_{\HS(L^2;H^\s)}^2 |t-r|,
\end{split}
\label{bbX2}
\end{align}

\noi
provided that \eqref{RG0} and \eqref{RG1}  hold
and $\ta > 0$ is sufficiently small.
Here, the second inequality in~\eqref{bbX2} follows
from carrying out case-by-case analysis
(where we used  \eqref{EN3}).
In the following, we use 
\begin{align}
 s + \s > 0
 \qquad \text{and}\qquad  \s > - \frac 12
 \label{RGYx}
\end{align}

\noi
where the first bound follows from \eqref{RGY2}
and the second one follows from \eqref{RGY6}.

We also note the following identity:
\begin{align}
\begin{split}
\min \big( s + (\s \wedge 0), (s\wedge 0) + \s\big)
& =  (s\wedge \s)  + \big((s\vee \s) \wedge 0\big)\\
& =  (s\wedge \s),  
\end{split}
\label{RGY4}
\end{align}

\noi
where the second equality follows under $s+\s > 0$.

\smallskip
\begin{itemize}
\item[(a)] 
$N\sim N_1 \ges N_{23}$.

\smallskip

\begin{itemize}
\item[(a.i)] 
$N \sim N_1 \sim N_{\max}:= \max(N, N_1, N_2, N_3)$. 
In this case, 
we carry out case-by-case analysis, depending on 
the sizes of $N_2$, $N_3$, and $N_{23}$, using \eqref{EN3}, 
and conclude that 
\eqref{bbX2} holds if 
\begin{align}
\begin{split}
s_0 
& <  \s + 1 
+ \min\Big(\big((s\wedge 0) + \s   +1\big)\wedge 0, \,
\big(s +  (\s\wedge 0) +1\big)\wedge 0\Big)\\
& = \s + 1 + \big(((s\wedge \s) + 1)\wedge 0\big).
\end{split}
\label{RGY1}
\end{align}

\noi
Here, we used the condition $s+\s > 0$
in \eqref{RGYx} and the identity \eqref{RGY4}.

\smallskip

\item[(a.ii)]
 $N_2 \sim N_3\sim N_{\max} \gg  N \sim N_1$.
In this case, \eqref{bbX2} holds if 
\begin{align}
s+\s > 0
\qquad \text{and}\qquad
s_0 < s+ 2\s + 1.
\label{RGY2}
\end{align}

\end{itemize}

\smallskip
\item[(b)] 
$N\sim N_{23} \ges N_1$.

\smallskip
\begin{itemize}
\item[(b.i)] 
$N\sim N_{23} \sim N_{\max}$. 
In this case, we have $N_2 \vee N_3 \sim N_{\max}$, 
and \eqref{bbX2} holds if 
\begin{align}
s_0 < (s\wedge \s)  + (\s \wedge 0) + \big((s\vee \s) \wedge 0\big)+ 2
= (s\wedge \s)  + (\s \wedge 0) +  2, 
\label{RGY3}
\end{align}

\noi
where we used \eqref{RGY4}.

\smallskip
\item[(b.ii)]
 $N_2 \sim N_3\sim N_{\max} \gg N\sim N_{23}$.
In this case, \eqref{bbX2} holds if 
\begin{align}
s+\s > 0
\qquad \text{and}\qquad 
s_0 < s + \s +( \s\wedge 0) + 2.
\label{RGY5}
\end{align}
\end{itemize}

\smallskip
\item[(c)] 
$N_1\sim N_{23} \ges N$.
\smallskip
\begin{itemize}
\item[(c.i)] 
$N_1\sim N_{23}  \sim N_{\max}$.
In this case, we have $N_2 \vee N_3 \sim N_{\max}$, 
and \eqref{bbX2} holds if 
\begin{align*}
(s\wedge \s) + \s + 1 > 0 \vee (s_0 - 1).
\end{align*}

\noi
Here,  we  used the identity \eqref{RGY4}.
Under $s + \s > 0$, this simplifies to 
\begin{align}
\s > - \frac 12 \qquad \text{and}\qquad 
s_0 < (s\wedge \s) + \s + 2 .
\label{RGY6}
\end{align}

\smallskip
\item[(c.ii)]
 $N_2 \sim N_3\sim N_{\max} \gg N_1\sim N_{23}$.
In this case, 
\eqref{bbX2} holds if 
\begin{align*}
s + \s > 0
\qquad \text{and}
\qquad  
    1 + 2\s + s > 0\vee (s_0 - 1), 
\end{align*}

\noi
which simplifies to 
\begin{align}
s_0 < s + 2\s +2
\label{RGY7}
\end{align}

\noi
under \eqref{RGYx}.

\end{itemize}

\end{itemize}

\noi
Putting \eqref{RGY1}, \eqref{RGY2}, \eqref{RGY3}, 
\eqref{RGY5}, \eqref{RGY6}, and \eqref{RGY7}, 
we obtain the conditions \eqref{RG0} and~\eqref{RG1}.

This concludes the proof of~\eqref{bbX1}
(and hence \eqref{bbX0aa} and \eqref{ZG1a}).
\end{proof}

\section{Deterministic nonlinear estimates}
\label{SEC:nonlin}

In this section, we establish the deterministic nonlinear  estimates 
to handle the nonlinear Duhamel integral part of the Duhamel formulation 
\eqref{mild0} 
(more precisely, at the level of the interaction representation
in the vector-valued formulation; see \eqref{mild2}).
The material presented in this section is standard
for those in dispersive PDEs, 
but we include some details, 
in particular for readers
in stochastic analysis who may not be familiar with the subject.
Moreover, 
our presentation in Subsection \ref{SUBSEC:str}
(see Proposition \ref{PROP:Str2}), 
following \cite[Subsection 3.1]{GKO1}, 
provides a suitable wave admissible pair $(q, r)$
(and a dual admissible pair $(\wt q, \wt r)$)
in a concise manner, which is of  interest
in its own right.

In Subsection~\ref{SUBSEC:nonl1}, 
we cover the one-dimensional case.
In this case, the wave equation is non-dispersive
and thus we treat the nonlinearity simply via Sobolev's embedding.
In Subsection~\ref{SUBSEC:str}, 
we go over the Strichartz estimate, following 
the presentation in \cite[Subsection~3.1]{GKO1}.
In Subsection~\ref{SUBSEC:nonl2}, 
we treat the case  $d\ge2$ by employing  the Strichartz estimates.

\subsection{One-dimensional case}
\label{SUBSEC:nonl1}

In this subsection, we consider the case $d=1$.
As mentioned above,  the one-dimensional wave equation  is not dispersive
and, as a consequence,  
  there are no Strichartz estimates. 
In this case,  the critical regularity $s_\text{crit}$ is given by the Sobolev exponent as in \eqref{crit3}. 
See \cite{CCT} for a further discussion.

\begin{proposition}
\label{PROP:nonlin1}

Let $d = 1$.
Given an integer $k \ge 2$,  let  $s_{\textup{crit}} = \frac 12 - \frac 1k$ be as in \eqref{crit3}.
Then, given any  $s\ge s_{\textup{crit}}$ and  $1 <  p, q <\infty$, 
we have 
\begin{align}
\begin{split}
\bigg\| \int_0^t \S(-t')
\begin{pmatrix}
0 \\
u^k_{t'} - v^k_{t'}
\end{pmatrix}
dt'
\bigg\|_{U^p_{T} \H^s_x}
& \les
T \Big( \|u  \|_{L^\infty_T H^s_x}
+ \|v \|_{L^\infty_T H^s_x} \Big)^{k-1}
\| u - v \|_{L^\infty_T H^s_x}\\
& \les
T \Big( \|u  \|_{U^q_T H^s_x}
+ \|v \|_{U^q_T H^s_x} \Big)^{k-1}
\| u - v \|_{U^q_T H^s_x}
\end{split}
\label{nonlin1}
\end{align}

\noi
for any $T > 0$, 
where $\S(t)$ is the matrix-valued linear wave propagator
defined in \eqref{Sdef}.

\end{proposition}

\begin{proof}

By duality and Sobolev's embedding, we have
\begin{align}
\begin{split}
\| f\|_{H^{-\s}}
& = \sup_{\|g\|_{H^{\s}}  \le 1}
\bigg|\int_\T fg dx\bigg|
\le \sup_{\|g\|_{H^{\s}}  \le 1} \|f \|_{L^1} \|g\|_{L^\infty}
\\
&
\les
   \|f \|_{L^1},
\end{split}
\label{1d1}
\end{align}

\noi
for $\s > \frac 12$.
Noting that $1 - s_{\text{crit}} > \frac 12$, 
it follows from 
 \eqref{lin3},  \eqref{1d1}, 
the fractional Leibniz rule (Lemma~\ref{LEM:leibniz}), and Sobolev's inequality
with \eqref{crit3}
that 
\begin{align}
\begin{split}
& \bigg\| \int_0^t \S(-t')
\begin{pmatrix}
0 \\
u^k_{t'} - v^k_{t'}
\end{pmatrix}
dt'
\bigg\|_{U^p_{T} \H^{s}_x}  \\
& \quad \les T \| \jb{\nabla}^{s-s_\textup{crit}}(u^k - v^k) \|_{L^\infty_T H^{s_\textup{crit}-1}_x } \\
& \quad \les T \| \jb{\nabla}^{s-s_\textup{crit}}(u^k - v^k) \|_{L^\infty_T L^1_x } \\
& \quad 
\les T \Big( \| \jb{\nabla}^{s-s_\textup{crit}} u\|_{L^\infty_T L^k_x}  + \| \jb{\nabla}^{s-s_\textup{crit}} v\|_{L^\infty_T L^k_x}  \Big)^{k-1} \| \jb{\nabla}^{s-s_\textup{crit}}(u-v) \|_{L^\infty_T L^k_x} \\
& \quad \les T \big(\| u \|_{L^\infty_T H^{s}_x} + \|v\|_{L^\infty_T H^{s}_x} \big)^{k-1} \| u-v \|_{L^\infty_T H^{s}_x}.
\end{split}
\label{1d2}
\end{align}

\noi
This proves the first inequality in \eqref{nonlin1}.
The second inequality in \eqref{nonlin1} follows from \eqref{Linfty}.
\end{proof}

\begin{remark}\label{REM:local}\rm
In proving the second inequality in \eqref{nonlin1}, 
we used \eqref{Linfty} 
to conclude 
\begin{align}
\|u\|_{L^\infty_T H^s_x}
\les 
\|u\|_{U^q_T H^s_x}.
\label{nonlin1x}
\end{align}

\noi
Strictly speaking, we need to use 
the definition 
\eqref{local} of the time restriction norm.
More precisely, 
given a function $u$  
on the time interval $[0, T]$, 
let $\wt u$ be an extension of $u$ on $\R$
such that $\wt u|_{[0, T]} = u$.
Then, by \eqref{Linfty}, we have 
\begin{align}
\|u\|_{L^\infty_T H^s_x}
\le  \|\wt u\|_{L^\infty_t (\R; H^s_x)}
\les 
\|\wt u\|_{U^q_t (\R;  H^s_x)}.
\notag 
\end{align}

\noi
By taking an infimum over all extensions $\wt u$ of $u$
with \eqref{local}, 
we  obtain \eqref{nonlin1x}.

In working with 
the time restriction norm in \eqref{local}, 
we need to consider  an extension, 
prove an estimate (while guaranteeing that the right-hand side is
independent of the choice of the extension), 
 and take an infimum over all extensions.
However, this procedure is standard
and we omit such details 
in the following.  

\end{remark}

\subsection{Strichartz estimates}
\label{SUBSEC:str}

In order to establish nonlinear estimates in the higher dimensional setting, 
we need to exploit dispersion in the form of the Strichartz estimates.
In this subsection, we recall the Strichartz estimates
and also establish 
a key proposition 
(Proposition~\ref{PROP:Str2})
which guarantees 
existence of suitable $s$-admissible and dual $s$-admissible pairs 
in studying the (deterministic) NLW \eqref{NLW}
via the Strichartz estimates.

We start by recalling the notion of wave admissible pairs.

\begin{definition}\label{DEF:admiss}
\rm

Let $d\ge2$.

\smallskip

\noi\textup{(i)} 
Let  $0<s < \frac d 2$.
We say that a pair $(q,r)$ is $s$-admissible if $2\le q \le \infty$ and $2 \le r <\infty$, 
\begin{align}
\begin{split}
&   \frac{1}{q} + \frac dr  = \frac d 2-  s,  \qquad \text{and} \qquad   \frac 2q + \frac{d-1}{r}   \leq \frac {d-1} 2.
\end{split}
\label{admis1a}
\end{align}

\smallskip

\noi\textup{(ii)} Let $0<s<1$. 
We say that a  pair $(\wt{q},\wt{r})$ is dual $s$-admissible\footnote{Here, we define
the notion of dual $s$-admissibility for the convenience of the presentation.
Note that $(\wt q, \wt r)$ is dual $s$-admissible
if and only if $(\wt q', \wt r')$ is $(1-s)$-admissible.}
if $1\le \wt{q} \le 2$, $1 < \wt{r} \le 2$, 
\begin{align}
\frac{1}{\wt{q}} + \frac{d}{\wt{r}} = \frac d 2 + 2 - s,
 \qquad \text{and} \qquad \frac{2}{\wt{q}} + \frac{d-1}{\wt{r}} \ge \frac{d+3}{2}.
\label{admis1aa}
\end{align}

\smallskip

\noi
We refer to the equalities in \eqref{admis1a}
and \eqref{admis1aa} as the scaling conditions, 
while we refer to 
 the inequalities 
 in  \eqref{admis1a} and \eqref{admis1aa} 
 as the admissibility conditions.

\end{definition}

The following lemma provides a more direct description of the 
(dual) wave admissible
 exponents; see \cite[Lemma~3.1]{GKO1} for the $d = 2$ case.

\begin{lemma}
\label{LEM:pair1}
Let $d \ge 2$.

\smallskip

\noi\textup{(i)} For $0<s < \frac d 2$, a pair $(q,r)$ is $s$-admissible 
if 
\begin{align}
\frac1q + \frac d r = \frac d 2 -s \quad \text{and} \quad 
\frac{2d}{d-s} \le r \le
\min\bigg\{ \frac{2d}{(d-1-2s)_+}, \frac{2(d+1)}{(d+1-4s)_+}  \bigg\},
\label{admis2aa}
\end{align}
and $r\ne\infty$,
where $x_+ := \max(x,0)$ with the understanding that $\frac10 = \infty$.

\smallskip

\noi\textup{(ii)} For
$0 < s <1$,
a pair $(\wt{q}, \wt{r})$ is dual $s$-admissible 
if 
\begin{align}
\begin{split}
& \frac1{\wt q} + \frac d{\wt r} = \frac d2 + 2  - s
\quad
\text{and} \quad\\
& \max\bigg\{1+,
\frac{2(d+1)}{d+5 - 4s}, \frac{2d}{d+3-2s}
 \bigg\}  \le   \wt r \le \frac{2d}{  d+ 2-2s}.
\end{split}
\label{admis3aa}
\end{align}

\end{lemma}

\begin{proof}
It suffices to show that 
the second conditions in \eqref{admis2aa}
and \eqref{admis3aa} imply
the admissibility conditions in \eqref{admis1a}
and \eqref{admis1aa}, respectively.

\smallskip

\noi
(i) 
From the scaling condition in \eqref{admis1a}, 
we have 
\begin{align*}
2 \le q
& \quad \iff
\quad 
r \le \frac{2d}{(d-1-2s)_+}, \\
\frac 2 q + \frac{d-1}{r} \le \frac{d-1}{2}
& \quad \iff
\quad   r \le \frac{2(d+1)}{(d+1-4s)_+}  .
\end{align*}

\noi
Hence, \eqref{admis2aa} implies \eqref{admis1a}.

\smallskip
\noi
(ii)
From the scaling condition in \eqref{admis1aa} with $0 < s < 1$, 
we have 
\begin{align*}
1 \le \wt{q}
& \quad \iff
\quad 
 \wt{r} \le \frac{2d}{d+2-2s}, \\
2 \ge \wt{q}
& \quad \iff
\quad  \wt{r} \ge \frac{2d}{d+3-2s} , \\
\frac{2}{\wt{q}} + \frac{d-1}{\wt{r}} \ge \frac{d+3}{2}
& \quad \iff
\quad  \wt{r} \ge \frac{2(d+1)}{d + 5 -4s}.
\end{align*}

\noi
Hence, \eqref{admis3aa} implies \eqref{admis1aa}.
\end{proof}

We now state the  Strichartz estimates for the wave equation.
The Strichartz estimates on $\R^d$ have been studied by many mathematicians. 
See, for example, \cite{GV95,LS95,KT98}.
See also \cite{TAO, KTV14}.
The first  estimate \eqref{Str1a}  for $\S(t)\vec{\phi}$ on $\T^d$ 
follows from 
\cite[Corollary~2.6]{KTV14} on $\R^d$   and the finite speed of propagation for the wave equation. 
The second estimate \eqref{Str2a} for $\S(-t)^\tf \vec{\psi}$ can be similarly obtained from 
those for the half wave equation in \cite[Corollary~2.5]{KTV14}.
The nonhomogeneous estimate
\eqref{Str3}
follows from 
\cite[Corollary~2.6]{KTV14} 
(with the finite speed of propagation)
and \eqref{Str1a};
see \cite[p.\,7348]{GKO1}.

\begin{lemma}[Strichartz estimates]
\label{LEM:Slin}
Let $d\ge2$.
Given  $0 < s_1 < \frac d2$ and $0<s_2<1$,
 let $(q_1,r_1)$ and $(q_2,r_2)$ be $s_1$-admissible and dual $s_2$-admissible, 
 respectively.
Then, we have
\begin{align}
\big\| \pi_1 [\S(t) \vec{\phi}\,]\big\|_{L^\infty_T  H^{s_1}_x } 
+ \big\| \pi_1 [ \S(t) \vec{\phi}\,]\big \|_{L^{q_1}_T L^{r_1}_x} &\les \|\vec{\phi} \|_{\H^{s_1}},
\label{Str1a}
\\
\big\| \pi_2 [\S(-t)^\tf \,\vec{\psi}\,] \big\|_{L^\infty_T  H^{1-s_2}_x} 
+ \big\| \pi_2  [\S(-t)^\tf \, \vec{\psi}\,] \big\|_{L^{q_2}_T L^{r_2}_x} &\les \|\vec{\psi} \|_{(\H^{s_2})^*}
\label{Str2a}
\end{align}

\noi
for any $0 < T \le 1$, 
where 
 $\pi_j \vec f $ denotes the projection onto the $j$th component of a vector $\vec f$
 as in \eqref{vec1}, 
$\S(t)^\tf$ denotes the matrix transpose of $\S(t)$  in \eqref{Sdef}, 
and $(\H^{s_2}(\T^d))^* = H^{-s_2}(\T^d) \times H^{1-s_2}(\T^d)$.

When $s= s_1 = s_2 \in (0, 1)$, 
we also have the following nonhomogeneous Strichartz estimates{\rm :}
\begin{align}
\| (\If(F), \dt \If(F))\|_{L^\infty_T \H^s_x}
+ \| \If(F)\|_{L^{q_1}_T L^{r_1}_x}
\les \|F\|_{L^{q_2}_TL^{r_2}_x}\wedge
\|F\|_{L^{1}_TH^{s-1}_x}
\label{Str3}
\end{align}

\noi
for any $0 < T \le 1$, 
where $\If(F)$ is given by 
\begin{align}
\If(F)(t) = \int_0^t S(t-t')F(t') dt'.
\label{If1}
\end{align}

\noi
Here, $S(t) = \frac{\sin(t\jb{\nb})}{\jb{\nb}}$ is as in \eqref{S0}.

\end{lemma}

From the transference principle 
(Lemma~\ref{LEM:trf} with $T_1 = (\pi_1 ,0)$ and $T_2= (0,\pi_2)$)
and   Lemma~\ref{LEM:Slin} with  \eqref{local}
(see also 
Remark \ref{REM:local}), 
we obtain the following 
version of the Strichartz estimates
on space-time functions.


\begin{lemma}\label{LEM:STR}

Let $d\ge2$.
Given  $0 < s_1< \frac d2$ and $0<s_2<1$,
 let $(q_1,r_1)$ and $(q_2,r_2)$ be $s_1$-admissible and dual $s_2$-admissible, 
 respectively.
Then, we have 
\begin{align}
\notag
\| \pi_1 \vec{u} \|_{L^{q_1}_T L^{r_1}_x} & \les \| \S(-t) \vec{u} \|_{U^{q_1}_T  \H^{s_1}_x},
 \\
\notag
\| \pi_2 \vec{v} \|_{L^{q_2}_T L^{r_2}_x} & \les \| \S(t)^\tf \,\vec{v} \|_{U^{q_2}_T (\H^{s_2}_x)^*}
\end{align}

\noi
for any $0<T \le 1$, 
where $\S(t)^\tf$ denotes the matrix transpose of $\S(t)$ in \eqref{Sdef}.

\end{lemma}

\medskip

Suppose that we can find
an $s$-admissible pair $(q,r)$ and a dual $s$-admissible pair $(\wt q,\wt r)$ so that 
\begin{align}
q \ge k \wt q \qquad \text{and} \qquad  r\ge k \wt r. 
\label{HK1}
\end{align}

\noi
Then,  H\"older's inequality  yields
\begin{align}
\|  u^k \|_{L^{\wt q}_T L^{\wt r}_x}
\les T^{\frac{1}{\wt q}- \frac{k}q} \|  u \|_{L^q_T L^r_x}^k. 
\label{HK2}
\end{align}

\noi
In this case, 
Proposition \ref{PROP:LWP}
on local well-posedness in $\H^s(\T^d)$ of NLW \eqref{NLW} 
follows
from  the Strichartz estimates (Lemma \ref{LEM:Slin}), 
 \eqref{HK2}, and a standard contraction argument.\footnote{
 Proposition \ref{PROP:Str2}, giving \eqref{HK2},  is restricted to $s < 1$.
See the proof of  Proposition \ref{PROP:LWPx} for 
relevant nonlinear estimates for $s \ge 1$.
 }
When $q > k \wt q$, 
 we have a positive power of $T$ in \eqref{HK2}, 
 in which case
 we can take $T = T(\| (\phi_0, \phi_1) \|_{\H^s}) > 0$.
Indeed, this is the case when $s$ is greater than the scaling critical regularity $s_\text{scaling}$.

The following proposition shows that
the condition 
 \eqref{HK1}
can be satisfied
when $0\le s_{\text{crit}} \le s <1$, where $s_\text{crit} = s_{\text{crit}}(d,k)$ denotes the critical exponent in \eqref{crit1} associated with NLW~\eqref{NLW}.
See 
 \cite[Section 3]{GKO1}
for the $d = 2$ case.

\begin{proposition}\label{PROP:Str2}

Let $d\ge2$ and $k\ge2$ 
such that $0 \le s_{\textup{crit}} <1$, where 
$s_{\textup{crit}} = s_{\textup{crit}}(d, k)$ is as in \eqref{crit1}.
Namely, $(d, k)$ satisfies 
\begin{align}
(d,k) \in 
\mathcal{A}_1 := 
\big\{(2,k): \, k\ge 2\big\} \cup \big\{ (3, 2), (3,3), (3, 4), (4, 2), (5,2) \big\}.
\notag 
\end{align}

\noi
Suppose that one of the following conditions holds\textup{:}

\smallskip

\begin{itemize}
\item[\textup{(i)}]
$s_\textup{crit} <  s < 1$,  when $(d,k) \in\{(2,2), (2,3), (3,2)\}$, or

\smallskip

\item[\textup{(ii)}]
$s_\textup{crit} \le s < 1$,  otherwise.

\end{itemize}

\smallskip

\noi
Then,
there exist
an $s$-admissible pair $(q,r)$ and a dual $s$-admissible pair $(\wt q,\wt r)$
  such that
\begin{align}
q \ge k \wt q \qquad \text{and} \qquad  r\ge k \wt r.
\label{X1}
\end{align}

\noi
Here,  the equality
$q =  k \wt q $ holds when 
$s = s_\textup{crit}=s_\textup{scaling} \,( \ge \frac 12)$, where~$s_\textup{scaling}$ is as 
in~\eqref{scaling}.
Namely, the equality 
$q =  k \wt q $ holds when
\begin{align}
\label{dk1}
(d,k)
\in 
\mathcal{A}_2
:= 
\{ (2,k): \, k \ge 5  \} \cup \{(3,3), (3,4), (5,2)\}.
\end{align}

\end{proposition}

We note that 
\[ \mathcal A_2 = \mathcal A_1 \setminus \mathcal A_0. \]

\noi
Here, 
$\mathcal A_0$ 
is as in 
\eqref{crit2x}, where  we have  $s_\text{scaling} <  s_\text{conf}$, 
where $ s_\text{conf}$ is as in \eqref{conf}.

When $d = 2$, Proposition \ref{PROP:Str2}
was  already established in 
 \cite[Section 3]{GKO1}, 
 and thus 
we only consider  the case $d \ge 3$
in the following.
In view of \eqref{X1},  we would like to  maximize 
\[   \min\Big\{ \frac{q}{\wt q}, \frac{r}{\wt r}  \Big\} \]

\noi
under the constraints of  Lemma \ref{LEM:pair1}. 
In view of \eqref{admis1a},  \eqref{admis1aa}, 
and 
 Lemma \ref{LEM:pair1}, 
this is essentially\footnote{Here, we allow $\wt r =1$ that is not admissible for the Strichartz estimates.} equivalent 
to  the following maximization problem for 
$ J_s(r,\wt r)$ defined by
\begin{align}
 J_s(r,\wt r) =
  \min\Big\{ \frac{q}{\wt q}, \frac{r}{\wt r}  \Big\}
  =
 \frac{r}{\wt r} \min\bigg\{ 1, \frac{(d+4-2s)\wt r-2d }{(d-2s)r -2d } \bigg\}
\label{op1}
\end{align}

\noi
over the set
\begin{align}
K(s) & = [r_1,r_2] \times [\wt{r}_1, \wt{r_2}], 
\label{op2}
\end{align}

\noi
where
\begin{align}
\begin{split}
 [r_1,r_2]&= \bigg[ \frac {2d}{d-s},\, 
  \min\Big\{\frac{2d}{(d-1-2s)_+},\frac{2(d+1)}{(d+1-4s)_+} \Big\}\bigg] , \\
[\wt{r}_1, \wt{r}_2] &  = \bigg[\max\Big\{1,
  \frac{2(d+1)}{d+5 - 4s}, \frac{2d}{d+3-2s}\Big\},
\, \frac{2d}{  d+ 2-2s}\bigg].
\end{split}
\label{op2x}
\end{align}

The next lemma describes the maximum of $J_s(r,\wt{r})$ restricted to $K(s)$ 
as well as the values of the maximizers $(r,\wt{r})$.
See \cite[Lemma 3.3]{GKO1} for the $d = 2$ case.

\begin{lemma}\label{LEM:max1}

Let $d =  3, 4,   5$
and set $s_d=\frac{d-3}{2(d-1)} \in \big[0, \frac 12\big)$.
Given   $s_d\le s< 1$,
let
$ J_s(r,\wt r)$ and
$K(s)$ be as in \eqref{op1} and \eqref{op2}, respectively.
Then, we have
\begin{align}
  \sup_{(r, \wt r) \in K(s)}  J_s(r, \wt r)
= \begin{cases}
\rule[-3mm]{0pt}{0pt}
\frac{d+5 - 4s}{d+1-4s}, & \text{if } s_d \le s \le \frac12, \\
\rule[-3mm]{0pt}{0pt}
\frac{d+4 - 2s}{d-2s}, & \text{if }
   \frac12 \leq s< 1,
\end{cases}
\label{op2a}
\end{align}

\noi
where the supremum is indeed attained in each case\textup{:}

\smallskip
\begin{itemize}
\item[\textup{(i)}]
 when $s_d \le s \le \frac12$, it is attained at
$(r, \wt r) = \big( \frac{2(d+1)}{d+1-4s},   \frac{2(d+1)}{d+5 - 4s}\big)$,

\smallskip
\item[\textup{(ii)}]
when $ \frac 12 \leq s < 1$,  it is attained in the set{\rm :}
\begin{align*}
 \frac{d+4 - 2s}{d-2s}
 \cdot
   \frac{2(d+1)}{d+5 - 4s}
\le r & \le
\begin{cases}
\frac{2(d+1)}{d+1-4s}, & \text{if } \frac12\le s \le \frac{d+1}{2(d-1)}, \\
\frac{d+4-2s}{d-2s} \frac{2d}{d+2-2s}, & \text{if } \frac{d+1}{2(d-1)} \le s <1
\end{cases}
\quad \\
\text{and}
\quad \wt r & =
 \frac{d-2s}{d+4 - 2s}  r.
\end{align*}

\end{itemize}

\noi
Moreover, given an integer $k \ge 2$, we have
\begin{align}
\max _{(r, \wt r) \in K(s)}  J_s(r, \wt r) \ge k,
\label{op2b}
\end{align}

\noi
provided that
\begin{align}
s \ge 
  \begin{cases}
s_\textup{conf}(d, k) , & \text{if } s_d \le s \le \frac12, \\
s_\textup{scaling}(d, k), & \text{if }
   \frac12 \leq s < 1.
\end{cases}
\label{op2c}
\end{align}

\noi
The equality in \eqref{op2b} holds if and only if
the equality in \eqref{op2c} holds.

\end{lemma}

 \begin{proof}

From \eqref{op1}, the maximum of $J_s(r, \wt r)$ on $K(s)$ is given by $\max\{J_1(s),J_2(s) \}$,
where $ J_1(s)$ and $ J_2(s)$ are given by 
\begin{align}
 J_1(s) &= \max \bigg\{ \frac{r}{\wt r} : \frac{r}{\wt r} \le \frac{d+4 - 2s}{d-2s}, \, (r,\wt r) \in K(s) \bigg\},
\label{op4} \\
J_2(s) &  = \max \bigg\{
\frac{d+4-2s-2d /\wt r}{d-2s -2d /r} :
 \frac{r}{\wt r} \ge \frac{d+4 - 2s}{d-2s}, \, (r,\wt r ) \in K(s) \bigg\}.
\label{op5}
\end{align}

\noi
Note that $J_1(s)$ corresponds to the case $\frac r{\wt r}\le \frac q{\wt q}$,
while $J_2(s)$ corresponds to the case $\frac r{\wt r}\ge \frac q{\wt q}$.

For $0 \le s < 1$, we have 
\begin{align}
\label{op5a}
r_2 &=
\begin{cases}
\frac{2(d+1)}{d+1-4s},  & \text{if } 0 < s \le \frac{d+1}{2(d-1)}, \\
\frac{2d}{d-1-2s},& \text{if } \frac{d+1}{2(d-1)} \le s < 1,
\end{cases}
\end{align}

\noi
while we have, for  $s\ge s_{d}$,
\begin{align}
\label{op5x}
\wt{r}_1  =
\frac{2(d+1)}{d+5-4s}.
\end{align}

\noi
We  note that $r_2 <\infty$ and 
that $\wt{r}_1 \ge  1$
with the equality if and only if $d= 3$ and $s = 0$.

From \eqref{op2}, \eqref{op2x},  \eqref{op5a}, and \eqref{op5x}, we have
\begin{align}
\label{op5b}
 \max \bigg\{ \frac{r}{\wt r} :  (r,\wt r) \in K(s) \bigg\} = \frac{r_2}{\wt{r}_1} =
\begin{cases}
\frac{d+5-4s}{d+1-4s}, & \text{if } s_d \le s \le \frac{d+1}{2(d-1)}, \\
\frac{d}{d+1} \cdot \frac{d+5-4s}{d-1-2s} ,& \text{if } \frac{d+1}{2(d-1)} \le s < 1.
\end{cases}
\end{align}
Moreover, the inequality
\begin{align}
\label{op5c}
\frac{r_2}{\wt{r_1}} \le \frac{d+4-2s}{d-2s}
\end{align}
holds if and only if $s_d \le s \le \frac12$ 
(recall $d \le 5$).
 Hence, for $s_d \le s \le \frac12$, we conclude from \eqref{op4}, \eqref{op5}, \eqref{op5b}, and \eqref{op5c} that
\begin{align}
\label{op7}
\begin{split}
\sup_{(r, \wt r) \in K(s)}  J_s(r, \wt r)
= \sup\bigg\{ \frac{r}{\wt r} : (r,\wt r) \in K(s)\bigg\}
= \frac{d+5 - 4s}{d+1-4s}.
\end{split}
\end{align}

Next, we consider the case  $\frac12 \le s < 1$.
On the one hand, from \eqref{op4}, we have  $J_1(s) \le \frac{d+4 - 2s}{d-2s}$.
On the other hand,
by minimizing $r$ and maximizing $\wt r$ under   $\frac{r}{\wt r} \ge \frac{d+4 - 2s}{d-2s}$,
we obtain
\begin{align}
\begin{split}
J_2(s)
& =  \max\bigg\{
\frac{d+4-2s-2d /\wt r}{d-2s -2d /r} :
\frac{r}{\wt r} =
\frac{d+4 - 2s}{d-2s} , \,  (r,\wt r) \in K(s) \bigg\}
\\
& = \frac{d+4 - 2s}{d-2s}.
\end{split}
\label{op8}
\end{align}

Therefore, \eqref{op2a} follows from \eqref{op7} and \eqref{op8}.
The claims in (i) and (ii) 
follow from \eqref{op5a}, 
\eqref{op5x}, and 
\eqref{op8}.
Lastly, 
we note that 
 \eqref{op2b} and \eqref{op2c}
follow from \eqref{op2a} with~\eqref{crit1}.
\end{proof}

We conclude this subsection by presenting a proof of
Proposition \ref{PROP:Str2}.

\begin{proof}
 [Proof of Proposition \ref{PROP:Str2}]

Let $J_s(r, \wt r)$ be as in 
\eqref{op1}.
Then, 
\eqref{X1} follows from
\eqref{op2b} in Lemma \ref{LEM:max1}
(see \cite[Lemma 3.3]{GKO1} for the $d = 2$ case)
with the fact that $s_\text{crit} \ge s_d$.
Moreover, 
it follows from the proof of Lemma \ref{LEM:max1} 
(see \cite[Lemma 3.3]{GKO1} for the $d = 2$ case)
that 
\begin{align*}
\max _{(r, \wt r) \in K(s)}  J_s(r, \wt r) = 
\begin{cases} 
\frac r{\wt r}, 
& \text{when $s_d \le s < \frac12$,}\\
\frac q{\wt q}, 
& \text{when $\frac 12 \le s <1$}.
\end{cases}
\end{align*}

\noi
By noting that 
a strict inequality holds in \eqref{op5c}
for $s < \frac 12$, 
we see that 
$\max _{(r, \wt r) \in K(s)}  J_s(r, \wt r)
= \frac r{\wt r} < \frac{q}{\wt q}$ in the first case above.
From this observation with 
\eqref{op2b} and \eqref{op2c} (also see \eqref{crit1}), 
we conclude that  
$q = k \wt q$ holds if and only if
$ s = s_\textup{crit}=s_\textup{scaling} \ge \frac 12$ .
\end{proof}

\begin{remark}\label{REM:q1}\rm
Fix $0 \le s_\text{crit} < 1$
and let $s < 1$ be as in 
Proposition \ref{PROP:Str2}.
Then, it follows from 
 from (the proof of) Lemma~\ref{LEM:max1}
(see \cite[Lemma 3.1]{GKO1} for $d = 2$)
with \eqref{admis1a} and \eqref{admis1aa}
that 
the $s$-admissible pair $(q,r)$ and the dual $s$-admissible pair $(\wt{q},\wt{r})$
from  Proposition~\ref{PROP:Str2}
satisfy
 $q>2$ and $\wt{q}<2$.

\end{remark}

\subsection{Higher dimensional case}
\label{SUBSEC:nonl2}

In this subsection, we establish nonlinear estimates
for $d \ge 2$, using  the Strichartz estimates
discussed in the previous subsection.

\begin{proposition}
\label{PROP:nonlin2}

Given $d\ge 2$ and an integer $k\ge2$, 
let $s_{\textup{crit}}$ be as in \eqref{crit1}.
Suppose that  $s>0$ satisfies one of the following conditions{\rm :}
\begin{align}
\begin{split}
\textup{(i) } & s>s_{\textup{crit}}, \  \text{ when }
(d,k) \in \{(2,2), (2,3), (3,2)\}, \\
\textup{(ii) } & s \ge s_{\textup{crit}},  \ \text{ otherwise.}
\end{split}
\label{nonlin2a}
\end{align}


\noi
\textup{(a)} Let $0\le s_{\textup{crit}} < 1$.
Then,   there exist $1 \le p_0 <2 < q_0 <\infty$  such that 
we have 
\begin{align}
\begin{split}
& \bigg\| \int_0^t \S(-t')
\begin{pmatrix}
0\\
(\pi_1[\S(t') \vec\ub_{t'}])^k - (\pi_1[\S(t') \vec\vb_{t'}])^k
\end{pmatrix}
dt'
\bigg\|_{U^p_{T }\H^s_x} \\
&\hphantom{XXXXXXXX}
  \les T^\theta \big( \| \vec\ub \|_{U^q_{T}\H^s_x} + \| \vec\vb \|_{U^q_{T} \H^s_x} \big)^{k-1} \| \vec\ub - \vec\vb\|_{U^q_{T}  \H^s_x}
\end{split}
\label{nonlin2}
\end{align}

\noi
for any $0 < T \le 1$, 
$ p_0 \le  p <\infty$, and $  1 \le q \le q_0$, 
where 
\begin{align}
\begin{split}
\textup{(a.i)} \ \ & \text{$\ta  = 0$ when $s=s_{\textup{crit}}
=  s_{\textup{scaling}}$}, 
\\ 
\textup{(a.ii)} \ \ & \text{$\ta > 0$ is sufficiently small, otherwise.}
\end{split}
\label{nonlin2y}
\end{align}

\noi
Here, 
 $\S(t)$ is the matrix-valued linear wave propagator
defined in \eqref{Sdef}, 
and $\pi_1 \vec f $ denotes the projection onto the first component of the vector $\vec f$.

\medskip

\noi
\textup{(b)} Let $ s_{\textup{crit}} \ge 1$.
Then, 
\eqref{nonlin2} holds
for any  $1 < p < \infty$, $1 \le q \le k+\eps_0$, and $\theta = \frac{\eps_0}{k+\eps_0}$, where
\begin{align}
\begin{split}
\textup{(b.i)} \ \ & \text{$\eps_0  = 0$ when $s=s_{\textup{crit}}=  s_{\textup{scaling}}$}, 
\\ 
\textup{(b.ii)} \ \ & \text{$\eps_0 > 0$ is sufficiently small when 
$s>s_{\textup{crit}}=  s_{\textup{scaling}}$.}
\end{split}
\label{nonlin2x}
\end{align}

\noi
In particular, 
it follows from \eqref{crit1} and \eqref{crit2x} that 
\eqref{nonlin2y} also holds in this case.

\end{proposition}

\begin{proof}

For simplicity, we prove \eqref{nonlin2} when $\vec\vb \equiv0$, 
since the general case follows from a straightforward modification.

\smallskip

\noi
(a) Fix $0 \le s_{\textup{crit}}<1$
and 
let $s$ be as in \eqref{nonlin2a}.
We first consider the case $ s < 1$.
By Proposition~\ref{PROP:Str2}, there exist an $s$-admissible pair $(q,r)$ and a dual $s$-admissible pair $(\wt{q},\wt{r})$ such that
\begin{align}
q\ge k \wt{q} \quad \text{and} \quad  r \ge k \wt{r},
\label{nonl2a}
\end{align}

\noi
where the equality $q=k\wt{q}$ holds only when $s=s_{\text{crit}} = s_{\text{scaling}}$, 
namely only when  $(d,k)$ satisfies~\eqref{dk1}.
Moreover, from 
 Remark \ref{REM:q1}, 
we have $q>2$ and $\wt{q}<2$.
Consequently, 
from~\eqref{lin2} in Lemma~\ref{LEM:lin}, H\"older's inequality, \eqref{nonl2a}, 
the Strichartz estimates (Lemma~\ref{LEM:STR}), and 
the embedding \eqref{Linfty}, we have
\begin{align}
\begin{split}
& \bigg\| \int_0^t \S(-t' )
\begin{pmatrix}
0 \\
(\pi_1[\S(t') \vec\ub_{t'}])^k
\end{pmatrix}
dt'
\bigg\|_{U^p_{T} \H^s_x} \\
&\quad 
 \le \sup_{\substack{
\|\vec\vb  \|_{V^{p'}_{\Box, T}(\H^s_x)^*} =1}} 
\bigg| \int_0^T \int_{\T^d} (\pi_1[\S(t) \vec\ub_{t}])^k\cdot  \pi_2 \vec\vb_{t} \, dx  dt \bigg| \\
&\quad 
 \le  \sup_{\substack{
\|\vec\vb  \|_{V^{p'}_{\Box, T}(\H^s_x)^*} =1}} 
 \|\pi_1[\S(t) \vec\ub_{t}] \|^k_{L^{k \wt{q}}_T L^{k \wt{r}}_x} \| \pi_2 \vec\vb \|_{L^{\wt{q}'}_T L^{\wt{r}'}_x} \\
&\quad  \les T^\theta  \sup_{\substack{
\|\vec\vb 
\|_{V^{p'}_{\Box, T}(\H^s_x)^*} =1}} 
 \| \pi_1[\S(t) \vec\ub_{t}] \|^k_{L^{q}_T L^{r}_x} \| \pi_2 \vec\vb \|_{L^{\wt{q}'}_T L^{\wt{r}'}_x} \\
& \quad \les T^\theta   \sup_{\substack{
\|\vec\vb  
\|_{V^{p'}_{\Box, T}(\H^s_x)^*} =1}} 
\|  \vec\ub_{t} \|^k_{U^{q}_T \H^s_x} \| \S(t)^\tf \, \vec\vb \|_{U^{\wt{q}'}_T (\H^s_x)^*} \\
&\quad  \les  T^\theta \| \vec\ub \|^k_{U^q_{T}\H^s_x},
\end{split}
\label{nonl2ab}
\end{align}

\noi
where
$V^{p'}_{\Box, T}(\H^s_x)^*
=V^{p'}_{\Box, T}(H^{-s}(\T^d) \times H^{1-s}(\T^d))$ is as in 
\eqref{Dl1} (see also \eqref{local})
and 
\begin{align}
\theta = \frac{1}{\wt{q}} - \frac{k}{q} \ge 0 
\qquad \text{and} \qquad  p> \wt{q}.
\label{nonl2aa}
\end{align}

\noi
Here, the equality $\ta = 0$
holds when
$s=s_\text{crit} = s_{\text{scaling}}$.
In view of the embedding \eqref{Linfty},
by taking 
$p_0=\wt q + <2$ and $q_0 = q >2$, 
we see that the estimate  \eqref{nonlin2} holds 
under the hypothesis on $p$ and $q$.

Next, we consider the case
$s \ge 1$.
Fix $s_{\text{crit}} < s_0 < 1$. From Proposition~\ref{PROP:Str2}, there exist an $s_0$-admissible pair $(q,r)$ and a dual $s_0$-admissible pair $(\wt{q}, \wt{r})$ such that  \eqref{nonl2a} holds
with the strict inequality $q >  k \wt{q}$.
Proceeding as in \eqref{nonl2ab}
with \eqref{lin2} in Lemma~\ref{LEM:lin}, H\"older's inequality, 
the fractional Leibniz rule (Lemma~\ref{LEM:leibniz}), 
 \eqref{nonl2a} with $q >  k \wt{q}$, 
 the Strichartz estimates (Lemma~\ref{LEM:STR}), and 
the embedding~\eqref{Linfty}, we have 
\begin{align}
\begin{split}
& \bigg\| \int_0^t \S(-t' )
\begin{pmatrix}
0 \\
\jb{\nabla}^{s-s_0} (\pi_1[\S(t') \vec\ub_{t'}])^k
\end{pmatrix}
dt'
\bigg\|_{U^p_{T} \H^{s_0}_x} \\
&\quad  \le  \sup_{
\|\vec\vb  \|_{V^{p'}_{\Box, T}(\H^{s_0}_x)^*} =1}
\| \jb{\nabla}^{s-s_0} \pi_1[\S(t) \vec\ub_{t}] \|^k_{L^{k \wt{q}}_T L^{k \wt{r}}_x} \| \pi_2 \vec\vb \|_{L^{\wt{q}'}_T L^{\wt{r}'}_x} \\
&\quad  \les T^\theta  \sup_{\substack{
\|\vec\vb  \|_{V^{p'}_{\Box, T}(\H^{s_0}_x)^*} =1}}
\| \jb{\nabla}^{s-s_0} \pi_1[\S(t) \vec\ub_{t}] \|^k_{L^{q}_T L^{r}_x} \| \pi_2 \vec\vb \|_{L^{\wt{q}'}_T L^{\wt{r}'}_x} \\
&\quad  \les  T^\theta \| \vec\ub \|^k_{U^q_{T} \H^s_x},
\end{split}
\label{nonl2ax}
\end{align}
provided  that \eqref{nonl2aa} holds.
Note that $\ta > 0$ in this case.

\medskip

\noi
(b)
Let $s\ge s_{\text{crit}} \ge 1$.
In this case, it follows from \eqref{crit1}
that  
\begin{align}
s_{\text{crit}} = s_{\text{scaling}} = \frac d2 - \frac{2}{k-1}
\qquad  \text{and} \qquad d\ge3.
\label{nonlin3}
\end{align}

Given  $\eps_0 \ge 0$ as in \eqref{nonlin2x}, 
we 
define $(q_j, r_j)$, $j = 1, 2$,  and $\ta \ge 0$
by setting
\begin{align}
\begin{split}
q_1 & = k + \eps_0, \\
r_1 & = \frac{2d(k+\eps_0)}{(d-2)(k+\eps_0) - 2}, \\
q_2 & = \frac{(k-1)(k+\eps_0)}{k+\eps_0-1}, \\
r_2 & = \frac{d(k-1)(k+\eps_0)}{k+\eps_0+1}, \\
\theta & = \frac{\eps_0}{k+\eps_0} \ge 0.
\end{split}
\label{nonl3c}
\end{align}

\noi
Note from \eqref{nonlin3} that 
 $(q_1,r_1)$ is a $1$-admissible pair
 (see Definition \ref{DEF:admiss})
 and 
that  we have 
\begin{align}
\begin{split}
\frac{1}{q_1} + \frac{k-1}{q_2}  = 1 , & \qquad   (q_2 \vee 2) \le q_1 , \\
\frac{1}{r_1} + \frac{k-1}{r_2}   = \frac12, 
& \qquad 2 \le r_1 \le r_2 , \\
\theta = 1 - \frac{k}{q_1}, 
& \qquad  \frac{s-1}{d} \ge \frac{1}{r_1} - \frac{1}{r_2}.
\end{split}
\label{nonl3b}
\end{align}

\noi
Then, by applying  \eqref{lin3} in Lemma~\ref{LEM:lin},
 the fractional Leibniz rule (Lemma~\ref{LEM:leibniz}), 
Sobolev's inequality, 
H\"older's inequality in time with \eqref{nonl3c}, the Strichartz estimates (Lemma~\ref{LEM:STR})
with~\eqref{Linfty}, 
we have
\begin{align}
\begin{split}
& \bigg\| \int_0^t \S(-t')
\begin{pmatrix}
0 \\
( \pi_1[\S(t') \vec\ub_{t'}])^k
\end{pmatrix}
dt'
\bigg\|_{U^p_{T}\H^{s}_x}\\
&\quad
\les \| \jb{\nabla}^{s-1}  (\pi_1[\S(t) \vec\ub_{t}])^k \|_{L^1_T L^2_x}  \\
&\quad \les \| \jb{\nabla}^{s-1} \pi_1[\S(t) \vec\ub_{t}] \|_{L^{q_1}_T L^{r_1}_x} \| \pi_1[\S(t) \vec\ub_{t}] \|_{L^{q_2}_T L^{r_2}_x}^{k-1} \\
&\quad \les T^\theta \| \jb{\nabla}^{s-1} \pi_1[\S(t) \vec\ub_{t}]  \|^k_{L^{q_1}_T L^{r_1}_x}\\
 & \quad \les T^\theta \| \vec\ub \|^k_{U^{q}_{T} \H^s_x}
\end{split}
\label{nonl2b}
\end{align}

\noi
for any $q \le q_1 = k+ \eps_0$.
Note that when $s=s_{\text{crit}}$, 
we can use 
the last inequality in \eqref{nonl3b}
with the definitions of $r_1$ and $r_2$ in \eqref{nonl3c}
to conclude $\eps_0 = 0$, 
which in turn implies $\ta = 0$, using \eqref{nonl3c}.
\end{proof}

When $0 \le s=s_{\textup{crit}}
=  s_{\textup{scaling}} < 1$
or $s=s_{\textup{crit}} \ge 1 $, 
Proposition \ref{PROP:nonlin2}
does not provide
any small power of $T$ (namely, $\ta = 0$; see 
\eqref{nonlin2y}
and \eqref{nonlin2x}).
In the next proposition, we consider 
a $k$-linear version of \eqref{nonlin2}
at the critical regularity $s = s_\text{crit}$, where at least one of the $k$-factors is smoother in space, 
and show that we can gain a small power of $T$ 
under this additional assumption. 

\begin{proposition}\label{PROP:nonlin3}

Given $d\ge 2$ and an integer $k\ge2$
such that  $(d,k)\notin \{(2,2), (2,3), (3,2)\}$, 
let $s_{\textup{crit}}$ be as in \eqref{crit1}.
Suppose that  one of the following conditions
holds{\rm :}
\begin{align*}
\textup{(i)}\ \ & 0\le s_{\textup{crit}} < 1,\\
\textup{(ii)}\ \ &  s_{\textup{crit}} \ge  1\qquad \text{and} \qquad k \ge 3.
\end{align*}

\noi
Then, given any small $\eps > 0$, 
we have 
\begin{align}
\begin{split}
& \bigg\| \int_0^t \S(-t')
\begin{pmatrix}
0 \\
\prod_{j=1}^k \pi_1[\S(t') \vec\ub_{t'}^{(j)}]
\end{pmatrix}
dt' \bigg\|_{U^{p}_{T} \H^{s_\textup{crit}}_x}\\
& \hphantom{XXXXXXXX}
 \les T^\eps \| \vec\ub^{(1)} \|_{U^q_{T} \H^{s_\textup{crit}+\eps}_x} \prod_{j=2}^k \|\vec\ub^{(j)} \|_{U^{q}_{T} \H^{s_\textup{crit}}_x}
\end{split}
\label{nonl00}
\end{align}

\noi
for any $0 < T \le 1$, 
where $p$ and $q$ are as in Proposition \ref{PROP:nonlin2}.
In particular, \eqref{nonl00} holds for some  $q > 2$.

Next, we consider the case{\rm:}
\begin{align}
\textup{(iii)}\ \ &  s_{\textup{crit}} \ge  1\qquad \text{and} \qquad k = 2.
\notag 
\end{align}

\noi
Then, given any small $\eps > 0$, we have 
\begin{align}
\label{nonl000}
\begin{split}
& \bigg\| \int_0^t \S(-t')
\begin{pmatrix}
0 \\
\prod_{j=1}^2 \pi_1[\S(t') \vec\ub_{t'}^{(j)}]
\end{pmatrix}
dt' \bigg\|_{U^{p}_{T} \H^{s_\textup{crit}}_x}\\
& \hphantom{XX}
 \les 
 T^\eps \| \vec\ub^{(1)} \|_{U^{q_*}_{T} \H^{s_\textup{crit}+\eps}_x}
 \min \Big(  \|\vec\ub^{(2)} \|_{U^{2}_{T} \H^{s_\textup{crit}}_x}, 
  \| \vec\ub^{(2)} \|_{U^{q_*}_{T} \H^{s_\textup{crit}+\eps}_x}\Big)
\end{split}
\end{align}

\noi
for any  
 $0 < T \le 1$
and $1 < p < \infty$, 
where 
 $q_* > 2$ is given by $\frac 1{q_*} = \frac 12 - \eps$.

\end{proposition}

\begin{remark}\label{REM:nonlin5}
\rm

When  $d \ge 6$ and $k = 2$ (namely, $s_{\textup{crit}} \ge  1$), 
the estimate \eqref{nonlin2} for $s =s_{\textup{crit}}$ holds only for $q\le 2$
(see \eqref{nonl3c}), which is too restrictive
to treat the rough case.
The point of the estimate \eqref{nonl000}
is that it 
does not only provide a small power of $T$
but also allows
us to place some factors in $U^{q_*}_T$, $q_* > 2$,  with a weaker temporal regularity, 
at the expense of a slight loss in spatial regularity.
This  plays a crucial role in the rough case presented
in  Section \ref{SEC:rough}.

\end{remark}

\begin{proof}[Proof of Proposition \ref{PROP:nonlin3}]
The estimates \eqref{nonl00}
and \eqref{nonl000} follow from small modifications
of the proof of Proposition \ref{PROP:nonlin2} with $s = s_\text{crit}$, 
 which we describe below.

We first discuss \eqref{nonl00}.
When $0\le s_{\textup{crit}} < 1$, 
we modify  \eqref{nonl2ab} in the proof of Proposition~\ref{PROP:nonlin2}\,(a)
with $s=s_\text{crit}$ as follows.
Fix small $\eps > 0$.
 Given the $s_\text{crit}$-admissible pair $(q,r)$ 
with $q>2$, 
set
  $ q_*>q>2$ 
by  $\frac{1}{q_*} = \frac1q - \eps >0$
 such that 
   $(q_*, r)$ is $(s_\text{crit} + \eps)$-admissible.
Then, 
 by H\"older's inequality and the Strichartz estimates (Lemma~\ref{LEM:STR})
 with the embedding~\eqref{Linfty}, we have
\begin{align}
\begin{split}
\| \pi_1 [ \S(t) \vec\ub_t^{(1)}] \|_{L^q_T L^r_x} 
&  \les T^\eps \| \pi_1 [ \S(t) \vec\ub_t^{(1)}] \|_{L^{q_*}_T L^r_x} 
 \les T^\eps \| \vec\ub^{(1)} \|_{U^{q_*}_{T} \H^{s_\text{crit} + \eps }_x }\\
&  \les T^\eps \| \vec\ub^{(1)} \|_{U^{q}_{T} \H^{s_\text{crit} + \eps }_x }.
\end{split}
\label{nonlin4a}
\end{align}

\noi
Then, \eqref{nonl00} follows from \eqref{nonl2ab}
with $s = s_\text{crit}$
by replacing the estimate on the first factor 
$\pi_1 [ \S(t) \vec\ub_t^{(1)}] $  by~\eqref{nonlin4a}.

When $ s_{\textup{crit}} \ge  1$ and $k \ge 3$, 
we modify  \eqref{nonl2b} in the proof of Proposition~\ref{PROP:nonlin2}\,(b)
with $s=s_\text{crit}$ as follows.
Fix small $\eps > 0$.
Given 
 the $1$-admissible pair 
 $(q_1,r_1)$
 in \eqref{nonl3c} (with $\eps_0=0$), 
 set $q_* > q_1 = k$
 by 
 $\frac{1}{q_*} = \frac{1}{q_1} - \eps$
 such that $(q_*,r_1)$ is $(1+\eps)$-admissible.
 Then, 
  by H\"older's inequality and the Strichartz estimates (Lemma~\ref{LEM:STR})
 with the embedding~\eqref{Linfty}, we have
\begin{align}
\begin{split}
\| \jb{\nabla}^{s_\text{crit}-1} \pi_1[\S(t) \vec\ub_t^{(1)}] \|_{L^{q_1}_T L^{r_1}_x} 
& \les T^\eps \| \jb{\nabla}^{s_\text{crit}-1} \pi_1[\S(t) \vec\ub_t^{(1)}] \|_{L^{q_*}_T L^{r_1}_x} \\
& \les T^\eps \|\vec\ub^{(1)} \|_{U^{q_*}_{T} \H^{s_\text{crit}+\eps}}
\les T^\eps \|\vec\ub^{(1)} \|_{U^{q}_{T} \H^{s_\text{crit}+\eps}}.
\end{split}
\label{nonlin4b}
\end{align}

\noi
Then, \eqref{nonl00} follows from \eqref{nonl2b}
with $s = s_\text{crit}$
by replacing the estimate on the first factor 
$\pi_1 [ \S(t) \vec\ub_t^{(1)}] $  by~\eqref{nonlin4b}.

Next, we consider the first bound in \eqref{nonl000}
for 
$s_{\textup{crit}} \ge  1$ and $k = 2$.
Our goal is not only to gain a small power of $T$
but also place one of the factors in a weaker space $U^{q_*}_T$.
As before, 
the estimate \eqref{nonl000} follows
from a modification of  \eqref{nonl2b} in the proof of Proposition~\ref{PROP:nonlin2}\,(b)
with $s=s_\text{crit}$.
Fix small $\eps > 0$.
Given 
 the $1$-admissible pair 
 $(q_1,r_1) = (2, \frac{2d}{d-3})$
 in \eqref{nonl3c} (with $\eps_0=0$), 
 set $q_* > q_1 = 2$
 by 
 $\frac{1}{q_*} =  \frac{1}{q_1} - \eps = \frac{1}{2} - \eps$
 such that $(q_*,r_1)$ is $(1+\eps)$-admissible.
 Then, by proceeding as in \eqref{nonl2b}
 and then 
  by applying H\"older's inequality and the Strichartz estimates (Lemma~\ref{LEM:STR}), 
we have 
\begin{align*}
& \bigg\| \int_0^t \S(-t')
\begin{pmatrix}
0 \\
\prod_{j=1}^2 \pi_1[\S(t') \vec\ub_{t'}^{(j)}]
\end{pmatrix}
dt' \bigg\|_{U^{p}_{T} \H^{s_\textup{crit}}_x}\\
& \quad \les \| \jb{\nabla}^{s_\text{crit}-1} \pi_1[\S(t) 
\vec\ub_{t}^{(1)}] \|_{L^{2}_T L^{r_1}_x} \| \jb{\nabla}^{s_\text{crit}-1} \pi_1[\S(t)
\vec\ub_{t}^{(2)}] \|_{L^{2}_T L^{r_1}_x} \\
& \quad  \les T^\eps \| \jb{\nabla}^{s_\text{crit}-1} \pi_1[\S(t) 
\vec\ub_{t}^{(1)}] \|_{L^{q_1^*}_T L^{r_1}_x} \| \jb{\nabla}^{s_\text{crit}-1} \pi_1[\S(t) \vec\ub_{t}^{(2)}] \|_{L^{2}_T L^{r_1}_x}  \\
& \quad \les T^\eps \| \vec\ub^{(1)} \|_{U^{q_*}_{T} \H^{s_\text{crit}+\eps}_x} 
\| \vec\ub^{(j)}\|_{U^2_{T} \H^{s_\text{crit}}_x} ,
\end{align*}

\noi
which yields the first bound in \eqref{nonl000}.
The second bound in \eqref{nonl000} follows from a similar computation.
\end{proof}

\section{Pathwise local well-posedness: Young case}
\label{SEC:young}

In this section, we establish pathwise local well-posedness of SNLW \eqref{SNLW} 
forced by  a 
fractional-in-time noise with the Hurst parameter $\frac 12 < \be < 1$
(Theorem~\ref{THM:1}).

For this purpose, we first consider the following 
 deterministic formulation of the problem. 
Let $(\phi_0, \phi_1) \in \H^s(\T^d)$ and  $\vXX \in C^\al_{2,T} \LOP(\H^s(\T^d))$ 
for some $\frac 12 < \al<1$ such that $\updl\vXX =0$, 
namely, $\vXX$ is a Young driver.
Then, 
we consider the following (deterministic) YDE
for 
the interaction representation 
 $\vec\ub(t) = \S(-t) \vec u(t)$:
\begin{align}
\label{duhamelY}
\vec{\ub}_t &= (\phi_0,\phi_1) + \int_0^t \S(-t')
\begin{pmatrix}
0 \\ (\pi_1 [\S(t') \vec\ub_{t'}])^k
\end{pmatrix}
dt'
+  
\I^{\vXX}(\vec\ub)(t), 
\end{align}

\noi
where 
$\pi_1 \vec f $ denotes the projection onto the first component of the vector $\vec f$
and $\I^{\vXX}(\vec\ub)$ denotes the Young integral of $\vec \ub$
with the Young driver $\vXX$.
It follows 
 from 
the embedding
$\U^{2 }_T\H^s(\T^d) \subset \V_{T}^{2} \H^s(\T^d)$ 
and Lemma~\ref{LEM:intY}
(note that $\al + \frac 12  > 1$)
that the Young integral 
$\I^{\vXX} (\vec \ub) \in 
\V_{T}^{\frac1\al} \H^s(\T^d)\subset 
\U^{2 }_T\H^s(\T^d)$ 
is well defined for each $\vec \ub 
\in 
\U^{2 }_T\H^s(\T^d)$.

Given $\frac 12 < \al < 1$ and $s \ge 0$, 
we define the space of $\Xc^{\al, s}(\T^d)$ of enhanced data sets by setting
\begin{align*}
\Xc^{\al, s}(\T^d)= 
\big\{ \Xi = (\phi_0, \phi_1, \vXX )  \in
\H^s(\T^d) \times C^{\al}_{2,1} \LOP(\H^s(\T^d)):\, 
\updl \vXX =0
\big\}, 
\end{align*}

\noi
endowed 
with the norm:
\begin{align}
\| \Xi \|_{\Xc^{\al, s}} = \| (\phi_0, \phi_1) \|_{\H^s} + \| \vXX \|_{C^{\al}_{2,1} \LOP(\H^s)}.
\label{enh0}
\end{align}

\noi
Then, we have the following local well-posedness result for the YDE~\eqref{duhamelY}.

\begin{proposition}
\label{PROP:LWP1}

Let $d\ge 1$, $k\ge2$, $ \frac12 < \al < 1$, and $s \ge 0$, satisfying 
\eqref{reg1}{\rm :}
\begin{align*}
\textup{(i) } 
& s>s_{\textup{crit}} \text{ when } (d,k) \in\{(2,2), (2,3), (3,2)\}, \quad \text{or}\\
\textup{(ii) } &  s \ge s_{\textup{crit}}  \text{ otherwise},
\end{align*}
where $s_{\textup{crit}}$ is as in \eqref{crit3} for $d=1$ and as in \eqref{crit1} for $d\ge2$. Then, the YDE \eqref{duhamelY} is locally well-posed in $\mathcal{X}^{\al, s}(\T^d)$. More precisely, given an enhanced data set
\begin{align*}
\Xi = (\phi_0, \phi_1, \vXX ) \in \Xc^{\al, s}(\T^d),
\end{align*}
there exist $0<T=T(\Xi)\le 1 $ and a unique solution 
$\vec\ub$ to \eqref{duhamelY} in the class{\rm :}
\[\vec\ub \in
\U^{2 }([0,T]; \H^s(\T^d))
= 
U^{2 }([0,T]; \H^s(\T^d)) \cap C([0,T]; \H^s(\T^d)).\]

\noi
Furthermore, the
solution map{\rm :}
\begin{align}
\Xi = (\phi_0, \phi_1, \vXX )  \in \Xc^{\al, s}(\T^d)
\longmapsto \vec\ub \in \U^{2 }([0,T]; \H^s(\T^d))
\label{sol1}
\end{align}

\noi
 is locally Lipschitz continuous.

\end{proposition}

We first present a proof of 
Theorem \ref{THM:1}, assuming Proposition \ref{PROP:LWP1}.

\begin{proof}[Proof of Theorem \ref{THM:1}]

Fix $s \ge 0$ and 
 $\s \in \R$, satisfying \eqref{reg1}  and \eqref{sigma0}.
Let  $(\phi_0, \phi_1) \in \H^s(\T^d)$ and  
$\Phi \in \HS(L^2(\T^d); H^\s(\T^d))$, 
satisfying \eqref{phi1}, and
$\vXX$
be a random driver as in~\eqref{introX}; see also \eqref{XX0} and \eqref{XX0a}.
We note that 
the condition \eqref{YG0} with $s_0 = s$
 is equivalent to \eqref{sigma0}.
Then, given $\frac 12 < \al < \be$, it follows from Proposition~\ref{PROP:driver1} that 
there exists $\Si \subset \Om$ with $\PP(\Si) = 1$
such that 
\begin{align}
\vXX^\o \in C^\al_{2,1} \LOP(\H^s(\T^d))  \qquad \text{and}\qquad
\updl \vXX^\o = 0
\notag 
\end{align}

\noi
for each $\o \in \Si$.
In particular, 
we have 
$(\phi_0, \phi_1, \vXX^\o) \in \Xc^{\al, s}(\T^d)$
for each $\o \in \Si$.

Fix $\o \in \Si$.
Then, from the discussion preceding Proposition \ref{PROP:LWP1}, 
we see that the Young integral 
$\I^{\vXX^\o} (\vec \ub) \in 
\V_{T}^{\frac1\al} \H^s(\T^d)\subset
\U^{2 }_T\H^s(\T^d)$ 
is well defined for each $\vec \ub 
\in 
\U^{2 }_T\H^s(\T^d)$.
Therefore, 
by interpreting 
the stochastic term $\vec{\pPsi} (\vec\ub)$ in \eqref{psi2}
as the Young integral
$\I^{\vXX^\o}(\vec \ub)$, 
pathwise local well-posedness of SNLW 
\eqref{mild2}, interpreted 
as the YDE \eqref{duhamelY}
with the enhanced data set
$(\phi_0, \phi_1, \vXX^\o) \in \Xc^{\al, s}(\T^d)$, 
follows from 
Proposition~\ref{PROP:LWP1}.
\end{proof}

We now present a  proof of Proposition~\ref{PROP:LWP1}.

\begin{proof}[Proof of Proposition~\ref{PROP:LWP1}]

Given 
$\Xi = (\phi_0, \phi_1, \vXX ) \in \mathcal{X}^{\al, s}(\T^d)$, 
we define the map $\vec{\Gamma} = \vec{\Gamma}_{\Xi}$  by
\begin{align}
\begin{split}
\vec{\Gamma}(\vec{\ub})(t)
& = (\phi_0, \phi_1)+ \int_0^t \S(-t')
\begin{pmatrix}
0\\
(\pi_1[\S(t') \vec\ub_{t'}])^k
\end{pmatrix}
dt'
+ 
\I^{\vXX}(\vec\ub)(t)\\
& = : (\phi_0, \phi_1)+ \vec{\NN}({\vec{\ub}})(t) + 
\I^{\vXX}(\vec\ub)(t).
\end{split}
\label{Gamma}
\end{align}

As for the last term in \eqref{Gamma}, 
it follows from 
\eqref{Linfty}, 
\eqref{intY1b} in Lemma~\ref{LEM:intY}, 
 and \eqref{con1x}
that 
\begin{align}
\begin{split}
\| \I^{\vXX}(\vec\ub) \|_{U^{2}_{T} \H^s_x} 
& \les \|   \I^{\vXX}(\vec\ub) \|_{V^{\frac1\al}_{T} \H^s_x} 
 \les \| \vXX\|_{\V^{\frac{1}{\al}}_{2,T} \LOP(\H^s)}
  \| \vec{\ub} \|_{{V}^{2}_{T} \H^s_x}
\\
& \les T^\al  \| \vXX\|_{{C}^{\al}_{2,T} \LOP(\H^s)} \| \vec{\ub} \|_{{U}^{2}_{T} \H^s_x} .
\end{split}
\label{nonlY1a}
\end{align}

\noi
Then, 
from \eqref{lin1} in Lemma \ref{LEM:lin}, 
 Proposition~\ref{PROP:nonlin2} with $p=q=2$ for $d\ge2$ 
 (and Proposition~\ref{PROP:nonlin1} for $d=1$), and \eqref{nonlY1a}, 
we have 
\begin{align}
\begin{split}
\| \vec{\Gamma}(\vec{\ub}) \|_{{U}^{2}_{T} \H^s_x} 
& \le \| (\phi_0, \phi_1) \|_{\H^s} 
+  \| \vec{\NN}(\vec{\ub}) \|_{U^{2}_{T} \H^s_x} 
+  \| \I^{\vXX}(\vec\ub)\|_{U^{2}_{T} \H^s_x}\\
& \les \| (\phi_0, \phi_1) \|_{\H^s} 
+ 
T^\theta \| \vec{\ub} \|^k_{U^{2}_{T} \H^s_x} 
+ T^\al  \| \vXX\|_{{C}^{\al}_{2,T} \LOP(\H^s)} \| \vec{\ub} \|_{{U}^{2}_{T} \H^s_x} 
\end{split}
\label{G2}
\end{align}

\noi
for some $\ta \ge 0$, 
satisfying \eqref{nonlin2y} when $d \ge 2$
(and $\ta > 0$ when $d = 1$).
A similar computation yields
the difference estimate:
\begin{align}
\begin{split}
\| \vec{\Gamma}(\vec{\ub}) - \vec{\Gamma}(\vec{\vb}) \|_{{U}^{2}_{T} \H^s_x} 
& \les 
T^\theta 
\Big(\|  \vec{\ub}  \|_{U^{2}_{T} \H^s_x} 
+ \|  \vec{\vb}  \|_{U^{2}_{T} \H^s_x} \Big)^{k-1}
\|  \vec{\ub} - \vec{\vb} \|_{U^{2}_{T} \H^s_x} \\
& \quad + T^\al  \| \vXX\|_{{C}^{\al}_{2,T} \LOP(\H^s)} \| \vec{\ub}  - \vec{\vb}\|_{{U}^{2}_{T} \H^s_x} .
\end{split}
\label{G3}
\end{align}

\noi
Hence, if $\ta > 0$, 
a contraction argument with \eqref{G2} and \eqref{G3}
yields local well-posedness of the YDE~\eqref{duhamelY}.
Moreover, 
local Lipschitz continuity of the solution map in \eqref{sol1}
follows from a slight modification of \eqref{G2}
and \eqref{G3} with 
\eqref{intY1c} in Lemma~\ref{LEM:intY}.

\begin{remark}\label{REM:conti}
\rm

The argument above showed that $\vec{\Gamma}$ is a contraction
on (a closed ball centered at the origin in) $U^2_T \H^s_x$
by choosing sufficiently small $T > 0$.
By noting that each term on the right-hand side of \eqref{duhamelY} is continuous
in time (recall Lemma \ref{LEM:intY}), 
we see that 
$\vec{\Gamma}$ is in fact a contraction
on (a closed ball centered at the origin in) 
the closed subspace 
$\U^2_T \H^s_x$ 
in \eqref{local2},  
consisting of continuous functions in time.
This proves continuity 
of the solution $\vec \ub$ to \eqref{duhamelY}.
A similar comment applies to 
the $\ta = 0$ case considered below
and 
the proof of Proposition~\ref{PROP:LWP2}
presented in Section \ref{SEC:rough}, 
and we omit this part of details.
See also Remark \ref{REM:conti2}.

\end{remark}

\medskip

It remains to consider the case $\ta = 0$, namely, 
when $d \ge 2$ and $s=s_{\text{crit}} = s_{\text{scaling}}$.
If initial data $(\phi_0, \phi_1)$ is sufficiently small, 
then a contraction argument, using 
\eqref{G2} and \eqref{G3} with $\ta = 0$, 
yields local well-posedness of the YDE \eqref{duhamelY}.
In the general case (namely, the large data case),  
we  proceed 
as in \cite[proof of Theorem 1.1]{HTT11}. 
While the required modifications are straightforward, we present 
some details for readers' convenience.

Fix $\Xi = (\phi_0, \phi_1, \vXX )  \in \Xc^{\al, s}(\T^d)$.
Given small $\dl_1 > 0$ (to be chosen later), it follows from the dominated convergence theorem
that there exists $N = N((\phi_0,\phi_1), \dl_1) \in \N$ such that 
\begin{align}
\label{X6aa}
\| \P^\perp_N(\phi_0,\phi_1) \|_{\H^s} \le \dl_1,
\end{align}

\noi
where $\P_N^\perp= \Id - \P_N$.
Given $K > 0$, small $\dl_2>0$, and $T > 0$,  
define the set $B_{K, \dl_2, T} = B_{K, \dl_2, T}(N)$ by 
\begin{align}
\begin{split}
B_{K, \dl_2, T}
& = \big\{ \vec \ub \in U^{2}([0,T]; \H^s(\T^d)): \,
\| \vec \ub \|_{U^{2}_{T} \H^s_x} \le K, \,
\| \P_N^\perp \vec \ub \|_{U^{2}_{T} \H^s_x} \le \dl_2 \big\}, 
\end{split}
\label{X6}
\end{align}

\noi
endowed with the metric $d(\vec \ub, \vec \vb) = \|\vec \ub - \vec \vb\|_{U^{2}_{T} \H^s_x}$.

With $u_t = \pi_1[\S(t) \vec\ub_t]$, we have
\begin{align}
\label{X6a}
\begin{split}
\P_N u_t  &= \P_N \pi_1 [ \S(t) \vec\ub_t] = \pi_1 [\S(t) \P_N \vec\ub_t ] , \\
 \P_N^\perp u_t &= \P_N^\perp \pi_1 [ \S(t) \vec\ub_t ] = \pi_1 [\S(t) \P_N^\perp \vec\ub_t].
\end{split}
\end{align}

\noi
By expanding $u^k = (\P_N u + \P^\perp_N u)^k$ with \eqref{X6a}, we write $\vec{\mathcal{N}}(\vec{\ub})$  in \eqref{Gamma} as 
\begin{align*}
\vec{\mathcal{N}}(\vec{\ub}) & = \vec{\mathcal{N}}^{(1)}(\vec{\ub}) + \vec{\mathcal{N}}^{(2)} (\vec{\ub}) ,
\end{align*}

\noi
where $\vec{\mathcal{N}}^{(1)}$ is at least quadratic in $\P_N^\perp u$ and 
$\vec{\mathcal{N}}^{(2)}$ is at most linear in $\P_N^\perp u$. 
Then, by Proposition \ref{PROP:nonlin2}
with $\theta=0$, \eqref{X6a}, and \eqref{X6}, we have
\begin{align}
\label{X7}
\| \vec \NN^{(1)}(\vec \ub) \|_{U^{2}_{T} \H^s_x} 
& \les \| \P_N^\perp \vec{\ub} \|_{U^{2}_{T} \H^s_x}^2 \| \vec{\ub} \|_{U^{2}_{T} \H^s_x}^{k-2}  \les   \dl_2^2 K^{k-2}
\end{align}

\noi 
for any $\vec{\ub} \in B_{K,\dl_2, T}$.
On the other hand, 
it follows from 
crudely estimating the contribution from 
$\vec \NN_2(\vec\ub)$ 
with 
 \eqref{lin3} in Lemma~\ref{LEM:lin}, 
the fractional Leibniz rule (Lemma~\ref{LEM:leibniz}),  Bernstein's inequality, 
and \eqref{Linfty}
that there exists  $\kk = \kk(d, k) > 0$ such that 
\begin{align}
\begin{split}
\| \vec \NN^{(2)}(\vec\ub)\|_{U^{2}_{T} \H^s_x}
& \les T \|u \|_{L^\infty_T H^s_x}
\|\P_N u \|_{L^\infty_{T}W^{s, \infty}_x}^{k-1}\\
& \les T N^\kk \| \S(t) \vec\ub \|_{L^\infty_T \H^s_x}^k\\
&
\les T N^\kk K^k
\end{split}
\label{X8}
\end{align}

\noi
for  any $\vec\ub \in B_{K, \dl_2, T}$.
%
%
%
%


Set 
\begin{align}
K = 2 \| \Xi \|_{\Xc^{\al, s}}
\qquad \text{and} \qquad
\dl_2 = 2\dl_1, 
\label{X8x}
\end{align}

\noi
where the $\Xc^{\al, s}$-norm is as in \eqref{enh0}.
Then, 
 from \eqref{Gamma}, 
\eqref{lin1}, 
 \eqref{X7}, \eqref{X8},
and \eqref{nonlY1a} with~\eqref{X8x},
we have 
\begin{align}
\begin{split}
\| \vec\G(\vec\ub) \|_{U^2_{T} \H^s_x}
& \le \| (\phi_0, \phi_1)\|_{\H^s}
+ C \dl^2_2 K^{k-2}\\
 &\quad  
 + C TN^\kk K^k 
+  C T^\al \|\vXX\|_{C^\al_{2,1} \LOP(\H^s)} K \\
& \le K, 
\end{split}
\label{X8a}
\end{align}

\noi
where the last inequality holds, provided that 
 $T = T(N, K) = T((\phi_0,\phi_1), \dl_1, K) > 0$
and $\dl_2 = \dl_2(K) >0$ are sufficiently small.
Similarly, with  \eqref{X6aa}, we have
\begin{align}
\begin{split}
\| \P_N^\perp \vec\G(\vec\ub) \|_{U^2_{T} \H^s_x}
& \le \dl_1 
+ C \dl^2_2 K^{k-2}\\
 &\quad   + C TN^\kk K^k 
+  C T^\al \|\vXX\|_{C^\al_{2,1} \LOP(\H^s)} K\\
& \le \dl_2, 
\end{split}
\label{X8b}
\end{align}

\noi
where the last inequality holds, provided that 
$T = T(N, K, \dl_2) = T((\phi_0,\phi_1), \dl_1, K) > 0$
and $\dl_2 = \dl_2(K) >0$ are sufficiently small.
A similar computation yields
the following difference estimate:
\begin{align}
\begin{split}
& \| \vec \G(\vec\ub) - \vec \G(\vec\vb)\|_{U^{2}_{T}\H^s_x}\\
& \quad \le C\Big(
  \dl_2 K^{k-2}
+ T N^\kk K^{k-1}
 + 
 T^\al \|\vXX\|_{C^\al_{2,1} \LOP(\H^s)} \Big)
 \|\vec\ub - \vec\vb\|_{U^{2}_{T } \H^s_x}\\
& \quad \le \frac 12 
 \|\vec\ub - \vec\vb\|_{U^{2}_{T } \H^s_x}, 
\end{split}
\label{X10}
\end{align}

\noi
where the last inequality holds, provided that 
$T = T(N, K) = T((\phi_0,\phi_1), \dl_1, K) > 0$
and $\dl_2 = \dl_2(K) >0$ are sufficiently small.

In view of \eqref{X8x}, 
we see that \eqref{X8a}, \eqref{X8b}, and \eqref{X10}
hold
by choosing 
$T = T\big((\phi_0,\phi_1),  \|\vXX\|_{C^\al_{2,1} \LOP(\H^s)}\big) > 0$
and $\dl_1 = \dl_1(K) >0$  sufficiently small.
In this case, 
it follows from 
\eqref{X8a}, \eqref{X8b}, and \eqref{X10}
that $\vec \G$ is a contraction on $B_{K, \dl_2, T}$
and thus local well-posedness of the YDE \eqref{duhamelY} follows
from Banach's fixed point theorem.
Local Lipschitz continuity of the solution map in \eqref{sol1}
follows from a slight modification of 
\cite[Part 3 of the proof of Theorem~1.1 on p.\,346]{HTT11} with \eqref{intY1c} in Lemma~\ref{LEM:intY}.
\end{proof}

\section{Pathwise local well-posedness: rough case}
\label{SEC:rough}

In this section, we establish pathwise local well-posedness of SNLW \eqref{SNLW} 
forced by  a 
white-in-time noise with the Hurst parameter $ \be = \frac 12$
(Theorem~\ref{THM:2}).

Following the presentation  in Section~\ref{SEC:young}  for the Young case, 
a natural approach is 
to  consider the following (deterministic) RDE
for 
the interaction representation 
 $\vec\ub(t) = \S(-t) \vec u(t)$:
\begin{align}
\vec{\ub}_t &= (\phi_0,\phi_1) + \int_0^t \S(-t')
\begin{pmatrix}
0 \\ (\pi_1 [\S(t') \vec\ub_{t'}])^k
\end{pmatrix}
dt'
+  
\I^{\vXX, \vbbX}(\vec\ub)(t), 
\label{duha2}
\end{align}

\noi
where 
$\I^{\vXX, \vbbX}(\vec\ub)$ denotes the rough integral of $\vec \ub$
with a given rough driver 
\begin{align}
(\vXX, \vbbX )  
& \in C^{\al}_{2,T} \LOP(\H^s(\T^d))
\times 
C^{2\al}_{2,T} \LOP(\H^s(\T^d))
\label{duha2a}
\end{align}

\noi
for  $\al < \frac12$ sufficiently close to $\frac 12$, 
satisfying Chen's relation \eqref{Chen2}.
Recall from Subsection~\ref{SUBSEC:sewR}
that  in order to construct
the rough integral $\I^{\vXX, \vbbX}(\vec\ub)$, 
we need to impose
an appropriate controlled structure on $\vec \ub$ 
as in Definition \ref{DEF:RP}\,(ii):
\begin{align}
(\updl \vec\ub)_{t,r} = \vXX_{t,r}( \vec\ub_r' )+ \vec{R}^{\vec\ub}_{t,r}
\label{duha3}
\end{align}

\noi
for some remainder term $\vec{R}^{\vec\ub} = \vec{R}^{\vXX, \vec\ub}$.
Thus, we are led to study the system \eqref{duha2} and \eqref{duha3}
of two equations.

Unfortunately, 
this approach fails in the 
 following case:
\begin{align}
\label{bad1}
d \ge 6, \quad
k=2, \quad
\text{and}
\quad
s=s_{\text{crit}} \ge 1 ,
\end{align}
where $s_{\text{crit}}$ is as in \eqref{crit1}.
On the one hand, 
as pointed out in Remark \ref{REM:nonlin5}, 
the nonlinear estimate \eqref{nonlin2} in Proposition \ref{PROP:nonlin2}
holds only for $q \le 2$.
On the other hand, 
assuming that 
the rough integral $\I^{\vXX, \vbbX}(\vec\ub)$ in \eqref{duha2}
makes sense, 
we expect 
$\vec\ub \in U^{q }_T \H^s(\T^d)\setminus U^2_T \H^s(\T^d)$
for $q= \frac 1\al >2$.
Namely, 
$\vec\ub$ does not have sufficient temporal regularity 
to apply the nonlinear estimate \eqref{nonlin2}.

We overcome  this issue by imposing a further structure on the unknown
$\vec \ub$. 
More precisely, we write $\vec \ub$ as  
\begin{align}
\vec\ub = \vec\vb + \vec\wb, 
\label{exp1}
\end{align} 

\noi
where $\vec\vb$,  $\vec\wb$, 
and 
the Gubinelli derivative $\vec\ub'$ (of $\vec \ub  =\vec\vb + \vec\wb$)
satisfy the following RDE-PDE system:
\begin{align}
\begin{split}
\vec\vb & = \I^{\vXX, \vbbX}(\vec\vb + \vec\wb),  \\
\vec\wb(t)& 
 = (\phi_0, \phi_1) + \vec\NN(\vec\vb + \vec\wb)(t),
\\
\big(\updl (\vec \vb + \vec \wb)\big )_{t,r} & = \vXX_{t,r} (\vec\ub'_r) + \vec{R}^{\vec\vb+ \vec \wb}_{t,r},
\qquad (t,r) \in \Dl_{2,T},
\end{split}
\label{RR1}
\end{align}

\noi
for some remainder term $\vec{R}^{\vec\vb+ \vec \wb}
= \vec{R}^{\vXX, \vec\vb+ \vec \wb}$, 
where 
$\vec\NN_t(\vec\ub)$ is given by 
\begin{align}
\vec\NN_t(\vec\ub) = 
\int_0^t \S(-t')
\begin{pmatrix}
0 \\
(\pi_1[\S(t') (\vec\ub_{t'})])^k
\end{pmatrix}
dt'.
\label{RR1x}
\end{align}

\noi
Here, 
$\I^{\vXX, \vbbX}(\vec\vb + \vec\wb)$
denotes the rough integral of $\vec\vb + \vec\wb$
with the rough  driver $(\vXX, \vbbX)$, given by 
\begin{align}
 \I^{\vXX, \vbbX}(\vec\vb + \vec\wb)(t) = \big[ (\Id - \Ld\updl)\vec\Taa \big]_{t,0},
\label{RR1a}
\end{align}

\noi
where 
$\vec\Taa$ is given by 
\begin{align}
\label{RTaa}
\vec\Taa_{t,r} = \vXX_{t,r} (\vec\vb_r + \vec\wb_r) + \vbbX_{t,r}( \vec\ub_r').
\end{align}

\noi
See \eqref{int2z} and \eqref{intRY0}.
From \eqref{RR1}, 
we see that 
 $\vec \vb$ is  
rough in time, 
while $\vec \wb$ is {\it smoother in time}.
A key observation is that, 
under an appropriate assumption, 
we can show that $\vec \vb$ is {\it smoother in space}
(see \eqref{RWP2} with \eqref{RRnorm})
such that we can apply Proposition \ref{PROP:nonlin3}.

Given $0 <\al <\frac12$, $s\ge0$, and small $\eps>0$, 
we define the space $\wt\Xc^{\al,s, \eps}(\T^d)$ 
of enhanced data sets by setting
\begin{align}
\begin{split}
 \wt\Xc^{\al,s, \eps}(\T^d)
&  = \Big\{
\Xi = (\phi_0, \phi_1, \vXX, \vbbX) \in
\H^s(\T^d) \times
\prod_{j = 1}^2 C^{j\al}_{2,1} \LOP(\H^s(\T^d); \H^{s+\eps}(\T^d))\\
& 
\hphantom{XXl} 
\text{  such that 
$(\vXX, \vbbX)$  satisfies Chen's relation \eqref{Chen2}}
\Big\}, 
\end{split}
\label{enh1}
\end{align}

\noi
endowed with the norm:
\begin{align}
\label{Xtilde}
\| \Xi \|_{\wt\Xc^{\al,s, \eps}}
=
\| (\phi_0, \phi_1) \|_{\H^s}
+ \|\vXX\|_{C^{\al}_{2,1} \LOP(\H^s;\H^{s+\eps})} + \| \vbbX\|_{C^{2\al }_{2,1} \LOP(\H^s;\H^{s+\eps})}.
\end{align}

\noi
In view of the embedding \eqref{con1x}, 
we see that 
 $(\vXX, \vbbX)$ is a $\frac 1\al$-variational rough path (see Definition~\ref{DEF:RP}\,(i)).

Next, we introduce a solution space.
Given $0 <\al <\frac12< \g < 1$, $s\ge0$, and small $\eps>0$, 
let $\vXX \in 
C^{\al}_{2,1} \LOP(\H^s(\T^d); \H^{s+\eps}(\T^d))$.
Fix  $0 < T \le 1$.
Let 
$\CRP^{\al, \g, s}_{\vXX, T} = 
\CRP^{\al, \g}_{\vXX, T}(\H^s(\T^d))$ 
be the space of paths controlled by $\vXX$
as defined in Definition~\ref{DEF:RP}\,(ii).
We then define 
 the ``subclass'' $\Yc^{\al,\g,s, \eps}_{\vXX, T} $ 
 by setting
 \begin{align}
 \begin{split}
 \Yc^{\al,\g,s, \eps}_{\vXX, T} 
&  = \big\{
 (\vec \vb, \vec \wb, \vec\ub') \in \V^{\frac1\al}_T \H^{s+\eps}(\T^d) \times \V^{\frac1\g}_{T} \H^s(\T^d) \times \V^{\frac1\al}_T \H^s(\T^d)\\ 
& \hphantom{XXll}\text{such that } \, (\vec\vb+\vec\wb, \vec\ub')\in \CRP^{ \al, \g, s}_{\vXX, T}
 \big\}, 
 \end{split}
 \label{YY1}
 \end{align}

 \noi
 endowed with the norm:
\begin{align}
\| (\vec\vb, \vec\wb, \vec\ub') \|_{\Yc^{\al,\g,s,\eps}_{\vXX, T}} 
& = \| \vec\vb \|_{{V}^{\frac1\al}_{T} \H^{s+\eps}_x} + \|\vec\wb \|_{V^{\frac1\g}_{T} \H^s_x} + \| (\vec\vb+\vec\wb, \vec\ub') \|_{\CRP^{\al, \g, s}_{\vXX, T}},
\label{RRnorm}
\end{align}

\noi
where the third norm is as in \eqref{CRPnorm}.

The main goal of this section is to prove the following
 local well-posedness for the RDE-PDE system \eqref{RR1}.

\begin{proposition}
\label{PROP:LWP2}

Let $d\ge 1$, $k\ge2$, 
$ 0 <  \al < \frac12 < \g < 1$ such that both $\al$ and $\g$ are sufficiently close to $\frac 12$ and 
\begin{align}
  \al + \g >1. 
  \label{RWP0}
\end{align}

 \noi
 Suppose that 
$s \ge 0$ satisfies~\eqref{reg1}{\rm :}
\begin{align*}
\textup{(i) } 
& s>s_{\textup{crit}} \text{ when } (d,k) \in\{(2,2), (2,3), (3,2)\},
 \quad \text{or}\\
\textup{(ii) } &  s \ge s_{\textup{crit}}  \text{ otherwise},
\end{align*}
where $s_{\textup{crit}}$ is as in \eqref{crit3} for $d=1$ and as in \eqref{crit1} for $d\ge2$. 
Then, given small $\eps > 0$, the RDE-PDE system~\eqref{RR1} is locally well-posed in $\wt{\mathcal{X}}^{\al,s, \eps}(\T^d)$.
 More precisely, given an enhanced data set
\begin{align*}
\Xi = (\phi_0, \phi_1, \vXX, \vec{\bbX}) &\in \wt{\mathcal{X}}^{\al,s, \eps}(\T^d),
\end{align*}

\noi
there exist $0< T  = T(\Xi) \le 1$ and a unique solution
$(\vec\vb, \vec\wb, \vec\ub')$ to \eqref{RR1} in the class{\rm :}
\begin{align}
(\vec\vb, \vec\wb, \vec\ub')
\in  \Yc^{\al,\g,s, \eps}_{\vXX, T},  
\label{RWP2}
\end{align}

\noi
satisfying
\begin{align}
\vec \ub' = \vec \vb + \vec \wb.
\label{RWP3}
\end{align}

\noi
Furthermore, the
solution map{\rm :}
\begin{align}
\begin{split}
& \Xi = (\phi_0, \phi_1, \vXX , \vbbX)  \in \wt \Xc^{\al, s, \eps}(\T^d)
\longmapsto
(\vec\vb, \vec\wb, \vec\ub')\in  \Yc^{\al,\g,s, \eps}_{\vXX, T} 
\end{split}
\label{sol2}
\end{align}

\noi
 is locally Lipschitz continuous.


\end{proposition}

\begin{remark}\label{REM:rough1}\rm

We point out that the decomposition \eqref{exp1}: $\vec \ub = \vec \vb + \vec \wb$
is needed only for the case~\eqref{bad1}.
In other situations, 
we can directly study the RDE \eqref{duha2}
coupled with \eqref{duha3} for two unknowns $\vec \ub$ and $\vec \ub'$.
See Remark \ref{REM:system2}.
For the simplicity of presentation, however, 
 we present a unified approach based on the decomposition \eqref{exp1}
 and study the RDE-PDE system~\eqref{RR1}
 for three unknowns
$\vec \vb$, $\vec \wb$, and $\vec \ub'$.

\end{remark}

We first present a proof of 
Theorem \ref{THM:2}.

\begin{proof}[Proof of Theorem \ref{THM:2}]

Fix $s \ge  0$ and 
 $\s \in \R$, satisfying \eqref{reg1}  and \eqref{sigma0}.
Fix   $(\phi_0, \phi_1) \in \H^s(\T^d)$ and  
$\Phi \in \HS(L^2(\T^d); H^\s(\T^d))$, 
satisfying \eqref{phi1}, and let
$\vXX$ and $\vbbX$
be random drivers as in~\eqref{introX}
and \eqref{bX0}; see also \eqref{XX0}, \eqref{XX0a}, 
 \eqref{bbX0}, and \eqref{bbX0a}.
We note that 
the conditions~\eqref{YG0}, \eqref{RG0}, and \eqref{RG1} with $s_0 = s + \eps$
(for small $\eps > 0$; see also \eqref{RG1x} in Remark \ref{REM:reg1})
 is equivalent to \eqref{sigma0}.\footnote{Note that the condition \eqref{sigma0}
 implies $\s > - \frac 12$.}
 Then, 
 given $0 < \al < \frac 12$, it follows from Proposition~\ref{PROP:driver2} that 
there exists $\Si \subset \Om$ with $\PP(\Si) = 1$
such that, for each $\o \in \Si$, we have  
\begin{align}
(\vXX^\o, \vbbX^\o )  
& \in C^{\al}_{2,1} \LOP(\H^s(\T^d);\H^{s+\eps}(\T^d))
\times 
C^{2\al}_{2,1} \LOP(\H^s(\T^d);\H^{s+\eps}(\T^d))
\notag 
\end{align}

\noi
and $(\vXX^\o, \vbbX^\o )$ satisfies Chen's relation \eqref{Chen2}.  
In particular, 
we have 
$(\phi_0, \phi_1, \vXX^\o , \vbbX^\o)  \in \wt \Xc^{\al, s, \eps}(\T^d)$
for each $\o \in \Si$.

Fix $\o \in \Si$.
Let  $ 0 <  \al < \frac12 < \g < 1$ such that both $\al$ and $\g$ are sufficiently close to $\frac 12$, 
satisfying \eqref{RWP0}.\footnote{Note that, by taking $\al$ and $\g$  sufficiently close to $\frac 12$, we
have 
$\al + \g =  (\al + \g) \wedge(3\al)$.}
Then, by taking $\eps > 0$ sufficiently small, 
it follows from Proposition \ref{PROP:LWP2}
that there exists 
a unique solution 
$(\vec\vb, \vec\wb, \vec\ub')
\in  \Yc^{\al,\g,s, \eps}_{\vXX^\o, T^\o}$ 
to the RDE-PDE system \eqref{RR1}
on the time interval $[0, T^\o]$, 
satisfying \eqref{RWP3}.
By setting 
\begin{align}
\vec \ub = \vec \vb + \vec \wb, 
\label{RWP4}
\end{align}
we see that 
$\vec \ub$ is a unique solution to the RDE
\eqref{duha2}.
Here, 
 the uniqueness for $\vec \ub$ refers
 to the uniqueness of 
 $\vec \ub' \, (= \vec \ub)$ given by Proposition \ref{PROP:LWP2}.\footnote{The uniqueness of   $\vec \vb$ and $\vec \wb$
  given by Proposition \ref{PROP:LWP2} also provides
certain uniqueness for $\vec \ub$   
 under the decomposition \eqref{RWP4}.
 }
Therefore, 
by interpreting 
the stochastic term $\vec{\pPsi} (\vec\ub)$ in \eqref{psi2}
as the rough integral
$\I^{\vXX, \vbbX}(\vec \ub)$, 
this proves pathwise local well-posedness of SNLW 
\eqref{mild2}, interpreted 
as the RDE \eqref{duha2}.
\end{proof}

We conclude this section by presenting a  proof of Proposition~\ref{PROP:LWP2}.

\begin{proof}[Proof of Proposition~\ref{PROP:LWP2}]

In the following, we only consider the case
$ d\ge 2$ and  $s=s_{\text{crit}}$.
 When  $d = 1$ or $s>s_{\text{crit}}$, 
the nonlinear estimates come with a small power of $T$
(see Propositions~\ref{PROP:nonlin1} and \ref{PROP:nonlin2})
and thus  
local well-posedness follows from a simpler argument.
More precisely, there is no need to exploit
smallness coming from the high frequency part of the initial data
(see~\eqref{RRBK2}) 
 and thus we can simply proceed with a contraction argument 
on the ball of radius $K$ in the space $\Yc^{\al,\g,s,\eps}_{\vXX, T}$ (see \eqref{YY1}) centered at the origin
(rather than working with a more complicated
set $B_{K,\dl_2,T}^{(\phi_0, \phi_1)}$ in~\eqref{RRBK}
which is designed to exploit
smallness  of the high frequency part of the initial data).

Let $s=s_{\text{crit}}$
and fix  $0 < T \le 1$.
Given $(\vec\vb, \vec\wb, \vec\ub') \in \Yc^{\al,\g,s,\eps}_{\vXX, T}$
(in particular,  $(\vec\vb+\vec\wb, \vec\ub') \in \CRP^{\al, \g, s}_{\vXX, T}$),
we write its  controlled structure 
as
\begin{align}
\big(\updl (\vec\vb+\vec\wb) \big)_{t,r} = \vXX_{t,r} (\vec\ub_r') + \vec R^{\vec\vb+\vec\wb}_{t,r}, 
\quad (t,r)\in\Dl_{2,T}, 
\label{YY2}
\end{align}

\noi
for some remainder term 
$\vec R^{\vec\vb+\vec\wb} 
= \vec R^{\vXX, \vec\vb+\vec\wb}
\in \V^{\frac1\g}_{2, T} \H^{s}(\T^d)$.
For simplicity of notation, we write
\begin{align}
\vec \zb = (\vec\vb, \vec\wb, \vec\ub')\quad \text{and}\quad 
\vec \zb^{(j)} = \big(\vec\vb^{(j)}, \vec\wb^{(j)}, (\vec\ub^{(j)})'\big), \ \ j = 1, 2, 
\label{YY3}
\end{align}

\noi
in the following.

Fix $\Xi = (\phi_0, \phi_1, \vXX, \vbbX) \in \wt\Xc^{\al,s,  \eps}(\T^d)$, 
where 
$\wt\Xc^{\al,s,  \eps}(\T^d)$ is as in \eqref{enh1}.
We 
define the map $\vec\bG = \vec\bG_\Xi$ on $\Yc^{\al,\g,s,\eps}_{\vXX, T}$ 
 by setting
\begin{align}
\label{RRG}
\begin{split}
\vec\bG (\vec \zb)
& = \big(\vec\G_{1}(\vec \zb), \vec\G_{2}(\vec \zb), \vec\vb+\vec\wb \big),
\\
\vec\G_{1}(\vec \zb)
& = \I^{\vXX, \vbbX}(\vec\vb + \vec\wb), 
\\
\vec\G_{2} (\vec \zb)(t) & =  (\phi_0, \phi_1) + \vec\NN(\vec\vb + \vec\wb)(t),
\end{split}
\end{align}

\noi
where $\vec\zb = 
(\vec\vb, \vec\wb, \vec\ub')$ is as in \eqref{YY3}
and $ \I^{\vXX, \vbbX}(\vec\vb + \vec\wb)$ and  $\vec\NN$ are as in 
\eqref{RR1a} and  \eqref{RR1x}, respectively.

Given small $\dl_1 > 0$ (to be chosen later), it follows from the dominated convergence theorem
that there exists $N = N((\phi_0,\phi_1), \dl_1) \in \N$ such that 
\begin{align}
\label{RRBK2}
\| \P^\perp_N(\phi_0, \phi_1) \|_{\H^s} \le \dl_1,
\end{align}

\noi
where $\P_N^\perp = \Id - \P_N$.
Given $K > 0$, small $\dl_2>0$, and $T > 0$,  
define the set $B_{K,\dl_2,T}^{(\phi_0, \phi_1)} = B_{K,\dl_2,T}^{(\phi_0, \phi_1)}(N)$ by setting
\begin{align}
\label{RRBK}
\begin{split}
B_{K,\dl_2,T}^{(\phi_0, \phi_1)} & = \big\{ (\vec\vb, \vec\wb, \vec\ub') \in \Yc^{\al, \g, s, \eps}_{\vXX, T}: \, \vec\vb_0 = 0, \ \vec\wb_0 = \vec\ub'_0= (\phi_0, \phi_1), \\
& \hphantom{XXl}   \|(\vec\vb, \vec\wb, \vec\ub') \|_{\Yc^{\al,\g,s,\eps}_{\vXX, T}} \le K,\  \| \P_N^\perp \vec\wb \|_{{V}^{\frac1\g}_{T} \H^{s}_x} \le \dl_2
   \big\},
\end{split}
\end{align}

\noi
endowed with the following metric: 
\[ d (\vec \zb^{(1)}, \vec \zb^{(2)})
= 
\big\|\big(\vec\vb^{(1)} - \vec\vb^{(2)}, \vec\wb^{(1)}- \vec\wb^{(2)}, 
(\vec\ub^{(1)})' - (\vec\ub^{(2)})'\big) \big\|_{\Yc^{\al,\g,s,\eps}_{\vXX, T}}, 
\]

\noi
where 
$\vec \zb^{(j)}$ is as in \eqref{YY3} and 
the $\Yc^{\al, \g, s, \eps}_{\vXX, T}$-norm is as  in \eqref{RRnorm}. 
In the following, we show that 
 $\vec\bG$ in \eqref{RRG} is a contraction on $B_{K,\dl_2,T}^{(\phi_0, \phi_1)}$
 by appropriately choosing the parameters.

\begin{remark}\label{REM:conti2}\rm

As in Remark \ref{REM:conti}, 
continuity in time of each component of 
$\vec\bG (\vec \zb)$ in~\eqref{RRG}
follows from 
Lemma \ref{LEM:intR}, 
the definition \eqref{RR1x} of $\vec \NN$, 
and the continuity of $\vec \vb$
and $\vec \wb$ (see \eqref{YY1} and \eqref{local2}).
We also note that 
the remainder term
$\vec R^{\vec\bG}_{t,r}$
defined in \eqref{RW1a}
is continuous in both $t$ and $r$, 
which follows from the continuity of 
$ \vec\G_j (\vec\zb)$, $j = 1, 2$, and 
the regularity of $\vXX$.
Moreover, 
as required in Definition \ref{DEF:V23}\,(iii), 
$\vec R^{\vec\bG}_{t,r}$
 vanishes on the diagonal (namely, when $t = r$), 
since 
$\vXX$
 vanishes on the diagonal
 (as a function in $C_{2, T}$).
Hence, 
we focus on establishing estimates
on various terms in the following.

\end{remark}

\medskip

\noi
$\bul$
{\bf Step 1: Controlled  structure of $\vec\bG (\vec \zb)$.}\\
\indent
Let $(\vec\vb, \vec\wb, \vec\ub') \in B^{(\phi_0, \phi_1)}_{K, \dl_2, T}$
with the controlled structure \eqref{YY2}.
We first show that $\vec\bG (\vec \zb) = 
\big(\vec\G_1(\vec\zb) ,  \vec\G_2(\vec\zb), \vec\vb+\vec\wb \big)$ 
in~\eqref{RRG}
is a controlled  path in $\CRP^{\al, \g, s}_{\vXX, T}$ 
in the sense of 
Definition~\ref{DEF:RP}\,(ii).
In view of 
 \eqref{RRG}, \eqref{RR1a}, and \eqref{RTaa}, 
we write   
\begin{align}
\label{RW1a}
\begin{split}
&
\big( \updl \vec\G_1 (\vec\zb) + \updl \vec\G_2 (\vec\zb)
\big)_{t,r} \\
&\quad 
=
\vXX_{t,r} (\vec\vb_r+\vec\wb_r) 
+ \Big\{ \vbbX_{t,r}(\vec\ub'_r) 
- ( \Ld\updl \vec\Taa)_{t,r}
+
 \big(\updl\vec\NN(\vec\vb+\vec\wb) \big)_{t,r} \Big\}
\\
 & \quad 
=: \vXX_{t,r} (\vec\vb_r+\vec\wb_r)  + \vec R^{\vec\bG}_{t,r},
\end{split}
\end{align}

\noi
where $\vec R^{\vec\bG}$ is a short-hand notation 
for 
$\vec R^{\vXX, \vec\G_1(\vec\zb)+\vec\G_2(\vec\zb)}$.

We need to show that $\vec R^{\vec\bG} \in \V^{\frac1\g}_{2,T} \H^s(\T^d)$. 
The continuity and the vanishing on the diagonal of 
$\vec R^{\vec\G}$ are clear from \eqref{RW1a}, 
\eqref{RRG}, 
\eqref{RR1x}, and the 
corresponding properties
of 
$\vXX$, $\vec\vb$, and $\vec\wb$.
See Remark \ref{REM:conti2}.


Let $\o^{(1)}_{X, p}$ and $\o^{(2)}_{X, p}$ be as in 
\eqref{ctrl0} and \eqref{ctrl1}, respectively.
Then,
it follows from 
 \eqref{con1a} with 
\begin{align}
 \kk := \al + \g >1
 \label{KK1}
\end{align}

\noi
and the fact that a sum of controls is a control
that 
\begin{align}
\begin{split}
\o: \!
& = \big(\o^{(2)}_{\L(\H^s), \frac1\al} (\vXX)\big)^\frac \al \kk
\big(\o^{(2)}_{\H^s, \frac1\g} (R^{\vec\vb+\vec\wb} )\big)^\frac \g \kk\\
& \quad  + \big(\o^{(2)}_{\L(\H^s), \frac1{2\al}} (\vbbX)\big)^\frac {2\al} \kk
\big(\o^{(1)}_{\H^s, \frac1\al} (\vec \ub')\big)^\frac \al\kk
\end{split}
\label{YY5}
\end{align}

\noi
is also a control.
It follows from 
\eqref{int2bb}
that 
\eqref{intR1ab}
holds for $\vec\Taa$ in \eqref{RTaa}
with $\o$ in \eqref{YY5}.
Then, 
from \eqref{RW1a}
and  \eqref{sew1} in the sewing lemma (Lemma \ref{LEM:sew}) with $\kk > 1$, 
 we have
\begin{align}
\label{Rw1b}
\| \vec R^{\vec\bG}_{t,r} \|_{  \H^s  }
\les \o(t, r)^\kk
+ \|\vbbX_{t, r}\|_{\LOP(\H^s)} 
\| \vec\ub'_r \|_{\H^s}
+
 \big\|  \big(\updl\vec\NN(\vec\vb+\vec\wb)\big)_{t,r} \big\|_{\H^s}
\end{align}

\noi
for any $(t, r) \in \Dl_{2, T}$.
Hence, from \eqref{Rw1b} and  \eqref{YY5} with \eqref{ctrl0a}, we obtain
\begin{align}
\label{Rw1c}
\begin{split}
\| \vec R^{\vec\bG} \|_{\V^\frac1\g_{2,T} \H^s_x}
&
\les \| \vXX\|_{\V^{\frac1\al}_{2,T} \LOP(\H^s)} 
\| \vec R^{\vec\vb+\vec\wb} \|_{\V^{\frac1\g}_{2, T} \H^s_x }
\\
&
\quad
+
\|\vbbX \|_{\V^{\frac{1}{2\al}}_{2, T}  \LOP(\H^s)} 
\| \vec\ub' \|_{V^{\frac1\al}_{T} \H^s_x}
+
\| \vec\NN(\vec\vb+\vec\wb) \|_{V^{\frac1\g}_{T} \H^s_x}.
\end{split}
\end{align}

\noi
Therefore, 
 $\big(\vec\G_1(\vec\zb) +   \vec\G_2(\vec\zb), \vec\vb+\vec\wb \big) 
\in \CRP^{\al, \g, s}_{\vXX, T}$ follows once 
we prove 
\begin{align}
\| \vec\NN(\vec\vb+\vec\wb) \|_{V^{\frac1\g}_{T} \H^s_x} < \infty.
\label{YY6}
\end{align}

\noi
We will prove \eqref{YY6} in  Step 2.

\medskip

\noi
$\bul$ {\bf Step 2: Nonlinear estimate.}\\ 
\indent 
In this step, we prove \eqref{YY6}.
By expanding the power $(\pi_1[\S(t')\vec\vb_{t'}]+\pi_1[\S(t')\vec\wb_{t'}])^k$ in~\eqref{RR1x}, we write
\begin{align}
\label{RR1add}
\vec\NN(\vec\vb + \vec\wb) = \vec\NN(\vec\wb)
+ \vec\NN^{(0)}(\vec\vb, \vec\wb),  
\end{align}
where $\vec\NN^{(0)}(\vec\vb, \vec\wb) $ is at least linear in $\vec\vb$.

Proceeding as in \eqref{X6a}-\eqref{X8} in the proof of Proposition~\ref{PROP:LWP1}
(recall $\g > \frac 12$ sufficiently close to $\frac 12$)
with \eqref{RR1x}, \eqref{RRBK},  and~\eqref{RRnorm}, 
we have 
\begin{align}
\label{RR1ae}
\| \vec\NN(\vec\wb) \|_{V^{\frac1\g}_T \H^s_x} 
\les \| \vec\NN(\vec\wb) \|_{U^\frac 1\g_T \H^s_x} 
\les 
 \dl_2^2 K^{k-2} + T N^\kk K^k
\end{align}
for some $\kk = \kk(d,k)>0$.
As for the second term on the right-hand side of \eqref{RR1add}, 
we apply 
Proposition \ref{PROP:nonlin3}
and obtain
\begin{align}
\label{RR1ad}
\| \vec\NN^{(0)}(\vec\vb, \vec\wb) \|_{V^{\frac1\g}_{T} \H^s_x} & \les T^\eps \|\vec\vb\|_{V^{\frac1\al}_{T} \H^{s+\eps}_x}
 \Big( \|\vec\vb\|_{V^{\frac1\al}_{T} \H^{s+\eps}_x} + \|\vec\wb\|_{V^{\frac1\g}_{T} \H^s_x} \Big)^{k-1}
\end{align}

\noi
by taking 
$\al$ sufficiently close to $\frac 12$.

Hence, 
from  \eqref{RR1add}, \eqref{RR1ae}, and \eqref{RR1ad} with \eqref{RRBK}, we obtain
\begin{align}
\label{RR1g}
\| \vec\NN(\vec\vb+ \vec\wb) \|_{V^{\frac1\g}_T \H^s_x}
 \les
 \dl_2^2 K^{k-2} + T^\ta N^\kk K^k < \infty
\end{align}
for some $\theta>0$.
This proves \eqref{YY6}.

\medskip

\noi
$\bul$ {\bf Step 3: Contraction.}\\
\indent
 From  \eqref{RRG}, \eqref{RRnorm},  and \eqref{CRPnorm},  we have
\begin{align}
\label{RR2a0}
\begin{split}
\| \vec\bG(\vec\zb) \|_{\Yc^{\al, \g, s, \eps}_{\vXX, T}}
&
\le
3
\| (\phi_0, \phi_1) \|_{\H^s}
+
 \|\I^{\vXX, \vbbX}(\vec\vb + \vec\wb) \|_{V^{\frac1\al}_{T} \H^{s+\eps}_x}\\
& \quad
+\| \vec\NN(\vec\vb+\vec\wb) \|_{{V}^{\frac1\g}_{T} \H^s_x} 
+ \frac14 \|\vec\vb+\vec\wb\|_{V^{\frac1\al}_T \H^s_x}\\
& \quad
+
\| \vec R^{\vec\bG} \|_{\V^{\frac1\g}_{2,T} \H^s_x}
\end{split}
\end{align}

\noi
for any  $\vec\zb= (\vec\vb, \vec\wb, \vec\ub') \in B_{K, \dl_2, T}^{(\phi_0, \phi_1)}$ defined in \eqref{RRBK}, 
where 
 $\vec R^{\vec\bG}$ is as is \eqref{RW1a}.

From \eqref{intR0ab} in Lemma~\ref{LEM:intR}
with 
\eqref{RRBK}, \eqref{RRnorm}, \eqref{YY2}, 
\eqref{CRPnorm}, 
\eqref{con1x}, 
 and 
\eqref{Xtilde}, 
we have
\begin{align}
\label{RR2aa}
\begin{split}
&  \|\I^{\vXX, \vbbX}(\vec\vb + \vec\wb) \|_{V^{\frac1\al}_{T} \H^{s+\eps}_x}\\
&\quad 
\les
\| \vXX\|_{\V^{\frac1\al}_{2,T} \LOP(\H^s;\H^{s+\eps})}
 \Big( \| \vec\vb+\vec\wb\|_{L^\infty_T \H^s_x}
+ \| \vec R^{\vec\vb+\vec\wb}\|_{\V^\frac1\g_{2,T} \H^s_x} \Big)
\\
&\quad 
\quad
+ \|\vbbX\|_{\V^{\frac{1}{2\al}}_{2,T} \LOP(\H^s;\H^{s+\eps})} \|\vec\ub'\|_{V^\frac1\al_{T} \H^s_x}
\\
&\quad 
\les T^\al
\| \Xi \|_{\wt\Xc^{\al,s, \eps}} K
\end{split}
\end{align}

\noi
for any  $\vec\zb=(\vec\vb, \vec\wb, \vec\ub') \in B_{K, \dl_2, T}^{(\phi_0, \phi_1)}$.

Set\begin{align}
K= 6 \| \Xi \|_{\wt\Xc^{\al, s , \eps}}\qquad \text{and}\qquad 
\dl_2 = 2\dl_1,
\label{RRr}
\end{align}

\noi
with the $\wt\Xc^{\al, s , \eps}$-norm as in \eqref{Xtilde}.
Then, 
from  \eqref{RR2a0}, \eqref{RR2aa}, \eqref{RR1g}, and \eqref{Rw1c} with \eqref{RRBK} and \eqref{RRr}, we have
\begin{align}
\begin{split}
\label{RR2ac}
\| \vec\bG(\vec\zb) \|_{\Yc^{\al, \g, s, \eps}_{\vXX, T}}
& \le 3 \| (\phi_0, \phi_1) \|_{\H^s}
+ CT^\al
\| \Xi \|_{\wt\Xc^{\al,s, \eps}} K\\
& \quad 
+ C \dl_2^2 K^{k-2}
+
CT^\theta N^\kk K^{k}
+
\frac12 K
\\
&  \le K, 
\end{split}
\end{align}

\noi
where the second inequality holds, provided that 
 $T = T(N, K) = T((\phi_0,\phi_1), \dl_1, K) > 0$
and $\dl_2 = \dl_2(K) >0$ are sufficiently small.
By 
  \eqref{RRG}, \eqref{RRBK2}, \eqref{RR1g}, 
and
\eqref{RRr}, 
we have 
\begin{align}
\label{RRF1}
\begin{split}
 \| \P^\perp_N \vec\G_2(\vec\zb) \|_{V^{\frac1\g}_{T} \H^s_x}  
 & \le \dl_1 + C \dl_2^2 K^{k-2} +C T^\theta N^\kk K^k  \\
 & \le \dl_2, 
\end{split}
\end{align}

\noi
where the second inequality holds, provided that 
$T = T(N, K, \dl_2) = T((\phi_0,\phi_1), \dl_1, K) > 0$
and $\dl_2 = \dl_2(K) >0$ are sufficiently small.

Given 
$\vec\zb^{(j)} = \big(\vec\vb^{(j)}, \vec\wb^{(j)}, (\vec\ub^{(j)})'\big) \in B_{K, \dl_2, T}^{(\phi_0, \phi_1)}$, $j=1,2$, 
let $\vec{R}^{(j)}$ be  the remainder term, 
defined by the following relation:
\begin{align}
\label{YY7}
\begin{split}
&\big( \updl \vec\G_1 (\vec\zb^{(j)}) + \updl \vec\G_2 (\vec\zb^{(j)})
\big)_{t,r} \\
&\quad 
=
\vXX_{t,r} (\vec\vb_r^{(j)}+\vec\wb_r^{(j)}) \\
& \quad \quad 
+ \Big\{ \vbbX_{t,r}\big((\vec\ub^{(j)})_r'\big) 
- ( \Ld\updl \vec\Taa^{(j)})_{t,r}
+
 \big(\updl\vec\NN(\vec\vb^{(j)}+\vec\wb^{(j)}) \big)_{t,r} \Big\}
\\
 & \quad 
= : \vXX_{t,r} (\vec\vb^{(j)}_r+\vec\wb^{(j)}_r)  + \vec R^{(j)}_{t,r}
\end{split}
\end{align}

\noi
for $(t, r) \in \Dl_{2, T}$, 
where 
$\vec\Taa^{(j)}$ is given by 
\begin{align}
\label{YY8}
\vec\Taa_{t,r}^{(j)} = \vXX_{t,r} (\vec\vb_r^{(j)} + \vec\wb_r^{(j)}) + \vbbX_{t,r}\big( (\vec\ub_r^{(j)})'\big).
\end{align}

\noi
From \eqref{YY7} and \eqref{YY8}, we have
\begin{align}
\begin{split}
\vec R^{(1)}_{t,r} - \vec R^{(2)}_{t,r}
& = \vbbX_{t,r}\big((\vec\ub^{(1)})'_r - (\vec\ub^{(2)})'_r\big) 
- ( \Ld\updl \vec\Taa^{(1, 2)})_{t,r}\\
& \quad +
 \big(\updl\vec\NN(\vec\vb^{(1)}+\vec\wb^{(1)}) \big)_{t,r} 
 -  \big(\updl\vec\NN(\vec\vb^{(2)}+\vec\wb^{(2)}) \big)_{t,r} , 
\end{split}
\label{YY9}
\end{align}

\noi
where 
$\vec\Taa^{(1, 2)}$ is given by 
\begin{align}
\label{YY10}
\begin{split}
\vec\Taa_{t,r}^{(1, 2)} 
& = \vec\Taa_{t,r}^{(1)} - \vec\Taa_{t,r}^{(2)}\\
& = \vXX_{t,r} \big((\vec\vb_r^{(1)} + \vec\wb_r^{(1)})
- (\vec\vb_r^{(2)} + \vec\wb_r^{(2)})\big)
 + \vbbX_{t,r}\big( (\vec\ub^{(1)})_r' - (\vec\ub^{(2)})_r'\big).
\end{split}
\end{align}

\noi
A direct computation with \eqref{YY10} and \eqref{YY7} (see \eqref{int2bb}) yields
\begin{align}
\begin{split}
(\updl \vec \Taa^{(1, 2)})_{t_1,t_2,t_3}
&
= - \vec \XX_{t_1,t_2} (R^{\vec\vb^{(1)} + \vec\wb^{(1)}}_{t_2,t_3} - R^{\vec\vb^{(2)} + \vec\wb^{(2)}}_{t_2,t_3})
 \\
& \quad -  \vec \bbX_{t_1,t_2} \big( (\updl [(\vec \ub^{(1)})'])_{t_2,t_3} - (\updl [(\vec \ub^{(2)})'])_{t_2,t_3} \big)
\end{split}
\label{YY11}
\end{align}

\noi
for $(t_1, t_2, t_3) \in \Dl_{3, T}$, 
where
$R^{\vec\vb^{(j)} + \vec\wb^{(j)}}$
denotes the remainder term
in the controlled structure of 
$\big(\vec\vb^{(j)}+ \vec\wb^{(j)}, (\vec\ub^{(j)})' \big) \in \CRP^{\al,\g, s}_{\vXX, T}$, 
$j = 1, 2$, 
as in \eqref{YY2}.

Define $\wt \o$ by setting
\begin{align}
\begin{split}
\wt \o
& = \big(\o^{(2)}_{\L(\H^s), \frac1\al} (\vXX)\big)^\frac \al \kk
\big(\o^{(2)}_{\H^s, \frac1\g} (R^{\vec\vb^{(1)}+\vec\wb^{(1)}} - R^{\vec\vb^{(2)}+\vec\wb^{(2)}} )\big)^\frac \g \kk\\
& \quad  + \big(\o^{(2)}_{\L(\H^s), \frac1{2\al}} (\vbbX)\big)^\frac {2\al} \kk
\big(\o^{(1)}_{\H^s, \frac1\al} ((\vec \ub^{(1)})'- (\vec \ub^{(2)})')\big)^\frac \al\kk, 
\end{split}
\label{YY11a}
\end{align}

\noi
where $\kk$ is as in \eqref{KK1}.
Then, $\wt \o$ is a control, just like 
$\o$ defined in 
\eqref{YY5}.
Then, 
by arguing as in \eqref{intR1ab} with \eqref{YY11}
and then applying 
 \eqref{sew1} in the sewing lemma (Lemma~\ref{LEM:sew}; recall $\kk > 1$), we have 
\begin{align}
\|\Ld\updl \vec\Taa^{(1, 2)}_{t, r}
\|_{ \H^s_x} \les \wt \o(t, r)^\kk
\label{YY12}
\end{align}

\noi
for any $(t, r) \in \Dl_{2, T}$.
Thus, from \eqref{YY12} with \eqref{YY11a}, \eqref{con1x}, 
 \eqref{RRnorm}, and
\eqref{CRPnorm} with \eqref{YY3}, 
we obtain
\begin{align}
\begin{split}
\|\Ld\updl \vec\Taa^{(1, 2)}
\|_{\V^{\frac1\g}_{2,T} \H^s_x}
& \les \| \vXX\|_{\V^{\frac1\al}_{2,T} \LOP(\H^s;\H^{s+\eps})}
\| R^{\vec\vb^{(1)}+\vec\wb^{(1)}} - R^{\vec\vb^{(2)}+\vec\wb^{(2)}}\|_{\V^{\frac1\g}_{2, T} \H^s_x }
\\
&
\quad
+
\|\vbbX \|_{\V^{\frac{1}{2\al}}_{2, T}  \LOP(\H^s;\H^{s+\eps}) } 
\big\| (\vec \ub^{(1)})'- (\vec \ub^{(2)})'\big\|_{V^{\frac1\al}_{T} \H^s_x}\\
& \les T^\al 
\| \Xi \|_{\wt\Xc^{\al,s, \eps}} 
\| \vec \zb^{(1)} - \vec \zb^{(2)} \|_{\Yc^{\al, \g, s, \eps}_{\vXX, T}} .
\end{split}
\label{YY12a}
\end{align}

Hence, putting 
\eqref{RRG}, 
\eqref{RR2aa}, 
\eqref{RR1g}, \eqref{YY9}, and \eqref{YY12a}
together, we obtain
\begin{align}
\label{RR2ad}
\begin{split}
& \| \vec\bG(\vec\zb^{(1)}) - \vec\bG(\vec\zb^{(2)}) \|_{\Yc^{\al, \g, s, \eps}_{\vXX, T}} \\
& \quad \le 
\bigg\{ C\Big( T^\al
\| \Xi \|_{\wt\Xc^{\al,s, \eps}} 
+ \dl_2 K^{k-2}
+ T^\ta N^\kk K^{k-1} \Big) + \frac 14\bigg\}
\| \vec \zb^{(1)} - \vec \zb^{(2)} \|_{\Yc^{\al, \g, s, \eps}_{\vXX, T}} 
\\
& \quad \le \frac 12 
\| \vec \zb^{(1)} - \vec \zb^{(2)} \|_{\Yc^{\al, \g, s, \eps}_{\vXX, T}} 
\end{split}
\end{align}

\noi
for any $\zb^{(j)}  \in B_{K, \dl_2, T}^{(\phi_0, \phi_1)}$, $j = 1, 2$, 
where the second inequality holds, provided that 
 $T = T(N, K) = T((\phi_0,\phi_1), \dl_1, K) > 0$
and $\dl_2 = \dl_2(K) >0$ are sufficiently small
(recall \eqref{RRr}).

In view of \eqref{RRr}, 
we see that 
 \eqref{RR2ac}, \eqref{RRF1}, and \eqref{RR2ad}
hold, 
by choosing 
$T = T\big((\phi_0,\phi_1),  \|\vXX\|_{C^\al_{2,1} \LOP(\H^s; \H^{s+\eps})}, 
 \|\vbbX\|_{C^{2\al}_{2,1} \LOP(\H^s;\H^{s+\eps})}
\big) > 0$
and $\dl_1 = \dl_1(K) >0$ are sufficiently small.
In this case, 
it follows from 
 \eqref{RR2ac}, \eqref{RRF1}, and \eqref{RR2ad}
that 
$\vec\bG$ defined in \eqref{RRG} is a contraction on $B_{K,\dl_2,T}^{(\phi_0, \phi_1)}$, 
and thus local well-posedness of \eqref{RR1} follows
from Banach's fixed point theorem.
From the definition \eqref{RRG}
of 
$\vec\bG$ with \eqref{YY3}, 
we conclude that 
$\vec \ub' = \vec \vb + \vec \wb$, yielding \eqref{RWP3}.
Local Lipschitz continuity of the solution map in \eqref{sol2}
follows from a slight modification of 
\cite[Part 3 of the proof of Theorem~1.1 on p.\,346]{HTT11} with \eqref{intR0ac} in Lemma~\ref{LEM:intR}.
We omit details.
\end{proof}

\begin{remark}\rm
\label{REM:CRPnorm}

We note that, in the second inequalities
of 
\eqref{RR2ac} and
\eqref{RR2ad}, 
the smallness of the 
constant (which is  $\frac 14$ in our case) in front of the $\|\vec{\ub}'\|_{V^{\frac1\al}_T \H^s_x}$-norm in 
the definition of the $\CRP^{\al,\g, s}_{\vXX, T}$-norm in \eqref{CRPnorm} 
played a crucial role
in showing that 
$\vec\bG$ is a contraction.

\end{remark}

\begin{remark}\label{REM:system2}\rm

Let us briefly discuss the case when \eqref{bad1} does {\it not} hold.
In this case, it suffices to study the system \eqref{duha2} and \eqref{duha3}, 
assuming the regularity \eqref{duha2a} of the driver $(\vXX, \vbbX)$.

We introduce a solution space in this setting.
Given $0 <\al <\frac12< \g < 1$, and $s\ge0$, 
let $\vXX \in 
C^{\al}_{2,1} \LOP(\H^s(\T^d))$.
Fix  $0 < T \le 1$.
Let 
$\CRP^{\al, \g, s}_{\vXX, T} = 
\CRP^{\al, \g}_{\vXX, T}(\H^s(\T^d))$ 
be the space of paths controlled by $\vXX$
as defined in Definition~\ref{DEF:RP}\,(ii).
We then define 
 the ``subclass'' $\wt \Yc^{\al,\g,s}_{\vXX, T} $ 
 by setting
 \begin{align}
 \begin{split}
\wt  \Yc^{\al,\g,s}_{\vXX, T} 
&  = \big\{
 (\vec \ub,  \vec\ub') \in \V^{\frac1\al}_T \H^{s}(\T^d) 
 \times \V^{\frac1\al}_T \H^s(\T^d)
\, \text{ such that } \, (\vec\ub, \vec\ub')\in \CRP^{ \al, \g, s}_{\vXX, T}
 \big\}, 
 \end{split}
 \label{SY1}
 \end{align}

 \noi
 endowed with the norm:
\begin{align}
\| (\vec\ub, \vec\ub') \|_{\wt \Yc^{\al,\g,s}_{\vXX, T}} 
& = \| \vec\ub \|_{{V}^{\frac1\al}_{T} \H^{s}_x} 
 + \| (\vec\ub, \vec\ub') \|_{\CRP^{\al, \g, s}_{\vXX, T}},
\label{SY2}
\end{align}

\noi
where the second norm is as in \eqref{CRPnorm}.
Then, under the assumption of
Proposition \ref{PROP:LWP2}, 
the system \eqref{duha2} and \eqref{duha3}
is locally well-posed in 
$ \wt \Xc^{\al, s, 0}(\T^d)$, 
where $\wt \Xc^{\al, s, 0}(\T^d)$ is as in \eqref{enh1} with $\eps = 0$.
 More precisely, given an enhanced data set
\begin{align*}
\Xi = (\phi_0, \phi_1, \vXX, \vec{\bbX}) &\in \wt{\mathcal{X}}^{\al,s, 0}(\T^d),
\end{align*}

\noi
there exist $0< T  = T(\Xi) \le 1$ and a unique solution
$(\vec\ub,  \vec\ub')
 \in \wt  \Yc^{\al,\g,s}_{\vXX, T}$ to 
the system \eqref{duha2} and~\eqref{duha3}, 
satisfying
\begin{align}
\vec \ub' = \vec \ub.
\label{SY3}
\end{align}

\noi
Furthermore, the
solution map{\rm :}
\begin{align*}
& \Xi = (\phi_0, \phi_1, \vXX , \vbbX)  \in \wt \Xc^{\al, s, 0}(\T^d)
\longmapsto
(\vec\ub,  \vec\ub')\in  \wt \Yc^{\al,\g,s}_{\vXX, T} 
\end{align*}

\noi
 is locally Lipschitz continuous.


This local well-posedness 
of the system \eqref{duha2} and~\eqref{duha3}
follows from 
proceeding as in the proof of Proposition \ref{PROP:LWP2}
(but with a simpler argument)
on the set 

\smallskip

\begin{itemize}
\item
 $\wt B_{K,T}^{(\phi_0, \phi_1)}$
when $s > s_\text{scaling}$, 

\smallskip

\item 
$\wt B_{K,\dl_2,T}^{(\phi_0, \phi_1)}$
when 
$s =  s_\text{scaling}$, 

\end{itemize}

\smallskip

\noi
for $K$ and $\dl_2$ as in \eqref{RRr}, 
where 
 $\wt B_{K,T}^{(\phi_0, \phi_1)}$ and 
$\wt B_{K,\dl_2,T}^{(\phi_0, \phi_1)}$
are defined below.
A required modification is straightforward
and thus we omit details.

Given $K > 0$ and $T > 0$,  
define the set $\wt B_{K,T}^{(\phi_0, \phi_1)}$ by setting
\begin{align*}
\wt B_{K,T}^{(\phi_0, \phi_1)} 
& = \big\{ (\vec\ub,  \vec\ub') \in \wt \Yc^{\al, \g, s}_{\vXX, T}: \,
 \vec\ub_0 = \vec\ub'_0= (\phi_0, \phi_1),
\,  \|(\vec\ub,  \vec\ub') \|_{\wt \Yc^{\al,\g,s}_{\vXX, T}} \le K
   \big\},
\end{align*}

\noi
endowed with the following metric: 
\begin{align}
\wt d \Big(\big(\vec\ub^{(1)},  (\vec\ub^{(1)})'\big), 
\big(\vec\ub^{(2)},  (\vec\ub^{(2)})'\big)\Big)
= 
\big\|\big(\vec\ub^{(1)} - \vec\ub^{(2)}, 
(\vec\ub^{(1)})' - (\vec\ub^{(2)})'\big) \big\|_{\wt \Yc^{\al,\g,s}_{\vXX, T}}, 
\label{SY6}
\end{align}

\noi
where 
the $\wt \Yc^{\al, \g, s}_{\vXX, T}$-norm is as  in \eqref{SY2}.

When $s = s_\text{scaling}$, the nonlinear estimate
in Proposition  \ref{PROP:nonlin2}
does not provide a small power of $T$ to run a contraction argument, 
and thus we need to make use of smallness of the high frequency part
(as in the proofs of Propositions \ref{PROP:LWP1}
and \ref{PROP:LWP2}).
Given small $\dl_1 > 0$, let 
 $N = N((\phi_0,\phi_1), \dl_1) \in \N$ be as in \eqref{RRBK2}.
Then, given $K > 0$, small $\dl_2>0$, and $T > 0$,  
define the set $\wt B_{K,\dl_2,T}^{(\phi_0, \phi_1)} = \wt B_{K,\dl_2,T}^{(\phi_0, \phi_1)}(N)$ by setting
\begin{align*}
\wt B_{K,\dl_2,T}^{(\phi_0, \phi_1)} 
 = \big\{ (\vec\ub,  \vec\ub') \in \wt \Yc^{\al, \g, s}_{\vXX, T}:
   \,&  \vec\ub_0 = \vec\ub'_0= (\phi_0, \phi_1), \\
&   \|(\vec\ub,  \vec\ub') \|_{\wt \Yc^{\al,\g,s}_{\vXX, T}} \le K,\ 
 \| \P_N^\perp \vec\ub \|_{{V}^{\frac1\al}_{T} \H^{s}_x} \le \dl_2
   \big\},
\end{align*}

\noi
endowed with the metric $\wt d$ in \eqref{SY6}.

\end{remark}

\appendix

\section{Pathwise solutions coincide with  Ito solutions}
\label{SEC:ito}

In this section, 
we focus on the rough case (= white-in-time case with $\be = \frac 12$)
and show that our pathwise solutions 
constructed in Theorem \ref{THM:2}
agree with classical Ito solutions.
In Subsection \ref{SUBSEC:ito1}, 
we first go over the Ito solution theory
for SNLW \eqref{SNLW} on $\T^d$
for readers' convenience, as it 
does not seem to be written  in the literature.
We then compare our pathwise solutions
with Ito solutions
in Subsection \ref{SUBSEC:ito2}.

\subsection{Review of the Ito solution theory}
\label{SUBSEC:ito1}

In this subsection, we briefly go over the Ito solution theory 
for SNLW \eqref{SNLW} on $\T^d$ (more precisely, 
its Duhamel formulation \eqref{mild0}).
Here, the Ito solution theory refers
to a solution theory
where we interpret the stochastic convolution $\Psi(u)$
in \eqref{psi1}
as an Ito integral.

\begin{proposition}\label{PROP:LWPx}
Under the assumption of Theorem \ref{THM:2}
\textup{(}see in particular \eqref{reg1} and \eqref{sigma0}\textup{)}, 
given  $\Phi \in \HS(L^2(\T^d); H^\s(\T^d))$, 
satisfying~\eqref{phi1}, 
 SNLW \eqref{SNLW}
 with a multiplicative white-in-time noise
is  locally well-posed in $\H^s(\T^d)$
in the Ito sense. 

\end{proposition}

Here, we stated Proposition 
\ref{PROP:LWPx} for SNLW on $\T^d$ but it also holds
on $\R^d$ (with the same proof).
See also Remark \ref{REM:white1}.

\begin{proof}[Proof of Proposition \ref{PROP:LWPx}]

Proposition \ref{PROP:LWPx} follows
from a standard contraction argument in $L^2_\text{ad}(\Om; X^s(T))$
(see~\eqref{it1}, \eqref{it3}, and \eqref{it4} below)
of processes adapted to the filtration generated by $W^\frac 12$
in \eqref{W0}, 
using 
the truncation method as in \cite{DD1, DD2} for stochastic NLS
with multiplicative noises, 
and we only indicate necessary ingredients.
Moreover, we restrict our attention to $d \ge 2$ in the following,
since the necessary nonlinear estimate for $d = 1$  follows more easily
without using the Strichartz estimate
(just by using~\eqref{1d2}).

We first establish   nonlinear estimates.

\smallskip

\noi
$\bul$ {\bf Case 1:}  $s_\text{crit} < 1$, 
where $s_\text{crit}$ is as in \eqref{crit1}.
\\
\indent
Let $s > 0$  be as in \eqref{reg1}.
We first consider the case $s < 1$.
Given $T > 0$, define the solution space $X^s(T)$
(= the Strichartz space) for $\vec u = (u, \dt u)$ via the norm:
\begin{align}
\|\vec u\|_{X^s(T)} = \|(u, \dt u) \|_{C_T\H^s_x}
+ \| u \|_{L^q_T L^r_x}, 
\label{it1}
\end{align}

\noi
where $(q, r)$ is the $s$-admissible pair
constructed in 
Proposition \ref{PROP:Str2}.
Then, from 
\eqref{Str3} in 
Lemma \ref{LEM:Slin} and \eqref{HK2}, we have
\begin{align}
\big\| \big(\If(u^k), \dt \If(u^k)\big) \big\|_{X^s(T)} \les  T^\ta \| \vec u\|_{X^s(T)}^k 
\label{it2}
\end{align}

\noi
for some $\ta \ge 0$, where $\If(F)$ is as in \eqref{If1}
and $\ta = 0$ if and only if $s = s_\text{crit} = s_\text{scaling}$.

Next, we consider the case $s \ge 1$.
Fix $s_0 \in (s_\text{crit}, 1)$
as in the second part of
the proof of Proposition \ref{PROP:nonlin2}\,(a).
In this case, we
define  $X^s(T)$
 via the norm:
\begin{align}
\|\vec u\|_{X^s(T)} = \|(u, \dt u) \|_{C_T\H^s_x}
+ \| u \|_{L^q_T W^{s- s_0, r}_x}, 
\label{it3}
\end{align}

\noi
where $(q, r)$ is the $s_0$-admissible pair
constructed in 
Proposition \ref{PROP:Str2}.
Then, 
the bound \eqref{it2} follows
 from 
\eqref{Str3} in 
Lemma \ref{LEM:Slin}, 
the fractional Leibniz rule (Lemma~\ref{LEM:leibniz}), 
Sobolev's inequality, and 
H\"older's inequality in time with \eqref{nonl3c};
see \eqref{nonl2ax}.

\smallskip

\noi
$\bul$ {\bf Case 2:}  $s_\text{crit} \ge  1$.
\\
\indent
In this case, we
define  $X^s(T)$
 via the norm:
\begin{align}
\|\vec u\|_{X^s(T)} = \|(u, \dt u) \|_{C_T\H^s_x}
+ \| u \|_{L^{q_1}_T W^{s- 1, r_1}_x}, 
\label{it4}
\end{align}

\noi
where $(q_1, r_1)$ is the $1$-admissible pair
from 
the proof of Proposition \ref{PROP:nonlin2}\,(b);
see \eqref{nonl3c}.
Then, 
the bound \eqref{it2} follows from 
\eqref{Str3} in 
Lemma \ref{LEM:Slin}, 
the fractional Leibniz rule (Lemma~\ref{LEM:leibniz}), 
and 
\eqref{X1};
see \eqref{nonl2b}.

\smallskip

Let $\eta \in C^\infty(\R_+; [0,1])$
be a smooth non-increasing function such that $\eta \equiv 1$ on $[0, 1]$
and $\eta \equiv 0 $ on $[2, \infty)$.
In the scaling subcritical case $s > s_\text{scaling}$, 
we 
consider the following truncated Duhamel formulation
(with  $R > 
 \|(\phi_0, \phi_1) \|_{\H^s_x}$):
\begin{align}
\begin{split}
u_R(t) & = \dt S(t) \phi_0 + S(t) \phi_1\\
& \quad  + \int_0^t S(t-t') 
\eta\bigg(\frac {\|\vec u_R\|_{X^s(t')}}{R}\bigg)
u_R^k(t') \, dt' + \Psi(u_R)(t)
\end{split}
\label{it5}
\end{align}

\noi
and the corresponding equation for $\dt u$, 
where $\Psi(u)$ is the stochastic convolution defined  in~\eqref{psi1}.
Then, from the nonlinear estimate \eqref{it2} (with $\ta > 0$), 
Lemma \ref{LEM:stoconv1} below
on the stochastic convolution $\Psi(u)$
in \eqref{psi1}
interpreted as an Ito integral, 
and
\cite[Lemma 3.3]{DD1} on the truncation operator, 
a standard contraction argument 
shows that \eqref{it5}
is locally well-posed in the Ito sense
with $\vec u_R = (u_R, \dt u_R) \in L^2_\text{ad}(\Om; X^s(T))$
for some (deterministic) $T  = T(R)> 0$.
Moreover, in view of the truncation, 
we see that \eqref{it5} is globally well-posed.

By setting a stopping time $\tau_R$ by 
\[ \tau_R = \inf\big\{t > 0: \|\vec u_R\|_{X^s(t)} \ge R\big\}, \]

\noi
we see that $u_R$ is an Ito solution to 
the (untruncated) Duhamel formulation \eqref{mild0}
on the time interval $[0, \tau_R]$.

In the scaling critical case $s =  s_\text{scaling}$, 
we have $\ta = 0$ in \eqref{it2}.
Define 
$X_\text{Str}(T)$
by setting
\begin{align*}
\| u \|_{X_\text{Str}(T)}
= \begin{cases}
\| u \|_{L^q_T L^r_x},  & \text{if $s_\text{crit} < 1$ and $s < 1$}, \\ 
\| u \|_{L^q_T W^{s- s_0, r}_x}, 
& \text{if $s_\text{crit} < 1 \le s$}, \\
\| u \|_{L^{q_1}_T W^{s- 1, r_1}_x}, & \text{if }s_\text{crit} \ge  1, 
\end{cases}
\end{align*}

\noi
where the parameters are as in \eqref{it1}, \eqref{it3}, and \eqref{it4}.
In this case, we consider the
following
 truncated Duhamel formulation:\footnote{Recall from Remark \ref{REM:foc}
 that we only consider the focusing case (with  $-u^k$ in \eqref{SNLW}.}
\begin{align}
\begin{split}
u_{\dl}(t) 
& = \dt S(t) \phi_0 + S(t) \phi_1\\
& \quad  + \int_0^t S(t-t') 
\eta\bigg(\frac {\|u_\dl\|_{X_\text{Str}(t')}}{\dl}\bigg)
u_{\dl}^k(t') \, dt' +  \Psi(u_{\dl})(t)
\end{split}
\label{it5b}
\end{align}

\noi
for small $\dl > 0$.
Proceeding as above
with the nonlinear estimate \eqref{it2} (where we replace
the $X^s(T)$-norm on the right-hand side by 
the $X_\text{Str}(T)$-norm), 
a contraction 
argument shows that 
\eqref{it5b} is locally well-posed
in the Ito sense
with 
$u_{\dl} \in L^2_\text{ad}(\Om; X_\text{Str}(T))$
by choosing $\dl > 0$ sufficiently small, 
where the local existence time
$T>0$ is chosen such that 
$\|\dt S(t) \phi_0 + S(t) \phi_1\|_{X_\text{Str}(T)}$
is sufficiently small.
We then use \eqref{it5b} to show that 
$(u_{\dl}, \dt u_{\dl}) \in L^2_\text{ad}(\Om; X^s(T))$.

For 
 small $\dl > 0$, 
 set  a stopping time $\tau_{\dl}$ by 
\[ \tau_{\dl} = \inf\big\{t \in (0, T]: 
\|u_\dl\|_{X_\text{Str}(t)} \ge \dl
\big\}, \]

\noi
we see that $u_{\dl}$ is an Ito solution to 
the (untruncated) Duhamel formulation \eqref{mild0}
on the time interval $[0, \tau_{\dl}]$.
\end{proof}

The next lemma establishes
basic bounds
for the stochastic convolution.
See Appendix~\ref{SEC:D}
for the
 Burkholder-Davis-Gundy inequality
 which plays a crucial role in the proof.

\begin{lemma}\label{LEM:stoconv1}

Let $s, \s \in \R$ satisfy \eqref{sigma0}{\rm :}
\begin{align}
\s \ge \max( - s, s-1).
\label{reg2}
\end{align}

\noi
Then, given any finite $p\ge2$, we have
\begin{align}
\big\|\| \vec \Psi(u)\|_{C_T \H^s_x} \big\|_{L^p(\Om)}
\les_p
\| \Phi\|_{\HS(L^2; H^\s)} \big\|\| u \|_{L^2_T H^s_x}\big\|_{L^p(\Om)}
\label{it6a}
\end{align}

\noi
for any adapted process $u \in L^p(\Om; L^2([0, T]; H^s(\T^d))$, 
where 
$\vec \Psi(u) = (\Psi(u), \dt \Psi(u))$ with $\Psi(u)$
 as in \eqref{psi1}.

In addition, given $0 < s_0\le 1$, let $(q, r)$ be an $s_0$-admissible pair
in the sense of 
Definition~\ref{DEF:admiss}.
\noi
Then, given any finite $p\ge2$, we have
\begin{align}
\big\|\| \Psi(u)\|_{L^q_T W^{s- s_0, r}_x} \big\|_{L^p(\Om)}
\les_p
\| \Phi\|_{\HS(L^2; H^\s)} \big\|\| u \|_{L^2_T H^s_x}\big\|_{L^p(\Om)}
\label{it6b}
\end{align}

\noi
for any adapted process $u \in L^p(\Om; L^2([0, T]; H^s(\T^d))$.

\end{lemma}

\begin{proof}

Given any finite $p \ge 2$, 
it follows from 
the Burkholder-Davis-Gundy inequality (Lemma~\ref{LEM:BDG})
 with \eqref{psi1} that 
\begin{align}
\E\big[\| \vec \Psi(u)\|_{C_T \H^s_x}^p\big]
\les_p
\sum_{\ta \in \{-1, 1\}}
 \E\Bigg[\bigg(\int_0^T\| S_{\ta} (-t')M_{u(t')} \Phi \|_{\HS(L^2; H^s)}^2dt'\bigg)^\frac p2\Bigg], 
\label{it7}
\end{align}

\noi
where $S_{\pm} (t) = e^{\pm i  t\jb{\nb}}{\jb{\nb}}^{-1}$
and $M_f$ denotes the multiplication operator by $f$.
From 
\eqref{phi1}, we have 
\begin{align}
\begin{split}
\| S_{\ta} (-t') & M_{u(t')} \Phi \|_{\HS(L^2; H^s)}^2
 = \sum_{n_1 \in \Z^d}
\| S_{\ta} (-t')M_{u(t')} \Phi (e_{n_1})\|_{H^s}^2\\
& \sim  \sum_{n_1 \in \Z^d}
\jb{n_1}^{2\s} |\phi_{n_1}|^2
\sum_{n \in \Z^d}
\jb{n}^{2s-2} \jb{n_1}^{-2\s}|\ft u(t', n - n_1)|^2\\
& \les 
\| \Phi\|_{\HS(L^2; H^\s)}^2 \| u(t') \|_{H^s_x}^2, 
\end{split}
\label{it8}
\end{align}

\noi
where, in the last step, 
we used
\begin{align}
\jb{n}^{2s-2} \jb{n_1}^{-2\s}
\les \jb{n - n_1}^{2s}
\label{it9}
\end{align}

\noi
 under the condition \eqref{reg2}.
Hence, the first bound \eqref{it6a}
follows from \eqref{it7} and \eqref{it8}.

Next, we prove \eqref{it6b}.
We  adapt the proof of  \cite[Lemma 2.1\,(ii)]{OPW}
(see also \cite{BLL24}), 
using the Strichartz estimates, to the current multiplicative noise setting.
From \eqref{psi1}
and the Burkholder-Davis-Gundy inequality 
(Lemma~\ref{LEM:BDG}), we have 
\begin{align}
\begin{split}
& \big\|\| \Psi(u)\|_{L^q_T W^{s- s_0, r}_x}\big\|_{L^p(\Om)}
= \big\| \|\jb{\nb}^{s- s_0} \Psi(u)\|_{L^q_T L^r_x}\big\|_{L^p(\Om)}\\
& \quad 
=  \Bigg\|\bigg\|\int_0^T
\ind_{[0, t]}(t')\cdot \jb{\nb}^{s- s_0}S(t-t') \big[ u(t') \Phi d  W^\frac 12 (t') \big]
\bigg \|_{L^q_t([0, T];   L^r_x)}\Bigg\|_{L^p(\Om)}\\
& \quad 
\les_p 
 \Big\|
 \| M_{\ind_{[0, t]}(t')}\jb{\nb}^{s- s_0}S(t-t') M_{u(t')} \Phi 
\|_{\g(L^2;  L^q_t([0, T]; L^r_x))} 
 \|_{  L^2_{t'}([0, T])}\Big\|_{L^p(\Om)}.
\end{split}
\label{it10}
\end{align}

\noi
By applying
Lemma \ref{LEM:G1}, 
Minkowski's integral inequality, 
and the Strichartz estimates (\eqref{Str1a} in Lemma~\ref{LEM:Slin})
and proceeding as in \eqref{it8}
with \eqref{it9} under \eqref{reg2}, we have
\begin{align}
\begin{split}
&  \| M_{\ind_{[0, t]}(t')}\jb{\nb}^{s- s_0}S(t-t') M_{u(t')} \Phi 
\|_{\g(L^2;  L^q_t([0, T]; L^r))} \\
& \quad \les 
\Bigg\|
\bigg\|
\sum_{n \in \Z^d}
e_n(x)  \frac{\sin((t-t')\jb{n})}{\jb{n}^{1-s+s_0}}
\ft u(t', n- n_1) \phi_{n_1} \bigg\|_{\l^2_{n_1} }
\Bigg\|_{L^q_t([0, T];  L^r_x)}\\
& \quad
\le
\Bigg\|
\bigg\|
\sum_{n \in \Z^d}
e_n(x)  \frac{\sin((t-t')\jb{n})}{\jb{n}^{1-s+s_0}}
\ft u(t', n- n_1) \phi_{n_1} 
\bigg\|_{L^q_t([0, T];  L^r_x)}
\Bigg\|_{\l^2_{n_1} }\\
& \quad \les
\|\jb{n}^{s-1}
\ft u(t', n- n_1) \phi_{n_1} \|_{\l^2_{n, n_1} }\\
& \quad \les \| \Phi\|_{\HS(L^2; H^\s)} \| u(t') \|_{H^s_x}.
\end{split}
\label{it11}
\end{align}

\noi
Hence, the second bound \eqref{it6b}
follows from \eqref{it10} and \eqref{it11}.
\end{proof}

\begin{remark}\label{REM:white1}\rm
Consider the case $\Phi = \jb{\nb}^\ld$ for $\ld \ge 0$.
In this case, instead of \eqref{it8}, we have
\begin{align}
\begin{split}
\| S_{\ta} (-t')  M_{u(t')} \Phi \|_{\HS(L^2; H^s)}^2
& \sim 
\sum_{n \in \Z^d}
\jb{n}^{2s-2}
 \sum_{n_1 \in \Z^d}
\jb{n_1}^{2\ld}|\ft u(t', n - n_1)|^2\\
& \les 
\| u(t') \|_{H^s_x}^2, 
\end{split}
\label{it12}
\end{align}

\noi
where the last inequality holds, provided that 
\begin{align}
 \ld \le s < 1 - \frac d2 - \ld. 
\label{it13}
\end{align}

\noi
Under the assumption  $\ld \ge 0$, 
\eqref{it12} holds only for $d = 1$, 
in which case \eqref{it6a} extends to 
$\Phi = \jb{\dx}^\ld$ 
for $0 \le \ld < \frac 14$
and $s$ satisfying \eqref{it13}.

Let $d = 1$, 
and set $X^s(T) = C([0, T]; \H^s(\T))$.
Then, from 
\eqref{1d2}, 
we see that 
the nonlinear estimate \eqref{it2}
holds for $s\ge s_{\rm sob}$, 
where $s_{\rm sob} = \frac 12 - \frac 1k$
is as in 
\eqref{crit3}.
Hence, from \eqref{it13} with~\eqref{crit3}, we conclude
that, 
when $\Phi = \jb{\dx}^\ld$,  the one-dimensional SNLW \eqref{SNLW} 
 with a multiplicative white-in-time noise
is locally well-posed in $\H^{s_\text{sob}}(\T)$
in the  Ito sense, 
provided that\footnote{When $d = 1$, 
the $\ld < 0$ case was already treated in Proposition \ref{PROP:LWPx}.} 
\begin{align*}
\ld \le \min \bigg(\frac 14 - , \frac 1k -, \frac 12 - \frac 1k \bigg).
\end{align*}

\noi
In particular, 
this includes the case of a space-time white noise 
(namely, $\Phi = \Id$).

\end{remark}

\begin{remark}\label{REM:system1}\rm

As mentioned in 
Remark \ref{REM:rough1}, 
our pathwise solution theory in the rough case
requires the decomposition \eqref{exp1}
if \eqref{bad1} hold.
In order to compare our pathwise solutions
with Ito solutions under \eqref{bad1}, 
we also need to proceed
with the decomposition 
$u = v+ w$ in the construction of Ito solution.
and 
study the following system:
\begin{align}
\begin{split}
v & = \Psi(v+w),  \\
w(t)& 
 = \dt S(t) \phi_0 + S(t)  \phi_1 
+ \int_0^t S(t-t') \big(v(t') + w(t')\big)^k \, dt'.
\end{split}
\label{Ito2}
\end{align}

\noi
Note that if $(v, w)$ is an Ito solution to \eqref{Ito2}, 
then $u= v+w$ satisfies
\eqref{mild0}
in the Ito sense.


Fix $s > 0$ and small $\eps > 0$.
Depending on the following three cases:
\[ \text{(i) ~$s_\text{crit} < 1$ and $s < 1$, \quad
(ii) ~$s_\text{crit} < 1 \le s$,
\quad and \quad
 (iii) $s_\text{crit} \ge 1$}, \]
let the relevant parameters $q$, $r$, $s_0$, $q_1$, and $r_1$
be 
are as in \eqref{it1}, \eqref{it3}, and~\eqref{it4}.

\smallskip

\begin{itemize}
\item
When (i)  holds, 
we set $q_*$ by $\frac 1{q_*} = \frac 1q - \eps$
such that 
$(q_*, r)$ is $(s + \eps)$-admissible. 

\smallskip
\item 
When (ii) holds, 
we set $q_*$ by $\frac 1{q_*} = \frac 1q - \eps$
such that $(q_*, r)$ is $(s_0 + \eps)$-admissible
(recall that 
$s_0 \in (s_\text{crit}, 1)$).

\smallskip
\item 
When (iii) holds, 
we set $q_*$ by $\frac 1{q_*} = \frac 1{q_1} - \eps$
such that $(q_*, r)$ is $(1 + \eps)$-admissible
(as in the proof of Proposition \ref{PROP:nonlin3}).

\end{itemize}

\smallskip

\noi
Given $T > 0$, 
define 
$Y^{s, \eps}(T)$
by setting
\begin{align*}
\|\vec u \|_{Y^{s, \eps}(T)}
= \begin{cases}
\|\vec u  \|_{C_T\H^{s+\eps}_x}
+ 
\| u \|_{L^{q_*}_T L^r_x},  & \text{if $s_\text{crit} < 1$ and $s < 1$}, \\ 
\|\vec u  \|_{C_T\H^{s+\eps}_x}
+ \| u \|_{L^{q_*}_T W^{s- s_0, r}_x}, 
& \text{if $s_\text{crit} < 1 \le s$}, \\
\|\vec u  \|_{C_T\H^{s+\eps}_x}
+ \| u \|_{L^{q_*}_T W^{s- 1, r_1}_x}, & \text{if }s_\text{crit} \ge  1.
\end{cases}
\end{align*}

\noi
Then, define $Z^{s, \eps}(T)$ via the norm:
\begin{align*}
\| (\vec v, \vec w)\|_{Z^{s, \eps}(T)}
= \| \vec v\|_{Y^{s, \eps}(T)}
+ \|  \vec w\|_{X^{s}(T)}, 
\end{align*}

\noi
where the $X^s(T)$-norm is as in \eqref{it1}, \eqref{it3}, and~\eqref{it4}.
With \eqref{it2} and suitable modifications as in the proof of Proposition \ref{PROP:nonlin3}
\begin{align}
\begin{split}
& \big\| \big(\If((v+w)^k), \dt \If((v+w)^k)\big) \big\|_{X^s(T)} \\
& \quad \les  T^\ta \| \vec w\|_{X^s(T)}^k 
+ T^\eps \Big(\big( \| \vec v\|_{Y^{s, \eps}(T)} +\| \vec w\|_{X^s(T)}\big)^k 
- \| \vec w\|_{X^s(T)}^k \Big)
\end{split}
\label{Ito2b}
\end{align}

\noi
for some $\ta \ge 0$, where $\If(F)$ is as in \eqref{If1}
and $\ta = 0$ if and only if $s = s_\text{crit} = s_\text{scaling}$.

Consider the following truncated system
(with  $R > 
 \|(\phi_0, \phi_1) \|_{\H^s_x}$):
\begin{align}
\begin{split}
v_R & = \Psi(v_R+w_R),  \\
w_R(t)& 
 = \dt S(t) \phi_0 + S(t)  \phi_1 \\
& \quad + \int_0^t S(t-t')
\eta\bigg(\frac {\|(\vec v_R, \vec w_R)\|_{Z^{s, \eps}(t')}}{R}\bigg)
 \big(v_R(t') + w_R(t')\big)^k \, dt'
\end{split}
\label{Ito3}
\end{align}

\noi
and the corresponding equations for $\dt v$ and $\dt w$. 
When $s > s_\text{scaling}$, 
it follows from 
a standard contraction argument with
the nonlinear estimate~\eqref{Ito2b} (with $\ta > 0$), 
Lemma \ref{LEM:stoconv1}, 
and
\cite[Lemma~3.3]{DD1} on the truncation operator
that 
the truncated system \eqref{Ito3}
is locally well-posed in the Ito sense
with $(\vec v_R, \vec w_R) \in L^2_\text{ad}(\Om; Z^{s, \eps}(T))$
for some (deterministic) $T  = T(R)> 0$.
Then, 
by setting a stopping time $\tau_R$ by 
\[ \tau_R = \inf\big\{t\in ( 0, T]: 
\|(\vec v_R, \vec w_R)\|_{Z^{s, \eps}(t)}
 \ge R\big\}, \]

\noi
we see that $(v_R, w_R)$ is an Ito solution to 
the (untruncated) system \eqref{Ito2}
on the time interval $[0, \tau_R]$.
Consequently, 
$u_R:= v_R + w_R$
 is an Ito solution to 
the (untruncated) Duhamel formulation \eqref{mild0}
on the time interval $[0, \tau_R]$.

A required modification for $s = s_\text{scaling}$
is straightforward as in the proof of Proposition~\ref{PROP:LWPx}, 
and thus we omit details.
\end{remark}

\subsection{Agreement with Ito solutions}
\label{SUBSEC:ito2}

The following proposition shows that our pathwise solutions
in the rough case (with $\be = \frac 12$)
agree with Ito solutions.

Let us introduce some notation.
Fix $s \ge  0$ and 
 $\s \in \R$, satisfying \eqref{reg1}  and \eqref{sigma0}.
Fix   $(\phi_0, \phi_1) \in \H^s(\T^d)$ and  
$\Phi \in \HS(L^2(\T^d); H^\s(\T^d))$, 
satisfying \eqref{phi1}, and let
$\vXX$ and $\vbbX$
be random drivers as in~\eqref{introX}
and \eqref{bX0}.
Given $0 < t \le 1$, let 
\begin{align}
\label{AC2}
\| \Xi \|_{\wt\Xc^{\al,s, \eps}(t)}
=
\| (\phi_0, \phi_1) \|_{\H^s}
+ \|\vXX\|_{C^{\al}_{2,t} \LOP(\H^s;\H^{s+\eps})} + \| \vbbX\|_{C^{2\al }_{2,t} \LOP(\H^s;\H^{s+\eps})}.
\end{align}

\noi
Namely, \eqref{AC2} is a local-in-time version of  \eqref{Xtilde}.
Then, 
given $K \gg \| (\phi_0, \phi_1) \|_{\H^s}$, 
 define a stopping time $\tau_1$ by 
\begin{align*}
\tau_1 = \inf\big\{ t > 0: 
6\| \Xi \|_{\wt\Xc^{\al,s, \eps}(t)} > K \big\}.
\end{align*}

\noi
See \eqref{RRr}.
Note that $\tau_1$ is almost surely positive.
Furthermore, 
  a straightforward modification of Proposition \ref{PROP:LWP2}
  (also see the proof of Theorem \ref{THM:2}), 
allows us to construct a solution $\vec u$ 
to~\eqref{SNLW} on the time interval $[0, \tau_1]$
in the pathwise manner.

\begin{proposition}
\label{PROP:ito}

Let  $d \ge 1$ and $\be = \frac 12$. 
  Given  an integer  $k \geq 2$, 
suppose that $s\ge 0$ and $\s \in \R$
satisfy
 \eqref{reg1} 
and \eqref{sigma0}, respectively.
Fix $(\phi_0, \phi_1) \in \H^s(\T^d)$
and 
  $\Phi \in \HS(L^2(\T^d); H^\s(\T^d))$, 
satisfying~\eqref{phi1}, 
and let 

\smallskip

\begin{itemize}
\item
$\vec u$
be the local-in-time pathwise  solution 
to \eqref{SNLW}
on the time interval $[0, \tau_1]$
from the discussion above, and

\smallskip
\item
 $\vec u_\ito$
 be the local-in-time Ito solution 
to \eqref{SNLW}
 on a time interval $[0, \tau_2]$, 
 constructed in Proposition~\ref{PROP:LWPx}, 
 where
 \noi
 $\tau_2$ is an 
almost surely positive stopping time.

\end{itemize}

\smallskip

\noi
Then, we have 
\begin{align}
\vec u = \vec u_\ito   \quad \text{in } \ C([0, \tau_1\wedge \tau_2]; \H^s(\T^d)), 
\label{ag1}
\end{align}

\noi
almost surely.

\end{proposition}

A key ingredient is the following lemma;
see
\cite[Proposition~5.1]{FH20}
in the scalar case.

\begin{lemma}
\label{LEM:Ito}
Let $d\ge 1$,  $\be = \frac 12$, 
$ 0 <  \al < \frac12 < \g < 1$ such that
$  (\al + \g) \wedge(3\al) >1$. 
Let $s \ge 0$, 
and  $\s \in \R$ satisfy
\eqref{sigma0}.
Given  
   $\Phi \in \HS(L^2(\T^d); H^\s(\T^d))$, 
satisfying~\eqref{phi1}, 
let $(\vXX, \vbbX)$ be the rough driver as in \eqref{introX}
and \eqref{bX0}, 
constructed in Proposition
  \ref{PROP:driver2}.

Let $ (\vec\ub, \vec\ub') \in 
\V^{\frac1\al}_\tau  \H^s(\T^d) \times \V^{\frac1\al}_{\tau} \H^s(\T^d)$
be a {\rm (}random{\rm )} process, 
adapted to the filtration
 generated by $W^\frac 12$ in \eqref{W0}, 
where $\tau$ denotes an almost surely positive stopping time.
Suppose that 
 $\vec \ub$ is almost surely controlled by $\vXX$
 on the time interval
 with the Gubinelli derivative $\vec \ub'$
such that 
$ (\vec\ub, \vec\ub')\in \CRP^{ \al, \g, s}_{\vXX, \tau}$, 
in the sense of Definition~\ref{DEF:RP}\,(ii).
Then, we have 
\begin{align}
\vec \pPsi (\ind_{[0, \tau]}\cdot  \vec \ub)
=  \I^{\vXX, \vbbX}(\vec \ub)
\quad 
\text{in \ $C([0, \tau]; \H^s(\T^d))$}, 
\label{ag2}
\end{align}

\noi
 almost surely, 
where 
$\vec{\pPsi}(\ind_{[0, \tau]}\cdot  \vec \ub) $ is the stochastic term
in \eqref{psi2} {\rm(}with $\b = \frac 12${\rm)} interpreted as an Ito integral
and 
$\I^{\vXX, \vbbX}(\vec \ub)$ denotes the rough integral of $\vec \ub$
with the rough driver $(\vXX, \vbbX)$.

\end{lemma}

We first prove Proposition \ref{PROP:ito}
by assuming Lemma \ref{LEM:Ito}.

\begin{proof}[Proof of Proposition \ref{PROP:ito}]
For simplicity of presentation, 
we consider the case when \eqref{bad1} does {\it not} hold;
see Remark \ref{REM:system3}.
Furthermore, we only consider the case $s > s_\text{scaling}$.
The case 
$s = s_\text{scaling}$ follows from a minor modification.

Under the current assumption,  Theorem \ref{THM:2} follows
from directly studying the system~\eqref{duha2} and~\eqref{duha3}
for 
the interaction representation $\vec \ub = \S(-t)  \vec u(t)$, 
using Proposition \ref{PROP:nonlin2}.
See Remarks~\ref{REM:rough1}
and  \ref{REM:system2}.
Namely, there exists $\Si_1 \subset \Om$ with $\PP(\Si_1) = 1$
such that,\footnote{For simplicity of notation, 
we often drop the $\o$-dependence in the following.} for each $\o \in \Si_1$, 
 $\vec \ub = \vec \ub(\o)$ is a solution to 
\begin{align}
\label{ag3}
\vec\ub(t)
=
\vec \phi_0
+
\vec{\NN}({\vec \ub})(t)
+
\I^{\vXX, \vbbX} (\vec\ub) (t)
\end{align}

\noi
on the time interval $[0, \tau_1]$, 
coupled with 
\begin{align}
(\updl \vec\ub)_{t,r} = \vXX_{t,r}( \vec\ub_r )+ \vec{R}^{\vXX, \vec\ub}_{t,r}
\notag 
\end{align}

\noi
for some remainder term $ \vec{R}^{\vXX, \vec\ub}$, 
where 
$\vec \phi_0 = (\phi_0, \phi_1)$, 
$\vec{\NN}({\vec \ub}) $ is the nonlinear term  in~\eqref{bd2}, 
and 
$\I^{\vXX, \vbbX} (\vec\ub)$
 denotes the rough integral of $\vec \ub$
with the rough driver $(\vXX, \vbbX)$.
Here, we used the fact that $\vec \ub' = \vec \ub$;
see \eqref{SY3}. 
Moreover, we have
\begin{align} 
(\vec \ub,  \vec\ub') \in \wt  \Yc^{\al,\g,s}_{\vXX, \tau_1} 
\subset
 \V^{\frac1\al}_{\tau_1} \H^{s}(\T^d) 
 \times \V^{\frac1\al}_{\tau_1} \H^s(\T^d)
\notag 
\end{align}

\noi
for $ 0 <  \al < \frac12 < \g < 1$
such that 
both $\al$ and $\g$ are sufficiently close to $\frac 12$, satisfying \eqref{RWP0}.
Here, 
$ \wt  \Yc^{\al,\g,s}_{\vXX, \tau_1} $ is as in \eqref{SY1}.

Let $\vec \ub_\ito = \S(-t)  \vec u_\ito(t)$
be the interaction representation of the Ito solution 
$\vec u_\ito$.
Then, 
there exists $\Si_2 \subset \Om$ with $\PP(\Si_2) = 1$
such that, for each $\o \in \Si_2$, 
$\vec \ub_\ito = \vec \ub_\ito(\o)$ is a solution to 
\begin{align}
\label{ag6}
\vec\ub_\ito(t)
=
\vec \phi_0
+
\vec{\NN}({\vec \ub_\ito })(t)
+
\vec{\pPsi} ( \vec\ub_\ito) (t), 
\end{align}

\noi
on the time interval $[0, \tau_2]$, 
where 
$\vec{\pPsi} ( \vec\ub_\ito)$ is the stochastic term
in \eqref{psi2} (with $\b = \frac 12$) interpreted as an Ito integral.
Under the assumption $s > s_\text{scaling}$, 
it follows from the proof of Proposition~\ref{PROP:LWPx} 
with a standard continuity argument that 
(pathwise) uniqueness of the solution $\vec \ub_\ito$  to \eqref{ag6}
holds
in the entire class:
\begin{align}
\S(-t) X^s(\tau_2) 
= \big\{\vec \vb : \, \S(t) \vec \vb \in X^s(\tau_2)\big\}, 
\label{ag7}
\end{align}

\noi
where
$X^s(\tau_2)$ is as in \eqref{it1}, \eqref{it3}, and \eqref{it4}.

The pathwise solution $\vec \ub$ 
was constructed via  a contraction argument.
Namely, $\vec \ub$  was constructed
as a limit of the Picard iterates.
Since each Picard iterate is 
adapted, 
we see that $\vec \ub$ is also adapted.
Then, from Lemma \ref{LEM:Ito}
with Lemma \ref{LEM:stoconv1}, 
there exists $\Si_3 \subset \Om$ with $\PP(\Si_3) = 1$
such that, for each $\o \in \Si_3$, we have
\begin{align}
\vec{\pPsi} (\ind_{[0, \tau_1]} \cdot \vec\ub)(\o)   = \I^{\vXX, \vbbX} (\vec\ub)(\o) 
\label{ag8}
\end{align}

\noi
in 
$\S(-t) X^s(\tau_1)$, 
where the left-hand side is interpreted as an Ito integral.

Let $\Si = \Si_1 \cap \Si_2 \cap \Si_3$.  Note that $\PP(\Si) = 1$.
Fix $\o \in \Si$.
Then, it follows from \eqref{ag3} and~\eqref{ag8}, 
that 
the pathwise solution $\vec \ub$ is also a solution to the following Ito formulation:
\begin{align}
\begin{split}
\vec\ub(t)
& =
\vec \phi_0
+
\vec{\NN}({\vec \ub })(t)
+
\vec{\pPsi} (\ind_{[0, \tau_1]}\cdot \vec\ub) (t)\\
& =
\vec \phi_0
+
\vec{\NN}({\vec \ub })(t)
+
\vec{\pPsi} ( \vec\ub) (t)
\end{split}
\notag 
\end{align}

\noi
on the time interval $[0, \tau_1]$. 
Hence, letting $\tau = \tau_1\wedge \tau_2$, 
we see that both $\vec \ub$
and $\vec \ub_\ito$
satisfy 
the same Ito formulation \eqref{ag6}
on the time interval $[0, \tau]$.

From 
Remark \ref{REM:q1}
and \eqref{nonl3c} (with $\eps_0 > 0$), 
we see that 
$q$ and $q_1$ appearing in 
\eqref{it1}, \eqref{it3}, and \eqref{it4}
satisfy $q, q_1 > 2$.
Hence, from 
\eqref{ag7}, 
Lemma \ref{LEM:STR}, and \eqref{Linfty}, 
we have
\begin{align*}
\vec \ub \in 
 \V^{\frac1\al}_{\tau} \H^{s}(\T^d) 
\subset \S(-t) X^s(\tau) 
\end{align*}

\noi
by choosing $\al < \frac 12$ sufficiently close to $\frac 12$.
Therefore, from the uniqueness
of the Ito solution in the class
$\S(-t) X^s(\tau)$, 
we conclude that  
\begin{align*}
\vec \ub \equiv \vec \ub_\ito
\end{align*}

\noi
in 
$\S(-t) X^s(\tau)$, which implies \eqref{ag1}.
\end{proof}

\begin{remark}\label{REM:system3}\rm
In the proof of Proposition \ref{PROP:ito}, 
we only considered the case when \eqref{bad1} does not hold.
When \eqref{bad1} holds, 
we needed to use the decomposition $\vec \ub
= \vec \vb + \vec \wb$ in 
\eqref{exp1} and study the RDE-PDE system \eqref{RR1}
in proving Theorem \ref{THM:2}.
In this case, 
by setting 
$\vec v = \S(t) \vec \vb$
and 
$\vec w = \S(t) \vec \wb$, 
we can use Lemma \ref{LEM:Ito}
to show that 
$(\vec v, \vec w)$
is an Ito solution to \eqref{Ito2}
(and the corresponding equations for $\dt v$ and $\dt w$), 
which can be then compared with the Ito solution 
$(\vec v_\ito, \vec w_\ito)$
to \eqref{Ito2}
(and the corresponding  equations for $\dt v$ and $\dt w$)
constructed in 
Remark \ref{REM:system1}.
After this reduction, 
the required analysis is analogous to that in the proof of Proposition 
\ref{PROP:ito} presented above, 
and thus 
we omit details.

\end{remark}

We now present  a
proof of  Lemma \ref{LEM:Ito}.

\begin{proof}[Proof of Lemma \ref{LEM:Ito}]

We follow the proof of \cite[Proposition~5.1]{FH20}.
There are, however,  several differences:

\smallskip

\begin{itemize}
\item[(i)]
Unlike the scalar case considered in 
\cite[Proposition~5.1]{FH20}, 
our integrator $\vXX$ is operator-valued, 
which depends on the variable of integration, not just endpoints;
see $\vec A_1(\Pi)$ in \eqref{J2x}
and \eqref{J3}.
See also \eqref{J5}.

\smallskip

\item[(ii)]
In  \cite[Proposition~5.1]{FH20}, 
almost sure agreement such as \eqref{ag2}
was shown for fixed $t \ge 0 $, 
where a good set of full probability may depend on $t$.
In order to prove \eqref{ag2}, 
we need to be able to find 
a good set of full probability
which works for all 
 $0 \le t \le \tau$.

\end{itemize}

\smallskip

In the following, we assume $0 \le \tau \le 1$.
Note that, when $s= s_0\ge 0$,  the conditions \eqref{YG0},  \eqref{RG0}, 
and \eqref{RG1} reduce to \eqref{sigma0}.
Hence, from Proposition \ref{PROP:driver2}, we have 
\begin{align*}
(\vXX, \vbbX )  \in C^{\al}_{2,1} \LOP(\H^s(\T^d))
\times 
C^{2\al}_{2,1} \LOP(\H^s(\T^d)), 
\end{align*}

\noi
almost surely, 
satisfying Chen's relation \eqref{Chen2}.
We define $\vXX^\tau$ and $\vbbX^\tau$
by setting
\begin{align}
\vXX^\tau_{t, r} = \vXX_{t\wedge \tau, r \wedge \tau}
\qquad \text{and}\qquad  
\vbbX^\tau_{t, r} = \vbbX_{t\wedge \tau, r \wedge \tau}
\label{agg1}
\end{align}

\noi
for $(t, r) \in \Dl_{2, 1}$.
Then, 
we have 
\begin{align}
\begin{split}
(\vXX^\tau, \vbbX^\tau )
&   \in C^{\al}_{2,1} \LOP(\H^s(\T^d))
\times 
C^{2\al}_{2,1} \LOP(\H^s(\T^d))\\
& \subset 
\V^{\frac 1{\al}}_{2,1} \LOP(\H^s(\T^d))
\times \V^{\frac 1{2\al}}_{2,1} \LOP(\H^s(\T^d)), 
\end{split}
\label{agg2}
\end{align}

\noi
almost surely, 
satisfying Chen's relation \eqref{Chen2}.
We also define $\vec \ub^\tau$ and $(\vec \ub')^\tau$ by setting
\begin{align}
\vec \ub^\tau(t) = \vec \ub(t\wedge \tau)
\qquad \text{and}\qquad  
(\vec \ub')^\tau(t)
=  \vec \ub'(t\wedge \tau)
\notag
\end{align}

\noi
for $0 \le t\le 1$.
Suppose that we have
\begin{align}
(\updl \vec\ub)_{t,r} = \vXX_{t,r}( \vec\ub_r ')+ \vec{R}^{\vXX, \vec\ub}_{t,r}, \quad (t, r) \in \Dl_{2, \tau}
\notag
\end{align}

\noi
for some remainder term $ \vec{R}^{\vXX, \vec\ub}$.
Then, by setting $\vec R^\tau$ by 
\begin{align}
\vec R^\tau_{t, r} =
 \vec R^{\vXX, \vec \ub}_{t\wedge \tau, r \wedge \tau}
\notag
\end{align}

\noi
\noi
for $(t, r) \in \Dl_{2, 1}$, 
we have 
\begin{align}
(\updl \vec\ub^\tau)_{t,r} = \vXX^\tau_{t,r}\big( (\vec\ub ')^\tau_r \big)+ \vec{R}^\tau_{t, r}, 
\quad (t, r) \in \Dl_{2, 1}.
\notag
\end{align}

\noi
Consequently, we have 
$ \big(\vec\ub^\tau, (\vec\ub ')^\tau \big)
\in \CRP^{ \al, \g, s}_{\vXX^\tau, 1}$
with 
$(\vec\ub^\tau)' = (\vec\ub ')^\tau$.
From these definitions with~\eqref{int2d}, 
we see that 
the rough integral 
$\I^{\vXX^\tau , \vbbX^\tau }(\vec \ub^\tau )$
of $\vec \ub^\tau $
with the rough driver  $(\vXX^\tau , \vbbX^\tau)$
exists on the time interval $[0,  1]$, almost surely, 
satisfying 
\begin{align}
\I^{\vXX^\tau , \vbbX^\tau }(\vec \ub^\tau )(t)
= \I^{\vXX, \vbbX}(\vec \ub)(t\wedge \tau)
\label{agg7}
\end{align}

\noi
for $0 \le t \le 1$.
On the other hand, we also have 
\begin{align}
\vec \pPsi (\ind_{[0, \tau]}\cdot  \vec \ub)(t)
= \vec \pPsi (\ind_{[0, \tau]}\cdot  \vec \ub)(t\wedge \tau)
\label{agg8}
\end{align}

\noi
for $0 \le t \le 1$.
Hence, in view of \eqref{agg7} and \eqref{agg8}, 
\eqref{ag2} follows 
once we prove
\begin{align*}
\vec \pPsi (\ind_{[0, \tau]}\cdot  \vec \ub)
=  \I^{\vXX^\tau, \vbbX^\tau}(\vec \ub^\tau)
\quad 
\text{in \ $C([0, 1]; \H^s(\T^d))$}, 
\end{align*}

\noi
almost surely.
In view of the continuity (in time)
of the Ito and rough integrals in \eqref{ag2}, 
it suffices to prove 
\begin{align}
\vec \pPsi (\ind_{[0, \tau]}\cdot  \vec \ub)(t)
=  \I^{\vXX^\tau, \vbbX^\tau}(\vec \ub^\tau)(t)
\label{agg9}
\end{align}

\noi
as elements in $\H^s(\T^d)$, 
almost surely, for each fixed rational time $t \in [0, 1]\cap \Q$.
This handles the point (ii) mentioned above.

In the following, we prove \eqref{ag2}, 
assuming 
\begin{align}
    \big\|\| \vec \ub\|_{C_\tau \H^s_x}\big\|_{L^2(\Om)}
 +\big\|\| \vec \ub'\|_{C_\tau \H^s_x}\big\|_{L^2(\Om)}
    <\infty.
\label{AC1}
\end{align}

\noi
For a general case (namely, if \eqref{AC1} does not hold), 
we can proceed with 
 a standard localization argument.
Given $N \in \N$, 
set 
\begin{align*}
\tau_N = 
\inf\big\{t>0: 
\|\vec \ub^\tau\|_{C_t \H^s} + \|(\vec \ub')^\tau\|_{C_t\H^s} \geq N\big\}.
\end{align*}

\noi
Note that  $\lim_{N\to\infty} (\tau_N \land \tau) = \tau$, almost surely.
Instead of working with the ``localized''
processes $\vec \ub^\tau$, etc.
defined above, 
we then work with  new localized processes:
\begin{align*}
   \vec\ub^{\tau, N}(t):= \vec\ub^\tau(t\wedge\tau_N)
= \vec \ub(t\wedge \tau\wedge \tau_N) 
\end{align*}

\noi
and 
$\vXX^{\tau, N}$, $\vbbX^{\tau, N}$, 
 $(\vec \ub')^{\tau, N}$, and $\vec R^{\tau, N}$
defined in an analogous manner.
Then, the argument presented below yields
\begin{align}
\vec \pPsi (\ind_{[0, \tau\wedge \tau_N]}\cdot  \vec \ub)
=  \I^{\vXX, \vbbX}(\vec \ub^{\tau, N}), 
\label{ag2x}
\end{align}

\noi
almost surely.
Then, using the continuity of 
$\I^{\vXX, \vbbX}(\vec \ub)$ in $\vec \ub$
(Lemma \ref{LEM:intR}\,(iii)), 
we conclude~\eqref{ag2}
from \eqref{ag2x} by sending $N \to \infty$.

\medskip

In the following, we assume \eqref{AC1}.
For notational simplicity, 
we drop
the superscript $\tau$ in 
$\vXX^\tau$, $\vbbX^\tau$, 
$\vec \ub^\tau$, $(\vec \ub')^\tau$, and $\vec R^\tau$.
Fix $t \in [0, 1]\cap \Q$.
Our goal is to prove \eqref{agg9}.
From 
Lemma~\ref{LEM:intR}
with \eqref{agg2}
(see also~\eqref{int2d}), 
 we have 
\begin{align}
\label{J2}
 \I^{\vXX, \vbbX}(\vec \ub)(t)
=
\lim_{|\Pi([0,t])| \to 0}
\sum_{j=0}^{k-1}\Big( \vXX_{t_{j},t_{j+1}} (\vec\ub_{t_{j+1}})
+ \vbbX_{t_{j},t_{j+1}}( \vec\ub_{t_{j+1}}')\Big)
\end{align}

\noi
in $\H^s(\T^d)$, 
almost surely, 
 where
 the limit is taken over a sequence of  partitions
 $\Pi([0,t])$
 of  $[0,t]$ of the form \eqref{part1}
 (with $r = 0$)
 whose mesh size 
  $|\Pi([0,t])| = \max_j |t_j - t_{j+1}|$
 tends to $0$.
  On the other hand, 
 it follows from 
\eqref{psi2} and 
\cite[Proposition 2.13 in Chapter IV]{RY}
that  
 the following limit in probability holds:
\begin{align}
\label{J3}
\begin{split}
& \vec \pPsi (\ind_{[0, \tau]}\cdot  \vec \ub)(t)\\
& \quad=
\lim_{|\Pi([0,t])|\to 0}
\sum_{j=0}^{k-1} \S(-t_{j+1})
\begin{pmatrix}
0 \\
\pi_1[ \S(t_{j+1}) \ind_{[0, \tau]}(t_{j+1})\cdot \vec\ub_{t_{j+1}}]
\end{pmatrix}
\Phi (\updl W^\frac12)_{t_{j}, t_{j+1}}\\
&\quad  =:
\lim_{|\Pi([0,t])| \to 0}
\sum_{j=0}^{k-1}\vXX_{t_{j},t_{j+1}}^\text{disc} (\ind_{[0, \tau]}(t_{j+1})\cdot \vec\ub_{t_{j+1}})
\end{split}
\end{align}

\noi
in $\H^s(\T^d)$, 
where $\vXX^\text{disc}$ denotes the discretized version 
of the driver $\vXX$ in \eqref{introX}.
Thus, by taking a subsequence, 
the convergence \eqref{J3} holds almost surely.
Hence, 
by setting  $\Pi = \Pi([0, t])$, 
from \eqref{J2} and 
\eqref{J3}, we have  
\begin{align}
\label{J2x}
\begin{split}
&  \I^{\vXX, \vbbX}(\vec \ub)(t)
 - \vec \pPsi (\ind_{[0, \tau]}\cdot  \vec \ub)(t)\\
& \quad =
\lim_{|\Pi| \to 0}
\sum_{j=0}^{k-1}
\Big( \vXX_{t_{j},t_{j+1}}
(\vec\ub_{t_{j+1}})
 - \vXX_{t_{j},t_{j+1}}^\text{disc}(\ind_{[0, \tau]}(t_{j+1})\cdot\vec\ub_{t_{j+1}})\Big)\\
& \quad  \quad 
+ \lim_{|\Pi| \to 0}
\sum_{j=0}^{k-1}
 \vbbX_{t_{j},t_{j+1}}( \vec\ub_{t_{j+1}}')\\
& \quad  = :  \lim_{|\Pi| \to 0}\vec A_1(\Pi) +  \lim_{|\Pi| \to 0}\vec A_2(\Pi), 
\end{split}
\end{align}

\noi
almost surely.

 From \eqref{bbX0} 
 with $\vec \ub' = (\ub', \vb')$, 
 we have
\begin{align}
\begin{aligned}
\vec A_2(\Pi)
& =
\sum_{j=0}^{k-1}
\begin{pmatrix}
\bbX_{t_{j},t_{j+1}}^{(1,1)} (\ub_{t_{j+1}}')
+ \bbX_{t_{j},t_{j+1}}^{(1,2)} (\vb_{t_{j+1}}')\rule[-3.5mm]{0pt}{0pt}\\ 
\bbX_{t_{j},t_{j+1}}^{(2,1)}( \ub_{t_{j+1}}')
+\bbX_{t_{j},t_{j+1}}^{(2,2)} (\vb_{t_{j+1}}')
\end{pmatrix}\\
& =:
\begin{pmatrix}
A_2^{(1,1)}(\Pi) + A_2^{(1,2)}(\Pi)
\\
A_2^{(2,1)}(\Pi) + A_2^{(2,2)}(\Pi)
\end{pmatrix}.
\end{aligned}
\label{J16}
\end{align}

\noi
In the following, we first  show the following convergence:
\begin{align}
\lim_{|\Pi|\to 0}
\big\|\|A_2^{(1,1)}(\Pi)\|_{H^s}\big\|_{L^2(\Om)} = 0.
\label{AZ1}
\end{align}

\noi
From 
\eqref{bbX0a} with \eqref{XX0a}, 
we have 
\begin{align*}
\F_x (A_2^{(1,1)} (\Pi) ) (n)
& = \sum_{j=0}^{k-1}
\sum_{n=n_{123}}
\frac{1}{\jb{n} \jb{n_{23}}}
 \ft{\ub'}(t_{j+1}, n_3)\phi_{n_1} \phi_{n_2}\\
& \quad \times  \int_{t_{j+1}}^{t_j}
\int_{t_{j+1}}^{t_1'}
\sin(t_1'\jb{n}) \cos(t_2'\jb{n_3})
\\
& \quad \hphantom{XXXXX} \times
\sin( (t_2' - t_1') \jb{n_{23}})
 dB_{n_{2}}(t_2') dB_{n_1}(t_1').
\end{align*}

\noi
Then, from the Ito isometry, we have 
\begin{align*}
&
\big\| \| A_2^{(1,1)} (\Pi) \|_{H^s} \big\|_{L^2(\Om)}^2
\\
&\quad 
\les
\sum_{j=0}^{k-1} (t_j - t_{j+1})^2 
\sum_{n\in \Z^d }\sum_{n=n_{123}}
\frac{\jb{n}^{2(s-1)}}{ \jb{n_{23}}^2  }
\|\ft{\ub'}(t_{j+1}, n_3) \phi_{n_1}  \phi_{n_2} \|_{L^2(\Om)}^2
\\
&\quad 
\les
M_{s, \s}\cdot 
|\Pi| \cdot 
 \|\Phi\|_{\HS(L^2;H^\s)}^4 
\big\|\| \ub'\|_{C_1 H^s_x}\big\|_{L^2(\Om)}^2 .
\end{align*}

\noi
Here, 
$M_{s, \s}$ is given by 
\begin{align*}
M_{s, \s} = \sup_{n,n_1,n_2,n_3 \in \Z^d} 
\frac{\ind_{n = n_{123}}\cdot \jb{n}^{2(s-1)}}{ \jb{n_1}^{2\s} \jb{n_2}^{2\s} \jb{n_3}^{2s} \jb{n_{23}}^{2}}, 
\end{align*}

\noi
which is finite, since 
 $\s > \max(-s, s-1)$; see also \eqref{bbX2}.
 This proves \eqref{AZ1}.
 
 By repeating similar computations
 for the other components in \eqref{J16}, we obtain
 \begin{align}
\label{J11b}
\lim_{|\Pi|\to 0} \big\| \| \vec{A}_2(\Pi) \|_{\H^s} \big\|_{L^2(\Om)} =0.
\end{align}

Next, we treat the contribution from 
$\vec A_1(\Pi) $ in \eqref{J2x}.
From \eqref{J2x}, \eqref{introX},  and \eqref{J3} with \eqref{agg1}, we have 
\begin{align}
 \pi_1[\vec A_1 (\Pi)]
 =  A_1^{(1,1)}(\Pi ) + A_1^{(1,2)}(\Pi ), 
\label{J5a}
\end{align}

\noi
where
$A_1^{(1,1)}(\Pi )$ and $A_1^{(1,2)}(\Pi )$ are given by 
\begin{align}
\begin{split}
\Ft_x \big( A_1^{(1,1)}(\Pi )  \big)(n)
&
=
\sum_{j=0}^{j_*} \int^{t_j\wedge \tau }_{t_{j+1}\wedge \tau }
\sum_{n=n_{12}} \frac{1}{\jb{n}} 
 F^{(1,1)}_{n,n_2}(r)\Big|_{t'}^{t_{j+1}}
 \ft{\ub}(t_{j+1}, n_2)
\phi_{n_1} dB_{n_1}(t')\\
&\quad  + \int_\tau^{t_{j_*}}
\sum_{n=n_{12}} \frac{1}{\jb{n}} 
 F^{(1,1)}_{n,n_2}(t_{j_*+1})
 \ft{\ub}(t_{j_*+1}, n_2)
\phi_{n_1} dB_{n_1}(t'),
\\
\Ft_x \big( A_1^{(1,2)}(\Pi )  \big)(n)
&
=
\sum_{j=0}^{j_*} \int^{t_j\wedge \tau }_{t_{j+1}\wedge \tau }
\sum_{n=n_{12}} \frac{1}{\jb{n}} 
 F^{(1,2)}_{n,n_2}(r)\Big|_{t'}^{t_{j+1}}
\frac{\ft{\vb}(t_{j+1}, n_2)}{\jb{n_2}}
\phi_{n_1} dB_{n_1}(t')\\
& \quad + \int_\tau^{t_{j_*}}
\sum_{n=n_{12}} \frac{1}{\jb{n}} 
 F^{(1,2)}_{n,n_2}(t_{j_*+1})
\frac{ \ft{\vb}(t_{j_*+1}, n_2)}{\jb{n_2}}
\phi_{n_1} dB_{n_1}(t'),
\end{split}
\label{J5}
\end{align}

\noi
Here,  $F^{(1,1)}$ and $F^{(1,2)}$ are given by
\begin{align}
F^{(1,1)}_{n,n_2}(t)
=
\sin(t\jb{n}) \cos(t\jb{n_2})
\qquad \text{and} \qquad
F^{(1,2)}_{n,n_2}(t)
=
 \sin(t\jb{n}) \sin(t\jb{n_2}), 
\notag
\end{align}

\noi
and $j_*$ is the largest integer such that
$t_{j_*} > \tau$.

From  the mean value theorem, we have 
\begin{align}
\notag 
\begin{split}
|F^{(1,k)}_{n,n_2} (t') - F^{(1,k)}_{n,n_2} (t_{j+1})|
&
\les |t' - t_{j+1} |^{\ta} (\jb{n} + \jb{n_2})^\ta\\
& \le |\Pi|^\ta (\jb{n} + \jb{n_2})^\ta ,
\end{split}
\end{align}

\noi
for any $t' \in [t_{j+1}, t_j]$, $0 \le \ta \le 1$, and $k = 1, 2$.
Then, from the Ito isometry, 
we have 
\begin{align}
\notag
\begin{split}
 \big\| \| A_1^{(1,1)}(\Pi) \|_{H^s}
\big\|_{L^2(\Om)}^2
\les
\wt M_{s, \s}
\cdot  |\Pi|^{2\ta} 
\cdot \| \Phi\|_{\HS(L^2;H^\s)}^2 \big\| \| \ub\|_{C_1 H^s_x} \big\|_{L^2(\Om)}^2, 
\end{split}
\end{align}

\noi
Here, $\wt M_{s, \s}$ is given by 
\begin{align*}
\wt M_{s, \s}
= \sup_{n,n_1,n_2 \in \Z^d } 
\frac{\ind_{n = n_{12}}\cdot \jb{n}^{2(s-1)} (\jb{n_1}+\jb{n_2})^{2\ta}}{ \jb{n_1}^{2\s} \jb{n_2}^{2s} } , 
\end{align*}

\noi
which is finite by taking $\ta > 0$ sufficiently small, 
 $\s > \max(-s, s-1)$;
see also \eqref{XX6a}.
This proves
\begin{align*}
\lim_{|\Pi|\to 0} \big\| \| A_1^{(1,1)}(\Pi) \|_{H^s} \big\|_{L^2(\Om)}= 0.
\end{align*}

\noi
 By repeating similar computations
 for $A_1^{(1,2)}(\Pi )$ in \eqref{J5a} 
 and $  \pi_2[\vec A_1 (\Pi)]$, 
we obtain 
 \begin{align}
\label{J10x}
\lim_{|\Pi|\to 0} \big\| \| \vec{A}_1(\Pi) \|_{\H^s} \big\|_{L^2(\Om)} =0.
\end{align}

Hence, from 
\eqref{J2x}, 
\eqref{J11b},  and 
 \eqref{J10x}, 
 we conclude that \eqref{agg9} holds, 
almost surely, 
 by taking a further subsequence of partitions $\Pi$ of $[0, t]$.
 This concludes the proof of Lemma~\ref{LEM:Ito}.
\end{proof}

\section{Optimizer for the time restriction $p$-variation norm}
\label{SEC:B}

\noi

Given a function $u \in V^p([r,t]; X )$ for some $-\infty < r < t < \infty$, 
 we
show that the extension~$\wt u$ defined in \eqref{ext1}
indeed attains the infimum in \eqref{local} with $I = [r, t]$.

We first note that given any  extension $v$ of $u$, 
we have 
\begin{align}
    \| v \|^p_{V^p(\R;X)}
     \ge \o_{X, p}^{(1)}(u;t, r) + \| u_t \|^p_X.
\label{JX2}
\end{align}

\noi
On the one hand, 
from \eqref{U2}, 
the left-hand side of \eqref{JX2} is
given under the supremum over
all finite partitions $\{t_k\}_{k=0}^K \in \Zc$ of the form 
\eqref{U1}.
On the other hand, 
it follows from \eqref{ctrl0}, \eqref{part1}, and 
the boundary condition $v_\infty = 0$
that the right-hand side of \eqref{JX2} is
given under the supremum
over finite partitions, 
containing $r$, $t$, and $\infty$
but no point in $(-\infty, r)$ and $(t,\infty)$.
Hence, \eqref{JX2} follows.

Let $\wt u$ be the extension in \eqref{ext1}:
\begin{align}
    \wt u = \ind_{(-\infty, r)} \cdot u_r + \ind_{[r,t)} \cdot u + \ind_{[t, \infty)} \cdot u_t.
\label{JX2a}
\end{align}

\noi
From \eqref{U2}, given any $\eps > 0$, 
 there 
exists a partition $\{t_k\}_{k=0}^K$ of $\R$
such that 
\begin{align}
    \| \wt u \|^p_{V^p(\R;X)}
    \le 
    \sum_{k=0}^{K-1} \| \wt u_{t_k} - \wt u_{t_{k+1}} \|^p_X + \eps. 
\label{JX3}
\end{align}

\noi
Noting that adding $\infty$  to the partition increases the sum 
on the right-hand side, 
we assume
$t_0 = \infty$.
Moreover, if $t_K > r$, 
then adding  $r$ to the 
the partition increases the sum 
on the right-hand side.
If $t_K < r$, 
we can drop those $t_k$'s with $t_k < r$ from the partition
and add $r$ to the partition (if $r$ is not already in the partition).
Thus, we assume $t_K = r$ in the following.

 Let $v$ be another extension of $u$.
We consider the following two cases.

\medskip

\noi
$\bul$ \textbf{Case 1}: $\{t_k\}_{k=0}^K\cap [t,\infty)
= \varnothing$.\\
\indent
In this case, we have $t_k\in [r,t]$ for any $1\le k\le K$.
Then, it follows from \eqref{JX3},  the fact that $\wt u =v$ on $[r,t]$, 
and the boundary condition $\wt u_\infty = v_\infty = 0$ 
that 
\begin{align}
\begin{aligned}
     \| \wt u \|^p_{V^p(\R;X)}
    & \le 
\sum_{k=1}^{K-1} \| \wt u_{t_k} - \wt u_{t_{k+1}} \|^p_X +   \| \wt u_{t_1}  \|^p_X + \eps \\
    & = \sum_{k=1}^{K-1} \| v_{t_k} - v_{t_{k+1}} \|^p_X +   \| v_{t_1}  \|^p_X + \eps \\
    & \le \| v \|^p_{V^p(\R;X)} + \eps.
\label{JX4}
\end{aligned}
\end{align}

\medskip

\noi
$\bul$ 
\textbf{Case 2}: 
$\{t_k\}_{k=0}^K\cap [t,\infty)
\neq \varnothing$.
 \\
\indent
Let $M \in \N$ be  the smallest integer such that 
$t_M < t \le t_{M-1}$. 
Then, from 
\eqref{JX2a}, 
\eqref{JX3},
\eqref{ctrl0},   and \eqref{JX2} we have 
\begin{align}
\begin{aligned}
    \| \wt u \|^p_{V^p(\R;X)}
    & \le \sum_{k = M-1}^{K-1} \| \wt u_{t_k} - \wt u_{t_{k+1}} \|^p_X +   \| \wt u_{t}  \|^p_X + \eps \\
    & \le 
    \o_{X, p}^{(1)}(u;t, r) + \| u_t \|^p_X 
    + \eps \\
    & \le \| v \|^p_{V^p(\R;X)} + \eps.
\end{aligned}
\label{JX5}
\end{align}

\noi
Since the choice of $\eps > 0$ was arbitrary, 
we conclude from \eqref{JX4} and \eqref{JX5} that 
\begin{align*}
 \| \wt u \|^p_{V^p(\R;X)}
 \le 
\| v \|^p_{V^p(\R;X)} 
\end{align*}

\noi
for any extension $v$ of $u$.
This shows that 
 the infimum in \eqref{local}
for the $V^p([r,t]; X )$-norm
is attained by $\wt u$ in \eqref{JX2a}.

\begin{remark}\rm
A similar argument shows that the extension
$\ind_{(-\infty, r)} \cdot u_r + \ind_{[r,t]} \cdot u$
of $u$
also attains the infimum in \eqref{local}
for the $V^p([r,t]; X )$-norm.

\end{remark}

\section{Proof of the sewing lemma}
\label{SEC:sew}

In this section, we present a sketch of a proof of the sewing lemma (Lemma \ref{LEM:sew}), 
by adapting the arguments from
 \cite[Proposition~2.3]{GT10}
and \cite[{\it second} argument in the proof of Lemma~4.2]{FH20}
to the current $p$-variation setting, 
with a particular attention 
in showing 
continuity and vanishing on the diagonal
(as required in Definition \ref{DEF:V23}\,(iii)).

\begin{proof}[Proof of Lemma \ref{LEM:sew}]

We first prove 
the claims (i) and (ii). 
Fix $0 < p = \frac 1\al< 1$, and  let $h\in \V^p_{3,T} X \cap \Im \updl \vert_{C_{2,T} X}$. 
From  \eqref{con2}, there exists a control $\o$ such that
\begin{align}\label{SS0}
\| h_{t_1,t_2,t_3} \|_X \le \o(t_1,t_3)^\frac 1p
\end{align}

\noi
for any $(t_1,t_2,t_3) \in \Dl_{3,T}$.
Our goal is to show that there exists  unique  
$M \in \V^p_{2,T} X$ such that $\updl M = h$ and 
\begin{align}
\| M_{t,r} \|_{X} \le C_p \, \o(t,r)^\frac 1p
\label{SS1}
\end{align}

\noi
for any $(t, r) \in \Dl_{2,T}$.

\smallskip 
\noi
$\bul$ {\bf Step 1:} Uniqueness of $M$.
\\
\indent
Given $M, \wt M \in \V^p_{2,T} X$ such that 
$\updl M = \updl \wt M = h$,  
let 
 $Q= M - \wt M$.
Then, it follows from  $\updl Q =0$
and
 the exactness of the cochain complex \eqref{chain1}
 that there exists  $f \in C_{1,T} X$ with $\updl f = Q$. 
Moreover, 
from 
\eqref{ctrl000a}
and  $Q \in \V^p_{2,T} X$, 
we have  $f \in \V^p_{T} X$.
Recalling $0 < p < 1$, 
we conclude from Remark \ref{REM:sew3}
that  $f$ is a constant function, 
 and thus $Q = \updl f = 0$.

\smallskip 
\noi
$\bul$ {\bf Step 2:} Construction of $M \in \V^p_{2,T}X$.
\\
\indent
From the exactness of the cochain complex \eqref{chain1}
with  $h \in \Im \updl \vert_{C_{2,T} X}$, there exists $B \in C_{2,T} X$ such that $\updl B = h$. 
However, there is no guarantee that $B \in \V^p_{2,T} X$.
In the following, we construct a suitable  ``increment'' $M$, based on $B$,
with the desired  regularity.

Given  $0\le r < t \le T$, 
let  $\Pi([r,t])$ be a partition of $[r, t]$
of the form  \eqref{part1}. 
 We  define the approximate increment $M^{\Pi([r, t])}_{t,r}$ by setting
\begin{align}
\label{Mpi}
M_{t,r}^{\Pi([r, t])}  = B_{t,r} - \sum_{j=0}^{k-1} B_{t_{j}, t_{j+1}} . 
\end{align}

\noi
We first prove the following maximal inequality:
\begin{align}
\| M_{t,r}^{\Pi([r, t])}  \|_X 
\le C_p \,\o(t, r)^{\frac1p}, 
\label{max2}
\end{align}

\noi
uniformly in  partitions $\Pi([r, t])$ of $[r, t]$
of the form  \eqref{part1}. 

Given  $k \in \N$, 
let  $\Pi_k = \Pi([r,t])$ be a partition of $[r, t]$
of the form  \eqref{part1}, consisting of $k$ intervals. 
When $ k = 1$, namely, 
 $\Pi_{1}$ consists of one interval $[r, t]$, 
 we have  $M^{\Pi_{1}} = 0$, from which~\eqref{max2} trivially follows.
Fix an integer $k \ge 2$.
Then, it follows from the super-additivity~\eqref{con1} of $\o$ in \eqref{SS0}
that  
\begin{align}\label{SS2}
\o (t_{\l-1}, t_{\l+1}) \le \frac{2 \o(t,r) }{k-1}
\end{align}

\noi
for some  $ \l = 1, \dots, k$.
Consider the partition $ \Pi_{k-1} := \Pi_k\setminus\{t_\l\}$
and  the corresponding approximate increment $M^{\Pi_{k-1}}$.
A direct computation with \eqref{Mpi} and  $\updl B = h$ yields
\begin{align}
M^{\Pi_{k-1}}_{t,r} -  M^{\Pi_k}_{t,r} 
= 
 - h_{t_{\l-1}, t_\l , t_{\l+1}}.
\label{Mdiff}
\end{align}

\noi
Then, from \eqref{Mdiff}, \eqref{SS0}, and \eqref{SS2},
we have  
\begin{align*}
\| M^{\Pi_{k-1}}_{t,r} - M^{\Pi_k}_{t,r} \|_X
&
= \| h_{t_{\l-1}, t_\l , t_{\l+1}} \|_X
\le 
\o(t_{\l-1}, t_{\l+1})^{\frac1p} 
\le \bigg( \frac{2\o(t,r)}{k-1} \bigg)^{\frac1p}. 
\end{align*}

We iterate this process.
Namely, with 
\begin{align*}
\Pi_j  = \big\{r = t_j^{(j)} < \dots < t_1^{(j)} <  t_0^{(j)} = t\big\}, 
\end{align*}

\noi
we set $\Pi_{j-1}  = \Pi_j \setminus \{t_\l^{(j)}\}$
for some  $ \l = 1, \dots, j $
such that 
\begin{align}
\o \big(t_{\l-1}^{(j)}, t_{\l+1}^{(j)}\big) \le \frac{2 \o(t,r) }{j-1}.
\label{SS3}
\end{align}

\noi
This leads to 
 a sequence
of partitions
$ \Pi_k \supset \Pi_{k-1} \supset \dots  \supset \Pi_1 = \{[r, t]\}$.
Noting from \eqref{Mpi} that $M^{\Pi_{1}} = 0$
and applying \eqref{Mdiff}
(for $j$ in place of $k$), \eqref{SS0}, 
and \eqref{SS3}
with $0 < p < 1$, we obtain
\begin{align}
\| M_{t,r}^{\Pi_k}  \|_X 
& = \| M_{t,r}^{\Pi_k} - M_{t,r}^{\Pi_1} \|_X 
\le \sum_{j=1}^{k-1} \| M_{t,r}^{\Pi_{j+1}} - M_{t,r}^{\Pi_{j}} \|_X 
\notag
\\
&
\le  \sum_{j=2}^k \| h_{t^{(j)}_{\l-1}, t^{(j)}_{\l}, t^{(j)}_{ \l+1}} \|_X
\le C_p \, \o(t, r)^{\frac1p}, 
\notag
\end{align}

\noi
which proves the maximal inequality \eqref{max2}.

Given  two  partitions 
$\Pi = \Pi([r, t])$ 
and $\wt \Pi = \wt \Pi([r, t])$
of $[r, t]$
as in \eqref{part1}, 
let 
$M_{t,r}^{\Pi}$
and $M_{t,r}^{\wt \Pi}$  
be as in \eqref{Mpi}.
Without loss of generality, 
we assume that 
$\Pi$ is a refinement of 
 $\wt  \Pi$ in the following computation.
Then, from \eqref{max2}, 
the super-additivity~\eqref{con1},  
and the uniform continuity on $\Dl_{2, T}$ of $\o$ 
with $\o_{u, u} = 0$
and $0 < p < 1$, 
 we have 
\begin{align}
\begin{split}
\| M^{\Pi}_{t,r} - M^{\wt \Pi}_{t,r} \|_X
&
=  \bigg\|\sum_{[u,v] \in \wt \Pi}  \bigg( B_{v,u} - \sum_{\substack{[u',v'] \in  \Pi \\ [u',v'] \subset [u,v]}} B_{v',u'} \bigg) \bigg\|_X
\\
&
\le
 \sum_{[u,v] \in \wt \Pi} \|M_{v,u}^{\Pi \cap [u,v]}\|_X
\les 
 \sum_{[u,v] \in \wt \Pi} \o(v, u)^{\frac1p}\\ 
& \le
\o(t, r)
\cdot 
\max_{[u,v] \in \wt \Pi}\o(v, u)^{\frac1p-1}
\too 0, 
\end{split}
\label{max3}
\end{align}

\noi
as the mesh size $|\wt \Pi|$ tends to $0$.
Here, the notation 
$[u,v] \in \wt \Pi$ means that $u, v  \in \wt \Pi$
and that they are consecutive elements in the partition 
$\wt \Pi$, while 
the notation $[u',v'] \subset [u,v]$ denotes the usual inclusion of intervals.
This allows us to define
\begin{align}
M_{t, r} = \lim_{n \to \infty} M^{\Pi_n}_{t, r}
\label{max4}
\end{align}

\noi
for  a given sequence $\{\Pi_n\}_{n \in \N}$
of partitions of $[r, t]$
whose mesh size tends to $0$ as $n\to \infty$.
Note from \eqref{max3} that 
 the limit in \eqref{max4} is independent 
of the choice of $\{\Pi_n\}_{n \in \N}$.
Moreover, it follows from \eqref{max2} that the limit satisfies
\eqref{SS1}
for any $(t, r) \in \Dl_{2,T}$.
Together with the 
super-additivity of $\o$, 
this  in particular implies
\begin{align*}
\o_{X, p}^{(2)}(M;T, 0) < \infty, 
\end{align*}

\noi
where  $\o_{X, p}^{(2)}$ is as in \eqref{ctrl1}.

\smallskip 
\noi
$\bul$ {\bf Step 3:}
 $\updl M = h$.
\\
\indent
Fix $(t_1, t_2, t_3) \in \Dl_{3, T}$.
Let 
$\big\{\Pi_n([t_3,t_2])\big\}_{n\in\N}$ and 
 $\big\{\Pi_n([t_2,t_1])\big\}_{n\in\N}$
 be sequences of partitions of $[t_3, t_2]$ and $[t_2, t_1]$, respectively, 
 whose mesh sizes tend to $0$ as $n\to \infty$.
Then, 
 $\wt \Pi_n = 
\Pi_n([t_3,t_2])\cup \Pi_n([t_2,t_1])$
is a partition of $[t_3, t_1]$
whose mesh size tends to $0$ as $n\to \infty$.
From~\eqref{Mpi} with $\updl B = h$, we have 
\begin{align*}
& 
M_{t_1,t_3}^{\wt \Pi_n} 
- M_{t_1,t_2}^{\Pi_n([t_2,t_1])}
- M_{t_2,t_3}^{\Pi_n([t_3,t_2])}
= (\updl B)_{t_1, t_2, t_3}
= h_{t_1, t_2, t_3}
\end{align*}

\noi
for any $n \in \N$.
Hence, by taking $n \to \infty$, we obtain $\updl M = h$.

\smallskip 
\noi
$\bul$ {\bf Step 4:}
Continuity and vanishing on the diagonal of $M$.
\\
\indent
From \eqref{SS1}
(with $r = t$)
and $\o(t, t) = 0$, 
we see that $M_{t, t} = 0$ for any $0 \le t \le T$.
Moreover, we have 
\begin{align}
\lim_{t \to r+}
M_{t, r} = 
\lim_{r \to t-}
M_{t, r} = 0. 
\label{SS5}
\end{align}


Let $0 \le r < t \le T$.
In the following, we only show continuity of $M_{t, r}$
in $t$.
From $\updl M= h$, 
 \eqref{SS5}, 
and \eqref{vani1},  
we have 
\begin{align*}
\lim_{t_0 \to 0+}\|M_{t+t_0, r}
-  M_{t, r}\|_X
& \le \lim_{t_0 \to 0+} \| M_{t+t_0, t}\|_X + 
\lim_{t_0 \to 0+}\|h_{t+t_0, t, r}\|_X\\
& = \|h_{t,  t, r}\|_X = 0. 
\end{align*}

\noi
A similar computation yields
$\lim_{t_0 \to 0-}\|M_{t+t_0, r}- M_{t, r}\|_X = 0$, 
from which we conclude continuity of $M_{t, r}$ in $t$.

\medskip

Next, we briefly go over (iii).
From (i), we have 
\begin{align*}
\updl (\Id - \Ld \updl) g 
= \updl g - \updl \Ld (\updl g) = 0.
\end{align*}

\noi
Then, from the exactness of the cochain complex \eqref{chain1}
 there exists a unique $f \in C_{T} X$ (up to an additive constant) such that
\begin{align}
\notag
\updl f = (\Id - \Ld \updl ) g . 
\end{align}

\noi
Lastly, we note that 
the identity \eqref{sew2} 
follows from 
\eqref{con3} in Remark \ref{REM:sew3}.
\end{proof}

\begin{remark}\label{REM:GT}\rm

We note that 
\cite[Step 3 in the proof of Proposition 2.3]{GT10}
and the proof of 
\cite[Lemma 2.2]{DGHT19}
only established 
(appropriate versions of) the maximal inequality \eqref{max2}.
While the convergence \eqref{max4}
easily follows from 
the maximal inequality \eqref{max2}, 
we decided to include the detail
for readers' convenience.

\end{remark}

\section{On the Burkholder-Davis-Gundy inequality}
\label{SEC:D}

In this section, we recall  the Burkholder-Davis-Gundy inequality
in a Banach space setting, 
which we used in the proof of Lemma \ref{LEM:stoconv1}.
For this purpose,  we first go over the basic definition
of $\g$-radonifying operators. 
These operators appear as  natural extensions of Hilbert-Schmidt operators
 to a Banach space setting. 
 A practical example is the theory of stochastic integrations 
in the Banach space setting,  where $\g$-radonifying operators suitably generalize
  the role played by Hilbert-Schmidt operators in the more ordinary Hilbert space setting 
  \cite{BS, VN3}. 
  See 
   \cite[Chapter 9]{HVNVW2} for a detailed discussion on $\g$-radonifying operators.

Let $H$ be a separable Hilbert space and 
$B$ be a Banach space.
An operator $S\in \L(H;B)$ is said to be of finite rank if it can be represented as 
\[ S(\cdot)=\sum_{n=1}^{N}\jb{\,\cdot\,, e_{n}}x_{n},\]

\noi
where $\jb{\,\cdot\,, \,\cdot\,}$ denotes the inner product in $H$, $\{e_{n}\}_{n=1}^{N}$ is orthonormal in $H$, and 
$\{ x_{n}\}_{n=1}^{N}\subset B$. 

\begin{definition} \rm
We define the space $\g(H;B)$ as the closure of the set of finite rank operators in $\L(H;B)$ under the norm: 
\begin{equation}
\| S\|_{\g(H;B)} =\sup_{N \in \N} \bigg( \E \bigg[ \Big\| \sum_{n=1}^{N} g_{n}
S(e_n) \Big\|_{B}^{2} \bigg] \bigg)^{\frac{1}{2}}, 
\label{g1}
\end{equation}

\noi
 where $\{e_{n}\}_{n\in \N}$ is an orthonormal basis for $H$ and 
 $\{g_{n}\}_{n\in \N}$ is  a sequence of independent standard real-valued Gaussian random variables. 
 An operator  $S\in \g(H;B)$ is called a $\g$-radonifying operator.
\end{definition}
 
The definition \eqref{g1} is  independent of a choice of orthonormal 
basis $\{e_{n}\}_{n\in \N}$ for $H$. 
In the following, it is understood that the results are independent of this choice. 
Note that  $(\g(H;B),\|\cdot\|_{\g(H;B)})$ is a Banach space and is separable whenever $B$ is separable.

\begin{remark}\rm 
The $L^{2}(\Om)$-norm appearing in \eqref{g1} can be replaced with $L^{p}(\Om)$ for any $1\leq p<\infty$. 
Strictly speaking,  this creates a family of spaces $\g_{p}(H;B)$, however,  
by the Kahane-Khintchine inequality
(Theorem 6.2.6 in \cite{HVNVW2}),  all these norms are equivalent. 
Hence, it suffices to consider the most natural choice  $p=2$ and set $\g(H;B) : = \g_{2}(H;B)$.
\end{remark}

\begin{lemma}[Theorem 9.1.17 in \cite{HVNVW2}] 
An operator $S\in \L(H;B)$ is $\g$-radonifying if and only if
  the sum $\sum_{n\in \N}g_{n}S(e_{n})$ converges in $L^{p}(\Om; B)$
  for some $1\leq p<\infty$  and  some orthonormal basis $\{ e_{n}\}_{n\in \N}$ for $H$.
If $S\in \g(H;B)$, 
then this sum also converges almost surely and defines a $B$-valued Gaussian random variable with covariance operator $SS^{*}$. Furthermore,  we have 
\begin{align}
\|S\|_{\g(H;B)}=\bigg(\E \bigg[\Big\| \sum_{n\in \N}g_{n}S(e_{n}) \Big\|^{2}_{B} \bigg] \bigg)^{\frac{1}{2}}
\label{g2}
\end{align}

\noi
for any orthonormal basis $\{ e_{n}\}_{n\in \N}$ for $H$.
\end{lemma}

Recall that $S\in \L(H;B)$ is Hilbert-Schmidt from $H$ to a Hilbert space $B$ if 
\begin{align*}
\|S\|_{\textup{HS}(H;B)}^{2}
:=\sum_{n\in \N}\|S(e_{n})\|_{B}^{2}<\infty
\end{align*}

\noi
for some (and hence any) orthonormal basis $\{ e_{n}\}_{n\in \N}$ for $H$. 
It is then clear from \eqref{g2} that  when $B$ is a Hilbert space, 
then we have  $\g(H;B)=\text{HS}(H;B)$ and $\|S\|_{\g(H;B)}=\|S\|_{\text{HS}(H;B)}.$
In particular,  when $B$ is a Hilbert space, 
we can express the $\g(H;B)$-norm  {\it without} the use of probability. 
When $B$ is a mixed Lebesgue space, 
we can also characterize the $\g(H;B)$-norm without the use of probability.
See  \cite[Theorem 9.3.6]{HVNVW2}.

\begin{lemma} \label{LEM:G1}

Let $1 \le q, r < \infty$ and  $T  > 0$.
Then, an operator $S$ is $\g$-radonifying from $L^2(\T^d)$
to $L^q([0, T]; L^r(\T^d))$
if and only if $\big(\sum_{n \in \Z^d} |S(e_n)|^2\big)^\frac 12$
belongs to $L^q([0, T]; L^r(\T^d))$, 
where 
$e_n(x) = e^{in\cdot x}$.
In this case, we have
\begin{align}
\| S\|_{\g(L^2; L^q_T L^r_x)}
\sim \big\|
\|S(e_n)\|_{\l^2_n}
\big\|_{L^q_T L^r_x}.
\label{g3}
\end{align}

\end{lemma}

Here, we used the fact that 
 $L^q([0, T]; L^r(\T^d))$ has cotype $q\vee r\vee 2$
 which is finite for $q, r < \infty$;
 see  \cite[Proposition 7.14 with Example 7.12]{HVNVW2}.
We note that, for $2 \le q, r < \infty$, 
the equivalence~\eqref{g3} also  follows from 
\eqref{g2}, Minkowski's integral inequality, 
and 
the Kahane-Khintchine inequality
(Theorem 6.2.6 in \cite{HVNVW2}).

\medskip

Next, we recall the definition of M-type $p$ Banach spaces.
See \cite[Definition 10.41]{Pisier2}
and \cite[Definition 3.5.23]{HVNVW1}.

\begin{definition}\rm 
Let $ 1\le p \le 2$.
We say that a Banach space $B$
has martingale type $p$ (M-type $p$,  for short)
if there exists $C_p > 0$ such that 
\[ \E \big[\|f_N\|_B^p\big]
\le C_p \sum_{n = 1}^N \E \big[\|f_n - f_{n-1}\|_B^p\big]\]

\noi
for any $B$-valued $L^p(\Om)$-martingale $\{f_n\}_{n = 0}^N$ (with $f_{0} = 0$).

\end{definition}

Recall from 
\cite[Proposition 3.5.6 and line 15 on p.\,239]{HVNVW1}
that 
 every Hilbert space (in particular, $\R$) has M-type 2.
Moreover, 
given 
a measure space $A$, 
if $B$ has M-type $p$ for some $p \in [1, 2]$, 
then 
$L^q(A;B)$, $1 < q < \infty$,  has M-type $p\wedge q$;
see \cite[Proposition 3.5.30]{HVNVW1}.
In particular, we have the following lemma.

\begin{lemma}\label{LEM:G2}
Let $2 \le q , r< \infty$ and $T  > 0$.
Then, 
 $L^q([0, T]; L^r(\T^d;\R))$
 has M-type $2$.
\end{lemma}

Finally, we state a version  of the Burkholder-Davis-Gundy inequality
for M-type 2 separable Banach spaces; see \cite{Pisier1,BS,  Ondre, Seidler}.
See also \cite{Br} for the same inequality with a slightly stronger assumption
(for type 2 UMD Banach spaces).

\begin{lemma}\label{LEM:BDG}
Let $2\le  p < \infty$ and $B$ be an M-type $2$ Banach space.
Then, there exists $ C = C(B) > 0$ such that 
\begin{align*}
\E\Bigg[\sup_{0 < t < \tau} \bigg\|\int_0^t \G (t') dW(t') \bigg\|_{B}^p\Bigg]
\leq Cp^\frac p2 \, \E \Bigg[\bigg(\int_0^\tau \| \G(t)\|_{\g(H; B)}^2 dt \bigg)^\frac p2 \Bigg]
\end{align*}

\noi
for any  stopping time $\tau$
and 
$\g(H;B)$-valued progressively measurable process $\Phi$, 
where $W$ is an $H$-cylindrical Wiener process.

\end{lemma}

For our application in the proof of Lemma \ref{LEM:stoconv1}, we take $H = L^2(\T^d)$ 
and thus $W$ is an $L^2$-cylindrical Wiener process $W^\frac 12$ in \eqref{W0}.

\begin{ackno}\rm
A.C.~and~T.O.~were supported by the European Research Council (grant no.~864138 ``SingStochDispDyn").
A.C.~also received support from CNRS-INSMI through a grant ``PEPS Jeunes chercheurs et jeunes chercheuses 2025''
and from 
the European Research Council (grant no.~101219486 ``CritPDEsRand").
T.O.~was also supported by the EPSRC Mathematical Sciences Small Grant  (grant no.~EP/Y033507/1)
and
acknowledges support from  
the NSFC (grant no.~W2531005).

\end{ackno}

\end{document}